\documentclass{crmm-l}

\usepackage{amssymb,amscd,enumerate}

\usepackage{tikz}
\usetikzlibrary{calc,arrows}
\usepackage{diagbox}

\usepackage{chngcntr}
\counterwithout{figure}{chapter}

\usepackage[cmtip,all]{xy}

\usepackage[bookmarks, colorlinks=true, pdftitle={Elliptic Boundary Value Problems with Fractional Regularity Data}, pdfauthor={Alex Amenta and Pascal Auscher}]{hyperref}

\newtheorem{thm}{Theorem}[chapter]
\newtheorem{lem}[thm]{Lemma}
\newtheorem{prop}[thm]{Proposition}
\newtheorem{cor}[thm]{Corollary}

\theoremstyle{definition}
\newtheorem{dfn}[thm]{Definition}

\theoremstyle{remark}
\newtheorem{rmk}[thm]{Remark}

\numberwithin{section}{chapter}
\numberwithin{equation}{chapter}

\newcommand{\mc}{\mathcal}
\newcommand{\mb}{\mathbf}
\newcommand{\nm}[1]{\| #1 \|} %standard sized norm - sometimes too small
\newcommand{\map}[3]{#1 \colon #2 \rightarrow #3}

\newcommand{\sm}{\setminus}
\newcommand{\td}{\tilde}
\newcommand{\wtd}{\widetilde}

\newcommand{\BMO}{\operatorname{BMO}}
\newcommand{\CR}{\operatorname{CR}}
\newcommand{\curl}{\operatorname{curl}}
\newcommand{\dist}{\operatorname{dist}}
\newcommand{\dv}{\operatorname{div}}
\newcommand{\id}{\operatorname{id}}
\newcommand{\Lip}{\operatorname{Lip}}
\newcommand{\loc}{\operatorname{loc}}
\newcommand{\re}{\operatorname{Re}}
\newcommand{\sgn}{\operatorname{sgn}}
\newcommand{\sgp}{\operatorname{sgp}}
\newcommand{\supp}{\operatorname{supp}}
\newcommand{\Tr}{\operatorname{Tr}}
\newcommand{\VMO}{\operatorname{VMO}}
\newcommand{\wellp}{\operatorname{WP}}

\newcommand{\ga}{\alpha}
\newcommand{\gb}{\beta}
\newcommand{\gc}{\chi}
\newcommand{\gd}{\delta}

\newcommand{\gf}{\varphi}
\newcommand{\gfv}{\phi}
\newcommand{\gh}{\eta}
\newcommand{\gi}{\iota}
\newcommand{\gk}{\kappa}
\newcommand{\gl}{\lambda}
\newcommand{\gm}{\mu}
\newcommand{\gn}{\nu}
\newcommand{\go}{\omega}
\newcommand{\gp}{\pi}
\newcommand{\gq}{\theta}
\newcommand{\gr}{\rho}
\newcommand{\gs}{\sigma}
\newcommand{\gt}{\tau}
\newcommand{\gx}{\xi}
\newcommand{\gy}{\psi}
\newcommand{\gz}{\zeta}
\newcommand{\gD}{\Delta}
\newcommand{\gF}{\Phi}
\newcommand{\gG}{\Gamma}
\newcommand{\tgG}{\widetilde\Gamma}
\newcommand{\gL}{\Lambda}
\newcommand{\gO}{\Omega}

\newcommand{\gY}{\Psi}

\newcommand{\bbB}{\mathbb{B}}
\newcommand{\bbC}{\mathbb{C}}
\newcommand{\bbD}{\mathbb{D}}
\newcommand{\bbH}{\mathbb{H}}
\newcommand{\bbN}{\mathbb{N}}
\newcommand{\bbP}{\mathbb{P}}
\newcommand{\bbQ}{\mathbb{Q}}
\newcommand{\bbR}{\mathbb{R}}
\newcommand{\bbS}{\mathbb{S}}
\newcommand{\bbX}{\mathbb{X}}
\newcommand{\bbZ}{\mathbb{Z}}

\newcommand{\bbVMO}{\mathbb{VMO}}

\def\barint_#1^#2{\mathchoice
            {\mathop{\vrule width 4.5pt
height 3 pt depth -2.5pt
                    \kern -7.4pt
\int_{#1}^{#2} \kern 0pt}}%
            {\mathop{\vrule width 5pt height
3 pt depth -2.6pt
                    \kern -6.5pt
\intop_{#1}^{#2} \kern -4pt}\nolimits_{#1}}%
            {\mathop{\vrule width 5pt height
3 pt depth -2.6pt
                    \kern -6pt
\intop_{#1}^{#2} \kern -4pt}\nolimits_{#1}}%
            {\mathop{\vrule width 5pt height
3 pt depth -2.6pt
          \kern -6pt \intop \kern -4pt}\nolimits_{#1}}}
          
\def\bariint_#1^#2{\mathchoice
            {\mathop{\vrule width 8pt
height 3 pt depth -2.5pt
                    \kern -10.5pt
\intop \kern -9pt\int_{#1}^{#2} \kern 0pt}}%
            {\mathop{\vrule width 9pt height
3 pt depth -2.6pt
                    \kern -10.5pt
\intop \kern -10pt\int_{#1}^{#2} \kern -4pt}}%
            {\mathop{\vrule width 9pt height
3 pt depth -2.6pt
                    \kern -10pt
\intop \kern -10pt\int_{#1}^{#2} \kern -4pt}}%
            {\mathop{\vrule width 9pt height
3 pt depth -2.6pt
          \kern -10pt \int_{#1}^{#2} \kern -10pt\intop \kern -4pt}
      }}
          
\newcommand{\dint}{\int \kern -9pt\int}

\makeindex

\allowdisplaybreaks

\begin{document}

\frontmatter

\title[Elliptic BVPs with Fractional Regularity Data]{Elliptic Boundary Value Problems with Fractional Regularity Data: The First Order Approach}

%    author one information
\author{Alex Amenta}
\address{Delft Institute of Applied Mathematics, Delft University of Technology, P.O. Box 5031, 2628 CD Delft, The Netherlands}
%\curraddr{}
\email{amenta@fastmail.fm}

%    author two information
\author{Pascal Auscher}
\address{Laboratoire de Math\'{e}matiques d'Orsay, Univ.\ Paris-Sud, CNRS, Universit\'{e} Paris-Saclay, 91405 Orsay, France}
\email{pascal.auscher@math.u-psud.fr}

%\curraddr{}
%\email{}

%    If any version of the Mathematics Subject Classification other
%    than the 2010 edition appears, then you have an old version
%    of the AMS-LaTeX collection and need to upgrade.  Download from
%    http://www.ams.org/tex/amslatex.html .


\begin{abstract}
In this monograph our main goal is to study the well-posedness of boundary value problems of Dirichlet and Neumann type for elliptic systems $\operatorname{div} A \nabla u = 0$  on the upper half-space with coefficients  independent of the transversal variable, and with boundary data in fractional Hardy--Sobolev and Besov spaces.
Our approach uses minimal assumptions on the coefficients $A$, and in particular does not require De Giorgi--Nash--Moser estimates.
Our results are completely new for the Hardy--Sobolev case, and in the Besov case they extend corresponding results recently obtained by Barton and Mayboroda.

First we develop a theory of Besov--Hardy--Sobolev spaces adapted to operators which are bisectorial on $L^2$, with bounded $H^\infty$ functional calculus on their ranges, and which satisfy $L^2$ off-diagonal estimates.
In particular, this theory applies to perturbed Dirac operators $DB$.
We then prove that for a nontrivial range of exponents (the \emph{identification region}) the Besov--Hardy--Sobolev spaces adapted to $DB$ are equal to those adapted to the unperturbed Dirac operator $D$ (which correspond to classical Besov--Hardy--Sobolev spaces).

Our main result is the classification of solutions to the elliptic system $\operatorname{div} A \nabla u = 0$ within a certain \emph{classification region} of exponents, defined in terms of the identification regions for certain perturbed Dirac operators associated with $A$.
More precisely, we show that if the conormal gradient of a solution belongs to a weighted tent space (or one of their real interpolants) with exponent in the classification region, and in addition vanishes at infinity in a certain sense, then it has a trace in a Hardy--Sobolev (or Besov) space, and can be represented as a semigroup evolution of this trace in the transversal direction.
As a particular corollary, any such solution can be represented in terms of an abstract layer potential operator. 
	
Within the classification region, we show that well-posedness of one of the boundary value problems under consideration is equivalent to a certain boundary projection being an isomorphism.
We derive various consequences of this characterisation, including interpolation, extrapolation, and duality properties of well-posedness, as well as some results on stability of well-posedness under coefficient perturbation.
These consequences are illustrated in various situations, including in particular that of the Regularity problem for real equations.

\end{abstract}

\maketitle

%    Dedication.  If the dedication is longer than a line or two,
%    remove the centering instructions and the line break.
%\cleardoublepage
%\thispagestyle{empty}
%    If this book uses the documentclass stml-l or mmono-s, change
%    13.5pc to 10.5pc.
%\vspace*{13.5pc}
%\begin{center}
%  Dedication text (use \\[2pt] for line break if necessary)
%\end{center}
%\cleardoublepage

%    Change page number to 7 if a dedication is present.
\setcounter{page}{5}

\tableofcontents

%    Include unnumbered chapters (preface, acknowledgments, etc.) here.
%\include{}

\mainmatter
%    Include main chapters here.
\chapter{Introduction}\label{chap:BHSintro}

\section{Introduction and context}

In this monograph we investigate well-posedness of boundary value problems associated with divergence-form elliptic equations\index{elliptic equation}
\begin{equation}\label{eqn:mainPDE}
	L_A u := \dv A \nabla u = 0,
\end{equation}
where the unknown $u$ is a $\bbC^m$-valued function on the upper half-space
\begin{equation*}
	\bbR^{1+n}_+ := \{(t,x) \in \bbR^{1+n} : t > 0\}.
\end{equation*}
We work in ambient dimension $1+n \geq 2$, with $m \geq 1$.
The equation \eqref{eqn:mainPDE} may be considered as system of $m$ scalar equations.
The value of $m$ is quite irrelevant to everything that we do, so for simplicity the reader may assume $m=1$ throughout.

We use the so-called `first order approach', following previous work by the second author with Hofmann, McIntosh, Mourgoglou, Ros\'en (Axelsson), and Stahlhut \cite{AAH08,AAM10,AAM10.2,AA11,AR12,AMM13,AM14,AM15,AS16}.
In these articles, boundary value problems are considered with (boundary) data of regularity $0$ (i.e. data in Lebesgue spaces $L^p$ and Hardy spaces $H^p$) or of regularity $-1$ (i.e. in Sobolev spaces $W^p_{-1}$, H\"older spaces $\dot{\gL}_{-1}^{\ga}$, or in $\dot{BMO}_{-1}$).
Here we address problems with \emph{fractional regularity data}: that is, data in Hardy--Sobolev spaces $\dot{H}^p_\gq$ or Besov spaces $\dot{B}^{p,p}_\gq$ with $\gq \in (-1,0)$.    % \footnote{We do not consider more general Triebel--Lizorkin spaces $\dot{F}^{p,q}_\gq$ with $q \neq 2$ (note that $\dot{H}^p_\gq = \dot{F}^{p,2}_\gq$) or Besov spaces $\dot{B}^{p,q}_\gq$ with $q \neq p$.}
These were considered in the previous articles only when $p=2$.

Boundary value problems for the equation $L_A u = 0$ with data in Besov spaces have recently been studied by Barton and Mayboroda \cite{BM16}.
Under the additional assumptions that $m=1$ and that solutions to \eqref{eqn:mainPDE} satisfy De Giorgi--Nash--Moser estimates (see \eqref{eqn:DGNM}), they establish various well-posedness results via the method of layer potentials.
One novelty of their approach is that they can also consider inhomogeneous problems $L_A u = f$, which we do not address.
For homogeneous problems, however, our results are more general.
The first order approach requires neither the existence of fundamental solutions (which is implied by De Giorgi--Nash--Moser estimates when $m=1$, and which is crucial in setting up the method of layer potentials) nor the validity of trace theorems (which hold for Besov spaces, but not for Hardy--Sobolev spaces).

Our approach is based on an abstract framework of \emph{adapted Besov--Hardy--Sobolev spaces}.
For applications to boundary value problems with data of regularity $0$ and $-1$, a theory of adapted Hardy spaces had been sufficiently developed by the second author and Stahlhut \cite{AS16}.
We extend this theory by exploiting properties of weighted tent spaces $T^p_\gq$ and their real interpolants, the \emph{$Z$-spaces} $Z^p_\gq$.

\subsection{The elliptic equation}

Consider again the elliptic equation \eqref{eqn:mainPDE}.
The gradient operator $\nabla$ maps $\bbC^m$-valued functions in $1+n$ variables $f$ to $\bbC^{m(1+n)}$-valued functions $\nabla f$ by writing $f = (f^j)_{j=1}^m$ as an $m$-tuple of $\bbC$-valued functions and acting as the usual gradient operator on each component $f^j$.
The divergence operator $\dv$ is similarly defined in terms of the usual divergence operator, sending $\bbC^{m(1+n)}$-valued functions to $\bbC^m$-valued functions.
These differential operators are interpreted in the distributional sense.
Vectors $v \in \bbC^{m(1+n)}$ are split into \emph{transversal} and \emph{tangential} parts $v = (v_\perp, v_\parallel)$ according to the decomposition
\begin{equation}\label{eqn:ttsplit}
	\bbC^{m(1+n)} = \bbC^m \oplus \bbC^{mn},
\end{equation}
and likewise $\bbC^{m(1+n)}$-valued functions $f$ can be split into transversal and tangential parts $f = (f_\perp, f_\parallel)$, valued in $\bbC^m$ and $\bbC^{mn}$ respectively.
We write $\nabla_\parallel$ and $\dv_\parallel$ for the corresponding tangential restrictions of $\nabla$ and $\dv$.

Throughout the monograph we assume that the coefficient matrix
\begin{equation*}
	A \in L^\infty(\bbR^{n+1}_+ : \mc{L}(\bbC^{m(1+n)}))
\end{equation*}
is bounded, measurable, and \emph{$t$-independent}, meaning that $A(t,x) = A(x)$ for almost every $(t,x) \in \bbR^{1+n}_+$.
Thus we may consider $A$ as an element of $L^\infty(\bbR^n : \mc{L}(\bbC^{m(1+n)}))$.
We also assume that $A$ is \emph{strictly accretive on curl-free vector fields}, meaning that there exists $\gk > 0$ such that
\begin{equation}\label{eqn:accretivity}
	\operatorname{Re} \int_{\bbR^n} \left( A(x) f(x), f(x) \right) \, dx \geq \gk \nm{f}_2^2
\end{equation}
for all $f \in L^2(\bbR^n : \bbC^{m(1+n)})$ such that $\curl_\parallel(f_\parallel) = 0$.
The round bracket in the integrand above is the usual Hermitean inner product on $\bbC^{m(1+n)}$.
By $\curl_\parallel(f_\parallel) = 0$ we mean that
\begin{equation*}
  \partial_j f_k = \partial_k f_j \qquad  (1 \leq j,k \leq n, \quad j \neq k),
\end{equation*}
where $\partial_j$ is the distributional partial derivative in the $j$-th coordinate direction of $\bbR^n$, acting componentwise on $\bbC^m$-valued functions.
The strict accretivity condition \eqref{eqn:accretivity} is weaker than the usual notion of pointwise strict accretivity
\begin{equation*}
	\operatorname{Re} (A(x)v,v) \geq \gk |v|^2 \qquad (v \in \bbC^{m(1+n)}, \quad x \in \bbR^n)
\end{equation*}
unless $m=1$ and $A$ is real, in which case these two notions are equivalent \cite[\textsection 2]{AAM10}.

As we have not assumed any regularity of $A$, we must consider weak solutions to \eqref{eqn:mainPDE}: we say that a function $u \in W^2_{1,\loc}(\bbR^n : \bbC^m)$ solves \eqref{eqn:mainPDE} if for all $\gf \in C_0^\infty(\bbR^{1+n}_+ : \bbC^m)$ we have
\begin{equation*}
	\dint_{\bbR^{1+n}_+} ( A(x) \nabla u(t,x), \nabla \gf(t,x)) \, dx \, dt = 0.
\end{equation*}

\subsection{Formulation of boundary value problems}\label{ssec:fbvp}\index{boundary value problem}

Various boundary value problems for the equation $L_A u = 0$ have been studied.
Here we reformulate some of these problems, and introduce some new ones.
In Definition \ref{dfn:bvps} we concisely synthesise these problems into two categories, each parametrised by a set of exponents and two (related) families of spaces of boundary data.

First, for $1 < p < \infty$, following the groundbreaking work of Dahlberg on the Laplace equation on Lipschitz domains \cite{D77}, we formulate the \emph{$L^p$-Dirichlet problem} for $L_A$, denoted by $(D_H)_{0,A}^p$:\index{boundary value problem!Dirichlet}
\begin{equation*}
  (D_H)_{0,A}^p :
  \left\{
    \begin{array}{l}
      L_A u = 0 \quad \text{in $\bbR_+^{1+n}$}, \\
      \nm{N_* u}_{L^p} \lesssim \nm{f}_{L^p}, \\
      \lim_{t \to 0} u(t,\cdot) = f \in L^p(\bbR^n: \bbC^m).
    \end{array} \right.
\end{equation*}
This should be read: 
\begin{quote}
for all $f \in L^p(\bbR^n : \bbC^m)$, \\
there exists $u \in W_{1,\loc}^2(\bbR^{1+n}_+ : \bbC^m)$ solving $L_A u = 0$, \\
with $\nm{N_* u}_p \lesssim \nm{f}_p$ (the \emph{interior estimate}), \\
and $u(t,\cdot) \to f$ in $L^p$ as $t \to 0$ (the \emph{boundary condition}). 
\end{quote}
Here $N_*$ is the non-tangential maximal function
\begin{equation*}
	N_* u (x) := \sup_{\substack{(t,y) \in \bbR^{1+n}_+ \\ y \in B(x,t)}} |u(t,y)|.
\end{equation*}
We say that the problem $(D_H)_{0,A}^p$ is \emph{well-posed} if for all $f \in L^p(\bbR^n : \bbC^m)$ there exists a unique $u$ satisfying these conditions.\footnote{Use of the non-tangential maximal function $N_*$ is appropriate for $(D_H)_{0,A}^p$ provided that solutions to $L_A u = 0$ have pointwise values.
Otherwise, it is better to use the modified non-tangential maximal function $\wtd{N}_*$ defined in \eqref{eqn:KPMF}.
Also, the boundary condition $\lim_{t \to 0} u(t,\cdot) = f$ is traditionally imposed as non-tangential convergence almost everywhere rather than $L^p$ convergence; we shall return to this point later.
For now, observe that there is \emph{a priori} no relation between the interior estimate and the boundary condition, as there are no trace theorems for the space $\{u  : \nm{N_* u}_p<\infty\}$.}

Well-posedness is defined analogously for all boundary value problems that we consider: for all boundary data in the specified function space, there must exist a unique solution to \eqref{eqn:mainPDE}---up to an additive constant, for Regularity and Neumann problems---which satisfies the conditions of the boundary value problem.
Of course, a boundary value problem may or may not be well-posed.

For $n/(n+1) < p < \infty$, we formulate the \emph{$H^p$-Regularity problem}\index{boundary value problem!Regularity} for $L_A$:
\begin{equation*}
  (R_H)_{0,A}^p :
  \left\{
    \begin{array}{l} 	L_A u = 0 \quad \text{in $\bbR_+^{1+n}$}, \\
      \nm{ \wtd{N}_* (\nabla u)}_{L^p} \lesssim \nm{ \nabla_\parallel f }_{H^p}, \\
      \lim_{t \to 0} \nabla_\parallel u(t,\cdot) = \nabla_\parallel f \in H^p(\bbR^n : \bbC^{mn}), \\
    \end{array} \right.
\end{equation*}
where $\wtd{N}_* u$ is the modified non-tangential maximal function
\begin{equation}\label{eqn:KPMF}
	\wtd{N}_* u(x) := \sup_{\substack{(t,y) \in \bbR^{1+n}_+ \\ y \in B(x,t)}} \bigg( \bariint_{\substack{t/2 < \tau < 2t \\ \xi \in B(y,t)}}^{} |u(\gt,\gx)|^2 \, d\gt \, d\gx\bigg)^{1/2} \qquad (x \in \bbR^n),
\end{equation}
and where $H^p(\bbR^n : \bbC^{mn})$ is the ($\bbC^{mn}$-valued) real Hardy space, which may be identified with $L^p(\bbR^n : \bbC^{mn})$ when $p > 1$.
The restriction on $p$ arises because a gradient $\nabla_\parallel f$ need not be in $H^p$ for $p \leq n/(n+1)$, as it will only have one vanishing moment.

\begin{rmk}\label{rmk:d-r-eq}
If $f$ is a distribution with $\nabla_\parallel f \in H^p(\bbR^n : \bbC^{mn})$, then $f$ may be identified with an element of $\dot{H}^p_1(\bbR^n : \bbC^m)$---the $\bbC^m$-valued homogeneous Hardy--Sobolev space of order $1$, defined in Section \ref{sec:hss}---and the boundary condition $\lim_{t \to 0} \nabla_\parallel u(t,\cdot) = \nabla_\parallel f \in H^p(\bbR^n : \bbC^{mn})$ is equivalent to the condition
\begin{equation*}
	\lim_{t \to 0} u(t,\cdot) = f \in \dot{H}^p_1(\bbR^n : \bbC^{m}).
\end{equation*}
Therefore, by considering functions $u$ rather than tangential gradients $\nabla_\parallel u$, the $H^p$-Regularity problem can be seen as a kind of $\dot{H}_1^p$-Dirichlet problem.
Conversely, by considering tangential gradients $\nabla_\parallel u$, the $L^p$-Dirichlet problem $(D_H)_{0,A}^p$ can be seen as a kind of $\dot{H}_{-1}^p$-Regularity problem.
For technical reasons we generally consider Regularity problems rather than Dirichlet problems.
\end{rmk}

For $n/(n+1) < p < \infty$, we also formulate the \emph{$H^p$-Neumann problem}\index{boundary value problem!Neumann} for $L_A$:
\begin{equation*}
  (N_H)_{0,A}^p :
  \left\{
    \begin{array}{l} 	L_A u = 0 \quad \text{in $\bbR_+^{1+n}$}, \\
      \nm{ \wtd{N}_*(\nabla u) }_{L^p} \lesssim \nm{ \partial_{\gn_A} f }_{H^p}, \\
      \lim_{t \to 0} \partial_{\gn_A} u(t,\cdot) = \partial_{\gn_A} f \in H^p(\bbR^n : \bbC^{m}). \\							
    \end{array} \right.
\end{equation*}
The \emph{$A$-conormal derivative}\index{conormal derivative} $\partial_{\gn_A}$ of $u$ is defined by
\begin{equation}\label{eqn:conormalderiv}
	\partial_{\gn_A} u(t,\cdot) = e_0 \cdot A\nabla u(\cdot,t),
\end{equation}
where $e_0 = (1,0,\ldots,0)$ is the normal vector to $\partial \bbR^{1+n}_+$ `pointing in the $t$-direction'.\footnote{The Regularity and Neumann problems for Laplace's equation on domains do not require use of $\wtd{N}_*$; in this case $N_*$ works fine (see Dahlberg and Kenig \cite{DK87} and Brown \cite{B95}). The modified non-tangential maximal function $\wtd{N}_*$ is needed for more general operators $L_A$. This modification was introduced by Kenig and Pipher \cite{KP93}, who obtained the first results in this direction.}

The boundary value problems $(D_H)_{0,A}^p$, $(R_H)_{0,A}^p$, and $(N_H)_{0,A}^p$ are all problems of \emph{order zero}: in each of these problems boundary data are assumed to be in the Lebesgue space $L^p$ or the Hardy space $H^p$.
One can also formulate Regularity and Neumann problems of \emph{order $-1$}.
For $1 < p < \infty$, the \emph{$\dot{H}^p_{-1}$-Regularity problem}, which is similar to the $L^p$-Dirichlet problem but with a different interior estimate\footnote{It is known that $\nm{\wtd{N}_* u}_{L^p} \lesssim \nm{\nabla u}_{T_{-1}^p}$ in the range of $p$ that we shall deal with (see \cite{AS16}), so the two problems are related. The converse inequality is not known except for the case of real equations \cite[Theorem 1.7]{HKMP15}.} and a decay condition at infinity (see Remark \ref{rmk:d-r-eq}), is
\begin{equation*}
  (R_H)_{-1,A}^p : \left\{
    \begin{array}{l}
      L_A u = 0 \quad \text{in $\bbR_+^{1+n}$}, \\
      \nm{ \nabla u }_{T^p_{-1}} \lesssim \nm{ \nabla_\parallel f }_{\dot{H}_{-1}^p}, \\
      \lim_{t \to \infty} \nabla_\parallel u(t,\cdot) = 0 \quad \text{in $\mc{Z}^\prime(\bbR^n : \bbC^{mn})$}, \\
      \lim_{t \to 0} \nabla_\parallel u(t,\cdot) = \nabla_\parallel f \in \dot{H}_{-1}^p(\bbR^n : \bbC^{mn}). \\							
    \end{array} \right.
\end{equation*}
Here $\mc{Z}^\prime(\bbR^n : \bbC^{mn})$ is the space of $\bbC^{mn}$-valued tempered distributions modulo polynomials, in which all homogeneous Hardy--Sobolev and Besov spaces are embedded.
The space $T^p_{-1}$ is a particular instance of a \emph{weighted tent space}: for $0 < p < \infty$, $\theta \in \bbR$, and $\map{f}{\bbR^{1+n}_+}{\bbC^N}$ measurable with $N \in \bbN_+$ fixed, we have
\begin{equation*}\label{eqn:wts-norm-intro}
  \nm{f}_{T^p_{\theta}} := \bigg( \int_{\bbR^n} \bigg( \int_0^\infty \int_{B(x,t)} |t^{-\theta} f(t,y)|^2 \, \frac{dy \, dt}{t^{n+1}} \bigg)^{p/2} \, dx \bigg)^{1/p}.
\end{equation*}

We can also formulate boundary value problems with data in $\BMO$-type spaces and homogeneous H\"older spaces $\dot{\gL}_\ga$, thus increase the range of exponents to `$p \geq \infty$'.
For $0 < \ga < 1$ we define
\begin{equation*}
  (R_H)_{-1,A}^{(\infty,\ga)} : \left\{
    \begin{array}{l}
      L_A u = 0 \quad \text{in $\bbR_+^{1+n}$}, \\
      \nm{ \nabla u }_{T^\infty_{-1;\ga}} \lesssim \nm{ \nabla_\parallel f }_{\dot{\gL}_{\ga-1}}, \\
      \lim_{t \to \infty} \nabla_\parallel u(t,\cdot) = 0 \quad \text{in $\mc{Z}^\prime(\bbR^n : \bbC^{mn})$}, \\
      \lim_{t \to 0} \nabla_\parallel u(t,\cdot) = \nabla_\parallel f \in \dot{\gL}_{\ga - 1}(\bbR^n : \bbC^{mn}) \\		
    \end{array} \right.
\end{equation*}
and furthermore, with $\ga = 0$,
\begin{equation*}
  (R_H)_{-1,A}^{(\infty,0)} : \left\{
    \begin{array}{l}
      L_A u = 0 \quad \text{in $\bbR_+^{1+n}$}, \\
      \nm{ \nabla u }_{T^\infty_{-1;0}} \lesssim \nm{\nabla_\parallel f}_{\dot{\BMO}_{-1}}, \\
      \lim_{t \to \infty} \nabla_\parallel u(t,\cdot) = 0 \quad \text{in $\mc{Z}^\prime(\bbR^n : \bbC^{mn})$}, \\
      \lim_{t \to 0} \nabla_\parallel u(t,\cdot) = \nabla_\parallel f \in \dot{\BMO}_{-1}(\bbR^n : \bbC^{mn}). \\
    \end{array} \right.
\end{equation*}

The spaces $\dot{\BMO}_{-1}$ and $\dot{\gL}_{\ga - 1}$ are most conveniently represented as the homogeneous Triebel--Lizorkin space $\dot{F}^{\infty,2}_{-1}$ and Besov spaces $\dot{B}^{\infty,\infty}_{\ga - 1}$ respectively, as their negative orders prevent traditional characterisations in terms of smoothness.
In these problems the limit in the boundary condition is imposed in the weak-star topology, which is possible since $\dot{\gL}_{\ga - 1}$ and $\dot{\BMO}_{-1}$ are the Banach duals of $\dot{H}_1^{n/(n+\ga)}$ and $\dot{H}_1^1$ respectively.
The spaces $T^\infty_{-1;\ga}$ are again instances of weighted tent spaces: for $\gq \in \bbR$ and $\ga \geq 0$, and for measurable $\map{f}{\bbR^{1+n}_+}{\bbC^N}$, we have the Carleson-type norm
\begin{equation}\label{eqn:wts-norm-carl-intro}
  \nm{f}_{T^{\infty}_{\gq;\ga}} := \sup_{B \subset \bbR^n} \frac{1}{r_B^\alpha} \bigg( \frac{1}{r_B^n} \int_0^{r_B} \int_{B(c_B,r_B - t)} |t^{-\theta} f(t,y)|^2 \, dy \, \frac{dt}{t} \bigg)^{1/2} 
\end{equation}
where the supremum is taken over all open balls $B = B(c_B,r_B) \subset \bbR^n$.

For $p$ and $\ga$ as above, we can define \emph{order $-1$ Neumann problems} $(N_H)_{-1,A}^p$, $(N_H)_{-1,A}^{(\infty,\ga)}$, and $(N_H)_{-1,A}^{(\infty,0)}$ in the same way, with tangential gradients $\nabla_\parallel$ replaced by $A$-conormal derivatives $\partial_{\gn_A}$ in the boundary condition (keeping $\nabla_\parallel$ in the decay condition at infinity).

Note that in the `order $-1$' problems above, we impose a tent space estimate on $\nabla u$ rather than a nontangential maximal function estimate.
We also impose a decay condition on the tangential gradient $\nabla_\parallel u$ at infinity.
For $p$ sufficiently small the decay condition is implied by the other conditions, and if $L_A$ satisfies a De Giorgi--Nash--Moser condition (see \eqref{eqn:DGNM}) then it is implied for all $p < \infty$, and also for some range of $\ga \geq 0$. (see Lemma \ref{lem:decaylem}).

\begin{rmk}
  We have not imposed any non-tangential convergence of solutions in the problems above.
  This is because the classification theorems of the second author and Mourgoglou, in particular \cite[Corollaries 1.2 and 1.4]{AM15}, automatically yield almost everywhere (a.e.) non-tangential convergence of Whitney averages (of either the solution or its conormal gradient, whichever is relevant).
  When $L_A$ satisfies a De Giorgi--Nash--Moser condition this can be improved to a.e. non-tangential convergence without Whitney averages.
\end{rmk}

Let us summarise the problems we have introduced so far.
There are Dirichlet problems of order $0$ and $1$ (interpreting the $H^p$-Regularity problem as a $\dot{H}_1^p$-Dirichlet problem), Regularity problems of order $0$ and $-1$, and Neumann problems of order $0$ and $-1$.

In their recent memoir \cite{BM16}, Barton and Mayboroda consider problems of \emph{intermediate order}.
They formulate Dirichlet problems of order $\gq \in (0,1)$ and Neumann problems of order $\gq \in (-1,0)$ as follows.
For $0 < \gq < 1$ and $n/(n+\gq) < p \leq \infty$,
\begin{equation*}
	(D_B)_{\gq,A}^{p} : \left\{ \begin{array}{l} 	L_A u = 0 \quad \text{in $\bbR_+^{1+n}$}, \\
	\nm{ \nabla u }_{L(p,\gq,2)} \lesssim \nm{ f }_{\dot{B}_{\gq}^{p,p}}, \\
	\Tr u = f \in \dot{B}_\gq^{p,p}(\bbR^n : \bbC^m) \\
	\end{array} \right.
\end{equation*}
and
\begin{equation*}
	(N_B)_{\gq-1,A}^{p} : \left\{ \begin{array}{l} 	L_A u = 0 \quad \text{in $\bbR_+^{1+n}$}, \\
	\nm{ \nabla u }_{L(p,\gq,2)} \lesssim \nm{ \partial_{\gn_A} f }_{\dot{B}_{\gq-1}^{p,p}}, \\
	\partial_{\gn_A}u|_{\partial \bbR^{1+n}_+} = \partial_{\gn_A} f \in \dot{B}_{\gq-1}^{p,p}(\bbR^n : \bbC^m). \\
	\end{array} \right.
\end{equation*}
The spaces $L(p,\gq,2)$ are defined by the norms
\begin{align*}
	\nm{ F }_{L(p,\gq,2)} := \bigg( \dint_{\bbR^{1+n}_+} \bigg( \bariint_{\substack{t/2 < \tau < 2t \\ \xi \in B(x,t)}}^{} |\gt^{1-\gq} F(\gt,\gx)|^2 \, d\gx \, d\gt \bigg)^{p/2} \, dx \, \frac{dt}{t} \bigg)^{1/p}
\end{align*}
with the usual modification when $p = \infty$.
We refer to these spaces as \emph{$Z$-spaces} starting from Section \ref{sec:zs} (the letter $L$ already being overused), with an indexing convention such that $Z^p_\gq = L(p,\gq+1,2)$.
The boundary condition for $(D_B)_{\gq,A}^p$ is phrased in terms of the trace operator, which Barton and Mayboroda show to be bounded from $\{u : \nabla u \in L(p,\theta,2)\}$ to $\dot{B}_\gq^{p,p}$ when $p > n/(n+\gq)$ \cite[Theorem 3.9]{BM16}.
A similar argument is used to define the boundary conormal derivative $\partial_{\gn_A}u|_{\partial \bbR^{1+n}_+}$.

\begin{rmk}
  We warn the reader that our indexing convention for boundary value problems is different to that in \cite{BM16}: we index our problems by the order of the function space in which the boundary data is assumed to lie. Thus Barton and Mayboroda refer to $(N_B)_{\gq-1,A}^p$ as $(N)_{\gq,A}^p$.
\end{rmk}

As we stated earlier, for technical reasons we consider Regularity problems rather than Dirichlet problems.
Thus for $-1 < \gq < 0$ and $n/(n+\gq+1) < p \leq \infty$ we define
\begin{equation*}
	(R_B)_{\gq,A}^{p} : \left\{ \begin{array}{l} 	L_A u = 0 \quad \text{in $\bbR_+^{1+n}$}, \\
	\nm{ \nabla u }_{Z^p_\gq} \lesssim \nm{ \nabla_\parallel f }_{\dot{B}_{\gq}^{p,p}}, \\
	\lim_{t \to \infty} \nabla_\parallel u(t,\cdot) = 0 \quad \text{in $\mc{Z}^\prime(\bbR^n : \bbC^{mn})$}, \\
	\lim_{t \to 0} \nabla_\parallel u(t,\cdot) = \nabla_\parallel f \in \dot{B}_\gq^{p,p}(\bbR^n : \bbC^{mn}) \\
	\end{array} \right.
\end{equation*}
and
\begin{equation*}
	(N_B)_{\gq,A}^{p} : \left\{ \begin{array}{l} 	L_A u = 0 \quad \text{in $\bbR_+^{1+n}$}, \\
	\nm{ \nabla u }_{Z^p_\gq} \lesssim \nm{\partial_{\gn_A} f}_{\dot{B}_{\gq}^{p,p}}, \\
	\lim_{t \to \infty} \nabla_\parallel u(t,\cdot) = 0 \quad \text{in $\mc{Z}^\prime(\bbR^n : \bbC^{mn})$}, \\
	\lim_{t \to 0} \partial_{\gn_A} u(t,\cdot) = \partial_{\gn_A} f \in \dot{B}_\gq^{p,p}(\bbR^n : \bbC^m), \\
	\end{array} \right.
\end{equation*}
replacing the trace conditions with limiting conditions for consistency with the `endpoint order' problems that we have already defined, writing $Z^p_\gq$ instead of $L(p,\gq+1,2)$, and including a decay condition at infinity.
As before, when $p = \infty$ we impose the boundary condition in the weak-star topology, using that $\dot{B}^{\infty,\infty}_\gq$ is the dual of $\dot{B}^{1,1}_{-\gq}$.

\begin{rmk}
  If we omit the decay condition, the Regularity problem $(R_B)_{\gq,A}^p$ is equivalent to the Dirichlet problem $(D_B)_{\gq+1,A}^p$ defined above by an argument similar to that of Remark \ref{rmk:d-r-eq}, and the Neumann problem $(N_B)_{\gq,A}^p$ is simply a rewriting of the previously-defined Neumann problem.
  As stated earlier, the decay condition is often redundant.
  Furthermore, the trace theorem of Barton and Mayboroda \cite[Theorem 3.9]{BM16} implies that the boundary limiting condition and the trace condition are equivalent given that $\nabla u \in Z_\theta^p$.
\end{rmk}

The Besov spaces $\dot{B}^{p,p}_\gq$ with $\gq \in (-1,0)$ are not the only function spaces situated between $H^p_0$ and $\dot{H}^p_{-1}$.
One of our new contributions to this problem is that we also consider the Hardy--Sobolev spaces $\dot{H}^p_\gq$ with $\gq \in (-1,0)$.
These are defined in Section \ref{sec:hss}; they may be identified with the homogeneous Triebel--Lizorkin spaces $\dot{F}^{p,2}_\gq$, whereas the Besov spaces $\dot{B}^{p,p}_\gq$ may be identified with $\dot{F}^{p,p}_\gq$ when $p < \infty$.
We use Hardy--Sobolev spaces to formulate the following Regularity and Neumann problems, with $-1 < \gq < 0$ and $n/(n+\gq+1) < p < \infty$:
\begin{equation*}
  (R_H)_{\gq,A}^{p} : \left\{
    \begin{array}{l}
      L_A u = 0 \quad \text{in $\bbR_+^{1+n}$}, \\
      \nm{ \nabla u }_{T^p_\gq} \lesssim \nm{\nabla_\parallel f}_{\dot{H}_{\gq}^{p}}, \\
      \lim_{t \to \infty} \nabla_\parallel u(t,\cdot) = 0 \quad \text{in $\mc{Z}^\prime(\bbR^n : \bbC^{mn})$}, \\
      \lim_{t \to 0} \nabla_\parallel u(t,\cdot) = \nabla_\parallel f \in \dot{H}_\gq^{p}(\bbR^n : \bbC^{mn}) \\
    \end{array} \right.
\end{equation*}
and
\begin{equation*}
  (N_H)_{\gq,A}^{p} : \left\{
    \begin{array}{l}
      L_A u = 0 \quad \text{in $\bbR_+^{1+n}$}, \\
      \nm{ \nabla u }_{T^p_\gq} \lesssim \nm{ \partial_{\gn_A} f }_{\dot{H}_{\gq}^{p}}, \\
      \lim_{t \to \infty} \nabla_\parallel u(t,\cdot) = 0 \quad \text{in $\mc{Z}^\prime(\bbR^n : \bbC^{mn})$}, \\
      \lim_{t \to 0} \partial_{\gn_A} u(t,\cdot) = \partial_{\gn_A} f \in \dot{H}_\gq^{p}(\bbR^n : \bbC^m). \\
    \end{array} \right.
\end{equation*}
Recall that the $T^p_\gq$ (quasi-)norm is defined in \eqref{eqn:wts-norm-intro}.
Furthermore, for $-1 < \gq < 0$ we formulate `endpoint' problems $(R_H)_{\gq,A}^\infty$ and $(N_H)_{\gq,A}^\infty$ by replacing $\dot{H}_\gq^p$ with the homogeneous $\BMO$-Sobolev space $\dot{\BMO}_\gq$, which may be identified with the homogeneous Triebel--Lizorkin space $\dot{F}^{\infty,2}_\gq$, and replacing $T^p_\gq$ with the weighted tent space $T^\infty_{\gq;0}$ (defined in \eqref{eqn:wts-norm-carl-intro}).
In this case the boundary condition is imposed in the weak-star topology, using that $\dot{\BMO}_\gq$ is the dual of $\dot{H}_{-\gq}^1$.
Unlike the function space $\{u : \nabla u \in Z_\gq^p\}$, there is no trace theorem for the function space $\{u : \nabla u \in T_\gq^p\}$.

Let us briefly summarise the Regularity and Neumann problems that we have introduced.
\begin{itemize}
\item
  At order $0$ we have problems $(R_H)_{0,A}^p$ and $(N_H)_{0,A}^p$ for $n/(n+1) < p < \infty$, with boundary data in $H^p$ and a modified non-tangential maximal estimate on the interior.
\item
  At order $-1$ we have $(R_H)_{-1,A}^p$ and $(N_H)_{-1,A}^p$ for $1 < p < \infty$, with boundary data in $\dot{H}_{-1}^p$ and a $T^p_{-1}$ interior estimate.
  Furthermore, for $0 \leq \alpha < 1$, we have $(R_H)_{-1,A}^{(\infty,\ga)}$ and $(N_H)_{-1,A}^{(\infty,\ga)}$ with boundary data in $\dot{\gL}_{\ga - 1}$ (or $\dot{\BMO}_{-1}$ when $\ga = 0$) and a $T^\infty_{-1;\ga}$ interior estimate.
\item
In between, i.e. for order $\gq \in (-1,0)$, and with $n/(n+\gq+1) < p \leq \infty$, we have $(R_B)_{\gq,A}^p$ and $(N_B)_{\gq,A}^p$ with boundary data in $\dot{B}^{p,p}_\gq$, and $(R_H)_{\gq,A}^p$ and $(N_H)_{\gq,A}^p$ with boundary data in $\dot{H}^p_\gq$ ($\dot{\BMO}_\gq$ when $p = \infty$). 
Here the interior estimates are in $Z^p_\gq$ and $T^p_\gq$ ($T^\infty_{\gq;0}$ when $p = \infty$) respectively.
\end{itemize}
For all problems of negative order we also impose a decay condition on $\nabla_\parallel u(t,\cdot)$ as $t \to \infty$ in the space $\mc{Z}^\prime$ of tempered distributions modulo polynomials, which is redundant in many cases (see Lemma \ref{lem:decaylem}).
Note that for $\gq \in (-1,0)$, the problems $(R_H)_{\gq,A}^2$ and $(R_B)_{\gq,A}^2$ (and likewise for Neumann problems) coincide, since $\dot{H}^2_\gq = \dot{B}^{2,2}_\gq$ and $Z^2_\gq = T^2_\gq$.

Since we assume no regularity of our coefficients, and since we consider weak solutions without further regularity properties, the only Regularity and Neumann problems which are meaningful are those of order between $-1$ and $0$.
Thus, within this context, this list of boundary value problems is essentially complete.

\subsection[Perturbed Dirac operators and CR systems]{Perturbed Dirac operators and Cauchy--Riemann systems}\label{ssec:fop}

Let $D$ denote the differential operator on $\bbC^{m(1+n)}$-valued functions given by
\begin{equation*}
	D := \begin{bmatrix} 0 & \dv_\parallel \\ -\nabla_\parallel & 0 \end{bmatrix}
\end{equation*}
with respect to the transversal/tangential splitting \eqref{eqn:ttsplit} of $\bbC^{m(1+n)}$.
We refer to $D$ as a \emph{Dirac operator}\index{Dirac operator}, because $D^2$ acts as the Laplacian $\gD$ on the range of $D$.
When $B \in L^\infty(\bbR^n : \mc{L}(\bbC^{m(1+n)}))$ is a coefficient matrix satisfying the same assumptions as $A$, we refer to the composition $DB$ as a \emph{perturbed Dirac operator}.\index{Dirac operator!perturbed}

The \emph{Cauchy--Riemann system}\index{Cauchy--Riemann system} associated with $DB$ is the first-order partial differential system
\begin{equation*}
	(\CR)_{DB} : \left\{ \begin{array}{r} \partial_t F + DB F = 0 \\ \curl_\parallel F_\parallel = 0\end{array} \right. \quad \text{in $\bbR^{1+n}_+$}
\end{equation*}
interpreted in the weak ($L_\text{loc}^2$) sense: that is, we say that $F \in L_{\loc}^2(\bbR^{1+n}_+ : \bbC^{m(1+n)})$ solves $(\CR)_{DB}$ if for all test functions $\gf \in C_c^\infty(\bbR^{1+n}_+ : \bbC^{m(1+n)})$
\begin{equation*}
	\dint_{\bbR^{1+n}_+} \left( F(t,x), \partial_t \gf(t,x) \right) \, dx \, dt = \dint_{\bbR^{1+n}_+} \left( F(t,x), B^*(x) D\gf(t,x) \right) \, dx \, dt,
\end{equation*}  
and for all  $\gy \in C_c^\infty(\bbR^{1+n}_+ : \bbC^{m})$ and $1 \leq j,k \leq n$, $j \neq k$,
\begin{equation*}
	\dint_{\bbR^{1+n}_+} \left( F_k(t,x), \partial_j \gy(t,x) \right) \, dx \, dt = -\dint_{\bbR^{1+n}_+} \left( F_j(t,x), \partial_k \gy(t,x) \right) \, dx \, dt.
\end{equation*}
The condition $\curl_\parallel F_\parallel = 0$ in $(\CR)_{DB}$ is preserved for all limits, so the Cauchy--Riemann system may be considered as an evolution equation in a restricted space involving a differential structure.

The first-order approach to boundary value problems for elliptic equations $L_A u = 0$ exploits a correspondence between these elliptic equations and Cauchy--Riemann systems $(\CR)_{DB}$.
Recall that $A \in L^\infty(\bbR^n : \mc{L}(\bbC^{m(1+n)}))$.
Write $A$ in matrix form with respect to the transversal/tangential splitting \eqref{eqn:ttsplit} of $\bbC^{m(1+n)}$ as
\begin{equation}\label{eqn:A}
	A = \begin{bmatrix} A_{\perp\perp} & A_{\perp\parallel} \\ A_{\parallel\perp} & A_{\parallel\parallel} \end{bmatrix},
\end{equation}
and using this representation of $A$ define auxiliary matrices
\begin{equation*}
	 \overline{A} := \begin{bmatrix} A_{\perp\perp} & A_{\perp\parallel} \\ 0 & I\end{bmatrix} \quad \text{and} \quad \underline{A} := \begin{bmatrix} I & 0 \\ A_{\parallel\perp} & A_{\parallel\parallel} \end{bmatrix}
\end{equation*}
in $L^\infty(\bbR^n : \mc{L}(\bbC^{m(1+n)}))$.
Strict accretivity of $A$ implies that $A_{\perp\perp}$ is invertible in $L^\infty(\bbR^n : \mc{L}(\bbC^m))$, and so $\overline{A}$ is invertible in $L^\infty(\bbR^n : \mc{L}(\bbC^{m(1+n)}))$.
Thus we may define
\begin{equation*}
	\hat{A} := \underline{A} \overline{A}^{-1}.
\end{equation*}
The transformed coefficient matrix $\hat{A}$ satisfies the same assumptions as $A$, and $\hat{\hat{A}} = A$ \cite[Proposition 3.2]{AAM10}.

The correspondence between elliptic equations $L_A u = 0$ and Cauchy--Riemann systems $(\CR)_{DB}$ is given by the following theorem.
See \cite[\textsection 3]{AAM10}, \cite[Proposition 4.1]{AA11}, \cite[\textsection 2]{aR13}, and \cite[Lemma 7.1]{AM15} for proofs and discussions.

\begin{thm}[Auscher--Axelsson--McIntosh]\label{thm:AAM}\index{elliptic equation!equivalence with Cauchy--Riemann system}\index{Cauchy--Riemann system!equivalence with elliptic equation}
  Let $A$ be as above, and let $B = \hat{A}$.
  If $u$ solves $L_A u = 0$, then the $A$-conormal gradient $\nabla_A u$ solves the Cauchy--Riemann system $(\CR)_{DB}$.
  Conversely, if $F$ solves $(\CR)_{DB}$, then there exists a function $u$, unique up to an additive constant, such that $L_A u = 0$ and $F = \nabla_A u$.
\end{thm}

The $A$-\emph{conormal gradient}\index{conormal gradient} $\nabla_A u$ of a function $\map{u}{\bbR^{1+n}_+}{\bbC^m}$ is defined by
\begin{equation}\label{eqn:conormalgrad}
	\nabla_A u = \begin{bmatrix} \partial_{\gn_A} u \\ \nabla_\parallel u \end{bmatrix},
\end{equation}
where the $A$-conormal derivative $\partial_{\gn_A}$ is defined in \eqref{eqn:conormalderiv}.
The components of $\nabla_A u$ appear in the boundary conditions of the Regularity and Neumann problems; hence our preference for Regularity problems over Dirichlet problems.

Thus in our consideration of elliptic equations we may focus instead on Cauchy--Riemann systems.
The principal advantage of Cauchy--Riemann systems over elliptic equations is that the Cauchy equation $\partial_t F + DB F = 0$ can be solved by semigroup methods.
We sketch how this is done, following \cite{AAM10} and \cite{AA11}.

Consider $D$ as an unbounded operator on $L^2 := L^2(\bbR^n : \bbC^{m(1+n)})$ with natural domain, and consider $B$ as a multiplication operator on $L^2$.
Then we have the following theorem \cite[Proposition 3.3 and Theorem 3.4]{AAM10}.
This is a very deep result: it is part of the framework developed by Axelsson, Keith, and McIntosh \cite{AKM06}, which encompasses the solution of the Kato square root problem \cite{AHLMT02}.

\begin{thm}[Axelsson--Keith--McIntosh]
  The perturbed Dirac operator $DB$ is bisectorial and has bounded $H^\infty$ functional calculus on the closure $\overline{\mc{R}(D)} = \overline{\mc{R}(DB)}$ of its range.
\end{thm}

These notions are properly discussed in Section \ref{section:bisectorial-hfc}.
Using the direct sum decomposition
\begin{equation*}
	L^2 = \mc{N}(DB) \oplus \overline{\mc{R}(DB)}
\end{equation*}
(which follows from bisectoriality of $DB$) and the bounded $H^\infty$ functional calculus of $DB$ on $\overline{\mc{R}(DB)}$, we obtain a decomposition
\begin{equation}\label{eqn:spectral-decomposition}
	L^2 = \mc{N}(DB) \oplus \overline{\mc{R}(DB)}^+ \oplus \overline{\mc{R}(DB)}^-.
\end{equation}
The \emph{positive and negative spectral subspaces} $\overline{\mc{R}(DB)}^\pm$ are the images of $\overline{\mc{R}(DB)}$ under the projections $\gc^\pm(DB)$, which are defined via the functions $\map{\gc^\pm}{\bbC \sm i\bbR}{\{0,1\}}$ given by
\begin{equation*}
	\gc^\pm(z) := \mb{1}_{z : \pm\operatorname{Re}(z) > 0}.
\end{equation*}
These are the characteristic functions of the right and left half-plane, restricted to $\bbC \sm i\bbR$.
They are bounded and holomorphic on every bisector, so they fall within the scope of the $H^\infty$ functional calculus.

% On the positive spectral subspace $\overline{\mc{R}(DB)}^+$ we can construct a strongly continuous semigroup $(e^{-tDB})_{t > 0}$ via the family of functions $(z \mapsto e^{-tz})_{t > 0}$, which are holomorphic and bounded on the right half-plane.
% Using this semigroup we may define a generalised \emph{Cauchy operator} $\map{C_{DB}^+}{\overline{\mc{R}(DB)}}{L^\infty(\bbR_+ : \overline{\mc{R}(DB)})}$ by
% \begin{equation*}
% (C_{DB}^+ f)(t) = e^{-tDB} \gc^+(DB) f,
% \end{equation*} 
% whose restriction to $\overline{\mc{R}(DB)}^+$ is given by the semigroup $e^{-tDB}$.

By way of the $H^\infty$ functional calculus, we may define a generalised \emph{Cauchy operator} $\map{C_{DB}^+}{\overline{\mc{R}(DB)}}{L^\infty(\bbR_+ : \overline{\mc{R}(DB)})}$ by
\begin{equation*}
(C_{DB}^+ f)(t) = e^{-tDB} \gc^+(DB) f,
\end{equation*}
corresponding to the family of functions $(z \mapsto e^{-tz} \gc^+(z))$, which are bounded and holomorphic on every bisector.
When restricted to the positive spectral subspace $\overline{\mc{R}(DB)}^+$, the Cauchy operator $C_{DB}^+$ acts as a strongly continuous semigroup.
The Cauchy operator is used in the following classification theorem, which is a combination of parts of \cite[Theorem 2.3]{AAM10} and \cite[Corollary 8.4]{AA11}.

\begin{thm}[Auscher--Axelsson--McIntosh]
If $f \in \overline{\mc{R}(DB)}^+$, then the Cauchy extension $C_{DB}^+ f$ solves $(\CR)_{DB}$, with
\begin{equation*}
	\nm{ \wtd{N}_*(C_{DB}^+ f) }_2 \simeq \nm{ f }_2 \qquad \text{and} \qquad \lim_{t \to 0} (C_{DB}^+ f)(t,\cdot) = f \quad \text{in $L^2$}.
\end{equation*}
Conversely, if $F$ solves $(\CR)_{DB}$ and $\wtd{N}_*(F) \in L^2$, then $F = C_{DB}^+ f$ for a unique $f \in \overline{\mc{R}(DB)}^+$.
\end{thm}

Combining this with Theorem \ref{thm:AAM} yields a characterisation of well-posedness of the boundary value problems $(R_H)_{0,A}^2$ and $(N_H)_{0,A}^2$.
Consider the $L^2$-Regularity problem $(R_H)_{0,A}^2$ and let $B = \hat{A}$.
A function $u$ solves $L_A u = 0$ with $\wtd{N}_*(\nabla u) \in L^2$ ($\nabla u$ and $\nabla_A u$ are interchangeable in this assumption) if and only if $\nabla_A u = C_{DB}^+ g$ for some $g \in \overline{\mc{R}(DB)}^+$, hence $\nabla_\parallel u(t,\cdot) = (C_{DB}^+ g)(t)_\parallel$ and $\lim_{t \to 0} \nabla_\parallel u(t,\cdot) = g_\parallel$.
Therefore $(R_H)_{0,A}^2$ is well-posed if and only if $g \mapsto g_\parallel$ is an isomorphism from $\overline{\mc{R}(DB)}^+$ to $L^2(\bbR^n : \bbC^{mn}) \cap \mc{N}(\curl_\parallel)$.
By the same argument, $(N_H)_{0,A}^2$ is well-posed if and only if $g \mapsto g_\perp$ is an isomorphism from $\overline{\mc{R}(DB)}^+$ to $L^2(\bbR^n : \bbC^{m})$.

By characterising solutions to $(\CR)_{DB}$ within various function spaces, we show that well-posedness of corresponding Regularity and Neumann problems is equivalent to the transversal and tangential projections being isomorphisms between certain `boundary function spaces'.
In this section we only described how to handle boundary value problems of order $0$ with $L^2$ boundary data.
We need to extend this technique to boundary value problems of more general order, and beyond $L^2$.
In the argument we just described, abstract semigroup theory (accessed via holomorphic functional calculus) did a lot of the work for us.
However, once we go beyond $L^2$, we cannot rely on abstract semigroup theory on general Banach spaces.
This would only classify solutions $F$ to $(\CR)_{DB}$ such that $F(t)$ is in a fixed function space for all $t > 0$, ruling out consideration of many tent spaces and $Z$-spaces.
Furthermore, abstract semigroup methods do not always allow us to move from initial value problems for $(\CR)_{DB}$ to boundary value problems for $L_A u = 0$.
On top of these defects we also need to consider quasi-Banach spaces, for which no semigroup theory seems to be available.
This makes things difficult.

\subsection{Adapted function spaces}

`Adapted' Hardy spaces $H^p_L$, with respect to which some operator $L$ has good properties (such as bounded $H^\infty$ functional calculus), have been developed in various contexts.
For example, Hardy spaces of differential forms on Riemannian manifolds were constructed by the second author with McIntosh and Russ \cite{AMR08} (these are adapted to the Hodge-Dirac operator $d + d^*$ on the de Rham complex); Hardy spaces adapted to non-negative self-adjoint operators satisfying Davies--Gaffney estimates on spaces of homogeneous type were studied by Hofmann, Lu, Mitrea, Mitrea, and Yan \cite{HLMMY11} (generalising the aforementioned example); Hardy spaces adapted to divergence-form elliptic operators on $\bbR^n$ were developed by Hofmann and Mayboroda \cite{HM09} and also McIntosh \cite{HMM11}.
Some further developments can be found, for example, in the work of Hyt\"onen, van Neerven, and Portal \cite{HvNP08}, Jiang and Yang \cite{JY10}, Anh and Li \cite{AL11}, and Duong and Li \cite{DL13}.
This is a very small sample of the work that has been done. 

Hardy spaces $\mb{H}^p_{DB}$ and Sobolev spaces $\mb{W}^p_{-1,DB}$ adapted to perturbed Dirac operators $DB$ were introduced by the second author and Stahlhut \cite{AS16} (See also Stahlhut's thesis \cite{sSthesis}, and for a different approach see Frey, McIntosh, and Portal \cite{FMP15}).
These spaces (defined along with more general spaces in Chapter \ref{chap:ahsb}) consist of $\bbC^{m(1+n)}$-valued functions (at least formally); the simplest case is
\begin{equation*}
	\mb{H}^2_D = \mb{H}^2_{DB} = \overline{\mc{R}(DB)} = \overline{\mc{R}(D)} \subset L^2(\bbR^n : \bbC^{m(1+n)}).
\end{equation*}
The bounded $H^\infty$ functional calculus of $DB$ on $\mb{H}^2_{DB}$ extends to $\mb{H}^p_{DB}$ and $\mb{W}_{-1,DB}^p$, yielding spectral decompositions
\begin{equation*}
	\mb{H}_{DB}^p = \mb{H}_{DB}^{p,+} \oplus \mb{H}_{DB}^{p,-}, \qquad \mb{W}_{-1,DB}^p = \mb{W}_{-1,DB}^{p,+} \oplus \mb{W}_{-1,DB}^{p,-}
\end{equation*}
analogous to \eqref{eqn:spectral-decomposition}.
Furthermore, the Cauchy operator $C_{DB}^+$ on $\overline{\mc{R}(DB)}$ extends to operators on $\mb{H}_{DB}^{p}$ and $\mb{W}_{-1,DB}^{p}$, both of which we denote by $\mb{C}_{DB}^+$.

The main application of these spaces, which incorporates results from work of the second author with both Stahlhut \cite{AS16} and Mourgoglou \cite{AM15}, is a classification of solutions to the Cauchy--Riemann system $(\CR)_{DB}$ with $L^p$-type interior estimates, for $p$ such that certain $DB$-adapted spaces may be identified with $D$-adapted spaces.
For simplicity we only state results for $1 < p < \infty$ in this introduction; corresponding results for $p \leq 1$ and `$p \geq \infty$' (i.e. for boundary data in $\BMO$-type and H\"older spaces) are also available, but these may not be stated so simply in terms of H\"older conjugate exponents.

\begin{thm}[Auscher--Mourgoglou--Stahlhut]\label{thm:AMS}\index{Cauchy--Riemann system!classification of solutions}
  Let $1 < p < \infty$.
  \begin{enumerate}[(i)]
  \item
    Assume $\mb{H}^p_{DB} \simeq \mb{H}_D^p$. If  $f \in \mb{H}^{p,+}_{DB}$, then $\mb{C}_{DB}^+ f$ solves $(\CR)_{DB}$, with 
    \begin{equation*}
      \nm{ \wtd{N}_* (\mb{C}_{DB}^+ f) }_p \simeq \nm{ f }_{H^p} \qquad \text{and} \qquad \lim_{t \to 0} \mb{C}_{DB}^+ f(t) = f \quad \text{in $H^p$}.
    \end{equation*}
    Conversely, if $F$ solves $(\CR)_{DB}$ and $\wtd{N}_* F \in L^p$, then $F = \mb{C}_{DB}^+ f$ for a unique $f \in \mb{H}_{DB}^{p,+}$.
  \item
    Assume   $\mb{H}^{p'}_{DB^*} \simeq \mb{H}_D^{p'}$.
    If  $f \in \mb{W}_{-1,DB}^{p,+}$, then $\mb{C}_{DB}^+ f$ solves $(\CR)_{DB}$, with
    \begin{equation*}
      \nm{ \mb{C}_{DB}^+ f }_{T^{p}_{-1}} \simeq \nm{ f }_{\dot{W}^{p}_{-1}} \qquad \text{and} \qquad \lim_{t \to 0} \mb{C}_{DB}^+ f(t) = f \quad \text{in $\dot{W}^{p}_{-1}$}.
    \end{equation*}
    Conversely, if $F \in T^{p}_{-1}$ solves $(\CR)_{DB}$ and $\lim_{t \to \infty} F(t)_\parallel = 0$ in $\mc{Z}^\prime(\bbR^n)$, then $F = \mb{C}_{DB}^+ f$ for a unique $f \in \mb{W}_{-1,DB}^{p,+}$.
  \end{enumerate}
\end{thm}

Furthermore, it is shown that for every $B$ there exists an open interval containing $2$, which we denote $I_0(\mb{H},DB)$, such that $\mb{H}_{DB}^p \simeq \mb{H}_D^p$ for all $p \in I_0(\mb{H},DB)$ \cite[Theorem 5.1]{AS16}.
Thus there is a nontrivial range of exponents for which Theorem \ref{thm:AMS} applies.

\begin{rmk}
  The assumption in part (ii) of the theorem is given in terms of the conjugate exponent $p^\prime$ and the adjoint coefficients $B^*$.
As part of our theory we show---as a corollary of what will eventually be termed `$\heartsuit$-duality'---that this is equivalent to $\mb{W}_{-1,DB}^{p}\simeq\mb{W}_{-1,D}^{p}$, thus unifying the assumptions of parts (i) and (ii).
\end{rmk}

As we described in the case where $p = 2$, Theorem \ref{thm:AMS} implies a characterisation of well-posedness of various Regularity and Neumann problems, both of order $0$ and order $-1$, in terms of certain transversal and tangential projections being isomorphisms.
We will not explicitly state this characterisation now; instead, we state our extension of this result in Theorem \ref{thm:wpchar-intro}.

The technical heart of this monograph is an extension of Theorem \ref{thm:AMS} to fractional order $\gq \in (-1,0)$, incorporating both Hardy--Sobolev spaces and Besov spaces.
To this end, we introduce Hardy--Sobolev spaces $\mb{H}_{\gq,L}^p$ and Besov spaces $\mb{B}_{\gq,L}^p$ adapted to operators $L$ satisfying `Standard Assumptions', which are satisfied in particular by the perturbed Dirac operators $DB$ and $BD$.
We define extension operators
\begin{equation*}
	(\bbQ_{\gf,L}f)(t) = \gf(tL)f \qquad (t > 0, f \in \overline{\mc{R}(L)})
\end{equation*}
for appropriate holomorphic functions $\gf$, and the adapted Hardy--Sobolev and Besov norms are then, roughly speaking, defined by
\begin{equation*}
	\nm{ f }_{\mb{H}_{\gq,L}^p} := \nm{ \bbQ_{\gf,L}f }_{T^p_\gq}, \qquad \nm{ f }_{\mb{B}_{\gq,L}^p} := \nm{ \bbQ_{\gf,L} f }_{Z^p_\gq}.
\end{equation*}
These definitions are reminiscent of the $\gf$-transform characterisations of Triebel--Lizorkin and Besov spaces due to Frazier and Jawerth \cite{FJ90} (the letter $\gf$ has a different meaning there), with functional calculus and tent/$Z$-spaces replacing discretised Littlewood--Paley decompositions and sequence spaces respectively.

Chapters \ref{chap:otp} to \ref{chap:diffops} are occupied with the construction of a sufficiently rich general theory of adapted Besov--Hardy--Sobolev spaces.
The theory is relatively straightforward once enough preliminaries have been collected, but getting to this stage takes some time.
We emphasise in particular the amount of work needed to quantify independence on $\gf$ of the spaces $\mb{H}_{\gq,L}^p$ and $\mb{B}_{\gq,L}^p$ (essentially all of Sections \ref{sec:intops} and \ref{sec:eco}) and the care which must be taken in discussing completions (Section \ref{sec:completions}), which is necessary to discuss and fully exploit interpolation.

\subsection[Solutions to CR systems and well-posedness]{Classification of solutions to CR systems, and applications to well-posedness}\label{ssec:csCR}

Our main theorem is the following classification of solutions to the Cauchy--Riemann system $(\CR)_{DB}$.
In this statement we restrict ourselves to $1 < p < \infty$. %mk
As with our statement of Theorem \ref{thm:AMS}, this is a simplification of the full result (Theorems \ref{thm:mainthm-leq2} and \ref{thm:mainthm-gtr2}): our theorem also allows for $p \leq 1$ and `$p \geq \infty$', but the corresponding results are better stated in terms of the `exponent notation' that we introduce in Section \ref{sec:exponents}.

\begin{thm}\label{thm:main}
	Let $-1 < \gq < 0$ and $1 < p < \infty$.
	\begin{enumerate}[(i)]
	\item Suppose that $\mb{H}_{\gq,DB}^p \simeq \mb{H}_{\gq,D}^p$.
	If $f \in \mb{H}_{\gq,DB}^{p,+}$, then $\mb{C}_{DB}^+ f$ solves $(\CR)_{DB}$, with
	\begin{equation*}
		\nm{ \mb{C}_{DB}^+ f }_{T^p_\gq} \simeq \nm{ f }_{\dot{H}^p_\gq} \qquad \text{and} \qquad \lim_{t \to 0} \mb{C}_{DB}^+ f(t) = f \quad \text{in $\dot{H}_\gq^p$,}
	\end{equation*}
	and furthermore $\lim_{t \to \infty} \mb{C}_{DB}^+ f(t)_\parallel = 0$ in $\mc{Z}^\prime(\bbR^n)$.
	Conversely, if $F \in T^p_\gq$ solves $(\CR)_{DB}$ and $\lim_{t \to \infty} F(t)_\parallel = 0$ in $\mc{Z}^\prime(\bbR^n)$, then $F = \mb{C}_{DB}^+ f$ for a unique $f \in \mb{H}_{\gq,DB}^{p,+}$.
	\item
	Suppose that $\mb{B}_{\gq,DB}^p \simeq \mb{B}_{\gq,D}^p$.
	If $f \in \mb{B}_{\gq,DB}^{p,+}$, then $\mb{C}_{DB}^+ f$ solves $(\CR)_{DB}$, with
	\begin{equation*}
		\nm{ \mb{C}_{DB}^+ f }_{Z^p_\gq} \simeq \nm{ f }_{\dot{B}^{p,p}_\gq} \qquad \text{and} \qquad \lim_{t \to 0} \mb{C}_{DB}^+ f(t) = f \quad \text{in $\dot{B}_\gq^{p,p}$,}
	\end{equation*}
	and furthermore $\lim_{t \to \infty} \mb{C}_{DB}^+ f(t)_\parallel = 0$ in $\mc{Z}^\prime(\bbR^n)$.
	Conversely, if $F \in Z^p_\gq$ solves $(\CR)_{DB}$ and $\lim_{t \to \infty} F(t)_\parallel = 0$ in $\mc{Z}^\prime(\bbR^n)$, then $F = \mb{C}_{DB}^+ f$ for a unique $f \in \mb{B}_{\gq,DB}^{p,+}$.
	\end{enumerate}
\end{thm}

Parts (i) and (ii) of this theorem are essentially identical, the only modifications being the replacement of (adapted) Hardy--Sobolev spaces with (adapted) Besov spaces, and of tent spaces with $Z$-spaces.
In fact, our arguments apply equally well to both parts, so we prove them simultaneously.
Although the theorem can be thought of as `intermediate to' Theorem \ref{thm:AMS}, it does not follow by any interpolation argument.
It is proven similarly, but the underlying techniques must be generalised, and this takes a considerable amount of work.
Neither direction is easy, but the `converse' direction---finding a `trace' $f$ given a solution $F$---is certainly more difficult.

A consequence of this theorem (and of the abstract theory that we construct) is that certain solution spaces for the equation $L_A u = 0$ form interpolation scales (Section \ref{sec:solspace-interpolation}).
This provides another viewpoint on the structure of the solution spaces, although it is weaker than the explicit description that we obtain.
Another consequence is a representation theorem for solutions $u$ of $L_A u = 0$, rather than for their conormal gradients $\nabla_{A}u$ (which solve the Cauchy-Riemann system $(\CR)_{DB}$).
This result is much more complicated to state; we refer the reader to Theorem \ref{thm:secondrepresentation}.

Identifying regions of exponents where Theorem \ref{thm:main} applies is important.
These are regions in which we can identify $DB$-adapted spaces with $D$-adapted spaces, so we call them \emph{identification regions}.
Starting from information on the intervals $I_0(\mb{H},DB), I_0(\mb{H},DB^*) \ni 2$ (as given by the second author and Stahlhut), a procedure of `$\heartsuit$-duality' and interpolation yields non-trivial regions of exponents $(p,\gq)$ for which Theorem \ref{thm:main} applies (Section \ref{sec:rel}). 

With Theorem \ref{thm:main} as a springboard, we extend the characterisation of well-posedness of Regularity and Neumann problems---described for $p = 2$ after the statement of Theorem \ref{thm:AAM} and then extended to $p \neq 2$ and $\gq \in \{-1,0\}$ by the second author with Mourgoglou and Stahlhut---as follows.
For all exponents $(p,\gq)$, we show that $\mb{H}_{\gq,D}^p$ is equal to the set of those $f \in \dot{H}_\gq^{p}(\bbR^n : \bbC^{m(1+n)})$ with $\curl_\parallel f_\parallel = 0$.
Let $N_\perp$ and $N_\parallel$ denote the projections from $\mb{H}_{\gq,D}^p$ onto $\mb{H}_{\gq,\perp}^p := \dot{H}_\gq^p(\bbR^n : \bbC^m)$ and $\mb{H}_{\gq,\parallel}^p := \dot{H}_\gq^p(\bbR^n : \bbC^{mn}) \cap \mc{N}(\curl_\parallel)$ respectively.
For $(p,\gq)$ as in Theorem \ref{thm:main} we have an identification of $\mb{H}^{p,+}_{\gq,DB}$ as a subset of $\mb{H}^p_{\gq,D}$, and so we can use $N_\perp$ and $N_\parallel$ to define
\begin{equation*}
\map{N_{H,DB,\parallel}^{(p,\gq)}}{\mb{H}^{p,+}_{\gq,DB}}{\mb{H}_{\gq,\parallel}^p} \qquad \text{and} \qquad \map{N_{H,DB,\perp}^{(p,\gq)}}{\mb{H}^{p,+}_{\gq,DB}}{\mb{H}_{\gq,\perp}^p}.
\end{equation*}
Corresponding definitions of $N_{B,DB,\parallel}^{(p,\gq)}$ and $N_{B,DB,\perp}^{(p,\gq)}$ are also made for Besov spaces.
It is important to understand that we can define these restrictions on $DB$-adapted spaces only after we have identified them with $D$-adapted spaces, on which the projections are initially defined (the only exception to this is when $B$ has a block structure).
These operators carry the well-posedness of Regularity and Neumann problems, as shown by the following theorem.
Again, this is a simplication of the full result (Theorem \ref{thm:WP-char}).
The $\gq \in \{-1,0\}$ endpoints follow from Theorem \ref{thm:AMS}.

\begin{thm}\label{thm:wpchar-intro}
	Let $B = \hat{A}$, $-1 \leq \gq \leq 0$, and $1 < p < \infty$.
	Suppose that $\mb{H}_{\gq,DB}^p \simeq \mb{H}_{\gq,D}^p$.
	Then $(R_H)_{\gq,A}^p$ (resp. $(N_H)_{\gq,A}^p$) is well-posed if and only if $N_{H,DB,\parallel}^{(p,\gq)}$ (resp. $N_{H,DB,\perp}^{(p,\gq)}$) is an isomorphism.
	The same results hold \emph{mutatis mutandis} for problems with Besov boundary data.
\end{thm}

The notion of well-posedness can be refined when considering boundary value problems with different exponents.
Consider $-1 \leq \gq_0,\gq_1 \leq 0$ and $1 < p_0, p_1 < \infty$.
We say that the problems $(R_H)_{\gq_0,A}^{p_0}$ and $(R_H)_{\gq_1,A}^{p_1}$ are \emph{mutually well-posed} if they are both well-posed, and if for all $\nabla_\parallel f \in \dot{H}_{\gq_0}^{p_0} \cap \dot{H}_{\gq_1}^{p_1}$ the unique solutions to $(R_H)_{\gq_0,A}^{p_0}$ and $(R_H)_{\gq_1,A}^{p_1}$ with boundary data $\nabla_\parallel f$ are equal.
This definition simply extends to all the boundary value problems that we consider.
Two well-posed boundary value problems need not be mutually well-posed: this phenomenon was first observed by Axelsson \cite{aA10}.
The concept of mutual well-posedness extends the notion of compatible well-posedness introduced by Barton and Mayboroda \cite[\textsection 2.4]{BM16}.
More precisely, mutual well-posedness defines an equivalence relation on the set of exponents $(p,\gq)$ for which a problem (e.g. $(R_H)_{\gq,A}^p$) is well-posed, and compatible well-posedness corresponds to the equivalence class of the `energy exponent' $(2,-1/2)$.
The problems $(R_H)_{-1/2,A}^2$ and $(N_H)_{-1/2,A}^2$ are always well-posed \cite[Theorems 3.2 and 3.3]{AMM13}.

By Theorem \ref{thm:wpchar-intro}, $(R_H)_{\gq_0,A}^{p_0}$ and $(R_H)_{\gq_1,A}^{p_1}$ are mutually well-posed if and only if $N_{H,DB,\parallel}^{(p_0,\gq_0)}$ and $N_{H,DB,\parallel}^{(p_1,\gq_1)}$ are isomorphisms whose inverses are equal on the intersection $\mb{H}_{\gq_0,\parallel}^{p_0} \cap \mb{H}_{\gq_1,\parallel}^{p_1}$ (and likewise for Neumann problems, and with Besov boundary data).
This allows us to interpolate mutual well-posedness as a straightforward corollary of Theorem \ref{thm:wpchar-intro}.
Furthermore, by using real interpolation instead of complex interpolation, we can deduce mutual well-posedness of boundary value problems with Besov boundary data from that of those with Hardy--Sobolev boundary data.
The following theorem makes this precise.
The full result is Theorem \ref{thm:cwp-int}.

\begin{thm}\label{thm:cwpint-intro}
	Suppose $-1 \leq \gq_0,\gq_1 \leq 0$, $1 < p_0,p_1 < \infty$, and $\ga \in (0,1)$, and let
	\begin{equation*}
		\frac{1}{p} = \frac{1-\ga}{p_0} + \frac{\ga}{p_1} \qquad \text{and} \qquad \gq = (1-\ga)\gq_0 + \ga \gq_1. 
	\end{equation*}
	\begin{enumerate}[(i)]
		\item
			If $\mb{H}_{\gq_j,DB}^{p_j} \simeq \mb{H}_{\gq_j,D}^{p_j}$ for $j \in \{0,1\}$, and if $(R_H)_{\gq_0,A}^{p_0}$ and $(R_H)_{\gq_1,A}^{p_1}$ are mutually well-posed, then $(R_H)_{\gq,A}^p$ is mutually well-posed with both $(R_H)_{\gq_0,A}^{p_0}$ and $(R_H)_{\gq_1,A}^{p_1}$, and furthermore if $\gq_0 \neq \gq_1$ then $(R_B)_{\gq,A}^p$ is mutually well-posed with both $(R_H)_{\gq_0,A}^{p_0}$ and $(R_H)_{\gq_1,A}^{p_1}$.
		\item If $\mb{B}_{\gq_j,DB}^{p_j} \simeq \mb{B}_{\gq_j,D}^{p_j}$ for $j \in \{0,1\}$, and if $(R_B)_{\gq_0,A}^{p_0}$ and $(R_B)_{\gq_1,A}^{p_1}$ are mutually well-posed, then $(R_B)_{\gq,A}^p$ is mutually well-posed with both $(R_B)_{\gq_0,A}^{p_0}$ and $(R_B)_{\gq_1,A}^{p_1}$.
	\end{enumerate}
	Corresponding results are also true for Neumann problems.
\end{thm}

Since invertibility is stable in complex interpolation scales, well-posedness of our boundary value problems is also stable in a way that preserves mutuality.
This is made precise in the following simplification of Theorem \ref{thm:wp-ext}.

\begin{thm}\label{thm:wp-extrap-intro}
  Let $-1 < \gq < 0$ and $1 < p < \infty$, and suppose that $\mb{H}_{\gq,DB}^{p} = \mb{H}_{\gq,D}^{p}$.
  Suppose also that $(R_H)_{\gq,A}^p$ is well-posed.
  Then $(R_H)_{\gq_1,A}^{p_1}$ and $(R_H)_{\gq,A}^{p}$ are mutually well-posed for all $(p_1,\gq_1)$ in some neighbourhood of $(p,\gq)$.
  Similar results hold for Neumann problems, and for problems with Besov boundary data.
\end{thm}

This theorem does not apply when $\gq \in \{-1,0\}$.
However, the result is true if $\gq \in \{-1,0\}$ and $\gq_1 = \gq$, with $p_1$ in a neighbourhood of $p$.
This is implicit in the work of the second author with Mourgoglou and Stahlhut, and the proof of Theorem \ref{thm:wp-extrap-intro} still applies in this case.

Finally, we have a duality result for well-posedness (this is a simplification of Theorem \ref{thm:wp-duality}).

\begin{thm}
	Let $-1 \leq \gq_0,\gq_1 \leq 0$ and $1 < p_0,p_1 < \infty$, and suppose that $\mb{H}_{\gq_0,DB}^{p_0} \simeq \mb{H}_{\gq_0,D}^{p_0}$ and $\mb{H}_{\gq_1,DB}^{p_1} \simeq \mb{H}_{\gq_1,D}^{p_1}$.
	Then $(R_H)_{\gq_0,A}^{p_0}$ and $(R_H)_{\gq_1,A}^{p_1}$ are mutually well-posed if and only if $(R_H)_{-\gq_0-1, A^*}^{p_0^\prime}$ and $(R_H)_{-\gq_1-1, A^*}^{p_1^\prime}$ are mutually well-posed.
	Similar results hold for Neumann problems and with Besov spaces.
\end{thm}

We may take $(p_0,\gq_0) = (p_1,\gq_1)$ in this result.
The mapping $(p,\gq) \mapsto (p^\prime, -\gq-1)$ is the reflection about the point $(1/2,-1/2)$ in the $(1/p,\gq)$-plane, which corresponds to the aforementioned `energy exponent'.
This reflection describes what we call `$\heartsuit$-duality of exponents'.

These theorems can be used to derive new well-posedness results for Regularity problems $(R_H)_{\gq,A}^p$ with fractional order $\gq \in (-1,0)$, and also to derive known results for $(R_B)_{\gq,A}^p$ which were recently obtained by different methods by Barton and Mayboroda \cite{BM16} under the De Giorgi-Nash-Moser assumption.
For details see Section \ref{sec:regproblem}.

\begin{rmk}
All of our results can be reformulated with the lower half-space replacing the upper half-space.
The main difference is that in this case positive spectral subspaces must be replaced with negative spectral subspaces.
\end{rmk}

\section{Summary of the monograph}

In Chapter \ref{chap:fs-prelims} we discuss two types of function spaces.
First, the `ambient spaces': tent spaces, $Z$-spaces, and slice spaces.
Many of the results here are new, or have not been used in this context.
We then consider the `smoothness spaces': Hardy--Sobolev spaces, Besov spaces, and so on.
After a quick review of these spaces, we characterise them in terms of tent spaces and $Z$-spaces (Theorem \ref{thm:BHS-Xspace}).
We also introduce a new system of notation for exponents: these are written as boldface letters, typically $\mb{p}$ and $\mb{q}$, and encode both integrability and regularity information.
This is not strictly necessary, but it truly cleans up the exposition of later parts of the monograph and makes the flow of ideas more transparent.

In Chapter \ref{chap:otp} we discuss basic operator theoretic notions.
The operators that we use in applications (i.e. the perturbed Dirac operators $DB$ and $BD$) are bisectorial, with bounded $H^\infty$ functional calculi on their ranges, and satisfy certain off-diagonal estimates.
Most of our abstract theory applies to any operator $A$ satisfying these `Standard Assumptions', so we work with such operators until we are forced to use more specific properties of perturbed Dirac operators.
We establish the boundedness of certain integral operators between tent spaces and $Z$-spaces.
Particular examples of these operators are given in terms of `extension' and `contraction' operators $\bbQ_{\gf,A}$ and $\bbS_{\gy,A}$, which we discuss.
This chapter culminates in Theorem \ref{thm:QS-compn-bddness}, which quantifies when operators of the form $\bbQ_{\gy,A} \gh(A) \bbS_{\gf,A}$ are bounded between different tent/$Z$-spaces, where $\gh$ is a holomorphic function (not necessarily bounded) on an appropriate bisector.

In Chapter \ref{chap:ahsb} we consider Besov--Hardy--Sobolev spaces adapted to an operator $A$ satisfying the aforementioned Standard Assumptions.
In Section \ref{sec:hsbbasic} we introduce `pre'-Besov--Hardy--Sobolev spaces $\bbH_{A}^\mb{p}$ and $\bbB_{A}^\mb{p}$ and establish their basic properties.
Mapping properties of the holomorphic functional calculus between these spaces, including boundedness for $H^\infty$ functions of $A$ and `regularity shifting' estimates for operators such as powers of $A$, are collected in Section \ref{sec:Xhfc-props}.
These all follow from Theorem \ref{thm:QS-compn-bddness}.
In Section \ref{sec:completions} we discuss completions.
This issue is more subtle than it initially seems.
We define `canonical completions' $\gy \mb{H}_A^\mb{p}$ and $\gy \mb{B}_A^\mb{p}$ in terms of auxiliary functions $\gy$, and show how these can be used to formulate satisfactory duality and interpolation results (Proposition \ref{prop:completion-duality} and Theorem \ref{thm:completed-interpolation}).
We also introduce `inclusion regions' of exponents $\mb{p}$ such that $\bbH_{A_0}^\mb{p} \hookrightarrow \bbH_{A_1}^\mb{p}$ (likewise for Besov spaces) for two operators $A_0$, $A_1$ satisfying the Standard Assumptions with $\overline{\mc{R}(A_0)} = \overline{\mc{R}(A_1)}$.
An interpolation result for these regions (Theorem \ref{thm:inclusion-interpolation}) is proven.
Finally, in Section \ref{sec:sgp-abstract} we show that the extended Cauchy operator $\mb{C}_A$ produces strong solutions of the Cauchy problem for $A$ with initial data in any completion of any adapted pre-Besov--Hardy--Sobolev space, and we also show the quasinorm equivalence
\begin{equation}\label{eqn:cauchybd}
	\nm{ f }_{\bbH_A^\mb{p}} \simeq \nm{ C_A f }_{T^\mb{p}} \qquad (f \in \bbH_A^{\mb{p},\pm})
\end{equation}
when $\mb{p} =  (p,s)$ with $p \leq 2$ and $s < 0$, and likewise for Besov spaces and $Z$-spaces (Corollary \ref{cor:cauchy-abstract}).

Up until this point, we work with $\bbC^N$-valued functions for an arbitrary $N \in \bbN$, as in this abstract setting we gain nothing from the transversal/tangential structure of $\bbC^{m(1+n)}$.

In Chapter \ref{chap:diffops} we consider the case when $A$ is a perturbed Dirac operator of the form $DB$ or $BD$ (and so we finally specialise to $\bbC^{m(1+n)}$-valued functions).
We show that for all exponents $\mb{p}$ the spaces $\bbH^\mb{p}_D$ and $\bbB^\mb{p}_D$ are equal to projections of classical smoothness spaces intersected with $L^2$ (Theorem \ref{thm:class-space-identn-new}), and so we may take projections of these classical smoothness spaces (without intersecting with $L^2$) as completions.
We denote the resulting spaces by $\mb{H}^\mb{p}_D$ and $\mb{B}^\mb{p}_D$.
Then we define `identification regions' $I(\mb{H},DB)$ and $I(\mb{B},DB)$, consisting of exponents $\mb{p}$ for which we can identify $\mb{H}^\mb{p}_D$ and $\mb{B}^\mb{p}_D$ as completions of $\bbH_{DB}^\mb{p}$ and $\bbB_{DB}^\mb{p}$ respectively.
These regions turn out to be open (with some minor restrictions; see Theorem \ref{thm:identn-openness}) and stable under interpolation and $\heartsuit$-duality (in a sense which interchanges $B$ and $B^*$; Corollary \ref{cor:dual-interval-identn}).
Finally, in Theorem \ref{thm:sgpnorm-concrete} we show that for $\mb{p} = (p,s) \in I(\mb{H},DB)$ with $p > 2$ and $s < 0$ we have boundedness of the Cauchy operator $C_{DB}^+$ from $\bbH_{DB}^\mb{p}$ to $T^\mb{p}$, extending the `abstract' estimate \eqref{eqn:cauchybd} (and likewise for Besov spaces and $Z$-spaces).
This is a long argument which requires various ad-hoc estimates.
The result is known to fail for $s=0$, so it does not follow by interpolation.

In Chapter \ref{chap:BVPCR} we turn our attention to differential equations. 
After presenting some basic properties of gradients of solutions to $L_A u = 0$ (or equivalently solutions of $(\CR)_{DB}$) we prove Theorems \ref{thm:mainthm-leq2} and \ref{thm:mainthm-gtr2}, which classify solutions to $(\CR)_{DB}$ in tent/$Z$-spaces with a decay condition at infinity (the decay condition is removed for certain exponents in Section \ref{sec:dos}).
This leads us to a range of exponents, related to the identification region, called the \emph{classification region}.
The argument is quite long, particularly for exponents $\mb{p} = (p,s)$ with $p > 2$, and uses all of the preceding material.
We have been (perhaps excessively) pedantic in citing dependence on previous results, so it should be possible to treat certain technical lemmas as `black boxes' in initial readings.
Although these results are `intermediate to' Theorem \ref{thm:AMS} and proven by similar arguments, they do not follow by any interpolation procedure.
The results must be reproven manually.
As a corollary of Theorems \ref{thm:mainthm-leq2} and \ref{thm:mainthm-gtr2} we prove that certain solution spaces for the equation $L_A u = 0$ are interpolation scales (Theorem \ref{thm:solspace-interpolation}), and that the Whitney averages of such solutions have non-tangential boundary limits (Theorem \ref{thm:Wavgnt}).

In Chapter \ref{chap:wp} we present applications to boundary value problems.
Most of these have already been summarised in the introduction (Subsection \ref{ssec:csCR}).
In particular, we derive ranges of well-posedness for Regularity and Neumann problems for various classes of coefficients in Section \ref{sec:regproblem}.
Our results for real coefficient scalar equations with boundary data in Hardy--Sobolev spaces are new.
In Section \ref{sec:wp-pert} we show that well-posedness of a boundary value problem is stable under perturbation of the coefficients, given certain \emph{a priori} assumptions which are known to hold in some cases.
Finally, we show the relationship between our approach and the method of layer potentials in Section \ref{sec:lps}.
For exponents in the classification region, all solutions to boundary value problems with gradients in the corresponding tent/$Z$-space and with the appropriate decay condition are given by (generalised) layer potentials.

\section{Notation}\label{section:notation}

	The following notation will be used throughout the monograph.
	
	We let $\bbN := \{0,1,2, \ldots\}$ denote the natural numbers (including $0$), and $\bbN_+ := \{1,2,\ldots\}$ denote the positive natural numbers.
	
	For $a,b \in \bbR$ and $t > 0$ we write
		\begin{equation*}
		m_a^b(t) := \left\{ \begin{array}{ll} t^a & (t \leq 1) \\ t^{-b} & (t \geq 1).\end{array} \right.
	\end{equation*}
	For $0 < p,q \leq \infty$, we define the number
	\begin{equation*}
		\gd_{p,q} := \frac{1}{q} - \frac{1}{p},
	\end{equation*}
	with the interpretation $1/\infty = 0$.
	
	We write the Euclidean distance on $\bbR^n$ as $d(x,y) = d(y,x) := |x-y|$, the open  ball with centre $x \in \bbR^n$ and radius $r > 0$ by $B(x,r) := \{y \in \bbR^n : d(x,y) < r\}$, and the (half closed, half open) annulus with centre $x \in \bbR^n$, inner radius $r_0 > 0$, and outer radius $r_1 > r_0$ by
	\begin{equation*}
		A(x,r_0,r_1) := B(x,r_1) \sm B(x,r_0) = \{y \in \bbR^n : r_0 \leq d(x,y) < r_1\}.
	\end{equation*}
	For subsets $E,F \subset \bbR^n$ we write
	\begin{equation*}
		d(E,F) := \dist(E,F) = \inf \{d(x,y) : x \in E, y \in F\}.
	\end{equation*}

	We let $L^0(\gO : E)$ denote the set of strongly measurable functions from a measure space $\gO$ to a Banach space $E$.
	As usual, we identify two functions if they agree almost everywhere.
	
	We write quasinorms as either $\nm{f}_X$ or $\nm{f \mid X}$ according to typographical need.
	For two quasinormed spaces $X$ and $Y$, we write $X \hookrightarrow  Y$ to mean that $X \subset Y$ (possibly after some identification has been made) and that the identity map is bounded.
	We write $X \simeq Y$ to mean that $X \hookrightarrow Y$ and $Y \hookrightarrow X$, i.e. to mean that $X$ and $Y$ are isomorphic, and $X = Y$ to mean that the sets $X$ and $Y$ are equal and that the associated quasinorms are equivalent.
	Often we refer to norms as `quasinorms' even though they are actually norms; for example, we refer to the $L^p$ quasinorm when $p \in (0,\infty]$, even though this is a norm when $p \geq 1$.
	For a quick introduction to quasi-Banach spaces the reader can consult the early sections of \cite{nK03}.
	When necessary, we label dual pairings by the space on the left: for example, by $\langle f,g \rangle_{L^p}$, we mean the usual duality pairing between $L^p$ and $L^{p^\prime}$, with $f \in L^p$ and $g \in L^{p^\prime}$.
	
	When $a$ and $b$ are real numbers, we write $a \lesssim b$ to mean that $a \leq Cb$ for some $C > 0$ which is independent of $a$ and $b$, and which may vary from line to line.
	If $C$ depends on some other quantities $c_1,c_2,\ldots$, we write $a \lesssim_{c_1,c_2,\ldots} b$.
	We write $a \simeq b$ to mean that $a \lesssim b$ and $b \lesssim a$.
	Context prevents us from confusing this meaning of the symbol $\simeq$ and that introduced in the previous paragraph (representing isomorphism of quasinormed spaces).
	
	We use a new and extensive notation for exponents, which appear in boldface $\mb{p}$.
	This is described in Section \ref{sec:exponents}.

\section*{Acknowledgements}

The first author acknowledges financial support from the Australian Research Council Discovery Grant DP120103692, the VIDI subsidy 639.032.427 of the Netherlands organisation for Scientific Research (NWO), and an Australian Mathematical Society Lift-off Fellowship.
Both authors were partially supported by the ANR project ``Harmonic analysis at its boundaries'' ANR-12-BS01-0013.
The majority of this work was completed while the first author was a doctoral student at the Australian National University and Universit\'e Paris-Sud.
Both authors thank Moritz Egert for valuable discussions on and around this topic.

\chapter{Function Space Preliminaries}\label{chap:fs-prelims}

Throughout this chapter we consider $\bbC^N$-valued functions for some fixed $N \in \bbN$, but since nothing really changes whether we choose $N = 1$ or $N \neq 1$ (this is justified in Remark \ref{rmk:pedantry}) we do not refer to $\bbC^N$ in the notation.
So we write $L^2(\bbR^n) = L^2(\bbR^n : \bbC^N)$, $T^\mb{p}(\bbR^n) = T^\mb{p}(\bbR^n : \bbC^N)$, and so on.
For $z \in \bbC^N$ we write $|z|$ in place of $\nm{z}_{\bbC^N}$.

\section{Exponents}\label{sec:exponents}\index{exponent}

Our theory makes heavy use of relationships between different function spaces, which may be seen as relationships between the exponents used in their parametrisation.
The most efficient way to keep track of these relationships, balancing economy of notation and clarity of ideas, is to introduce an enriched notation for exponents right at the beginning, and to work with it consistently.

Our exponent notation depends implicitly on a fixed dimension $n \in \bbN_+$.
The set of \emph{exponents} is defined to be the disjoint union
\begin{equation*}
  \mb{E} :=  \mb{E}_\text{fin} \sqcup \mb{E}_\infty,
\end{equation*}
where $\mb{E}_\text{fin} := \{(p,s) : p \in (0,\infty), s \in \bbR\}$ and $\mb{E}_\infty :=  \{(\infty,s;\ga) : s \in \bbR, \ga \geq 0\}$.
We say that an exponent is \emph{finite}\index{exponent!finite} if it is in $\mb{E}_\text{fin}$, and \emph{infinite}\index{exponent!infinite} if it is in $\mb{E}_\infty$.

We define two functions $\map{i}{\mb{E}}{(0,\infty]}$, $\map{r}{\mb{E}}{\bbR}$, representing \emph{integrability} and a kind of \emph{regularity},  by
\begin{align*}
	i(p,s) &:= p, & i(\infty,s;\ga) &:= \infty, \\
	r(p,s) &:= s, & r(\infty,s;\ga) &:= s + \ga.
\end{align*}
We also define functions $\map{j,\gq}{\mb{E}}{\bbR}$ by
\begin{align*}
	j(p,s) &:= 1/p, & j(\infty,s;\ga) &:= -\ga/n \\
	\gq(p,s) &:= s, & \gq(\infty,s;\ga) &:= s.
\end{align*}

Note that $\mb{p}$ is finite if and only if $j(\mb{p})$ is positive, and furthermore every exponent $\mb{p}$ is uniquely determined by the pair $(j(\mb{p}), \gq(\mb{p}))$.
Thus we may consider the set of exponents as being parametrised by the points in the $(j,\theta)$-plane, as pictured in Figure \ref{fig:exponents}.

For $r \in \bbR$ and $\mb{p} \in \mb{E}$, define $\mb{p} + r \in \mb{E}$ to be the unique exponent satisfying
\begin{equation*}
	j(\mb{p}+r) = j(\mb{p}) \quad \text{and} \quad \gq(\mb{p} + r) = \gq(\mb{p}) + r.
\end{equation*}
We similarly define $\mb{p} - r$.

For every exponent $\mb{p}$, we define the \emph{dual exponent}\index{duality!of exponents} $\mb{p}^\prime \in \mb{E}$ to be the unique exponent satisfying $j(\mb{p}^\prime) + j(\mb{p}) = 1$ and $\gq(\mb{p}^\prime) + \gq(\mb{p}) = 0$.
For finite exponents we have
\begin{equation*}
  (p,s)^\prime := \left\{
    \begin{array}{ll}
      (p^\prime,-s) & (p > 1) \\
      \big(\infty, -s; n(\frac{1}{p} - 1)\big) & (p \leq 1),
    \end{array}
  \right.
\end{equation*}
where $p^\prime$ is the usual H\"older conjugate of $p$.
Clearly $\mb{p}^{\prime\prime} = \mb{p}$.
We also define the \emph{$\heartsuit$-dual\index{duality!heart@$\heartsuit$-} exponent}
\begin{equation*}
  \mb{p}^\heartsuit := \mb{p}^\prime - 1,
\end{equation*}
and a quick computation shows that $\mb{p}^{\heartsuit\heartsuit} = \mb{p}$.
			
For two exponents $\mb{p}, \mb{q} \in \mb{E}$, we write $\mb{p} \hookrightarrow \mb{q}$\index{embedding!of exponents} to mean that
\begin{equation*}
  \gq(\mb{p}) \geq \gq(\mb{q}) \quad \text{and} \quad \gq(\mb{q}) - \gq(\mb{p}) = n(j(\mb{q}) - j(\mb{p})).
\end{equation*}
We always have $\mb{p} \hookrightarrow \mb{p}$.
Observe that $\mb{p} \hookrightarrow \mb{q}$ and $\mb{q} \hookrightarrow \mb{r}$ implies $\mb{p} \hookrightarrow \mb{r}$, and $\mb{p} \hookrightarrow \mb{q}$ if and only if $\mb{q}^\prime \hookrightarrow \mb{p}^\prime$.
This notation reflects embedding properties of function spaces.
	
For $\gh \in \bbR$, define $[\mb{p},\mb{q}]_\gh \in \mb{E}$\index{interpolation!of exponents} to be the unique exponent satisfying
\begin{align*}
  j([\mb{p},\mb{q}]_\gh) &= (1-\gh)j(\mb{p}) + \gh j(\mb{q}), \\
  \gq([\mb{p},\mb{q}]_\gh) &= (1-\gh)\gq(\mb{p}) + \gh \gq(\mb{q}).
\end{align*}
Note that $[\mb{p},\mb{q}]_0 = \mb{p}$ and $[\mb{p},\mb{q}]_1 = \mb{q}$.
This notation is particularly useful for interpolation results.
	
\begin{lem}\label{lem:expext}
  Suppose $\mb{p}$ and $\mb{q}$ are exponents with $\mb{p} \hookrightarrow \mb{q}$.
  Then $[\mb{p},\mb{q}]_{\gh_0} \hookrightarrow [\mb{p},\mb{q}]_{\gh_1}$ whenever $\gh_0 \leq \gh_1$.
\end{lem}
	
\begin{proof}
  Write
  \begin{align}
    \gq([\mb{p},\mb{q}]_{\gh_1}) - \gq([\mb{p},\mb{q}]_{\gh_0})
    &= \left((1-{\gh_1})\gq(\mb{p}) + {\gh_1}\gq(\mb{q})\right) - \left( (1-{\gh_0})\gq(\mb{p}) + {\gh_0}\gq(\mb{q}) \right) \nonumber \\
    &= (\gh_1 - \gh_0) (\gq(\mb{q}) - \gq(\mb{p})) \label{line:pos} \\
    &= n(\gh_1 - \gh_0) (j(\mb{q}) - j(\mb{p})) \label{line:embcon} \\
    &= n \left(\left( (1-{\gh_1})j(\mb{p}) + {\gh_1}j(\mb{q}) \right) - \left( (1-{\gh_0})j(\mb{p}) + {\gh_0}j(\mb{q}) \right)\right) \nonumber \\
    &= n(j([\mb{p},\mb{q}]_{\gh_1}) - j([\mb{p},\mb{q}]_{\gh_0})) \nonumber.
  \end{align}
  Line \eqref{line:embcon} follows from $\mb{p} \hookrightarrow \mb{q}$.
  Furthermore, line \eqref{line:pos}, $\gh_1 - \gh_0 \geq 0$, and $\gq(\mb{p}) \geq \gq(\mb{q})$ imply that $\gq([\mb{p},\mb{q}]_{\gh_0}) \geq \gq ([\mb{p},\mb{q}]_{\gh_1})$.
  Thus $[\mb{p},\mb{q}]_{\gh_0} \hookrightarrow [\mb{p},\mb{q}]_{\gh_1}$.
\end{proof}

A straightforward computation shows the following lemma.

\begin{lem}\label{lem:extid}
  Suppose $\mb{p}, \mb{q} \in \mb{E}$ and $\gh_0, \gh_1,\gl \in \bbR$.
  Then
  \begin{equation*}
    [[\mb{p},\mb{q}]_{\gh_0}, [\mb{p},\mb{q}]_{\gh_1}]_{\gl} = [\mb{p},\mb{q}]_{(1-\gl)\gh_0 + \gl \gh_1}.
  \end{equation*}
  In particular this implies
  \begin{align*}
    \mb{p} &= [[\mb{p},\mb{q}]_{-1}, \mb{q}]_{1/2} \\
    \mb{q} &= [\mb{p}, [\mb{p},\mb{q}]_2]_{1/2}.
  \end{align*}
\end{lem}

The most convenient way of visualising exponents is by identifying them with points in the $(j,\gq)$ plane.
In Figure \ref{fig:exponents} we show two exponents $\mb{p}$ and $\mb{q}$ with $\mb{p} \hookrightarrow \mb{q}$, their dual exponents, their $\heartsuit$-duals, and various other exponents which may be constructed from them.
The operations $\mb{p} \mapsto \mb{p}^\prime$ and $\mb{p} \mapsto \mb{p}^\heartsuit$ are given by reflection about the marked points at $(1/2,0)$ and $(1/2,-1/2)$ respectively.
The exponent $(1/2,-1/2)$ is special: in Section \ref{sec:wpc} we introduce it as the `energy exponent'. Certain boundary value problems associated with this exponent are automatically well-posed.
Observe that $\mb{p} \hookrightarrow \mb{q}$ if and only if the line segment from $(j(\mb{p}),\gq(\mb{p}))$ to $(j(\mb{q}),\gq(\mb{q}))$ is parallel to that from $((n+1)/n,0)$ to $(1,-1)$, with the same orientation.
	
\begin{figure}
\caption{Various exponents in the $(j,\theta)$ plane.}\label{fig:exponents}
\begin{center}
\begin{tikzpicture}[scale=2]
	%\draw [help lines] (-1,-2) grid (4,2);

	%\draw [thin] (0,2) -- (4.5,2); %theta=0 axis
	%\draw [thin] (0,0) -- (4.5,0); %thetya=-1 axis
	
	%axes
	\draw [thick,->] (-0.25,0) -- (3,0); %j axis
	\draw [thick,->] (0,-2) -- (0,1.5); %theta axis

	%guidelines
	\draw [thin,dashed] (1,1.5) -- (1,-2); %j=1/2 reflection axis
	\draw [thin,dashed] (0,-1) -- (3,-1); %theta=-1/2 reflection axis
	\draw [thin,dotted] (2,1.5) -- (2,-2); %j=1 guideline
	\draw [thin,dotted] (0,-2) -- (3,-2); %theta = -1 guideline
	\draw [thin,dotted] (0,1) -- (3,1); %theta = 1/2 guideline
	\draw [thin,dotted] (3,1) -- (2,-2); %inclusion slope
	
	%marked points and labels
	\draw [fill=white] (1,0) circle [radius = 1pt]; %dual reflection point
	
	\draw [fill=white] (1,-1) circle [radius = 1pt]; %heartdual reflection point
	
	\draw [fill=black] (1.75,-0.5) circle [radius = 1pt]; %p
	\node [right] at (1.75,-0.5) {$\mb{p}$};
	
	\draw [fill=black] (1.5,-1.25) circle [radius = 1pt]; %q
	\node [right] at (1.5,-1.25) {$\mb{q}$};
	
	\draw [fill=black] (0.25,0.5) circle [radius = 0.75pt]; %p'
	\node [right] at (0.25,0.5) {$\mb{p}^\prime$};
	
	\draw [fill=black] (0.5,1.25) circle [radius = 0.75pt]; %q'
	\node [right] at (0.5,1.25) {$\mb{q}^\prime$};
	
	\draw [fill=black] (0.25,-1.5) circle [radius = 0.75pt]; %p heart
	\node [right] at (0.25,-1.5) {$\mb{p}^\heartsuit$};
	
	\draw [fill=black] (0.5,-0.75) circle [radius = 0.75pt]; %q heart
	\node [right] at (0.5,-0.75) {$\mb{q}^\heartsuit$};
	
	\draw [fill=black] (2,0.25) circle [radius = 0.75pt]; %[p,q]_{-1}
	\node [right] at (2,0.25) {$[\mb{p},\mb{q}]_{-1}$};
	
	\draw [fill=black] (1.375,-1.625) circle [radius = 0.75pt]; %[p,q]_{3/2}
	\node [right] at (1.375,-1.625) {$[\mb{p},\mb{q}]_{\frac{3}{2}}$};
	
	%containment arrows - slope is 4, put (.2,.8) buffer from nodes
	\draw [left hook-latex] (1.73,-0.58) -- (1.52,-1.17); %p->q
	\draw [left hook-latex] (0.48,1.17) -- (0.27,0.58); %q'->p'
	\draw [left hook-latex] (0.48,-0.83) -- (0.27,-1.42); %qheart -> pheart
	
	%axis labels
	\node [above] at (0,1.5) {$\theta$};
	\node [left] at (0,1) {$\frac{1}{2}$};
	\node [left] at (0,-1) {$-\frac{1}{2}$};
	\node [left] at (0,-2) {$-1$};
	\node [below] at (1.1,0) {$\frac{1}{2}$};
	\node [below] at (2.1,0) {$1$};
	\node [below] at (2.66,0) {$\frac{n+1}{n}$};
	\node [right] at (3,0) {$j$};
	
\end{tikzpicture}
\end{center}
\end{figure}
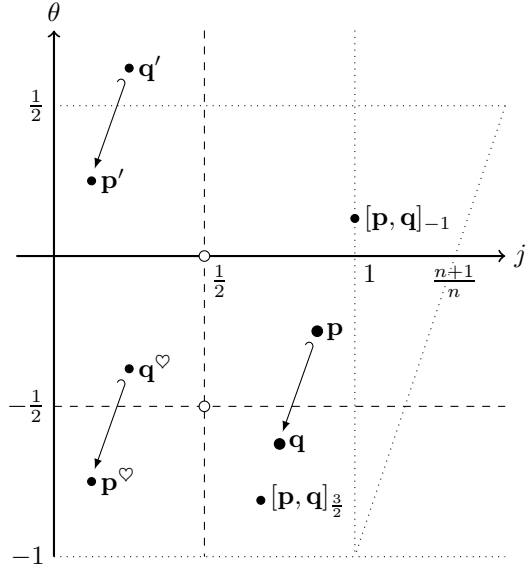

In Subsection \ref{sec:exponent-space-table} we provide a convenient table detailing the function spaces associated with different exponents.
	
\section{Tent spaces}\label{sec:ts}\index{space!tent}

The most fundamental function spaces in this monograph are the tent spaces.
These were first introduced by Coifman, Meyer, and Stein \cite{CMS83, CMS85}, and they have since proven their worth in harmonic analysis and PDE.
The other `ambient spaces' that we use---$Z$-spaces and slice spaces---are closely related to tent spaces, so a solid knowledge of tent spaces will be useful.

For $x \in \bbR^n$ we define the cone with vertex $x$ by
\begin{equation*}
  \gG(x) := \{(t,y) \in \bbR^{1+n}_+ : y \in B(x, t)\},
\end{equation*}
and for each open ball $B \subset X$ we define the \emph{tent with base $B$} by
\begin{equation*}
  T(B) := \bbR^{1+n}_+ \sm \left( \bigcup_{x \notin B} \gG(x) \right).
\end{equation*}
Equivalently, $T(B)$ is the set of points $(y,t) \in \bbR^{1+n}_+$ such that $B(y,t) \subset B$.

The tent space quasinorms are defined in terms of the \emph{Lusin operator}\index{Lusin operator} $\mc{A}$ and \emph{Carleson operators}\index{Carleson operator} $\mc{C}_\ga$, which are defined as follows: for all $\ga \geq 0$ and $f \in L^0(\bbR^{1+n}_+)$, 
\begin{equation}\label{eqn:lusin-q}
  \mc{A} f(x) := \bigg( \dint_{\gG(x)} |f(t,y)|^2 \, \frac{dy \, dt}{t^{n+1}} \bigg)^{1/2} \qquad (x \in \bbR^n)
\end{equation}
and
\begin{equation*}
  \mc{C}_\ga f(x) := \sup_{B \ni x} \frac{1}{r_B^{\ga}} \bigg( \frac{1}{r_B^n} \dint_{T(B)} |f(t,y)|^2 \, dy \, \frac{dt}{t} \bigg)^{1/2} \qquad (x \in \bbR^n).
\end{equation*}
For each $s \in \bbR$ we define an operator $\gk^s$ on $L^0(\bbR^{1+n}_+)$ by
\begin{equation*}
  (\gk^s f)(t,x) := t^{s} f(t,x) \qquad ((t,x) \in \bbR^{1+n}_+).
\end{equation*}
This operator is used to define `weighted' spaces.
	
In the following definition, and indeed throughout the whole monograph from this point, we use the exponent notation from Section \ref{sec:exponents}.

\begin{dfn}
  For a finite exponent $\mb{p} = (p,s)$, the \emph{tent space} $T^\mb{p} = T^\mb{p}(\bbR^n)$ is the set
  \begin{equation*}
    T^\mb{p} = T^p_s := \{f \in L^0(\bbR^{1+n}_+) : \mc{A} (\gk^{-s} f) \in L^p(\bbR^n)\}
  \end{equation*}
  equipped with the quasinorm
  \begin{equation*}
    \nm{f}_{T^p_s} := \nm{\mc{A} (\gk^{-s} f)}_{L^p(\bbR^n)}.
  \end{equation*}
  For an infinite exponent $\mb{p} = (\infty,s;\ga)$ we define $T^\mb{p}$ by
  \begin{equation*}
    T^\mb{p} = T^{\infty}_{s;\ga} := \{f \in L^0(\bbR^{1+n}_+) : \mc{C}_\ga (\gk^{-s} f) \in L^\infty(\bbR^n)\}
  \end{equation*}
  with its natural norm.
\end{dfn}
	
\begin{rmk}
  The spaces $T^p_s$ agree with those defined by Hofmann, Mayboroda, and McIntosh \cite[\textsection 8.3]{HMM11}, and with the spaces $T^{p,2}_{2,s}$ of Huang \cite{yH16}.
  Our spaces $T^\infty_{s;0}$ agree with Huang's spaces $T^{\infty,2}_{2,s}$.
\end{rmk}

All tent spaces are quasi-Banach spaces (Banach when $i(\mb{p}) \geq 1$).
For all finite exponents $\mb{p}$, the subspace $T^{\mb{p};c} \subset T^\mb{p}$ of compactly supported functions is dense in $T^\mb{p}$, and $L_c^2(\bbR^{1+n}_+)$ is densely contained in $T^\mb{p}$.
	
\begin{dfn}\label{dfn:ts-atom}
  Let $\mb{p}$ be an exponent with $i(\mb{p}) \leq 1$, and suppose $B \subset \bbR^n$ is a ball.
  We say that a function $a \in L^0(\bbR^{1+n}_+)$ is a \emph{$T^\mb{p}$-atom\index{atom!Tent space} (associated with $B$)} if $a$ is essentially supported in $T(B)$ and if
  \begin{equation*}
    \nm{a}_{T^2_s} \leq |B|^{\gd_{p,2}},
  \end{equation*}
  where $\gd_{p,2} = \frac{1}{2} - \frac{1}{p}$ (as defined in Section \ref{section:notation}).
\end{dfn}

An atomic decomposition for our tent spaces follows immediately from the original tent space atomic decomposition theorem \cite[Theorem 1c]{CMS85}.

\begin{thm}[Atomic decomposition]\index{atomic decomposition!of tent spaces}
  Let $\mb{p}$ be an exponent with $i(\mb{p}) \leq 1$.
  Then a function $f \in L^0(\bbR^{1+n}_+)$ is in $T^\mb{p}$ if and only if there exists a sequence $(a_k)_{k \in \bbN}$ of $T^\mb{p}$-atoms and a sequence $\gl \in \ell^{p}(\bbN)$ such that
  \begin{equation}\label{eqn:atomic-decomposition}
    f = \sum_{k \in \bbN} \gl_k a_k
  \end{equation}
  with convergence in $T^\mb{p}$.
  Furthermore we have
  \begin{equation*}
    \nm{f}_{T^\mb{p}} \simeq \inf \nm{\gl}_{\ell^{p}(\bbN)},
  \end{equation*}
  where the infimum is taken over all such decompositions.
\end{thm}
	
The following duality theorem includes all finite exponents; our exponent notation allows us to state the result without needing to separate the cases $i(\mb{p}) \leq 1$ and $i(\mb{p}) > 1$.

\begin{thm}[Duality]\label{thm:ts-duality}\index{duality!of tent spaces}
  Suppose that $\mb{p}$ is a finite exponent.
  Then for all $f,g \in L^0(\bbR^{1+n}_+)$ we have
  \begin{equation}\label{eqn:L2ip}
    \dint_{\bbR^{1+n}_+} |(f(t,x),g(t,x))| \, dx \, \frac{dt}{t} \lesssim \nm{f}_{T^\mb{p}} \nm{g}_{T^{\mb{p}^\prime}},
  \end{equation}
  and the pairing 
  \begin{equation}\label{eqn:l2pair}
    \langle f,g \rangle := \dint_{\bbR^{1+n}_+} (f(t,x),g(t,x)) \, dx \, \frac{dt}{t}
  \end{equation}
  identifies the Banach space dual of $T^\mb{p}$ with $T^{\mb{p}^\prime}$.
\end{thm}
	
Note in particular that the integral in \eqref{eqn:L2ip} converges absolutely.
	
\begin{rmk}
  Throughout this monograph we refer to the pairing in \eqref{eqn:l2pair} as the \emph{$L^2$ duality pairing}.\index{duality pairing!L2@$L^2$-}
\end{rmk}

When $\mb{p}$ is finite and $i(\mb{p}) \geq 2$, $T^\mb{p}$ may also be characterised in terms of the Carleson operator $\mc{C}_0$.\index{Carleson operator}
This is a straightforward extension of \cite[Theorem 3]{CMS85}.

\begin{thm}[Carleson characterisation of $T^\mb{p}$]\label{thm:tent-AC}
  Suppose $\mb{p}$ is a finite exponent with $i(\mb{p}) > 2$.
  Then for all $f \in L^0(\bbR^{1+n}_+)$ we have
  \begin{equation*}
    \nm{f}_{T^{\mb{p}}} \simeq \nm{\mc{C}_0 (\gk^{-\gq(\mb{p})} f)}_{L^p(\bbR^n)}.
  \end{equation*}
\end{thm}

The following `change of aperture' theorem was proven by Coifman, Meyer, and Stein for $T_0^p$ \cite[Proposition 4]{CMS85}, and the extension to the more general tent spaces here is immediate.

\begin{thm}[Change of aperture]
  For $\gb \in (0,\infty)$ and $x \in \bbR^n$ define
  \begin{equation*}
    \gG_\gb(x) := \{(t,y) \in \bbR^{1+n}_+ : y \in B(x,\gb t)\},
  \end{equation*}
  and for $f \in L^0(\bbR^{1+n}_+)$ define $\mc{A}_\gb f(x)$ as in \eqref{eqn:lusin-q}, with $\gG_\gb(x)$ in place of $\gG(x)$.
  Then for $\gb \in (0,\infty)$ and each finite exponent $\mb{p}$ we have an equivalence of quasinorms
  \begin{equation*}
    \nm{f}_{T^\mb{p}} \simeq \nm{\mc{A}_\gb (\gk^{-\gq(\mb{p})} f)}_{L^{i(\mb{p})}(\bbR^n)}.
  \end{equation*}
\end{thm}

The following embedding theorem, which can be seen as a tent space analogue of the Hardy--Littlewood--Sobolev embedding theorem, was proven by the first author \cite[Theorem 2.19]{aA16}.
The case where $\mb{p}$ and $\mb{q}$ are both infinite is not explicitly shown there, but it follows by the same argument.

\begin{thm}[Embeddings]\label{thm:ts-embeddings}\index{embedding!of tent spaces}
  Let $\mb{p}$ and $\mb{q}$ be exponents with $\mb{p} \hookrightarrow \mb{q}$.
  Then we have the embedding $T^\mb{p} \hookrightarrow T^\mb{q}$.
\end{thm}
	
The following complex interpolation theorem was proven by Hofmann, Mayboroda, and McIntosh for finite exponents \cite[Lemma 8.23]{HMM11}, and the extension to one infinite exponent follows by duality \cite[Theorem 2.1]{aA16}.
	
\begin{thm}[Complex interpolation]\label{thm-wts-c-interpolation}\index{interpolation!of tent spaces, complex}
  Suppose $\mb{p}$ and $\mb{q}$ are exponents with $j(\mb{p}), j(\mb{q}) \geq 0$ (with equality for at most one exponent), and let $0 < \gh < 1$.
  Then
  \begin{equation*}
    [T^\mb{p}, T^\mb{q}]_\gh = T^{[\mb{p},\mb{q}]_\gh}.
  \end{equation*}
\end{thm}

\begin{rmk}
  Here, and throughout the monograph, by `complex interpolation' we mean the Kalton--Mitrea complex interpolation method, introduced in \cite[\textsection 3]{KM98}.
  This agrees with the usual (Calder\'on) complex interpolation method on couples of Banach spaces, and is well-defined for all quasi-Banach couples.
\end{rmk}
	
\begin{rmk}\label{rmk:pedantry}
  In contrast with the earlier work of the first author \cite{aA16}, we define the operator $\gk^s$ in terms of powers of $t$ rather than powers of volumes of balls, so our tent spaces $T_{s}^{p}(\bbR^n)$ correspond to his tent spaces $T_{s/n}^{p,2}(\bbR^n)$.
  We also use $\bbC^N$-valued functions instead of $\bbC$-valued functions.
  This does not change the validity of previous results, as one can always split $T^{p}_s(\bbR^n:\bbC^N) \simeq \oplus_{j=1}^N T^{p}_s(\bbR^n:\bbC)$ and apply the results to each summand individually.
  This reduction would fail if we were to replace $\bbC^N$ with a general Banach space, but thankfully we have no need for such generality.
\end{rmk}

\section{$Z$-spaces}\label{sec:zs}\index{space!Z@$Z$-}

We now introduce a class of function spaces, called \emph{$Z$-spaces}, which are related to tent spaces by real interpolation.
The $Z$-spaces play the role for Besov spaces $\dot{B}^{p,p}_s$ that the tent spaces play for Hardy--Sobolev spaces $\dot{H}^p_s$.
	
\begin{dfn}\label{defn:Zpqs}
  We refer to a pair
  \begin{equation*}
    c = (c_0,c_1) \in (0,\infty) \times (3/2,\infty)
  \end{equation*}
  as a \emph{Whitney parameter}.\index{Whitney parameter}
  To each Whitney parameter $c$ and each $(t,x) \in \bbR^{1+n}_+$ we associate the \emph{Whitney region}\index{Whitney region}
  \begin{equation*}
    \gO_{c}(t,x) := (c_1^{-1} t, c_1 t) \times B(x,c_0 t) \subset \bbR^{1+n}_+,	
  \end{equation*}
  and for $f \in L^0(\bbR^{1+n}_+)$ we define the $L^2$-\emph{Whitney averages}\index{Whitney average}
  \begin{equation*}
    \mc{W}_{c} f (t,x) := \bigg( \bariint_{\gO_{c}(t,x)}^{} |f(\gt,\gx)|^2 \, d\gx \, d\gt\bigg)^{1/2} \qquad ((t,x) \in \bbR^{1+n}_+).
  \end{equation*}
  For an exponent $\mb{p}$ and a Whitney parameter $c$, and for all $f \in L^0(\bbR^{1+n}_+)$, we define the quasinorm
  \begin{equation}\label{eqn:zqn}
    \nm{f}_{Z^\mb{p}_c} := \nm{\mc{W}_{c} (\gk^{-r(\mb{p})} f)}_{L^{i(\mb{p})}(\bbR^{1+n}_+)}
  \end{equation}
  and a corresponding function space
  \begin{equation*}
    Z^\mb{p}_c = Z^\mb{p}_c(\bbR^n) := \{ f \in L^0(\bbR^{1+n}_+) : \nm{f}_{Z^\mb{p}_c} < \infty\}.
  \end{equation*}
\end{dfn}
	
We write $\gO(t,x) := \gO_{(1,2)}(t,x)$ when a particular Whitney parameter is not needed.
	
\begin{rmk}
  The spaces $Z^\mb{p}_c$ coincide with the spaces $L(i(\mb{p}), r(\mb{p})+1,2)$ introduced by Barton and Mayboroda \cite{BM16}.
  We use these spaces for the same purpose: as an ambient space for the gradient of a solution to an elliptic BVP with boundary data in a Besov space.
  The connection with tent spaces presented here extends that established by the first author \cite{aA16}.
\end{rmk}
	
\begin{rmk}
  The restriction $c_1 > 3/2$ is for technical reasons.
  The first time that it is actually needed is in our proof of the atomic decomposition theorem.
  It is possible to allow for $c_1 > 1$ by a straightforward covering argument, but this would take extra work, and $c_1 > 3/2$ is sufficient for our applications.
\end{rmk}
	
The following real interpolation theorem appears in \cite[Theorem 2.9]{aA16}.
In Theorem \ref{thm:ts-rint-full} we extend it to infinite exponents.

\begin{thm}[Real interpolation for tent spaces with finite exponents]\label{thm:ts-rint}\index{interpolation!of tent spaces, real}
  Suppose that $\mb{p}$ and $\mb{q}$ are finite exponents with $\gq(\mb{p}) \neq \gq(\mb{q})$, and suppose $0 < \gh < 1$.
  Then for all Whitney parameters $c$ we have
  \begin{equation*}
    (T^\mb{p}, T^\mb{q})_{\gh,p_\gh} = Z_c^{[\mb{p},\mb{q}]_\gh},
  \end{equation*}
  where $p_\gh = i([\mb{p},\mb{q}]_\gh)$.
\end{thm}

Thus for finite $\mb{p}$ the $Z$-spaces $Z^\mb{p}_c$ are complete and independent of $c$ (up to equivalence of quasinorms); we extend this to infinite exponents later.
Hence we write $Z^\mb{p}$ in place of $Z_c^\mb{p}$.

We establish further properties of the $Z$-spaces `by hand' rather than arguing by interpolation, because this yields stronger results.
In particular, it yields absolute convergence of $L^2$ duality pairings, while interpolation would only prove this on dense subspaces.
This is important in applications.
Our main tool is a dyadic characterisation of the $Z^\mb{p}$-quasinorm, as stated and used by Barton and Mayboroda \cite[Proof of Theorem 4.13]{BM16} but without proof.
To establish this characterisation we need some notation and a preliminary counting lemma.

Let $\mc{Q}(\bbR^n)$ be a system of (open) dyadic cubes in $\bbR^n$.
For every cube $Q \in \mc{Q}(\bbR^n)$ and for $k \in \bbZ$, define the \emph{Whitney cube}\index{Whitney cube}
\begin{equation*}
  \overline{Q}^k := (2^k \ell(Q), 2^{k+1} \ell(Q)) \times Q,
\end{equation*}
and the \emph{Whitney grid}\index{Whitney grid}
\begin{equation*}
  \mc{G}^k := \{\overline{Q}^k : Q \in \mc{Q}(\bbR^n)\}.
\end{equation*}
For each $k \in \bbZ$, $\mc{G}^k$ is a partition of $\bbR^{1+n}_+$ up to a set of measure zero.

For each Whitney parameter $c$, each $k \in \bbZ$, and each Whitney cube $\overline{Q}^k \in \mc{G}^k$, we define
\begin{equation*}
  \mc{G}_{c}(\overline{Q}^k) := \{\overline{R}^k \in \mc{G}^k : \text{$\overline{R}^k \cap \gO_{c}(t,x) \neq \varnothing$ for some $(t,x) \in \overline{Q}^k$}\}.
\end{equation*}

\begin{lem}\label{lem:W-cubes}
  Let $c$ be a Whitney parameter and $k \in \bbN$.
  Then for all $\overline{Q}^k \in \mc{G}^k$ we have
  \begin{equation*}
    |\mc{G}_{c}(\overline{Q}^k)| \lesssim_{c,k,n} 1
  \end{equation*}
  (where $| \cdot |$ denotes cardinality).
\end{lem}

\begin{proof}
  The condition $\overline{R}^k \cap \gO_{c}(t,x) \neq \varnothing$ may be rewritten as
  \begin{equation*}
    \ell(R) \in [t/2^{k+1}c_1, 2^{-k}c_1 t] \quad \text{and} \quad \dist(R,x) < c_0 t.
  \end{equation*}
  By rescaling and translating, the number of $R \in \mc{Q}(\bbR^n)$ such that this condition is satisfied is equal to the number of $R \in \mc{Q}(\bbR^n)$ such that
  \begin{equation*}
    \ell(R) \in (1/2^{k+1} c_1, 2^{-k} c_1) \quad \text{and} \quad \dist(R,0) < c_0, 
  \end{equation*}
  which is finite and depends only on $c$, $k$, and $n$.
\end{proof}

\begin{prop}[Dyadic characterisation]\label{prop:Zpqs-seq-equivalence}\index{dyadic characterisation!of Z-space@of $Z$-space}
  Let $\mb{p}$ be a finite exponent, $c$ a Whitney parameter, and $k \in \bbZ$.
  Then
  \begin{equation*}
    \nm{f}_{Z^{\mb{p}}_{c}} \simeq_{c,k,\mb{p}} \left\| \ell(Q)^{-r(\mb{p})} [|f|^2]_{\overline{Q}^k}^{1/2} \right\|_{\ell^p(\mc{G}^k, \ell(Q)^n)},
  \end{equation*}
  where $[|f|^2]^{1/2}_{\overline{Q}^k} = \nm{ f \mid L^2(\overline{Q}^k,d\gt d\gx / \gt^{1+n}) }$.
\end{prop}

\begin{proof}
  Write $\mb{p} = (p,s)$ and estimate
  \begin{align}
    \nm{f}_{Z^{\mb{p}}_c}^p
    &= \sum_{\overline{Q}^k \in \mc{G}^k} \dint_{\overline{Q}^k} \mc{W}_{c} (\gk^{-s} f)(t,x)^p \, \frac{dt}{t} \, dx \nonumber \\
    &\simeq \sum_{\overline{Q}^k \in \mc{G}^k} (2^k \ell(Q))^{-ps} \dint_{\overline{Q}^k} \nm{f \mid L^2(\gO_{c}(t,x), d\gt d\gx/\gt^{1+n})}^p \, \frac{dt}{t} \, dx \label{line:Zpqs-size} \\
    &\lesssim \sum_{\overline{Q}^k \in \mc{G}^k} (2^k \ell(Q))^{-ps} \dint_{\overline{Q}^k} \sum_{\overline{R}^k \in \mc{G}_{c}(\overline{Q}^k)} \nm{f \mid L^2(\overline{R}^k,d\gt d\gx/\gt^{1+n})}^p \, \frac{dt}{t} \, dx \label{line:Zpqs-cubes} \\
    &\simeq_{k,p,s} \sum_{\substack{\overline{Q}^k \in \mc{G}^k \\ \overline{R}^k \in \mc{G}_{c}(\overline{Q}^k)}} \ell(R)^{n-ps} \nm{f \mid L^2(\overline{R}^k,d\gt d\gx/\gt^{1+n})}^p \label{line:Zpqs-RQ} \\
    &\simeq \sum_{\overline{R}^k \in \mc{G}^k} \ell(R)^{n} \bigg(\ell(R)^{-s} \nm{f \mid L^2(\overline{R}^k,d\gt d\gx/\gt^{1+n})}\bigg)^p \label{line:Zpqs-end}.
  \end{align}
  The equivalence \eqref{line:Zpqs-size} comes from the fact that $\gt \simeq 2^k \ell(Q)$ when $(\gt,\gx) \in \gO_{c}(t,x)$ and $(t,x) \in \overline{Q}^k$.
  The upper bound \eqref{line:Zpqs-cubes} comes from covering $\gO_{c}(t,x)$ with the Whitney cubes $\overline{R}^k \in \mc{G}_{c}(\overline{Q})$, of which there are boundedly many by Lemma \ref{lem:W-cubes}.
  The equivalence \eqref{line:Zpqs-RQ} comes from noting that $\ell(R) \simeq \ell(Q)$ when $\overline{R}^k \in \mc{G}_{c}(\overline{Q}^k)$.
  Finally, \eqref{line:Zpqs-end} follows from the fact that every cube $\overline{R}^k \in \mc{G}^k$ appears at least once, and at most a bounded number of times, in the multiset $\{\overline{R}^k \in \mc{G}_{c}(\overline{Q}^k) : \overline{Q}^k \in \mc{G}\}$.
  
  To prove the converse statement, we need only prove the converse direction of \eqref{line:Zpqs-cubes}.
  This follows from the existence of a Whitney parameter $\wtd{c}$ such that whenever $\overline{R}^k \in \mc{G}_{c}(\overline{Q}^k)$ and $(t,x) \in \overline{Q}^k$, we have $\overline{R}^k \subset \gO_{\wtd{c}}(t,x)$.
  To this end one can take $\wtd{c}_0 = 2(c_0 + 2^{-k}\sqrt{n}(c_1 + 1))$ and $\wtd{c}_1 = 4c_1$.	
  This, along with the independence of $Z^\mb{p}_{c}$ on $c$, completes the proof.
\end{proof}

\begin{rmk}
  The same proof works for infinite exponents once we show that the corresponding $Z$-space norms are independent of $c$.
\end{rmk}

The dyadic characterisation of the $Z$-space quasinorm can be used to prove a duality theorem when $i(\mb{p}) > 1$.
As with the corresponding result for tent spaces, this is not just an abstract identification of dual spaces (which could be deduced by real interpolation), but also includes absolute convergence of the $L^2$ duality pairing.

\begin{prop}[Duality: reflexive range]\label{prop:Z-duality-reflexive}\index{duality!of Z-spaces@of $Z$-spaces}
  Suppose $i(\mb{p}) \in (1,\infty)$.
  Then for all $f,g \in L^0(\bbR^{1+n}_+)$ we have
  \begin{equation}\label{eqn:Z-duality-reflexive}
    \dint_{\bbR^{1+n}_+} |(f(t,x),g(t,x))| \, dx \, \frac{dt}{t} \lesssim \nm{f}_{Z^\mb{p}} \nm{g}_{Z^{\mb{p}^\prime}},
  \end{equation}
  and the $L^2$ duality pairing identifies the Banach space dual of $Z^\mb{p}$ with $Z^{\mb{p}^\prime}$.
\end{prop}

\begin{proof}
  Let $\overline{Q}_0 = (1,2) \times (0,1)^n$ be equipped with the measure $dx \, dt / t^{1+n}$, and for each function $f \in L^0(\bbR^{1+n}_+)$ and each cube $\overline{Q} \in \mc{G}$, let $f_{\overline{Q}}$ be the function on $\overline{Q}_0$ which is the affine reparametrisation of $\mb{1}_{\overline{Q}} f$, so that $[|f|^2]_{\overline{Q}}^{1/2} = \nm{f_{\overline{Q}}}_{L^2(\overline{Q}_0)}$.
  Then by Proposition \ref{prop:Zpqs-seq-equivalence}, writing $\mb{p} = (p,s)$, we have
  \begin{align*}
    \nm{f}_{Z^\mb{p}} &\simeq \nm{\ell(Q)^{-s} [|f|^2]_{\overline{Q}}^{1/2}}_{\ell^p(\mc{G}, \ell(Q)^n)} \\
                      &\simeq \nm{\ell(Q)^{-s} f_{\overline{Q}} }_{\ell^p(\mc{G},\ell(Q)^n : L^2(\overline{Q}_0))} \\
                      &=: \nm{f_{\overline{Q}}}_{\ell_s^p(\mc{G},\ell(Q)^n : L^2(\overline{Q}_0))}.
  \end{align*}
  Evidently the map $f \mapsto (f_{\overline{Q}})_{\overline{Q} \in \mc{G}}$ is an isomorphism between $Z^\mb{p}$ and $\ell_s^p(\mc{G},\ell(Q)^n : L^2(\overline{Q}_0))$.
  Furthermore, for all $f,g \in L^0(\bbR^{1+n}_+)$ we have
  \begin{equation*}
    \int_{\bbR^{1+n}_+} |(f(t,x),g(t,x))| \, dx \, \frac{dt}{t}
    \simeq \sum_{\overline{Q} \in \mc{G}} \ell(Q)^n \dint_{\overline{Q}_0} |(f_{\overline{Q}}(t,x),g_{\overline{Q}}(t,x))| \, \frac{dx\,dt}{t^{1+n}},
  \end{equation*}
  so $f \mapsto (f_{\overline{Q}})_{\overline{Q} \in \mc{G}}$ identifies the $L^2(\bbR^{1+n}_+)$ and $\ell^2(\mc{G},\ell(Q)^n : L^2(\overline{Q}_0))$ duality pairings (up to a constant).
  
  Since we have
  \begin{equation*}
    \sum_{\overline{Q} \in \mc{G}} \ell(Q)^n |(f_{\overline{Q}},g_{\overline{Q}})_{L^2(\overline{Q}_0)}|
    \lesssim \nm{f_{\overline{Q}}}_{\ell_s^p(\mc{G},\ell(Q)^n:L^2(\overline{Q}_0))} \nm{g_{\overline{Q}}}_{\ell_{-s}^{p^\prime}(\mc{G},\ell(Q)^n:L^2(\overline{Q}_0))}
  \end{equation*}
  and since the $\ell^2(\mc{G},\ell(Q)^n : L^2(Q_0))$ duality pairing identifies $\ell_s^{p^\prime}(\mc{G},\ell(Q)^n:L^2(Q_0))$ as the dual of $\ell_{-s}^{p}(\mc{G},\ell(Q)^n:L^2(Q_0))$, the corresponding results for $Z^\mb{p}$ follow.
\end{proof}

The dyadic characterisation may also be used to prove an atomic decomposition theorem for $Z$-spaces.

\begin{dfn}
  Let $\mb{p} = (p,s)$ be a finite exponent and $c$ a Whitney parameter.
  We say that a function $a \in L^0(\bbR^{1+n}_+)$ is a \emph{$Z^\mb{p}_{c}$-atom\index{atom!Z-space@$Z$-space} associated with the point $(t,x) \in \bbR^{1+n}_+$} if $a$ is essentially supported in $\gO_{c}(t,x)$ and if
  \begin{equation*}
    \nm{\gk^{-s} a}_{L^2(\gO_{c}(t,x), dx\, dt/t)} \leq t^{n\gd_{p,2}}.
  \end{equation*}
  (recall that $\gd_{p,2} = \frac{1}{2} - \frac{1}{p}$ is defined in Section \ref{section:notation}).
\end{dfn}

\begin{lem}\label{lem:Z-atom-est}
  Let $\mb{p}$ be a finite exponent and suppose $a$ is a $Z_c^\mb{p}$-atom associated with $(t_0,x_0) \in \bbR^{1+n}_+$.
  Then 
  \begin{equation*}
    \nm{a}_{Z^\mb{p}} \lesssim_{c,\mb{p}} 1.
  \end{equation*}
\end{lem}

\begin{proof}
  A reasonably quick computation shows that
  \begin{equation*}
    \{(t,x) \in \bbR^{1+n}_+ : \gO_{c}(t,x) \cap \gO_{c}(t_0,x_0) \neq \varnothing\} \subset \gO_{\wtd{c}}(t_0,x_0)
  \end{equation*}
  where $\wtd{c_0} = c_0(1+c_1^2)$ and $\wtd{c_1} = c_1^2$.
  Hence we can estimate, using the support and size conditions for $a$ and writing $\mb{p} = (p,s)$,
  \begin{align*}
    \nm{a}_{Z_{c}^\mb{p}}
    &\leq \bigg( \dint_{\gO_{\wtd{c}}(t_0,x_0)} \bigg( \frac{1}{t^{1+n}} \dint_{\gO_{c}(t_0,x_0)} \gt |\gt^{-s} a(\gt,\gx)|^2 \, d\gx \, \frac{d\gt}{\gt} \bigg)^{p/2} \, dx \, \frac{dt}{t} \bigg)^{1/p} \\
    &\lesssim_{c} t_0^{n\gd_{p,2}} \bigg( \dint_{\gO_{\wtd{c}}(t_0,x_0)} t_0^{-np/2} \, dx \, \frac{dt}{t} \bigg)^{1/p}
      \simeq_{c,p} t_0^{n\gd_{p,2} - \frac{n}{2} + \frac{n}{p}}
      = 1
  \end{align*}
  as required.
\end{proof}

\begin{thm}[Atomic decomposition of $Z$-spaces]\label{thm:Z-atomic-decompn}\index{atomic decomposition!of Z-spaces@of $Z$-spaces}
  Suppose $\mb{p} = (p,s)$ with $p \leq 1$, and let $c$ be a Whitney parameter.
  Then a function $f \in L^0(\bbR^{1+n}_+)$ is in $Z^\mb{p}$ if and only if there exists a sequence $(a_k)_k \in \bbN$ of $Z_{c}^\mb{p}$-atoms and a sequence $\gl \in \ell^p(\bbN)$ such that
  \begin{equation*}
    \sum_{k \in \bbN} \gl_k a_k = f
  \end{equation*}
  with convergence in $Z^\mb{p}$.
  Furthermore we have
  \begin{equation*}
    \nm{f}_{Z^\mb{p}} \simeq \inf \nm{\gl}_{\ell^p(\bbN)},
  \end{equation*} 
  where the infimum is taken over all such decompositions.
\end{thm}

\begin{proof}
  Given such a decomposition of $f$, we have
  \begin{equation*}
    \nm{f}_{Z^\mb{p}}^{p}
    =	\bigg\| \sum_{k \in \bbN} \gl_k a_k \bigg\|_{Z^\mb{p}}^{p}
    \lesssim \nm{\gl}_{\ell^p(\bbN)}^{p}
  \end{equation*}
  by Lemma \ref{lem:Z-atom-est}, and so $\nm{f}_{Z^\mb{p}} \lesssim \inf \nm{\gl}_{\ell^p(\bbN)}.$
  It remains to prove the reverse estimate.
  For each $k \in \bbZ$ we can write
  \begin{equation}\label{eqn:z-series}
    f = \sum_{\overline{Q}^k \in \mc{G}^k} f_{\overline{Q}^k}
  \end{equation}
  where $f_{\overline{Q}^k} = \mb{1}_{\overline{Q}^k} f$ (note that this notation differs from that in the proof of Proposition \ref{prop:Z-duality-reflexive}).
  By Proposition \ref{prop:Zpqs-seq-equivalence}, and using dominated convergence, this sum converges in $Z^\mb{p}$.
  If $k \geq \log_2(c_0^{-1} \sqrt{n}/3) + 1$ and if $c_1 > 3/2$ (this is the first place where we actually use this assumption), then
  \begin{equation*}
    \overline{Q}^k \subset \gO_{c}(c_Q,t_Q)
  \end{equation*}
  for all $Q \in \mc{Q}(\bbR^n)$, where $c_Q$ is the center of $Q$ and $t_Q$ is the midpoint of $2^k \ell(Q)$ and $2^{k+1} \ell(Q)$.
  Therefore, under this condition on $k$, each $f_{\overline{Q}^k}$ satisfies the support condition required of a $Z_{c}^\mb{p}$-atom.
  By Proposition \ref{prop:Zpqs-seq-equivalence} the norms
  \begin{equation*}
    \nm{\gk^{-s} f_{\overline{Q}^k}}_{L^2(\gO_{c}(c_{\overline{Q}}, t_{\overline{Q}}),dx \, dt/t)}
  \end{equation*}
  are all finite, so we can define
  \begin{equation*}
    \gl_{\overline{Q}^k} := t_{\overline{Q}^k}^{-n\gd_{p,2}} \nm{\gk^{-s}  f_{\overline{Q}^k}}_{L^2(\gO_{c}(c_{\overline{Q}}, t_{\overline{Q}}),dx \, dt/t)}
  \end{equation*}
  and
  \begin{equation*}
    a_{\overline{Q}^k} := \left\{ 
      \begin{array}{ll} 
        \gl_{\overline{Q}^k}^{-1} f_{\overline{Q}^k} & (f_{\overline{Q}^k} \neq 0) \\ 0 & (f_{\overline{Q}^k} = 0). \end{array} 
    \right.
  \end{equation*}
  Each $a_{\overline{Q}^k}$ is a $Z_{c}^\mb{p}$-atom and
  \begin{equation*}
    f = \sum_{\overline{Q}^k \in \mc{G}^k} \gl_{\overline{Q}^k} a_{\overline{Q}^k}
  \end{equation*}
  with convergence in $Z^\mb{p}$, and furthermore
  \begin{align*}
    \nm{(\gl_{\overline{Q}^k})}_{\ell^p(\mc{G}^k)}
    &= \left\| t_{\overline{Q}^k}^{-n\gd_{p,2}} \nm{\gk^{-s} f_{\overline{Q}}^k}_{L^2(\gO_{c}(c_{\overline{Q}^k},t_{\overline{Q}^k}),dx \, dt/t)} \right\|_{\ell^p(\mc{G}^k)} \\
    &\simeq \left\| (2^k \ell(Q))^{-n\gd_{p,2} - s} |Q|^{1/2} [|f|^2]_{\overline{Q}^k}^{1/2} \right\|_{\ell^p(\mc{G}^k)} \\
    &\simeq \left\| \ell(Q)^{(n/p)-s} [|f|^2]_{\overline{Q}^k}^{1/2} \right\|_{\ell^p(\mc{G}^k)} \\
    &\simeq \nm{f}_{Z^\mb{p}}
  \end{align*}
  again using Proposition \ref{prop:Zpqs-seq-equivalence}.
\end{proof}

In contrast with tent spaces, it is very easy to construct atomic decompositions of functions $f \in Z^\mb{p}$: as in the proof of the theorem, simply decompose $f$ via the Whitney grid $\mc{G}^k$ for sufficiently large $k$.
This works for all finite $\mb{p}$, even if $i(\mb{p}) > 1$.
Abstract decompositions are useful in the proof of $Z^\mb{p}$-$Z^{\mb{p}^\prime}$ duality when $i(\mb{p}) \leq 1$, which we now build towards.

\begin{lem}
  For all Whitney parameters $c$ and all $f \in L^0(\bbR^{1+n}_+)$, the function $\mc{W}_c f$ is lower semicontinuous.
\end{lem}

\begin{proof}
  Fix $M > 0$ and suppose that $\mc{W}_c f(t,x) > M$.
  Then there exists a small $\varepsilon > 0$ such that $\mc{W}_{(c_0 - \varepsilon, c_1 - \varepsilon)} f(t,x) > M$ also.
  A short computation shows that if $\td{x} \in B(x,\varepsilon t/2)$ and if $|\td{t} - t| < (c_1/(c_1 - \varepsilon) - 1)t$, then $\gO_{c}(\td{t},\td{x})$ contains $\gO_{(c_0 - \varepsilon, c_1 - \varepsilon)}(t,x)$, so for all such $(\td{t},\td{x})$ we have
  \begin{equation*}
    \mc{W}_c f(\td{t},\td{x}) \geq \mc{W}_{(c_0 - \varepsilon,c_1 - \varepsilon)} f(t,x) > M.
  \end{equation*}
  Therefore the set $\{(t,x) \in \bbR^{1+n}_+ : \mc{W}_c f(t,x) > M\}$ is open.
\end{proof}

\begin{cor}\label{cor:Zinfty-sup}
  Let $\mb{p}$ be an infinite exponent.
  Then
  \begin{equation*}
    \nm{f}_{Z_c^\mb{p}} = \sup_{(t,x) \in \bbR^{1+n}_+} \mc{W}_c (\gk^{-r(\mb{p})} f)(t,x),
  \end{equation*}
  i.e. the essential supremum in the definition of the $Z^\mb{p}_c$-norm can be replaced with a supremum.
\end{cor}

\begin{proof}
  Lower semicontinuity of $\mc{W}_c (\gk^{-r(\mb{p})} f)$ implies that if
  \begin{equation*}
    \mc{W}_c (\gk^{-r(\mb{p})} f)(t,x) > M
  \end{equation*}
  for some $M < \infty$ at one point $(t,x)$, then it continues to hold in an open neighbourhood of $(t,x)$, and in particular on a set of positive measure.
\end{proof}

We can finally prove the duality theorem for $Z^\mb{p}$ with $i(\mb{p}) \leq 1$.
As with the other duality results so far, this includes absolute convergence of the $L^2$ duality pairing.

\begin{thm}[Duality: non-reflexive range]\index{duality!of Z-spaces@of $Z$-spaces}
  Suppose $i(\mb{p}) \leq 1$ and let $c$ be a Whitney parameter.
  Then for all $f,g \in L^0(\bbR^{1+n}_+)$ we have
  \begin{equation}\label{eqn:Z-pairing-1}
    \dint_{\bbR^{1+n}_+} |(f(t,x),g(t,x))| \, dx\,\frac{dt}{t} \lesssim \nm{f}_{Z_c^{\mb{p}}} \nm{g}_{Z_c^{\mb{p}^\prime}},
  \end{equation}
  and the $L^2$ duality pairing identifies the Banach space dual of $Z_c^\mb{p}$ with $Z_c^{\mb{p}^\prime}$.
\end{thm}

\begin{proof}
  Write $\mb{p} = (p,s)$, so that $\mb{p}^\prime = (\infty,-s,n\gd_{p,1})$.
  First suppose $a$ is a $Z_{c}^\mb{p}$-atom associated with a point $(t_0,x_0) \in \bbR^{1+n}_+$.
  Then we have
  \begin{align*}
    \dint_{\bbR^{1+n}_+} |(a(t,x),g(t,x))| \, dx\,\frac{dt}{t}
    &\leq \nm{\gk^{-s} a}_{L^2(\bbR^{1+n}_+, dx\,dt/t)} \nm{\gk^s g}_{L^2(\gO_{c}(t_0,x_0),dx\,dt/t)} \\
    &\lesssim t_0^{n\gd_{p,2} + n\gd_{1,p} + (n/2)} \nm{\gk^{s-n\gd_{1,p}} g}_{L^2(\gO_{c}(t_0,x_0), dx\,dt / t^{1+n})} \\
    &\leq \nm{g}_{Z_c^\mb{p}}
  \end{align*}
  by Corollary \ref{cor:Zinfty-sup}.
  For general $f \in Z^\mb{p}$, write $f$ as a sum of $Z_c^\mb{p}$-atoms as in Theorem \ref{thm:Z-atomic-decompn}, so that
  \begin{align*}
    \dint_{\bbR^{1+n}_+} |(f(t,x),g(t,x))| \, dx\,\frac{dt}{t}
    &\leq \sum_{k \in \bbN} |\gl_k| \dint_{\bbR^{1+n}_+} |(a_k(t,x),g(t,x))| \, dx\,\frac{dt}{t} \\
    &\lesssim \nm{g}_{Z_c^{\mb{p}^\prime}}\nm{\gl}_{\ell^p(\bbN)}
  \end{align*}
  since $p \leq 1$.
  Taking the infimum over all atomic decompositions of $f$ proves \eqref{eqn:Z-pairing-1}.
  
  Now suppose that $\gfv \in (Z_{c}^\mb{p})^\prime$.
  By the same technique as in the proof of Proposition \ref{prop:Z-duality-reflexive}, we find that there exists a sequence $(g_{\overline{Q}}) \in \ell_{-s}^\infty(\mc{G} : L^2(Q_0))$ corresponding to the induced action of $\gfv$ on $\ell_s^p(\mc{G},\ell(Q)^n : L^2(Q_0))$ (since $\ell^p(\bbN)^\prime = \ell^\infty(\bbN)$ for $p \leq 1$).
  Hence there exists a function $G_\gfv \in L^0(\bbR^{1+n}_+)$ corresponding to the action of $\gfv$ on $Z_{c}^\mb{p}$.
  We need to show that $G_\gfv$ is in $Z_{c}^{\mb{p}^\prime}$.
  
  Suppose $(t,x) \in \bbR_+^{1+n}$.
  Then we can estimate
  \begin{align*}
    \mc{W}_c(\gk^{s-n\gd_{1,p}} G_\gfv)(t,x)
    &\simeq t^{-n\gd_{1,p}} \nm{G_\gfv}_{L_{-s}^2(\gO_c(t,x),d\gx\,d\gt/\gt^{1+n})} \\
    &= t^{-n\gd_{1,p}} \sup_{\substack{F \in L_s^2(\gO_c(t,x),d\gx\,d\gt/\gt^{1+n}) \\ \nm{F}\leq 1}} \left| (F,G_\gfv) \right| \\
    &\lesssim t^{-n\gd_{1,p} -(n/2) - n\gd_{p,2}}\nm{\gfv}_{(Z_s^p)^\prime}  \sup_{\substack{F \in L_s^2(\gO_c(t,x),d\gx\,d\gt/\gt) \\ \nm{F} \leq t^{n\gd_{p,2}}}} \nm{F}_{Z_s^p} \\
    &\leq \nm{\gfv}_{(Z^\mb{p})^\prime},
  \end{align*}
  using $n\gd_{1,p} + (n/2) + n\gd_{p,2} = 0$, the fact that the condition in the final supremum implies that $F$ is a $Z_{c}^\mb{p}$-atom, and Lemma \ref{lem:Z-atom-est}.
  Therefore we have
  \begin{equation*}
    \nm{G_\gfv}_{Z_{c}^{\mb{p}^\prime}}
    = \sup_{(t,x) \in \bbR^{1+n}_+} \mc{W}_c(\gk^{s-n\gd_{1,p}}G_\gfv)(t,x)
    \lesssim \nm{\gfv}_{(Z_c^\mb{p})^\prime}
  \end{equation*}
  as desired.
\end{proof}

\begin{cor}
  For all infinite exponents $\mb{p}$ and all Whitney parameters $c$, the $Z_c^\mb{p}$ norms are mutually equivalent.
  Hence for all exponents $\mb{p}$ we write $Z^\mb{p}$ in place of $Z_{c}^\mb{p}$.
\end{cor}

Having identified the duals of all $Z^\mb{p}$ spaces for finite $\mb{p}$, we can give a full interpolation theorem.

\begin{thm}[Real interpolation of tent spaces: full range]\label{thm:ts-rint-full}\index{interpolation!of tent spaces, real}
  Suppose that $\mb{p}$ and $\mb{q}$ are exponents with $\gq(\mb{p}) \neq \gq(\mb{q})$, and $0 < \gh < 1$.
  Then
  \begin{equation*}
    (T^\mb{p}, T^\mb{q})_{\gh,p_\gh} = Z^{[\mb{p},\mb{q}]_\gh},
  \end{equation*}
  where $p_\gh = i([\mb{p},\mb{q}]_\gh)$.
\end{thm}

\begin{proof}
  For finite exponents this is Theorem \ref{thm:ts-rint}.
  If $1 < i(\mb{p}), i(\mb{q}) \leq \infty$, this follows by writing
  \begin{equation*}
    (T^{\mb{p}}, T^{\mb{q}})_{\gh,p_\gh} = ((T^{\mb{p}^\prime})^\prime, (T^{\mb{q}^\prime})^\prime)_{\gh,p_\gh} = (T^{\mb{p}^\prime}, T^{\mb{q}^\prime})_{\gh,p_\gh^\prime}^\prime
  \end{equation*}
  via the duality theorem for real interpolation \cite[Theorem 3.7.1]{BL76}, using that $T^{\mb{p}^\prime} \cap T^{\mb{q}^\prime}$ is dense in both $T^{\mb{p}^\prime}$ and $T^{\mb{q}^\prime}$, and then noting that
  \begin{equation*}
    p_\gh^\prime = i([\mb{p},\mb{q}]_\gh)^\prime = i([\mb{p}^\prime, \mb{q}^\prime]_\gh).
  \end{equation*}
  The full result follows by Wolff reiteration \cite[Theorem 1]{tW82}.
\end{proof}

As an immediate consequence we can prove a real interpolation theorem for $Z$-spaces.
This follows from Theorem \ref{thm:ts-rint-full} along with the reiteration theorem for real interpolation \cite[Theorem 5.2.4]{BL76}.

\begin{prop}[Real interpolation of $Z$-spaces]\label{prop:Z-rint-full}\index{interpolation!of Z-spaces@of $Z$-spaces, real}
	Let $\mb{p}$ and $\mb{q}$ be exponents which are not both infinite and with $\gq(\mb{p}) \neq \gq(\mb{q})$, and let $\gh \in (0,1)$.
	Then
	\begin{equation*}
		(Z^\mb{p}, Z^\mb{q})_{\gh,p_\gh} = Z^{[\mb{p},\mb{q}]_\gh},
	\end{equation*}
	where $p_\gh = i([\mb{p}, \mb{q}]_\gh)$.
\end{prop}

We can also establish a complex interpolation theorem for $Z$-spaces.
Within the Banach range, this has already been done by Barton and Mayboroda \cite[Theorem 4.13]{BM16}.
However, in the quasi-Banach range---that is, for exponents $\mb{p}$ with $i(\mb{p}) < 1$---one must argue differently.
We remind the reader that we perform complex interpolation using the Kalton--Mitrea method from \cite[\textsection 3]{KM98}, which is well-defined for all quasi-Banach couples.

\begin{prop}[Complex interpolation of $Z$-spaces]\label{prop:Z-cint}\index{interpolation!of Z-spaces@of $Z$-spaces, complex}
  Let $\mb{p}$ and $\mb{q}$ be exponents which are not both infinite, and let $\gh \in (0,1)$.
  Then
  \begin{equation*}
    [Z^\mb{p}, Z^\mb{q}]_{\gh} = Z^{[\mb{p},\mb{q}]_\gh}.
  \end{equation*}
\end{prop}

We defer the proof to Section \ref{sec:Z-space-appendix}, as it requires the introduction of spaces $Z^{p,q}_s$ with $q \neq 2$.

\begin{rmk}
  The $Z$-spaces can be seen as Wiener amalgam spaces\index{space!Wiener amalgam} associated with the semidirect product $\bbR_+ \ltimes \bbR^n$ corresponding to the dilation action of the multiplicative group $\bbR_+$ on $\bbR^n$.
  Topologically $\bbR_+ \ltimes \bbR^n = \bbR^{1+n}_+$, and the group operation is given by $(t,x) \cdot (s,y) := (ts, x + ty)$.
  Thus many of the properties above can be deduced from properties of abstract Wiener amalgam spaces.
  For a review of these spaces, see \cite{cH03} and the references therein.
  However, if we were to use Wiener amalgam space arguments, we would not obtain absolute convergence of $L^2$ duality pairings (only abstract duality pairings).
  Furthermore, these arguments would not show the connection with tent spaces.
  On the other hand, this observation shows a previously unrecognised link between tent spaces and certain Wiener amalgam spaces, which is interesting in its own right, and which merits further investigation.
\end{rmk}

\section{Unification: tent spaces, $Z$-spaces, and slice spaces}\label{sec:unification}

Tent spaces and $Z$-spaces share the same fundamental properties.
To make this explicit, we write $X$ as a placeholder for either $T$ or $Z$ when a statement holds for both tent spaces and $Z$-spaces.\index{space!X-@$X$-}
When considering two different spaces, either of which can be a tent space or a $Z$-space independently, we use subscripts $X_0$, $X_1$.
For example, one can concisely write the conclusions of Theorem \ref{thm:ts-rint-full} and Proposition \ref{prop:Z-rint-full} as
\begin{equation*}
	(X^\mb{p}, X^\mb{q})_{\gh,p_\gh} = Z^{[\mb{p},\mb{q}]_\gh},
\end{equation*}
and the tent space and $Z$-space duality results can be written extremely concisely as
\begin{equation*}
	(X^\mb{p})^\prime \simeq X^{\mb{p}^\prime} \qquad \text{($\mb{p}$ finite).}
\end{equation*}
In this section we establish further properties of tent spaces and $Z$-spaces, including some relations between the two.

First, we point out that for all $s \in \bbR$ we have $X^2_s =  L^2_s(\bbR^{1+n}_+)$, where
\begin{equation}\label{eqn:L2sdefn}
	\nm{f}_{L^2_s} := \nm{\gk^{-s} f}_{L^2(\bbR^{1+n}_+)}.
\end{equation}

The following embedding theorem extends Theorem \ref{thm:ts-embeddings} not only to $Z$-spaces, but also to combinations of tent and $Z$-spaces.

\begin{thm}[Mixed embeddings]\label{thm:mixed-emb}\index{embedding!of tent- and Z-spaces@of tent- and $Z$-spaces}
	Let $X_0,X_1 \in \{T,Z\}$ and let $\mb{p} \hookrightarrow \mb{q}$ with $\mb{p} \neq \mb{q}$.
	Then we have the embedding
	\begin{equation*}
	 	(X_0)^\mb{p} \hookrightarrow (X_1)^\mb{q}.
	\end{equation*}
\end{thm}

\begin{proof}
	When $X_0 = X_1 = T$, this is Theorem \ref{thm:ts-embeddings}.
	
	Let $\mb{r} = [\mb{p}, \mb{q}]_2$, so that $\mb{p} \hookrightarrow \mb{r}$ and $[\mb{p},\mb{r}]_{1/2} = \mb{q}$ (by Lemmas \ref{lem:expext} and \ref{lem:extid}).
	Then we have embeddings $T^\mb{p} \hookrightarrow T^{\mb{p}}$ (trivially) and $T^\mb{p} \hookrightarrow T^\mb{r}$ (Theorem \ref{thm:ts-embeddings}).
	Hence
	\begin{equation*}
		T^\mb{p} \hookrightarrow (T^\mb{p}, T^\mb{r})_{1/2,i([\mb{p},\mb{r}]_{1/2})} = Z^{[\mb{p},\mb{r}]_{1/2}} = Z^\mb{q}
	\end{equation*}
	by Theorem \ref{thm:ts-rint-full}, using that $\mb{p} \neq \mb{q}$ and $\mb{p} \hookrightarrow \mb{q}$ imply $\gq(\mb{p}) \neq \gq(\mb{q})$.
	Similarly, putting $\mb{s} = [\mb{p},\mb{q}]_{-1}$, we have $T^\mb{s} \hookrightarrow T^{\mb{q}}$ and $T^{\mb{q}} \hookrightarrow T^{\mb{q}}$, so
	\begin{equation*}
		Z^\mb{p} = (T^\mb{s}, T^\mb{q})_{1/2,i([\mb{s},\mb{q}]_{1/2})} \hookrightarrow T^\mb{q}.
	\end{equation*}
	Finally, putting $\mb{t} = [\mb{p},\mb{q}]_{1/2}$ and using the previous results, we have
	\begin{equation*}
		Z^\mb{p} \hookrightarrow T^\mb{t} \hookrightarrow Z^\mb{q},
	\end{equation*}
	which completes the proof.
\end{proof}

We also have convenient mixed embeddings for fixed exponents.

\begin{lem}\label{lem:TZ-fixedexp-emb} %former lem:TZ-inf-emb
	If $i(\mb{p}) \leq 2$ then
	\begin{equation*}
		Z^\mb{p} \hookrightarrow T^\mb{p},
	\end{equation*}
	and if $i(\mb{p}) \geq 2$ (and in particular if $\mb{p}$ is infinite) then
	\begin{equation*}
		T^{\mb{p}} \hookrightarrow Z^\mb{p}
	\end{equation*}
\end{lem}

\begin{proof}
	The first embedding is proven in \cite[Corollary 2.16]{aA16}.
	The second then follows by duality.
\end{proof}

\begin{prop}[Density of intersections]\label{prop:mixed-x-density}
	Let $\mb{p}$ and $\mb{q}$ be exponents.
	If $\mb{p}$ is finite then $(X_0)^\mb{p} \cap (X_1)^\mb{q}$ is dense in $(X_0)^{\mb{p}}$.
	Otherwise, $(X_0)^\mb{p} \cap (X_1)^\mb{q}$ is weak-star dense in $(X_0)^{\mb{p}}$.
\end{prop}

\begin{proof}
	This follows immediately from the fact that $L^2_c(\bbR^{1+n}_+)$ is (weak-star) dense in $T^\mb{r}$ for (infinite) exponents $\mb{r}$, and likewise in $Z^\mb{r}$ (this can be proven directly, or by real interpolation, or by the embeddings of Theorem \ref{thm:mixed-emb}).
\end{proof}
	
	For all $r \in \bbR_+$, define a `downward shift'\index{downward shift} operator $S_r$ on $L^0(\bbR^{1+n}_+)$ by
	\begin{equation*}
		(S_r f)(t,y) := f(t+r,y) 
	\end{equation*}
	for all $f \in L^0(\bbR^{1+n}_+)$.
	These operators are well-behaved on certain tent spaces and $Z$-spaces, as shown in the following proposition.
	
	\begin{prop}[Uniform boundedness of downward shifts]\label{prop:s12shift}
		Let $\mb{p}$ be an exponent.
		\begin{enumerate}[(i)]
			\item
			If $i(\mb{p}) \leq 2$ and $\gq(\mb{p}) < -1/2$, then the operators $(S_r)_{r \in \bbR_+}$ are uniformly bounded on $X^\mb{p}$.
			\item
			If $i(\mb{p}) \in (2,\infty]$ and $r(\mb{p}) < -(n+1)/2$, then the operators $(S_r)_{r \in \bbR_+}$ are uniformly bounded on $X^\mb{p}$.
		\end{enumerate}
	\end{prop}
	
	\begin{rmk}
	Note that the assumptions for $i(\mb{p}) \leq 2$ and $i(\mb{p}) > 2$ are quite different: there is a sudden jump in dimensional dependence at $i(\mb{p}) > 2$.
	We do not have a good explanation for this behaviour, and there is no interpolation procedure to obtain stronger results when $2 < i(\mb{p}) < \infty$.
	Note that we can include endpoints when considering tent spaces (i.e. we can include $\gq(\mb{p}) = -1/2$ or $r(\mb{p}) = -(n+1)/2$ respectively).
	However, to realise the spaces $Z^\mb{p}$ as interpolants of tent spaces, we need to interpolate between tent spaces $T^{\mb{p}_0}$ and $T^{\mb{p}_1}$ with $\gq(\mb{p}_0) \neq \gq(\mb{p}_1)$, and so the endpoint $Z$-space results cannot be proven by this argument.
	\end{rmk}
	
	\begin{proof}[Proof of Proposition \ref{prop:s12shift}]
		It suffices to prove the tent space results, as the $Z$-space results follow by real interpolation.
		
		First we prove boundedness on tent spaces for $i(\mb{p}) \leq 1$ and for $i(\mb{p}) = 2$; the rest of part (i) follows by complex interpolation.
		Suppose $\mb{p} = (p,s)$ with $p \leq 1$ and let $a$ be a $T^\mb{p}$-atom associated with a ball $B$ of radius $r_B$.
		Then $S_r a$ is supported on $T(B)$, and we have
		\begin{align*}
			\nm{S_r a}_{T^2_s} &\simeq \left( \int_0^{r_B - r} \int_{\bbR^n} t^{-2s-1} |a(t+r,x)|^2 \, dx \, dt \right)^{1/2} \\
			&\leq \left( \int_0^{r_B - r} \int_{\bbR^n} (t+r)^{-2s-1} |a(t+r,x)|^2 \, dx \, dt \right)^{1/2}  \\
			&= \left( \int_r^{r_B} \int_{\bbR^n} |\gt^{-s} a(\gt,x)|^2 \, dx \, \frac{d\gt}{\gt} \right)^{1/2}  \\
			&\leq \nm{a}_{T_s^2}  \\
			&\leq |B|^{\gd_{p,2}} 
		\end{align*}
		using that $-2s-1 > 0$.
		Therefore $S_r a$ is, up to a uniform constant, a $T^\mb{p}$-atom associated with $B$.
		Hence if $f = \sum_{k \in \bbN} \gl_k a_k$ is an atomic decomposition of $f$ in $T^\mb{p}$, then $S_r f = \sum_{k \in \bbN} \gl_k (S_r a_k)$ is an atomic decomposition of $S_r f$ in $T^\mb{p}$ up to a uniform constant.
		Therefore the operators $(S_r)_{r \in \bbR_+}$ are uniformly bounded on $T^\mb{p}$.
		A similar argument (without need of atoms) works for $\mb{p} = (2,s)$ provided $s < -1/2$.
		
		Now let $\mb{p} = (p,s)$ with $p \in (2,\infty)$ and $s < -(n+1)/2$, and fix $f \in L^0(\bbR^{1+n}_+)$.
		First we estimate $S_r f$ in $T^\mb{p}$:
		\begin{align*}
			\nm{ S_r f }_{T^\mb{p}}
			&= \bigg( \int_{\bbR^n} \bigg( \dint_{\gG(x)} t^{-2s - n - 1} |f(t+r,y)|^2 \, dy \, dt \bigg)^{p/2} \, dx \bigg)^{1/p} \\
			&\leq \bigg( \int_{\bbR^n} \bigg( \dint_{\gG(x)} (t+r)^{-2s-n-1} |f(t+r,y)|^2 \, dy \, dt \bigg)^{p/2} \, dx \bigg)^{1/p} \\
			&= \bigg( \int_{\bbR^n} \bigg( \dint_{\gG(x) + r} \gt^{-2s-n-1} |f(\gt,y)|^2 \, dy \, d\gt \bigg)^{p/2} \, dx \bigg)^{1/p} \\
			&\leq \nm{ f }_{T^\mb{p}}
		\end{align*}
		using that $-2s-n-1 > 0$ and $\gG(x) + r \subset \gG(x)$, where $\gG(x) + r$ is the `vertically translated cone'
		\begin{equation*}
			\gG(x) + r := \{(t,y) \in \bbR^{1+n}_+ : (t-r,y) \in \gG(x)\}.
		\end{equation*}
		This proves part (ii) in the case where $\mb{p}$ is finite.
		
		Now suppose $\mb{p} = (\infty,s;\ga)$ with $s+\ga < -(n+1)/2$, and let $B = B(c,R) \subset \bbR^n$ be a ball.
		If $r \leq R$, then we can write
		\begin{align*}
			&R^{-\ga-n/2} \bigg( \dint_{T(B)} t^{-2s-1} |f(t+r,y)|^2 \, dy \, dt \bigg)^{1/2} \\
			&\leq R^{-\ga-n/2} \bigg( \dint_{T(B)} (t+r)^{-2s-1} |f(t+r,y)|^2 \, dy \, dt \bigg)^{1/2} \\
			&\lesssim (R+r)^{-\ga-n/2} \bigg( \dint_{T(B(c,R+r))} \gt^{-2s-1} |f(\gt,y)|^2 \, dy \, d\gt \bigg)^{1/2} \\
			&\lesssim \nm{f}_{T^\infty_s}
		\end{align*}
		using that $-2s-1 > 0$.
		If $r > R$ then instead we write
		\begin{align*}
			&R^{-\ga-n/2} \bigg( \dint_{T(B)} t^{-2s-1} |f(t+r,y)|^2 \, dy \, dt \bigg)^{1/2} \\
			&= R^{-\ga-n/2} \bigg( \dint_{T(B) + r} \gt^{-2s-1} \bigg( \frac{\gt-r}{\gt} \bigg)^{-2s-1} |f(\gt,y)|^2 \, dy \, d\gt \bigg)^{1/2} \\
			&\leq R^{-\ga-n/2} \bigg( \frac{R}{r} \bigg)^{-s-1/2} \bigg( \dint_{T(B) + r} \gt^{-2s-1} |f(\gt,y)|^2 \, dy \, d\gt \bigg)^{1/2} \\
			&\leq \bigg( \frac{R+r}{R} \bigg)^{\ga + n/2} \bigg( \frac{R}{r} \bigg)^{-s-1/2} \nm{f}_{T^\infty_\ga} \\
			&\lesssim \nm{f}_{T^\infty_\ga}
		\end{align*}
		using that $s + \ga \leq -(n+1)/2$ in the last line, where $T(B) + r$ is defined analogously to $\gG(x) + r$.
		These estimates imply that $\nm{ S_r f }_{T^\infty_\ga} \lesssim \nm{ f }_{T^\infty_\ga}$ as desired, completing the proof.	
	\end{proof}
	
	Now we define the \emph{slice spaces}.\index{space!slice}
	These were introduced in connection with tent spaces and boundary value problems by the second author and Mourgoglou \cite{AM15}.
	The name comes from the fact that functions in slice spaces are, roughly speaking, horizontal `slices' of functions in tent or $Z$-spaces (this is made precise in Proposition \ref{prop:slice-ret}).
	
	\begin{dfn}
		Suppose $\mb{p}$ is an exponent and $t > 0$.
		For $f \in L^0(\bbR^n)$ we define
		\begin{equation*}
			\nm{f}_{E^\mb{p}(t)} := t^{-r(\mb{p})} \nm{x \mapsto \nm{f}_{L^2(B(x,t),dy/t^n)}}_{L^{i(\mb{p})}(\bbR^n)}.
		\end{equation*}
		These quasinorms define the \emph{slice spaces}
		\begin{equation*}
			E^{i(\mb{p})}_{r(\mb{p})}(t) = E^{\mb{p}}(t) = E^{\mb{p}}(t)(\bbR^n) := \{f \in L^0(\bbR^n) : \nm{f}_{E^{\mb{p}}(t)} < \infty\}.
		\end{equation*}
	\end{dfn}
		
	For $f \in L^0(\bbR^n)$, $t > 0$, and $h > 3/2$ (this restriction corresponds to that in the definition of Whitney parameter), define $\gi_{t,h}(f) \in L^0(\bbR^{1+n}_+)$ by setting 
	\begin{equation*}
		\gi_{t,h}(f)(s,x) := f(x) \mb{1}_{[t,ht]}(s)
	\end{equation*}	
	for all $(s,x) \in \bbR^{1+n}_+$, and for $g \in L^0(\bbR^{1+n}_+)$ define $\gp_t(g) \in L^0(\bbR^n)$ by 
	\begin{equation*}
		\gp_{t,h}(g)(x) := \int_t^{ht} g(s,x) \, \frac{ds}{s}.
	\end{equation*}\
	for all $x \in \bbR^n$.
	
	\begin{prop}\label{prop:slice-ret}
	For all exponents $\mb{p}$, the operators
	\begin{equation*}
		E^\mb{p}(t) \stackrel{\gi_{t,h}}{\longrightarrow} X^\mb{p} \stackrel{\gp_{t,h}}{\longrightarrow} E^\mb{p}(t),
	\end{equation*}
	are bounded uniformly in $t$.
	Furthermore, the compositions of these operators are identity maps.
	\end{prop}
	
	\begin{proof}
		The tent space results with $\gq(\mb{p}) = 0$ are already stated in \cite[\textsection 3]{AM15}; the extension to all tent spaces is simple.
		Likewise, the composition statement is clear.
		The proof for $Z$-spaces is a straightforward (one page) argument that we omit.
	\end{proof}
	
	Therefore we can view the spaces $E^\mb{p}(t)$ as retracts of $X^\mb{p}$.
	Consequently, properties of tent spaces and $Z$-spaces descend to slice spaces.
	
	\begin{prop}\label{prop:slice-change-of-parameters}
		If $0 < t_0,t_1 < \infty$ and $\mb{p}$, $\mb{q}$ are exponents with $i(\mb{p}) = i(\mb{q})$, then $E^\mb{p}(t_0) = E^\mb{q}(t_1)$ with equivalent quasinorms.
	\end{prop}
	
	This follows from change of aperture for tent spaces \cite[Lemma 3.5]{AM15}.
	For $p \in (0,\infty]$ we write $E^p := E^\mb{p}(1)$ for any $\mb{p}$ with $i(\mb{p}) = p$; all $E^\mb{p}(t)$ quasinorms are equivalent to the $E^p$ quasinorm (but not uniformly in $t$ or $\mb{p}$).
	
	We have a duality theorem for slice spaces, and of course this includes absolute convergence of the $L^2$ duality pairing (now on $\bbR^n$ rather than $\bbR^{1+n}_+$).
	This is proven in \cite[Lemma 3.2]{AM15}.
	
	\begin{prop}[Duality]\label{prop:slice-duality}\index{duality!of slice spaces}
		Fix $t > 0$ and let $\mb{p}$ be a finite exponent.
		Then we have
		\begin{equation}\label{eqn:slice-duality-estimate}
			\int_{\bbR^n} |(f(x), g(x))| \, dx \lesssim \nm{f}_{E^\mb{p}} \nm{g}_{E^{\mb{p}^\prime}},
		\end{equation}
		and the $L^2(\bbR^n)$ duality pairing identifies the Banach space dual of $E^\mb{p}$ with  $E^{\mb{p}^\prime}$.
	\end{prop}
	
	The tent space and $Z$-space embedding results also descend to slice spaces, though for slice spaces the `regularity' parameters are not important.
	
	\begin{prop}[Embeddings]\label{prop:slice-embeddings}\index{embedding!of slice spaces}
		Suppose $0 < p_0 \leq p_1 \leq \infty$.
		Then $E^{p_0} \hookrightarrow E^{p_1}$.
	\end{prop}
	
	\begin{proof}
		Fix $\mb{p}_0$ and $\mb{p}_1$ with $i(\mb{p}_0) = p_0$, $i(\mb{p}_1) = p_1$, and $\mb{p}_0 \hookrightarrow \mb{p}_1$.
		Then we have bounded operators
		\begin{equation*}
			E^{p_0} \stackrel{\gi_{1,2}}{\longrightarrow} X^{\mb{p}_0} \hookrightarrow X^{\mb{p}_1} \stackrel{\gp_{1,2}}{\longrightarrow} E^{p_1}
		\end{equation*}
		whose composition is the identity map, with the inclusion following from Theorem \ref{thm:mixed-emb}.
	\end{proof}
	
	Slice spaces contain the Schwartz functions, and are contained in the space of tempered distributions.
	This is contained in \cite[Lemma 3.6]{AM15} (only the case $p < \infty$ is included there, but the argument extends to $p = \infty$).
	
	\begin{prop}\label{prop:slice-distribution}
		For all $p \in (0,\infty]$ we have $\mc{S} \subset E^p \subset \mc{S}^\prime$.
	\end{prop}
	
	We also have a straightforward integration by parts formula for functions in slice spaces.
	This is part of \cite[Lemma 3.8]{AM15}.
	
	\begin{prop}[Integration by parts in slice spaces]\label{prop:slice-IBP}
		Let $\mb{p}$ be a finite exponent and suppose that $\partial$ is a first-order differential operator with constant coefficients, and let $\partial^*$ be the adjoint operator for the $L^2$ duality pairing on $\bbR^n$.
		If $f, \partial f \in E^\mb{p}$ and $g, \partial^* g \in E^{\mb{p}^\prime}$, then
		\begin{equation*}
			\int_{\bbR^n} ( \partial f(x),  g(x) ) \, dx = \int_{\bbR^n} ( f(x), \partial^* g(x) ) \, dx.
		\end{equation*}
	\end{prop}
	
	Finally, we have an equivalent dyadic quasinorm for the slice spaces.
	This follows from the dyadic characterisation of $Z$-spaces (Proposition \ref{prop:Zpqs-seq-equivalence} and the remark following it) and Proposition \ref{prop:slice-ret}.
	
	\begin{prop}[Dyadic characterisation]\label{prop:slice-discrete}\index{dyadic characterisation!of slice space}
		For all $p \in (0,\infty]$ and $k \in \bbZ$ we have
		\begin{equation*}
			\nm{f}_{E^p} \simeq \nm{(\nm{f}_{L^2(Q)})_{Q \in \mc{D}_k}}_{\ell^p(\mc{D}_k)} 
		\end{equation*}
		where $\mc{D}_k$ is any grid of dyadic cubes in $\bbR^n$ with sidelength $2^k$.
	\end{prop}
	
	\begin{rmk}\label{rmk:e-w}
		The slice spaces $E^p$ are equal to the Wiener amalgam spaces\index{space!Wiener amalgam} $W(L^2, L^{p})(\bbR^n)$ when $p \geq 1$ (see \cite{cH03} and the references therein).
		Therefore, as with $Z$-spaces, many properties of slice spaces can be deduced from properties of Wiener amalgam spaces, but this would not emphasise the connection with tent- and $Z$-spaces.
	\end{rmk}

\section{Homogeneous smoothness spaces}\label{sec:hss}

We only give a quick definition of these, and state specific properties that we will need.
These definitions are special cases of the Littlewood--Paley definitions of Triebel--Lizorkin and Besov spaces, which we do not need in full generality.
For a quick introduction the reader can consult Grafakos \cite[Chapter 6]{lGMFA}; the less time-pressed reader may consult Triebel's oeuvre (for example \cite[\textsection 5]{hT83}).

Let $\mc{Z}(\bbR^n) \subset \mc{S}(\bbR^n)$ be the set of Schwartz functions $f$ such that $D^\ga f(0) = 0$ for every multi-index $\ga$, and let $\mc{Z}^\prime(\bbR^n)$ be the topological dual of $\mc{Z}(\bbR^n)$. The space $\mc{Z}^\prime(\bbR^n)$ can be identified with the quotient space $\mc{S}^\prime(\bbR^n) / \mc{P}(\bbR^n)$, where $\mc{P}(\bbR^n)$ is the space of polynomials on $\bbR^n$.\index{tempered distributions modulo polynomials}
In this section all our spaces will be of objects defined on $\bbR^n$, so we omit this from our notation, writing $\mc{Z}$, $\mc{Z}^\prime$, and so on.
We denote the Fourier transform and inverse Fourier transform of a function $\gf$ by $\hat{\gf}$ and $\check{\gf}$ respectively.

\begin{dfn}\label{dfn:classicalHSB}
Let $\gY \in \mc{S}$ be a radial bump function with
	\begin{equation*}
		\hat{\gY} \geq 0, \quad \supp \hat{\gY} \subset A(0,6/7,2), \quad \text{and} \quad\hat{\gY}|_{A(0,1,12/7)} = 1
	\end{equation*}
(of course these precise parameters are not so important),
and for $j \in \bbZ$ let $\gD_j$ denote the associated Littlewood--Paley operators $\gD_j(f) := f \ast (2^{jn} \gY(2^j \cdot))$. 
	
For $f \in \mc{Z}^\prime$, $\ga \in \bbR$, and $0 < p < \infty$ define
\begin{equation*}
	\nm{f}_{\dot{H}^p_\ga} := \nm{ \nm{j \mapsto 2^{j\ga} (\gD_j f)(\cdot)}_{\ell^2(\bbZ)} }_{L^p},
\end{equation*}	
and for $0 < p \leq \infty$ define
\begin{equation*}
	\nm{f}_{\dot{B}^{p,p}_\ga} := \nm{j \mapsto 2^{j\ga} \nm{\gD_j(f)}_{L^p}}_{\ell^p(\bbZ)}.
\end{equation*}
The homogeneous \emph{Hardy--Sobolev spaces}\index{space!Hardy--Sobolev} $\dot{H}_\ga^p = \dot{F}_\ga^{p,2}$ and \emph{Besov spaces}\index{space!Besov} $\dot{B}_\ga^{p,p}$ are then the sets of those $f \in \mc{Z}^\prime$ for which the corresponding quasinorms are finite.

These quasinorms are independent of the choice of $\gY$ (up to equivalence), and $\dot{H}_\ga^p$ and $\dot{B}_\ga^{p.p}$ are Banach spaces (quasi-Banach when $p < 1$).
\end{dfn}

The Hardy--Sobolev spaces $\dot{H}_\ga^p$ may be characterised by Fourier multipliers in tent spaces: this characterisation can be found in \cite[page 180]{hT88}.

\begin{thm}\label{thm:Hp-tent}
	Suppose $p \in (0,\infty)$ and $\ga \in \bbR$.
	Fix $a > n/\min(p,2)$ and $s_0,s_1 \in \bbR$ with $s_0 + a < s < s_1$.
	Fix also two Schwartz functions $h,H \in \mc{S}$ with
	\begin{equation*}
		\begin{array}{ll}
		\supp h \subset \{x \in \bbR^n : |x| \leq 2\}, & \text{$h(x) = 1$ if $|x| \leq 1$,} \\
		\supp H \subset \{x \in \bbR^n : 1/4 \leq |x| \leq 4\}, & \text{$H(x) = 1$ if $1/2 \leq |x| \leq 2$.}
		\end{array}
	\end{equation*}
	Let $\map{\gy}{\bbR^n \sm \{0\}}{\bbC}$ be a measurable function satisfying the conditions
	\begin{align}
		|\gy(x)| > 0 \qquad \text{if $1/2 \leq |x| \leq 2$,}& \label{line:triebel-1}\\
		\int_{\bbR^n} \left| \left( \frac{\gy(\cdot) h(\cdot)}{|\cdot|^{s_1}} \right)^{\vee}(y) \right| (1 + |y|)^a \, dy &< \infty, \label{line:triebel-2}\\
		\sup_{m = 1,2,\ldots} 2^{-ms_0} \int_{\bbR^n} \left| (\gy(2^m \cdot)H(\cdot))^{\vee}(y) \right| (1 + |y|)^a \, dy &< \infty. \label{line:triebel-3}
	\end{align}
	Then for all $f \in \mc{Z}^\prime$,
	\begin{equation*}
		\nm{t \mapsto (\gy(t\cdot)\hat{f})^\vee }_{T^p_\ga} \simeq \nm{f}_{\dot{H}_\ga^p}.
	\end{equation*}
\end{thm}
	
Conditions \eqref{line:triebel-2} and \eqref{line:triebel-3} are a bit complicated, but if $\check{\gy} \in \mc{Z}$ then they are automatically satisfied for any choice of $a$, $s_0$, $s_1$, $h$, and $H$.
In the following corollary we rephrase Theorem \ref{thm:Hp-tent} in convolution form with this simplification.
This is a consequence of the identity $\gf \ast f = (\hat{\gf} \hat{f})^\vee$.
	
\begin{cor}\label{cor:Hp-tent-easy}
	Suppose $p \in (0,\infty)$ and $\ga \in \bbR$.
	Let $\gf \in \mc{Z}(\bbR^n)$ be such that $\hat{\gf}(\gx) > 0$ if $1/2 \leq |\gx| \leq 2$.
	Then for all $f \in \mc{Z}^\prime$,
	\begin{equation*}
		\nm{t \mapsto t^{-n} \gf(t^{-1} \cdot) \ast f}_{T^p_\ga} \simeq \nm{f}_{\dot{H}_\ga^p}.
	\end{equation*}
\end{cor}

We extend this to more general smoothness spaces in Theorem \ref{thm:BHS-Xspace}.
Before that, we will go through some other equivalent norms.

For $f \in \mc{Z}^\prime$ and $\ga \in \bbR$, the \emph{Riesz potential}\index{Riesz potential} $I_\ga f \in \mc{Z}^\prime$, given by
\begin{equation*}
	I_\ga f(x) := (|\cdot|^{-\alpha} \hat{f}(\cdot))^{\vee}(x) \qquad (x \in \bbR^n),
\end{equation*}
is well-defined.
These can be used to characterise Hardy--Sobolev spaces.

\begin{thm}\label{thm:HS-rieszpot} 
	Suppose $p \in (0,\infty)$, $\ga \in \bbR$, and $f \in \mc{Z}^\prime$.
	Then $f \in \dot{H}_\ga^p$ if and only if $I_{-\ga} f \in L^p$, and $f \mapsto \nm{I_{-\ga} f}_{L^p}$ is an equivalent norm on $\dot{H}_\ga^p$.
	Furthermore, for all $s \in \bbR$, $I_\ga$ is an isomorphism from $\dot{H}_s^p$ to $\dot{H}_{s+\ga}^p$.
\end{thm}

We can also characterise Hardy--Sobolev and Besov spaces by integrals of differences.
For all $p \in [1,\infty]$, $g \in L^0$, and $s \in \bbR$, define
\begin{equation*}
	\mc{D}_s^p g(x) := \left( \int_{\bbR^n} \frac{|g(x+y) - g(x)|^p}{|y|^{n+ps}} \, dy \right)^{1/p} \qquad (x \in \bbR^n).
\end{equation*}

\begin{thm}\label{thm:stein}
	Suppose $\ga \in (0,1)$ and $p \in (2n/(n+\ga), \infty)$.
	Then for all $f \in L^2$,
	\begin{equation}\label{eqn:stlem}
		\nm{f}_{\dot{H}_\ga^p} \simeq \nm{\mc{D}_\ga^2 f}_{L^p}.
	\end{equation}
\end{thm}

\begin{proof}
	Whenever $f = I_\ga \gf$ for some $\gf \in C_c^\infty$, the estimate \eqref{eqn:stlem} follows from a lemma of Stein \cite[Lemma 1]{eS61} combined with the Riesz potential characterisation of $\dot{H}_\ga^p$ (Theorem \ref{thm:HS-rieszpot}).
	A density argument, using the fact that elements of $\dot{H}_\ga^p$ may be represented as $L^2_\text{loc}$ functions when $\ga \in (0,1)$, completes the proof.
\end{proof}

The corresponding characterisation for Besov spaces is in \cite[Theorem 5.2.3.2]{hT83}.

\begin{thm}\label{thm:besov-difference}
	Suppose $\ga \in (0,1)$ and $p \in [1, \infty)$.
	Then for all $f \in L^2$,
	\begin{equation*}
		\nm{f}_{\dot{B}^{p,p}_\ga} \simeq \nm{\mc{D}^p_\ga f}_{L^p}.
	\end{equation*}
\end{thm}

For $\alpha \in (0,1)$, the Besov space $\dot{B}_\ga^{\infty,\infty}$ may be identified with the more familiar homogeneous \emph{H\"older--Lipschitz space}\index{space!H\"older--Lipschitz} $\dot{\gL}_\ga$: this is the space of functions $f$ on $\bbR^n$ such that
\begin{equation*}
	\nm{f}_{\dot{\gL}_\ga} := \sup_{x,y \in \bbR^n} \frac{|f(x) - f(y)|}{|x-y|^\ga} < \infty,
\end{equation*}
modulo constants.
Such functions are continuous.

\begin{rmk}
  We do not obtain any well-posedness results for Besov spaces $\dot{B}_\alpha^{\infty,\infty}$ with $\alpha \notin (0,1)$, but nonetheless these spaces fit into our abstract framework.
\end{rmk}

We must also consider the Triebel--Lizorkin spaces $\dot{F}^{\infty,2}_\ga$ for $\ga \in \bbR$, which are the subspaces of $\mc{Z}^\prime$ determined by the quasinorms
\begin{equation*}
	\nm{f}_{\dot{F}_\ga^{\infty,2}} := \inf \nm{ \nm{j \mapsto 2^{j\ga} |f_j(\cdot)|}_{\ell^2(\bbZ)}}_{L^\infty},
\end{equation*}
where the infima are taken over all decompositions 
\begin{equation*}
	f = \sum_{j \in \bbZ} \gD_j f_j \qquad \text{in $\mc{Z}^\prime$}
\end{equation*}
with each $f_j \in L^\infty$, where $\gD_j$ are Littlewood--Paley operators as in Definition \ref{dfn:classicalHSB}. 
For all $\ga \in \bbR$ we refer to these spaces as homogeneous \emph{$\BMO$-Sobolev spaces}\index{space!BMO-Sobolev@$\BMO$-Sobolev} $\dot{\BMO}_\ga := \dot{F}_\ga^{\infty,2}$.
When $\ga \geq 0$, $\dot{\BMO}_\ga$ is usually defined as the image of $\BMO$ under the Riesz potential $I_\ga$ defined above, as a subspace of $\mc{Z}^\prime$, with a corresponding norm; this definition agrees with the Triebel--Lizorkin definition that we just made.
Of course $\dot{\BMO}_0 = \BMO$.
See Strichartz \cite{rS80} and Triebel \cite[\textsection 5.1.4]{hT83} for further information.
In particular, we have the following characterisation of $\dot{\BMO}_\ga$ for $\ga \in (0,1)$ \cite[Theorem 3.3]{rS80}.

\begin{thm}\label{thm:strichartz-bmo}
	Suppose $\ga \in (0,1)$.
	Then for all $f \in L^2$ we have
	\begin{equation*}
		\nm{f}_{\dot{\BMO}_\ga} \simeq \sup_Q \left( \frac{1}{|Q|} \int_Q \int_Q \frac{|f(x)-f(y)|^2}{|x-y|^{n+2\ga}} \, dy \, dx \right)^{1/2},
	\end{equation*}
	where the supremum may be taken over all cubes or all balls.
\end{thm}

We introduce some unconventional but useful notation for these spaces.
For a finite exponent $\mb{p} = (p,s)$, define
\begin{equation*}
	\mb{H}^\mb{p} := \dot{H}^p_s = \dot{F}^{p,2}_s \quad \text{and} \quad \mb{B}^\mb{p} := \dot{B}^{p,p}_s.
\end{equation*}
For $\mb{p} = (\infty,s;0)$, define
\begin{equation*}
	\mb{H}^\mb{p} := \dot{F}^{\infty,2}_s = \dot{\BMO}_s \quad \text{and} \quad \mb{B}^\mb{p} := \dot{B}^{\infty,\infty}_s.
      \end{equation*}
      recalling that $\dot{B}^{\infty,\infty}_s = \dot{\gL}_s$ when $s \in (0,1)$.
Finally, for $\mb{p} = (\infty,s;\ga)$ with $\ga > 0$, define
\begin{equation*}
	\mb{H}^\mb{p} := \mb{B}^\mb{p} := \dot{B}^{\infty,\infty}_{s + \ga}.
\end{equation*}
As a consequence of these definitions and the various duality identifications for classical smoothness spaces, for all finite exponents $\mb{p}$ we have
\begin{equation*}
	(\mb{X}^\mb{p})^\prime \simeq \mb{X}^{\mb{p}^\prime}
\end{equation*}
whenever $\mb{X}$ denotes either $\mb{H}$ or $\mb{B}$.\index{space!smoothness}

We also have the following interpolation theorem.
This is a combination of standard results (see for example Mendez and Mitrea \cite[Theorem 11]{MM00}, Triebel \cite[Theorems 8.1.3 and 8.3.3a]{hT73}), and Bergh and L\"ofstr\"om \cite[Theorem 6.4.5]{BL76}).\footnote{The results cited in \cite{hT73} and \cite{BL76} are for inhomogeneous spaces. As always, the same technique proves the result for homogeneous spaces. To obtain the stated results for Besov spaces with $\gq(\mb{p}) = \gq(\mb{q})$, write $\dot{B}^{p,p}_\gq = \dot{F}^{p,p}_\gq$ and use the interpolation results for Triebel-Lizorkin spaces.}

\begin{thm}\label{thm:int-clas}\index{interpolation!of smoothness spaces}
	Let $\mb{p}$ and $\mb{q}$ be finite exponents, and suppose $\gh \in (0,1)$ and $p_\gh := i([\mb{p},\mb{q}]_\gh)$.
	Then we have
	\begin{equation*}
		[\mb{H}^\mb{p}, \mb{H}^\mb{q}]_\gh = \mb{H}^{[\mb{p},\mb{q}]_\gh}
	\end{equation*}
	and (also allowing infinite exponents)
	\begin{equation*}
		[\mb{B}^\mb{p}, \mb{B}^\mb{q}]_\gh = \mb{B}^{[\mb{p},\mb{q}]_\gh} \qquad \text{and} \qquad (\mb{B}^\mb{p}, \mb{B}^\mb{q})_{\gh,p_\gh} = \mb{B}^{[\mb{p},\mb{q}]_\gh}.
	\end{equation*}
	Furthermore if $\gq(\mb{p}) \neq \gq(\mb{q})$, then we have
	\begin{equation*}
		 (\mb{H}^\mb{p}, \mb{H}^\mb{q})_{\gh,p_\gh} = \mb{B}^{[\mb{p},\mb{q}]_\gh}.
	\end{equation*}
\end{thm}

Combining these duality and interpolation results with Corollary \ref{cor:Hp-tent-easy}, and using our unconventional notation, we obtain a typographically slick characterisation of all the smoothness spaces that we consider.
This theorem and its proof foreshadow our approach to adapted Besov--Hardy--Sobolev spaces.

\begin{thm}\label{thm:BHS-Xspace}
	Let $\mb{p}$ be an exponent, and let $\gf \in \mc{Z}$ be such that $\hat{\gf}(\gx) > 0$ if $1/2 \leq |\gx| \leq 2$.
	Then for all $f \in \mc{Z}^\prime$,
	\begin{equation*}
		\nm{t \mapsto t^{-n} \gf(t^{-1} \cdot) \ast f}_{X^\mb{p}} \simeq \nm{f}_{\mb{X}^\mb{p}},
	\end{equation*}
	where $X = T$ if $\mb{X} = \mb{H}$, and $X = Z$ if $\mb{X} = \mb{B}$.
\end{thm}

\begin{proof}
	Let $Q_\gf$ be the operator $Q_\gf f(t) := t^{-n} \gf(t^{-1} \cdot) \ast f$.
	Corollary \ref{cor:Hp-tent-easy} says that
	\begin{equation}\label{eqn:Q-classic}
		\nm{Q_\gf f}_{T^\mb{p}} \simeq \nm{f}_{\mb{H}^\mb{p}}
	\end{equation}
	when $\mb{p}$ is finite, so in this case $Q_\gf$ is bounded from $\mb{H}^\mb{p}$ to $T^\mb{p}$.
	A quick computation shows that the adjoint $\map{(Q_\gf)^*}{T^{\mb{p}^\prime}}{\mb{H}^{\mb{p}^\prime}}$ with respect to the $L^2(\bbR^{1+n}_+)$ and $L^2(\bbR^n)$ duality pairings is given by
	\begin{equation*}
		((Q_\gf)^*G)(x) = S_{\td{\gf}} G(x) = \int_0^\infty (t^{-n} \td{\gf}(t^{-1} \cdot) \ast G(t,\cdot))(x) \, \frac{dt}{t} \qquad (G \in T^{\mb{p}^\prime})
	\end{equation*}
	where $\td{\gf}(y) := \overline{\gf(-y)}$, and where the integral converges in the weak-star topology on $\mb{H}^{\mb{p}^\prime}$.
	Since $\map{Q_{\td{\gf}}}{\mb{H}^\mb{p}}{T^\mb{p}}$ is bounded for all finite $\mb{p}$, we find that $\map{S_\gf}{T^{\mb{q}}}{\mb{H}^{\mb{q}}}$ is bounded for all $\mb{q}$ with $i(\mb{q}) > 1$.
	One can extend this to all finite $\mb{q}$: this is done by Coifman, Meyer, and Stein for $\gq(\mb{q}) = 0$ \cite[Theorem 6]{CMS85},\footnote{This is stated for compactly supported $\gf$ there, but the argument extends to $\gf \in \mc{Z}$.} and the reduction to this case for general $\mb{q}$ is done by writing
	\begin{equation*}
		(-\gD)^{-s} S_\gf G = t^{-2s} S_{\gy} G
	\end{equation*}
	with $\gy = (-\gD)^{-s} \gf$.
	
	By interpolation, for all $\mb{p}$, $Q_\gf$ is bounded from $\mb{B}^\mb{p}$ to $Z^\mb{p}$ and $S_\gf$ is bounded from $Z^\mb{p}$ to $\mb{B}^\mb{p}$.
	The Calder\'on reproducing formula says that there exists a function $\gy \in \mc{Z}$, also with $\hat{\gy}(\gx) > 0$ when $1/2 \leq |\gx| \leq 2$, such that
	\begin{equation*}
		S_\gy Q_\gf f = f
	\end{equation*}
	for all $f \in \mc{Z}^\prime$, and hence also for all $f \in \mb{H}^\mb{p}$ or $f \in \mb{B}^\mb{p}$.
	It follows that $\mb{X}^\mb{p}$ may be identified with $Q_\gf S_\gy X^\mb{p}$ for all $\mb{p}$, and by the retraction/coretraction interpolation theorem \cite[\textsection 1.2.4]{hT78} we get
	\begin{equation*}
		\nm{Q_\gf f}_{X^\mb{p}} \simeq \nm{f}_{\mb{X}^\mb{p}}
	\end{equation*}
	for all $\mb{p}$ and all $f \in \mc{Z}^\prime$, which completes the proof.
\end{proof}

\begin{rmk}
	We were a bit sketchy in the above proof, as we will carry out the argument in much greater detail, and in an abstract setting, in the following chapters.
	See in particular Section \ref{sec:eco} and Chapter \ref{chap:ahsb}.
      \end{rmk}

      \section{Factorisation and interpolation of $Z$-spaces}\label{sec:Z-space-appendix}

      In this section we prove the complex interpolation result for $Z$-spaces, including the quasi-Banach range (Proposition \ref{prop:Z-cint}).
      We argue by factorisation, exploiting the following `Calder\'on product formula', which is a special case of \cite[Theorem 3.4]{KM98}.
  
      \begin{thm}[Kalton--Mitrea]\label{thm:KM}
        Let $X_0$, $X_1$ be a pair of quasi-Banach function spaces on $\bbR^{1+n}_+$.
        Suppose that both $X_0$ and $X_1$ are A-convex and separable.
        Then
        \begin{equation*}
          [X_0,X_1]_\theta = X_0^{1-\theta} X_1^\theta,
        \end{equation*}
        where $X_0^{1-\theta} X_1^\theta$ is the quasi-Banach function space consisting of those $h \in L^0(\bbR^{1+n}_+)$ such that the quasinorm
        \begin{equation*}
          \nm{h}_{X_0^{1-\theta} X_1^\theta} := \inf\{ \nm{f}_{X_0}^{1-\theta} \nm{g}_{X_1}^\theta : |h| \leq |f|^{1-\theta} |g|^\theta, \, f \in X_0, \, g \in X_1\}
        \end{equation*}
        is finite, with the usual convention that the infimum of an empty set is $+\infty$.
      \end{thm}

A quasi-Banach function space $X$ is A-convex if and only if it is lattice $r$-convex for some $r > 0$ (see \cite[Theorem 2.2]{nK84} and \cite[Theorem 4.4]{nK86}), which means that for all finite collections $f_1,\ldots,f_k \in X$
\begin{equation*}
  \bigg\| \bigg( \sum_{i=1}^k |f_i|^r \bigg)^{1/r} \bigg\|_X \lesssim \bigg( \sum_{i=1}^k \nm{f_i}_X^r \bigg)^{1/r}.
\end{equation*}
See \cite{KM98} for further discussion of A-convexity.

We will use the following extension of the $Z$-space scale.

\begin{dfn}
  For $p,q \in (0,\infty]$ and $s \in \bbR$, we define $Z^{p,q}_s$ to be the set of all $f \in L^0(\bbR^{1+n}_+)$ such that the quasinorm
  \begin{align*}
    \nm{f}_{Z^{p,q}_s} &:= \bigg\| \ell(Q)^{-s} [|f|^q]^{1/q}_{\overline{Q}} \bigg\|_{\ell^p(\mc{G}, \ell(Q)^n)} \\
                         &= \bigg( \sum_{\overline{Q} \in \mc{G}} \ell(Q)^n \big|  \ell(Q)^{-s} [|f|^q]^{1/q}_{\overline{Q}} \big|^p \bigg)^{1/p}
  \end{align*}
  is finite, with the usual modifications when $p$ or $q$ is infinite.
\end{dfn}

When $q=2$, $Z^{p,q}_s$ is just the $Z$-space $Z^{\mb{p}}$ seen through the dyadic characterisation (Proposition \ref{prop:Zpqs-seq-equivalence}), with $\mb{p} = (p,s)$ when $p < \infty$ and $\mb{p} = (\infty,s;0)$ when $p = \infty$.
It is straightforward to prove that $Z^{p,q}_s$ is a quasi-Banach function space (Banach when $p,q \geq 1$), and that $Z^{p,q}_s$ is $\min(p,q)$-convex, hence A-convex.
Furthermore, $Z^{p,q}_s$ is separable when $p,q < \infty$.

We will prove the following interpolation theorem, from which Proposition \ref{prop:Z-cint} follows by taking $q_0 = q_1 = 2$. 
\begin{thm}\label{thm:appendix-Z-Cint}
  	Suppose $s_0,s_1 \in \bbR$, $p_0,p_1,q_0,q_1 \in (0,\infty)$, and $\theta \in (0,1)$.	
	Then
	\begin{equation*}
		[Z^{p_0,q_0}_{s_0}, Z^{p_1,q_1}_{s_1}]_\theta = Z^{p_\theta,q_\theta}_{s_\theta},
	\end{equation*}
	where $1/p_\theta = (1-\theta)/p_0 + \theta/p_1$, $1/q_\theta = (1-\theta)/q_0 + \theta/q_1$, and $s_\theta = (1-\theta)s_0 + \theta s_1$.
\end{thm}
  This was conjectured by Barton and Mayboroda, who proved the result when $p_0,p_1 \geq 1$ and $q_0 = q_1 \geq 1$ \cite[Theorem 4.13]{BM16}.
  To deduce this from Theorem \ref{thm:KM} we need to establish multiplication and factorisation results.

	\begin{dfn}
		Let $s_0,s_1 \in \bbR$ and $p_0,p_1,q_0,q_1 \in (0,\infty]$.
		We write
		\begin{equation*}
		Z^{p,q}_s \leftrightarrow Z^{p_0,q_0}_{s_0} \cdot Z^{p_1,q_1}_{s_1}
		\end{equation*}
		to mean that the following \emph{multiplication} and \emph{factorisation} properties hold:
		\begin{itemize}
			\item  if $f \in Z^{p_0,q_0}_{s_0}$ and $g \in Z^{p_1,q_1}_{s_1}$, then $fg \in Z^{p,q}_s$ with
				\begin{equation}\label{eqn:mult}
				\nm{fg}_{Z^{p,q}_s} \lesssim \nm{f}_{Z^{p_0,q_0}_{s_0}} \nm{g}_{Z^{p_1,q_1}_{s_1}},
				\end{equation}
			\item 
				if $h \in Z^{p,q}_s$ then there exist $F \in Z^{p_0,q_0}_{s_0}$ and $G \in Z^{p_1,q_1}_{s_1}$ such that $h = FG$ and
				\begin{equation}\label{eqn:fact}
				\nm{F}_{Z^{p_0,q_0}_{s_0}}\nm{G}_{Z^{p_1,q_1}_{s_1}} \lesssim \nm{h}_{Z^{p,q}_s}.
				\end{equation}
		\end{itemize}
		We abbreviate the multiplication and factorisation properties as
		\begin{equation*}
			Z^{p,q}_s \leftarrow Z^{p_0,q_0}_{s_0} \cdot Z^{p_1,q_1}_{s_1} \quad \text{and} \quad Z^{p,q}_s \rightarrow Z^{p_0,q_0}_{s_0} \cdot Z^{p_1,q_1}_{s_1}
		\end{equation*}
		respectively.
              This notation is easily extended to more than two factors, and formal computations involving commutativity and associativity (as in the proof of Proposition \ref{prop:factorisation}) are valid.
              \end{dfn}
		
		\begin{rmk}
                  In all the cases that we cover here, the implicit constants in \eqref{eqn:mult} and \eqref{eqn:fact} can be taken to be $1$.
                  This is because we use the dyadic definition of the $Z^{p,q}_s$-quasinorm, and it need not be true for other equivalent quasinorms.
		\end{rmk}

	\begin{prop}[Multiplication]\label{prop:mult}
		Suppose $s_0,s_1 \in \bbR$ and $p_0,p_1,q_0,q_1 \in (0,\infty]$, and let
	\begin{equation*}
		\frac{1}{p} := \frac{1}{p_0} + \frac{1}{p_1}, \quad \frac{1}{q} := \frac{1}{q_0} + \frac{1}{q_1}, \quad s := s_0 + s_1.
	\end{equation*}
	Then $Z^{p,q}_s \leftarrow Z^{p_0,q_0}_{s_0} \cdot Z^{p_1,q_1}_{s_1}$.
	\end{prop}
	
	\begin{proof}
	Let $f \in Z^{p_0,q_0}_{s_0}$ and $g \in Z^{p_1,q_1}_{s_1}$.
	Using H\"older's inequality twice and the assumptions on the exponents we have
	\begin{align*}
		&\nm{ \ell(Q)^{-s} [|fg|^q]^{1/q}_{\overline{Q}} }_{\ell^p(\mc{G}, \ell(Q)^n)} \\
		&\leq \nm{ \ell(Q)^{-s_0} [|f|^{q_0}]^{1/q_0}_{\overline{Q}} \ell(Q)^{-s_1} [|g|^{q_1}]^{1/q_1}_{\overline{Q}} }_{\ell^p(\mc{G}, \ell(Q)^n)} \\
		&\leq \nm{ \ell(Q)^{-s_0} [|f|^{q_0}|^{1/q_0}_{\overline{Q}} }_{\ell^{p_0}(\mc{G},\ell(Q)^n)} \nm{ \ell(Q)^{-s_1} [|f|^{q_1}|^{1/q_1}_{\overline{Q}} }_{\ell^{p_1}(\mc{G},\ell(Q)^n)} \\
		&= \nm{f}_{Z^{p_0,q_0}_{s_0}} \nm{g}_{Z^{p_1,q_1}_{s_1}},
	\end{align*}
	proving the multiplication property.
      \end{proof}

      \begin{lem}[Single-exponent factorisation]\label{lem:intermed-factor}
		Suppose $p_0,p_1 \in (0,\infty]$ and $s_0,s_1\in \bbR$, and let
		\begin{equation*}
			\frac{1}{p} := \frac{1}{p_0} + \frac{1}{p_1} \quad \text{and} \quad s := s_0 + s_1.
		\end{equation*}
		Then
		\begin{equation*}
			Z^{p,\infty}_s \rightarrow Z^{p_0,\infty}_{s_0} \cdot Z^{p_1,\infty}_{s_1}
		\quad \text{and} \quad
			Z^{\infty,p}_s \rightarrow Z^{\infty,p_0}_{s_0} \cdot Z^{\infty,p_1}_{s_1}.
		\end{equation*}
	\end{lem}
	
	\begin{proof}	
		First suppose $f \in Z^{p,\infty}_s$; we may assume without loss of generality that $f$ is nonnegative.
		Let
		\begin{equation}\label{eqn:decomp}
			F := \sum_{\overline{Q} \in \mc{G}}  \mb{1}_{\overline{Q}} \ell(Q)^{-s\frac{p}{p_0} + s_0} f^{p/p_0}
		\quad \text{and} \quad
			G := \sum_{\overline{Q} \in \mc{G}} \mb{1}_{\overline{Q}} \ell(Q)^{s\frac{p}{p_0} - s_0} f^{1 - p/p_0},
		\end{equation}
		so that $FG = f$.
		A straightforward estimate gives
		\begin{equation*}
			\nm{F}_{Z^{p_0,\infty}_{s_0}} \leq \nm{f}_{Z^{p,\infty}_s}^{p/p_0}
			\quad \text{and} \quad
			\nm{G}_{Z^{p_1,\infty}_{s_1}} \leq \nm{f}_{Z^{p,\infty}_s}^{p/p_1},
		\end{equation*}
		and so
		\begin{equation*}
			\nm{F}_{Z^{p_0,\infty}_{s_0}} \nm{G}_{Z^{p_1,\infty}_{s_1}} \leq \nm{f}_{Z^{p,\infty}_s}.
		\end{equation*}
		Now assume $g \in Z^{\infty,p}_s$ is nonnegative and define $F^\prime$ and $G^\prime$ as in \eqref{eqn:decomp}, but with $f$ replaced by $g$.
		The same argument as before yields
		\begin{equation*}
			\nm{F^\prime}_{Z^{\infty,p_0}_{s_0}} \nm{G^\prime}_{Z^{\infty,p_1}_{s_1}} \leq \nm{g}_{Z^{\infty,p}_s},
		\end{equation*}
		completing the proof.
              \end{proof}

                \begin{prop}[Factorisation]\label{prop:factorisation}
  	Suppose $s_0,s_1 \in \bbR$ and $p_0,p_1,q_0,q_1 \in (0,\infty]$, and let
	\begin{equation*}
		\frac{1}{p} := \frac{1}{p_0} + \frac{1}{p_1}, \quad \frac{1}{q} := \frac{1}{q_0} + \frac{1}{q_1}, \quad s := s_0 + s_1.
	\end{equation*}
	Then
	\begin{equation*}
		Z^{p,q}_s \rightarrow Z^{p_0,q_0}_{s_0} \cdot Z^{p_1,q_1}_{s_1}.
	\end{equation*}
  \end{prop}
  
  \begin{proof}
	It suffices to show that
	\begin{equation}\label{eqn:infty-factor}
		Z^{p,q}_s \rightarrow Z^{\infty,q}_{s_0} \cdot Z^{p,\infty}_{s_1},
	\end{equation}
	because by Proposition \ref{prop:mult} and Lemma \ref{lem:intermed-factor} along with \eqref{eqn:infty-factor} we have
	\begin{align*}
		Z^{p,q}_s &\rightarrow Z_{s/2}^{p,\infty} \cdot Z_{s/2}^{\infty,q}
		\\ &\rightarrow Z_{s_0/2}^{p_0,\infty} \cdot Z_{s_1/2}^{p_1,\infty} \cdot Z_{s_0/2}^{\infty,q_0} \cdot Z_{s_0/2}^{\infty,q_1}
		\\ &= Z_{s_0/2}^{p_0,\infty} \cdot Z_{s_0/2}^{\infty,q_0} \cdot Z_{s_1/2}^{p_1,\infty} \cdot Z_{s_1/2}^{\infty,q_1}
		\\ &\rightarrow Z_{s_0}^{p_0,q_0} \cdot Z_{s_1}^{p_1,q_1}.
	\end{align*}
	
	We now prove \eqref{eqn:infty-factor}.	
	Given $h \in Z^{p,q}_s$, let
	\begin{equation*}
		F := \sum_{\overline{Q} \in \mc{G}} \mb{1}_{\overline{Q}} \ell(Q)^{s_0} h / [|h|^q]^{1/q}_{\overline{Q}}
		\quad \text{and} \quad
		 G := \sum_{\overline{Q} \in \mc{G}} \mb{1}_{\overline{Q}} \ell(Q)^{-s_0} [|h|^q]^{1/q}_{\overline{Q}},
	\end{equation*}
	so that $FG = h$.
	We immediately have
	\begin{align*}
		\nm{F}_{Z^{\infty,q}_{s_0}} &= \sup_{\overline{Q} \in \mc{G}} \ell(Q)^{-s_0} \ell(Q)^{s_0} = 1
	\end{align*}
	and
	\begin{align*}
		\nm{G}_{Z^{p,\infty}_{s_1}} &= \nm{\ell(Q)^{-s_1} \ell(Q)^{-s_0} [|h|^q]^{1/q}_{\overline{Q}}}_{\ell^p(\mc{G},\ell(Q)^n)} \\
		&= \nm{\ell(Q)^{-s} [|h|^q]^{1/q}_{\overline{Q}} }_{\ell^p(\mc{G},\ell(Q)^n)} \\
		&= \nm{h}_{Z^{p,q}_s},
	\end{align*}
	proving \eqref{eqn:infty-factor}.
	\end{proof}
  
  	Now we can prove Theorem \ref{thm:appendix-Z-Cint}.
	
	\begin{proof}[Proof of Theorem \ref{thm:appendix-Z-Cint}]
		Since $Z^{p_0,q_0}_{s_0}$ and $Z^{p_1,q_1}_{s_1}$ are A-convex and separable, by Theorem \ref{thm:KM} we have
		\begin{equation*}
			[Z^{p_0,q_0}_{s_0},Z^{p_1,q_1}_{s_1}]_\theta = (Z^{p_0,q_0}_{s_0})^{1-\theta} (Z^{p_1,q_1}_{s_1})^\theta.
		\end{equation*}
		We will show that
		\begin{equation}\label{eqn:prodident}
			(Z^{p_0,q_0}_{s_0})^{1-\theta} (Z^{p_1,q_1}_{s_1})^\theta = Z^{p_\theta,q_\theta}_{s_\theta}.
		\end{equation}
		This follows from Propositions \ref{prop:mult} and \ref{prop:factorisation} in a standard way, but we include the argument for completeness.
		Suppose that $h \in L^0(\bbR^{1+n}_+)$ with $|h| \leq |f|^{1-\theta} |g|^\theta$ for some $f \in Z^{p_0,q_0}_{s_0}$, $g \in Z^{p_1,q_1}_{s_1}$.
		Then by Proposition \ref{prop:mult} we have
		\begin{align*}
			\nm{h}_{Z^{p_\theta,q_\theta}_{s_\theta}} &\leq \nm{|f|^{1-\theta} |g|^\theta}_{Z^{p_\theta,q_\theta}_{s_\theta}} \\
			&\leq \nm{|f|^{1-\theta}}_{Z^{p_0/(1-\theta), q_0/(1-\theta)}_{s_0(1-\theta)}} \nm{|g|^\theta}_{Z^{p_1/\theta, q_1/\theta}_{s_1 \theta}} \\
			&= \nm{f}_{Z^{p_0,q_0}_{s_0}}^{1-\theta} \nm{g}_{Z^{p_1,q_1}_{s_1}}^{\theta},
		\end{align*}
		so by taking the infimum over all such $f$ and $g$ we get
		\begin{equation*}
			\nm{h}_{Z^{p_\theta,q_\theta}_{s_\theta}} \leq \nm{h}_{ (Z^{p_0,q_0}_{s_0})^{1-\theta} (Z^{p_1,q_1}_{s_1})^\theta}.
		\end{equation*}
		Conversely, suppose that $h \in Z^{p_\theta,q_\theta}_{s_\theta}$.
		Then by Proposition \ref{prop:factorisation} there exist $F \in Z^{p_0/(1-\theta), q_0/(1-\theta)}_{s_0(1-\theta)}$ and $G \in Z^{p_1/\theta, q_1/\theta}_{s_1 \theta}$ such that $h = FG$ with
		\begin{align*}
			\nm{h}_{Z^{p_\theta,q_\theta}_{s_\theta}} &\geq \nm{F}_{Z^{p_0/(1-\theta), q_0/(1-\theta)}_{s_0(1-\theta)}} \nm{G}_{Z^{p_1/\theta, q_1/\theta}_{s_1 \theta}} \\
			&= \nm{|F|^{1/(1-\theta)}}_{Z^{p_0,q_0}_{s_0}}^{1-\theta} \nm{|G|^{1/\theta}}_{Z^{p_1,q_1}_{s_1}}^{\theta}.
		\end{align*}
		Setting $f = |F|^{1/(1-\theta)}$ and $g = |G|^{1/\theta}$ we see that $f^{1-\theta}g^\theta = |h|$, $f \in Z^{p_0,q_0}_{s_0}$, and $g \in Z^{p_1,q_1}_{s_1}$, so we find that
		\begin{equation*}
			\nm{h}_{Z^{p_\theta,q_\theta}_{s_\theta}} \geq \nm{h}_{(Z^{p_0,q_0}_{s_0})^{1-\theta} (Z^{p_1,q_1}_{s_1})^\theta}.
		\end{equation*}
		This completes the proof of \eqref{eqn:prodident}, and hence also that of Theorem \ref{thm:appendix-Z-Cint}.
              \end{proof}

              \newpage
      \section{Table of exponents and function spaces}\label{sec:exponent-space-table}

      Here we provide a small `dictionary' detailing precisely what is meant by the data $(X^{\mb{p}}, \mb{X}^{\mb{p}})$.
      Here $\mb{p}$ is an exponent, and $(X,\mb{X})$ is either $(T,\mb{H})$ or $(Z, \mb{B})$.
      The definition of the spaces $T^\mb{p}$, $Z^\mb{p}$, $\mb{H}^{\mb{p}}$, and $\mb{B}^{\mb{p}}$ change depending on whether $\mb{p}$ is finite or infinite, with a special case when $j(\mb{p}) = 0$ (i.e. $\mb{p} = (\infty,s;0)$).
      We advise the reader to keep a copy of this table at hand while reading the monograph.
      In the table we assume that $\alpha > 0$.

      \vspace{1cm}
      
      \begin{center}
        \begingroup
        \renewcommand*{\arraystretch}{1.4}
        \begin{tabular}{|c||c|c|}\hline
         \backslashbox{$\mb{p}$}{$(X,\mb{X})$}
            & $(T,\mb{H})$ & $(Z,\mb{B})$ \\
          \hline\hline
          $(p,s)$ & $(T^p_s, \dot{H}^p_s)$ & $(Z^p_s, \dot{B}^{p,p}_s)$ \\
          \hline
          $(\infty,s;0)$  & $(T^\infty_{s;0},\dot{\BMO}_s)$ & $(Z^\infty_s, \dot{B}_s^{\infty,\infty})$\\
          \hline
          $(\infty,s;\alpha)$ & $(T^\infty_{s;\alpha}, \dot{B}_{s + \alpha}^{\infty,\infty})$ & $(Z^\infty_{s + \alpha} , \dot{B}_{s+\alpha}^{\infty,\infty})$\\
          \hline
        \end{tabular}
        \endgroup
      \end{center}

      \vspace{1cm}

      In Section \ref{chap:diffops} we will define $\mb{X}^{\mb{p}}_D$ to be the image of $\mb{X}^{\mb{p}}$ under the projection $\bbP_D$ onto the closure of the range of the Dirac operator $D$, $\overline{\mc{R}(D)}$, which may be characterised by a tangential curl-free condition.
      Thus $\mb{X}^\mb{p}_D$ may be thought of as the set of tangential curl-free functions in $\mb{X}^\mb{p}$.

\chapter{Operator Theoretic Preliminaries}\label{chap:otp}

\section[Bisectorial operators and holomorphic functional calculus]{Bisectorial operators and holomorphic functional calculus}\label{section:bisectorial-hfc}

The material of this section is not new, but we present it here to fix notation.
Useful standard references are \cite{aM86,MY90,ADM95,mH06}, and a particularly nice recent exposition which focuses on bisectorial operators on reflexive Banach spaces is contained in Egert's thesis \cite[Chapter 3]{mE15}.

Let $0 < \go < \gp/2$.
The open bisector of angle $\go$ is the set
\begin{equation*}
	S_\go := \{z \in \bbC \sm \{0\} : \text{$|\arg(z)| < \go$ or $|\arg(-z)| < \go$}\} \subset \bbC,
\end{equation*}
where the argument $\arg(z)$ takes values in $(-\gp,\gp]$.
The closed bisector of angle $\go$ is the topological closure $\overline{S_\go}$ of $S_\go$ in $\bbC$.

Throughout this section we write $L^2 = L^2(\bbR^n)$.

\begin{dfn}
Let $0 \leq \go < \gp/2$.
A closed linear operator $A$ on $L^2$ is called \emph{bisectorial of angle $\go$}\index{bisectorial operator} if $\gs(A) \subset \overline{S_\go}$, and if for all $\gm \in (\go,\gp/2)$ and all $z \in \bbC \sm \overline{S_\gm}$ we have the resolvent bound
\begin{equation}\label{bsres}
	\nm{(z - A)^{-1}}_{\mc{L}(L^2)} \lesssim_\gm |z|^{-1}.
\end{equation}
\end{dfn}

Closedness of $A$ is included in this definition; this is not standard, but it is convenient.
Generally the precise angle $\go$ is not important, so we may simply refer to $A$ as \emph{bisectorial}.

The following proposition is proven in \cite[Proposition 3.2.2]{mE15} (except for the adjoint statement, which is a simple computation).

\begin{prop}\label{prop:bisectorialfacts}
	Let $A$ be a bisectorial operator on $L^2$.
	Then $A$ is densely-defined, and we have a topological (not necessarily orthogonal) splitting
		\begin{equation}\label{eqn:NR-decomposition}
			L^2 = \mc{N}(A) \oplus \overline{\mc{R}(A)}.
		\end{equation}
	Furthermore, $A^*$ is also bisectorial.
\end{prop}

We let $\bbP_A$ denote the projection onto $\overline{\mc{R}(A)}$ along $\mc{N}(A)$.

The procedure of constructing an operator $\gf(A)$ from a given bisectorial operator $A$ and holomorphic function $\gf$ on an appropriate bisector, known as \emph{holomorphic functional calculus},\index{functional calculus!holomorphic} plays a central role in our theory.
In order to introduce holomorphic functional calculus properly, we must first define some classes of holomorphic functions.

For an angle $\gm \in (0,\gp/2)$, the set of holomorphic functions $\map{\gf}{S_\gm}{\bbC}$ is denoted by $H(S_\gm)$.
For $\gs,\gt \in \bbR$ and $\gf \in H(S_\gm)$ we define
\begin{equation*}
	\nm{\gf}_{\gY_\gs^\gt(S_\gm)} = \nm{z \mapsto \gf(z)/m_{\gs}^{\gt}(|z|)}_{L^\infty(S_\gm)}
\end{equation*}
(the function $m_\gs^\gt$ is defined in Section \ref{section:notation}) and
\begin{equation*}
	\gY_\gs^\gt(S_\gm) := \{\gf \in H(S_\gm) : \nm{\gf}_{\gY_\gs^\gt(S_\gm)} < \infty\}.
\end{equation*}
Each $\gY_\gs^\gt(S_\gm)$ is a Banach space when normed by $\nm{\cdot}_{\gY_{\gs}^{\gt}(S_\gm)}$, consisting of all holomorphic functions on $S_\gm$ which decay of order $\gs$ at $0$ and of order $\gt$ at $\infty$, interpreting growth as decay of negative order.
An important special case is $\gY_0^0(S_\gm) = H^\infty(S_\gm)$, the set of bounded holomorphic functions on $S_\gm$.
We surpress reference to $S_\gm$ in this notation when the relevant bisector is clear from context.

The spaces $\gY_\gs^\gt$ are decreasing in $\gs$ and $\gt$, in the sense that if $\gs < \gs^\prime$ and $\gt < \gt^\prime$, then $\gY_{\gs^\prime}^{\gt^\prime} \hookrightarrow \gY_\gs^\gt$.
For $\gs,\gt \in \bbR$ we define the set
\begin{equation*}
	\gY_\gs^{\gt+} := \bigcup_{\gt^\prime > \gt} \gY_\gs^{\gt^\prime},
\end{equation*}
and we define the sets $\gY_{\gs+}^\gt$ and $\gY_{\gs+}^{\gt+}$ analogously.
The set $\gY_+^+ := \gY_{0+}^{0+}$ is particularly important: it is the set of holomorphic functions (on the relevant bisector) with polynomial decay of some positive order at $0$ and $\infty$.
We also define
\begin{equation*}
	\gY_\gs^{\infty} := \bigcap_{\gt} \gY_\gs^\gt,
\end{equation*}
the set of functions with polynomial decay of arbitrarily large order at $\infty$.
Similarly we define $\gY_{\infty}^\gt$, $\gY_\infty^\infty$, $\gY_{\gs +}^\infty$, and so on.

There are a few holomorphic functions which we use extensively.
First define $\gc^+,\gc^- \in H^\infty(S_\mu)$ by
\begin{equation}\label{eqn:chidefn}
	\gc^+(z) := \mb{1}_{z : \re(z) > 0}(z) \quad \text{and} \quad \gc^-(z) := \mb{1}_{z : \re(z) < 0}(z) \qquad (z \in S_\gm).
\end{equation}
These are the indicator functions of the two halves of the bisector $S_\gm$.
Define also
\begin{align*}
	[z] := \left\{ \begin{array}{rl} z & (\re(z) > 0) \\ -z & (\re(z) < 0) \end{array} \right. = (\gc^+(z) - \gc^-(z))z.
\end{align*}
This lets us define a bounded version of the exponential map,
\begin{equation}\label{eqn:semigp}
	\operatorname{sgp} := [z \mapsto e^{-[z]}] \in \gY_0^\infty.
\end{equation}
For $\gl \in \bbR \sm \{0\}$ we define the power function
\begin{equation*}
	[z \mapsto z^\gl] \in \gY_{\gl}^{-\gl}
\end{equation*}
via a branch cut on the half-line $i(-\infty,0] \subset \bbC$.

We say that a function $\gf \in H(S_\gm)$ is \emph{nondegenerate}\index{nondegenerate!holomorphic function} if it does not vanish on any open subset of $S_\gm$.
All the holomorphic functions defined above are nondegenerate except for $\gc^+$ and $\gc^-$.

Let us introduce some useful operations on holomorphic functions.
Let $\gf \in H(S_\gm)$.
There is a natural involution $\gf \mapsto \td{\gf}$ on $H(S_\gm)$ defined by
\begin{equation*}
	\td{\gf}(z) := \overline{\gf(\overline{z})} \quad (z \in S_\gm).
\end{equation*}
This involution is isometric on $\gY_\gs^\gt$ for all $\gs,\gt \in \bbR$.
For $t > 0$ we define the dilation $\gf_t \in H(S_\gm)$ by
\begin{equation*}
	\gf_t(z) := \gf(tz).
\end{equation*}
The following lemma is a simple consequence of the above definitions.

\begin{lem}\label{lem:dilation-norm-control}
	Let $\gs \in \bbR$.
	Then for all $t > 0$ we have
	\begin{equation*}
		\nm{\gf_t}_{\gY_\gs^{-\gs}} = t^\gs \nm{\gf}_{\gY_\gs^{-\gs}}.
	\end{equation*}
\end{lem}

Fix an angle $\go \in [0,\gp/2)$ and let $A$ be an $\go$-bisectorial operator on $L^2$.
If $\gm \in (\go, \gp/2)$ and $\gf \in \gY_+^+(S_\gm)$, then we can define an operator $\gf(A)$ on $L^2$ by
\begin{equation}\label{dunford}
	\gf(A)f := \frac{1}{2\gp i} \int_{\partial S_\gn} \gf(z)(z-A)^{-1} f \, dz \qquad (f \in L^2)
\end{equation}
for any choice of $\gn \in (\go,\gm)$, where $\partial S_\gn$ is oriented counterclockwise.
The integral \eqref{dunford} is well-defined and independent of the choice of $\gn$, and we have
	\begin{equation*}
		\nm{\gf(A)}_{\mc{L}(L^2)} \lesssim_{A,\gs,\gt,\gm} \nm{\gf}_{\gY_\gs^\gt(S_\gm)}.
	\end{equation*} 

The proof is straightforward.
The following $\ast$-homomorphism property holds: for all $\gf, \gy \in \gY_
+^+$, we have $(\gf \gy)(A) = \gf(A) \gy(A)$, and $\gy(A)^* = \td{\gy}(A^*)$.

Often it is convenient to assume that the operator $A$ is injective and has dense range.
In our applications this does not hold, but the splitting $L^2 = \mc{N}(A) \oplus \overline{\mc{R}(A)}$ from Proposition \ref{prop:bisectorialfacts} shows that the restriction $A|_{\overline{\mc{R}(A)}}$, acting on the Hilbert space $\overline{\mc{R}(A)}$ (with inner product induced by that of $L^2$), is injective and has dense range.
One can also show that $A|_{\overline{\mc{R}(A)}}$ is bisectorial.
Thus we bypass the problem by working on $\overline{\mc{R}(A)}$ instead of the whole space $L^2$.

The integral in \eqref{dunford} converges whenever $\gf \in \gY_+^+$, but if $\gf \in H^\infty$ is merely bounded, convergence is not guaranteed.
Nevertheless, we are able to construct operators $\gf(A)$ when $\gf \in H^\infty$ for certain $A$.

\begin{dfn}
Let $A$ be a bisectorial operator on $L^2$.
We say that $A$ has \emph{bounded $H^\infty$ functional calculus on $\overline{\mc{R}(A)}$}\index{functional calculus!bounded H-infinity@bounded $H^\infty$} if for all $\gf \in \gY_+^+$ and all $f \in \overline{\mc{R}(A)}$,
\begin{equation*}
	\nm{\gf(A)f}_{2} \lesssim \nm{\gf}_\infty \nm{f}_{2}.
\end{equation*}
\end{dfn}

The property of having bounded $H^\infty$ functional calculus on $\overline{\mc{R}(A)}$ is equivalent to certain quadratic estimates: this is an important theorem due to McIntosh (see \cite[\textsection 7 and \textsection 8]{aM86} and the other references at the start of this section).

\begin{thm}[McIntosh]\label{thm:fc-quadest}
	Let $A$ be a bisectorial operator on $L^2$.
	Then $A$ has bounded $H^\infty$ functional calculus on $\overline{\mc{R}(A)}$ if and only if the quadratic estimate
	\begin{equation}\label{eqn:quadratic-estimate}
		\nm{f}_{2} \simeq_\gf \left( \int_0^\infty \nm{\gf_t(A) f}_{2}^2 \, \frac{dt}{t} \right)^{1/2}
	\end{equation}
	holds for all $f \in \overline{\mc{R}(A)}$ and some (equivalently, all) nondegenerate $\gf \in \gY_+^+$.
\end{thm}

Note that the quadratic estimate \eqref{eqn:quadratic-estimate} need not hold for $\gf \in H^\infty$.

If $A$ has bounded $H^\infty$ functional calculus on $\overline{\mc{R}(A)}$, then for all $\gf \in H^\infty$ we can define a bounded operator $\gf(A)$ on $\overline{\mc{R}(A)}$ by
\begin{equation}\label{eqn:mcintosh}
	\gf(A)f := \lim_\ga \gf_\ga(A) f \qquad (f \in \overline{\mc{R}(A)}),
\end{equation}
where $(\gf_\ga)$ is a net in $\gY_+^+$ which converges to $\gf$ in $H^\infty$.
We then have
\begin{equation*}
	\nm{\gf(A)}_{\mc{L}(\overline{\mc{R}(A)})} \lesssim \nm{\gf}_{H^\infty},
\end{equation*}
and for $\gf,\gy \in H^\infty$ we have the homomorphism property $\gf(A) \gy(A) = (\gf\gy)(A)$.
For further details see \cite{aM86, MY90}.
Thus we may define bounded operators $\gc^\pm(A)$ and $e^{-t[A]} = \sgp_t(A)$ (for all $t > 0$) on $\overline{\mc{R}(A)}$, using the corresponding $H^\infty$ functions defined in \eqref{eqn:chidefn} and \eqref{eqn:semigp}.

If $\gf \in \gY_+^0$, then we can extend $\gf(A)$ from $\overline{\mc{R}(A)}$ to all of $L^2$ by 
\begin{equation*}
	\gf(A)f := \gf(A)\bbP_{A}f.
\end{equation*}
The operator $\gf(A)$ then maps $L^2$ into $\overline{\mc{R}(A)}$.
We have
\begin{equation*}
	\nm{\gf(A)}_{\mc{L}(L^2)} \lesssim \nm{\bbP_{A}}_{\mc{L}(L^2)} \nm{\gf(A)}_{\mc{L}(\overline{\mc{R}(A)})},
\end{equation*}
and the homomorphism property $\gy(A) \gf(A) = (\gy\gf)(A)$ persists for all $\gy \in H^\infty$.

We refer to the operators $\gc^+(A)$ and $\gc^-(A)$ on $\overline{\mc{R}(A)}$ as the \emph{positive} and \emph{negative spectral projections}\index{spectral projection} associated with $A$.
These are complementary projections.
We define the \emph{positive} and \emph{negative spectral subspaces}\index{spectral subspaces} as their images
\begin{equation*}
	\overline{\mc{R}(A)}^\pm := \gc^\pm(A)\overline{\mc{R}(A)},
\end{equation*}
and we have a topological direct sum decomposition
\begin{equation*}
	\overline{\mc{R}(A)} = \overline{\mc{R}(A)}^+ \oplus \overline{\mc{R}(A)}^-.
\end{equation*}

We define the \emph{Cauchy operators}\index{Cauchy operator!on RA@on $\overline{\mc{R}(A)}$} $\map{C_A}{\overline{\mc{R}(A)}}{L^\infty(\bbR \sm \{0\} : \overline{\mc{R}(A)}})$ by
\begin{equation}\label{eqn:cauchyoperator}
	C_A f(t) := e^{-tA}\gc^{\sgn(t)}(A)f = e^{-|t|[A]}\gc^{\sgn(t)}(A)f = \sgp_{|t|}(A) \gc^{\sgn(t)}(A) f,
\end{equation}
where we write $\gc^{\pm1}(A) = \gc^\pm(A)$.
Write $C_A^+$ and $C_A^-$ for the restrictions of $C_A$ to $\bbR_+$ and $\bbR_-$ respectively.
These are solution operators for the Cauchy problems associated with $A$ on the upper and lower half-spaces \cite{AAM10}.

\begin{prop}\label{prop:sgp-cauchy}
	Suppose that $A$ has bounded $H^\infty$ functional calculus on $\overline{\mc{R}(A)}$.
	If $f \in \overline{\mc{R}(A)}$, then
	\begin{equation*}
	F := C_A f \in C^\infty(\bbR \sm \{0\} : \overline{\mc{R}(A)})
	\end{equation*}
	solves the Cauchy problem
	\begin{equation*}
		\partial_t F(t) + AF(t) = 0, \quad \lim_{t \to 0^\pm} F(t) = \gc^\pm(A) f, \quad \lim_{t \to \pm \infty} F(t) = 0
	\end{equation*}
	with limits taken in $\overline{\mc{R}(A)}$.
\end{prop}

Finally we discuss unbounded operators arising from holomorphic functional calculus, and some situations in which their compositions are bounded.
Suppose that $\gf \in \gY_\gs^\gt$ with $\min(\gs,\gt) \leq 0$, so that the integral \eqref{dunford} need not be absolutely convergent.
We can define an unbounded operator $\gf(A)$ on $\overline{\mc{R}(A)}$ as follows.
Fix $\gd > \max(-\gs,-\gt) \geq 0$ and define $\gh^\gd \in \gY_\gd^\gd$ by
\begin{equation*}
	\gh^\gd(z) := \left( \frac{z}{(i+z)^2} \right)^\gd.
\end{equation*}
Then $\gh^\gd \gf \in \gY_+^+$, so $(\gh^\gd \gf)(A)$ is defined by \eqref{dunford}.
We also have $\gh^\gd \in \gY_+^+$, so $\gh^\gd(A)$ is also defined by \eqref{dunford}, and since $A$ is injective with dense range on $\overline{\mc{R}(A)}$, so is $\gh^\gd(A)$.
Therefore the unbounded operator $\gh^\gd(A)^{-1}$ is defined, with $\mc{D}(\gh^\gd(A)^{-1}) := \mc{R}(\gh^\gd(A))$.
We then define the unbounded operator
\begin{equation}\label{eqn:unbdd-fc}
	\gf(A) := \gh^\gd(A)^{-1} (\gh^\gd \gf)(A)
\end{equation}
with domain
\begin{equation*}
	\mc{D}(\gf(A)) := \{f \in \overline{\mc{R}(A)} : (\gh^\gd \gf)(A)f \in \mc{D}(\gh^\gd(A)^{-1})\}.
\end{equation*}
The operator $\gf(A)$ is closed, densely-defined, and independent of the choice of $\gd$.
Of course, if $\min(\gs,\gt) > 0$, then we can take $\gd = 0$ in the definition \eqref{eqn:unbdd-fc} and recover the original definition of $\gf(A)$ by the Cauchy integral \eqref{dunford}.

Now suppose $\gy \in \gY_{\gs_1}^{\gt_1}$ and $\gf \in \gY_{\gs_2}^{\gt_2}$.
Then a quick computation shows that $\gf(A) \gy(A) \subseteq (\gf \gy)(A)$.
If $\gs_1 + \gs_2 > 0$ and $\gt_1 + \gt_2 > 0$, then the operator $(\gf \gy)(A)$ is bounded and given by \eqref{dunford}, while the operator $\gf(A) \gy(A)$ is not \emph{a priori} given by such a representation.
This observation is convenient in what follows.

\section{Off-diagonal estimates and the Standard Assumptions}\label{sec:ode}

For $x \in \bbR$, write $\langle x \rangle := \max\{1,|x|\}$.
We continue to write $L^2 = L^2(\bbR^n)$.

\begin{dfn}\label{dfn:ODEs}
	Suppose $\gO \subset \bbC \sm \{0\}$, and let $(S_z)_{z \in \gO}$ be a family of operators in $\mc{L}(L^2)$.
	Let $M \geq 0$.
	We say that \emph{$(S_z)$ satisfies off-diagonal estimates of order $M$}\index{off-diagonal estimates} if for all Borel subsets $E,F \subset \bbR^n$, all $z \in \gO$, and all $f \in L^2$,
	\begin{equation}\label{off-diagonal-estimates}
		\nm{\mb{1}_F S_z(\mb{1}_E f)}_2 \lesssim \left\langle \frac{d(E,F)}{|z|} \right\rangle^{-M} \nm{\mb{1}_E f}_2.
	\end{equation}
\end{dfn}

Many families of operators constructed from first-order differential operators (in particular, certain families of resolvents) satisfy off-diagonal estimates of some order.
The following theorem shows that certain families constructed via holomorphic functional calculus of a bisectorial operator $A$ satisfy off-diagonal estimates, under the assumption that a certain resolvent family satisfies off-diagonal estimates.
This is a slight extension of \cite[Proposition 2.7.1]{sSthesis}

\begin{thm}[Off-diagonal estimates for families constructed by functional calculus]\label{thm-fcode}
	Fix $0 \leq \go < \gn < \gm < \gp/2$, $M \geq 0$, and $\gs,\gt > 0$.
	Let $A$ be an $\go$-bisectorial operator on $L^2$ with bounded $H^\infty$ functional calculus on $\overline{\mc{R}(A)}$, such that $((I + \gl A)^{-1})_{\gl \in \bbC \sm S_\gn}$ satisfies off-diagonal estimates of order $M$.
	Suppose that $(\gh(t))_{t > 0}$ is a continuous family of functions in $H^\infty(S_\gm)$ which is uniformly bounded.\footnote{Continuity isn't really needed here. We only assume it to avoid thinking about measurability.}
	If $\gy \in \gY_\gs^\gt(S_\gm)$, then the family of operators $(\gh(t)(A) \gy_t(A))_{t > 0}$ satisfies off-diagonal estimates of order $\min\{\gs,M\}$, with constants depending linearly on $\nm{\gy}_{\gY_\gs^\gt} \nm{\gh}$, where $\nm{\gh} := \sup_{t > 0} \nm{\gh(t)}_\infty$, and also depending on $A$, $M$, $\gs$, and $\gt$.
\end{thm}
	
\begin{proof}
	Fix Borel sets $E,F \subset \bbR^n$.
	Because $\gy_t(A)$ maps into $\overline{\mc{R}(A)}$ for each $t$, we can apply $\gh(t)(A)$ to $\gy_t(A)(\mb{1}_E f)$ for each $t >0$.
	We must prove the estimate
	\begin{equation*}
		\nm{\mb{1}_F \gh(t)(A) \gy_t(A) (\mb{1}_E f)}_2 \lesssim_{A,M,\gs,\gt} \nm{\gh} \nm{\gy}_{\gY_\gs^\gt (S_\gm)} \left\langle \frac{d(E,F)}{t} \right\rangle^{-\min\{\gs,M\}} \nm{\mb{1}_E f}_2
	\end{equation*}
	for all $f \in L^2$.
	Fix $\gn^\prime \in (\gn,\gm)$ throughout the proof.
	
	If $d(E,F) \leq t$, then $\langle d(E,F)/t \rangle \simeq 1$, and so we have
	\begin{align}
		\nm{\mb{1}_F \gh(t)(A) \gy_t(A) (\mb{1}_E f)}_2
		&\leq \int_{\partial S_{\gn^\prime}} |\gh(t)(z)||\gy(tz)| \nm{\mb{1}_F (z-A)^{-1} (\mb{1}_E f)}_2 \, |dz| \label{s1}\\
		&\lesssim_A \nm{\gh} \nm{\gy}_{\gY_\gs^\gt} \nm{\mb{1}_E f}_2 \int_{\partial S_{\gn^\prime}} m_{\gs}^{\gt}(|z|) \, \frac{|dz|}{|z|} \label{bsuse}\\
		&\simeq_{\gs,\gt,M} \nm{\gh} \nm{\gy}_{\gY_\gs^\gt} \bigg\langle \frac{d(E,F)}{t} \bigg\rangle^{-\min\{\gs,M\}} \nm{\mb{1}_E f}_2, \nonumber
	\end{align}
	using the resolvent bound coming from bisectoriality of $A$ in \eqref{bsuse}.
	
	Now suppose that $d(E,F) > t$.
	Rearranging \eqref{s1} and using the assumed off-diagonal estimates for $((I + \gl A)^{-1})_{\gl \in \bbC \sm S_\gn}$, we have
	\begin{align}
		&\nm{\mb{1}_F \gh(t)(A) \gy_t(A) (\mb{1}_E f)}_2 \nonumber \\
		&\lesssim_A \nm{\gh} \nm{\gy}_{\gY_\gs^\gt} \nm{\mb{1}_E f}_2 \int_{\partial S_{\gn^\prime}} m_{\gs}^{\gt}(|z|) \bigg\langle \frac{d(E,F)}{t/|z|} \bigg\rangle^{-M} \, \frac{|dz|}{|z|} \nonumber \\
		&\lesssim \nm{\gh} \nm{\gy}_{\gY_\gs^\gt} \nm{\mb{1}_E f}_2 (\mb{I}_0 + \mb{I}_\infty), \label{rn}
	\end{align}
	where
	\begin{equation*}
		\mb{I}_0 := \int_0^{t/d(E,F)} m_{\gs}^{\gt}(\gl) \, \frac{d\gl}{\gl}
	\end{equation*}
	and
	\begin{equation*}
		\mb{I}_\infty := \int_{t/d(E,F)}^\infty m_{\gs}^{\gt}(\gl) \bigg( \frac{\gl}{t} d(E,F)\bigg)^{-M} \, \frac{d\gl}{\gl}.
	\end{equation*}
	The integral $\mb{I}_0$ is estimated by
	\begin{equation}\label{i0est}
		\mb{I}_0 \leq \int_0^{t/d(E,F)} \gl^\gs \, \frac{d\gl}{\gl} 
		\simeq_\gs \bigg(\frac{t}{d(E,F)}\bigg)^\gs \lesssim_M \bigg\langle \frac{d(E,F)}{t} \bigg\rangle^{-\min(\gs,M)}.
	\end{equation}
	To estimate $\mb{I}_\infty$, we use that $t/d(E,F) \leq 1$ to write
	\begin{equation*}
		\mb{I}_\infty \simeq_M \bigg\langle \frac{d(E,F)}{t} \bigg\rangle^{-M} \bigg( \int_{t/d(E,F)}^1 \gl^{\gs-M} \, \frac{d\gl}{\gl} + C(\gt,M) \bigg)
	\end{equation*}
	where
	\begin{equation*}
		C(\gt,M) = \int_1^\infty \gl^{-\gt-M} \, \frac{d\gl}{\gl}.
	\end{equation*}
	If $\gs \leq M$, then we have
	\begin{equation*}
		\int_{t/d(E,F)}^1 \gl^{\gs - M} \, \frac{d\gl}{\gl} \lesssim_{\gs,M} \bigg( \frac{t}{d(E,F)}\bigg)^{\gs - M},
	\end{equation*}
	and in this case
	\begin{align}
		\mb{I}_\infty &\lesssim_{\gs,M} \bigg\langle \frac{d(E,F)}{t} \bigg\rangle^{-M} \bigg( \bigg\langle \frac{d(E,F)}{t} \bigg\rangle^{M - \gs} + C(\gt,M) \bigg) \nonumber \\
		&\lesssim_{\gt,M} \bigg\langle \frac{d(E,F)}{t} \bigg\rangle^{-\min\{\gs,M\}} \label{iinfest}.
	\end{align}
	Otherwise, we have
	\begin{equation*}
		\int_{t/d(E,F)}^1 \gl^{\gs - M} \, \frac{d\gl}{\gl} \lesssim_{\gs,M} 1,
	\end{equation*}
	and this also yields \eqref{iinfest}.
	Putting the estimates \eqref{i0est} and \eqref{iinfest} into \eqref{rn} completes the proof.
\end{proof}

Off-diagonal estimates can also be used to deduce uniform boundedness and convergence results for families of operators on slice spaces.
These propositions were proven by the second author and Mourgoglou \cite[\textsection 4]{AM15}.

\begin{prop}[Uniform boundedness of families on slice spaces]\label{prop:slice-bddness}
	Let $p \in (0,\infty]$.
	If $(T_s)_{s > 0}$ is a family of operators on $L^2$ satisfying   off-diagonal estimates of order greater than $n\max(|\gd_{p,2}|, 1/2)$, then $T_s$ extends to a bounded operator on $E^p_{0}(t)$ uniformly in $0 < s \leq t$. 
\end{prop}

\begin{prop}[Strong convergence in slice spaces]\label{prop:slice-convergence}
Let $p \in (0,\infty)$.
Suppose $(T_s)_{s > 0}$ is a family of operators on $L^2$ satisfying off-diagonal estimates of order greater than $n\max(1/p, 1/2)$, and such that $\lim_{s \to 0} T_s = I$ strongly in $L^2$.
Then $\lim_{s \to 0} T_s = I$ strongly in $E^p$.
\end{prop}

Throughout the `abstract' part of this work, the following assumptions will be sufficient.
They can be a bit of a mouthful if stated in full, so we give them a name.

\begin{dfn}\label{dfn:std-assns}
	We say that an operator $A$ satisfies the \emph{Standard Assumptions}\index{Standard Assumptions} if 
	\begin{itemize}
		\item $A$ is a $\go$-bisectorial operator on $L^2$ for some $\go \in [0,\gp/2)$, 
		\item $A$ has bounded $H^\infty$ functional calculus on $\overline{\mc{R}(A)}$, and
		\item for all $\gn \in (\go,\gp/2)$ the family $((I + \gl A)^{-1})_{\gl \in \bbC \sm S_\gn}$ satisfies off-diagonal estimates of arbitrarily large order.
	\end{itemize}
\end{dfn}

The main examples we have in mind are perturbed Dirac operators.

\begin{thm}\label{thm:DB-fundamental-props}\index{Dirac operator!perturbed!satisfies Standard Assumptions}
	Suppose $D$ and $B$ are as in Subsection \ref{ssec:fop} of the introduction.
	Then the perturbed Dirac operators $DB$ and $BD$ satisfy the Standard Assumptions (see Definition \ref{dfn:std-assns}).
\end{thm}

This was proven by the second author and Stahlhut (\cite[Proposition 2.1]{AS14.1} and \cite[Lemma 2.3, Propositions 3.1 and 3.2]{AS16}).
The off-diagonal estimates stated there are in a different but equivalent form.
We mention once more that establishing bounded $H^\infty$ functional calculus is a deep result due to Axelsson, Keith, and McIntosh \cite{AKM06}.

\section{Integral operators on tent spaces}\label{sec:intops}

Let $(S_{t,\gt})_{t,\gt > 0}$ be a continuous two-parameter family of bounded operators on $L^2 = L^2(\bbR^n)$, and for all $f \in L_c^2(\bbR_+ : L^2)$ define $Sf \in L^0(\bbR_+ : L^2)$ by 
\begin{equation}\label{kernop}
	Sf(t) := \int_0^\infty S_{t,\gt} f(\gt) \, \frac{d\gt}{\gt}.
\end{equation}
Since $f$ is compactly supported in $\bbR_+$, the Cauchy--Schwarz inequality shows that this integral is absolutely convergent.
We write $S \sim (S_{t,\gt})_{t,\gt > 0}$ to say that $S$ is given by the kernel $(S_{t,\gt})_{t,\gt > 0}$.

We would like to know when $S$ can be extended from $L_c^2(\bbR_+ : L^2)$ to an operator between various tent spaces and $Z$-spaces.
A first step is given by the following Schur-type lemma.
Recall that $L^2_s$ is defined in \eqref{eqn:L2sdefn} and coincides with $X_s^2$.

\begin{lem}\label{lem:schur-lem}
	Let $s,\gd \in \bbR$, and consider $S \sim (S_{t,\gt})_{t,\gt > 0}$ on $L^2$ as above.
	Suppose that there exists $\gamma \in L^1(\bbR_+: \bbR)$ (where $\bbR_+$ is equipped with the Haar measure $dt/t$) such that for all $t,\gt > 0$,
	\begin{equation*}
		\nm{\gt^{-\gd} S_{t,\gt}}_{\mc{L}(L^2)} \leq \gamma(t/\gt) (t/\gt)^{s + \gd}.
	\end{equation*}
	Then for all $f \in L^2_c(\bbR_+ : L^2)$,
	\begin{equation}\label{eqn:schur-est}
		\nm{Sf}_{L^2_{s + \gd}} \lesssim \nm{\gamma}_{L^1(\bbR_+)} \nm{f}_{L^2_{s}}.
	\end{equation}
\end{lem}

\begin{proof}
	We argue by duality.
	For all $g \in L^{2}_{-(s+\gd)}$, using the assumed estimate on the operators $\tau^{-\delta} S_{t,\tau}$, we can estimate
	\begin{align*}
		\left| \langle Sf, g \rangle \right|
		&\leq \int_0^\infty \int_0^\infty \nm{S_{t,\gt} f(\gt)}_2 \nm{g(t)}_2 \, \frac{d\gt}{\gt} \, \frac{dt}{t} \\
		&\leq \int_0^\infty \int_0^\infty \gamma(t/\gt) \nm{\gt^{-s} f(\gt)}_2 \nm{t^{s+\gd} g(t)}_2 \, \frac{d\gt}{\gt} \, \frac{dt}{t} \\
		&\leq \bigg( \int_0^\infty \int_0^\infty \gamma(t/\gt) \, \frac{dt}{t} \nm{\gt^{-s} f(\gt)}_2^2 \, \frac{d\gt}{\gt} \bigg)^{1/2}  \\
		&\qquad \cdot \bigg( \int_0^\infty \int_0^\infty \gamma(t/\gt) \, \frac{d\gt}{\gt} \nm{t^{s+\gd} g(t)}_2^2 \, \frac{dt}{t} \bigg)^{1/2} \\
		&\leq \nm{\gamma}_{L^1(\bbR_+)} \nm{f}_{L^2_{s}} \nm{g}_{L^2_{-s-\gd}},
	\end{align*}
	which implies \eqref{eqn:schur-est}.
\end{proof}

For certain kernels $(S_{t,\gt})$, assuming an $L^2_{s} \to L^2_{s + \gd}$ estimate such as \eqref{eqn:schur-est} and some off-diagonal estimates, we can prove boundedness of $S$ from $T^{\mb{p}}$ to $T^{\mb{p} + \gd}$ for some exponents $\mb{p}$ with $i(\mb{p}) \in (0,1]$.
This is a generalisation of an argument of the second author, McIntosh, and Russ \cite{AMR08} (see also \cite{HMM11}).

\begin{thm}[Extrapolation of $L^2$ boundedness to tent spaces]\label{thm:AMR}
	Let $\mb{p} = (p,s)$ be an exponent with $p \leq 1$, let $\gd \in \bbR$, and let $(S_{t,\gt})_{t,\gt > 0}$ be a continuous two-parameter family of bounded operators on $L^2$ such that for all $t_0,\gt_0 > 0$ the one-parameter families $(t^{-\gd} S_{t,\gt_0})_{t \in (\gt_0,\infty)}$ and $(\gt^{-\gd} S_{t_0,\gt})_{\gt \in (t_0,\infty)}$ both satisfy off-diagonal estimates of order $M$, with implicit constant $K$ uniform in $\gt_0$ and $t_0$ respectively.
	Suppose $a,b \in \bbR$, and let $\mb{S} \sim (m_a^b(t/\gt) S_{t,\gt})_{t,\gt > 0}$.
	
	If
	\begin{equation}\label{eqn:AMR-apriori}
		\nm{\mb{S}f}_{L^2_{s + \gd}} \lesssim \nm{f}_{L^2_{s}}
	\end{equation}
	holds for all $f \in L_c^2(\bbR_+ : L^2)$, with
	\begin{equation}\label{AMRparassn}
		-n\gd_{p,2} < b + s < M \quad \text{and} \quad a > s + \gd,
	\end{equation}
	then
	\begin{equation*}
		\nm{\mb{S}f}_{T^{\mb{p}+\gd}} \lesssim \nm{f}_{T^\mb{p}}
	\end{equation*}
	for all $f \in L_c^2(\bbR_+ : L^2)$, and the implicit constant is a linear combination of $K$ and $\nm{\mb{S}} := \nm{\mb{S}}_{L_{s}^2 \to L_{s + \gd}^2}$.
\end{thm}

\begin{proof}
	\textbf{Step 1: an estimate for compactly supported atoms.}
		Suppose that $f$ is a compactly supported $T^\mb{p}$-atom associated with a ball $B = B(c,r) \subset \bbR^n$ (see Definition \ref{dfn:ts-atom}).
		Then $f \in L^2_c$, and so $\mb{S}f$ is defined.
		We will show that $\mb{S}f$ is in $T^{\mb{p} + \gd} = T^p_{s + \gd}$ with quasinorm bounded independently of $f$.
		To do this we will exhibit an atomic decomposition of $\mb{S}f$, and we will estimate $\nm{\mb{S}f}_{T^{\mb{p} + \gd}}$ using the coefficients of this decomposition.
		
		Let $T_1 := T(4B)$ and $T_k := T(2^{k+1}B) \sm T(2^k B)$ for all integers $k \geq 2$.
		Then define $F_k := \mb{1}_{T_k} \mb{S}f$ for all $k \in \bbN$, so that $\mb{S}f = \sum_{k=1}^\infty F_k$ almost everywhere.
		For each $k \in \bbN$ the function $F_k$ is supported in a tent, so we can renormalise by writing $F_k = \gl_k f_k$ for some $\gl_k \in \bbC$ and some $T^{\mb{p} + \gd}$-atom $f_k$.
		We need only estimate the coefficients $\gl_k$.
		
		\textbf{Estimate for the local part.}
		For $k = 1$, since $F_1$ is supported in $T(4B)$, we must estimate $\nm{F_1}_{L^2_{s + \gd}}$ in terms of $|4B|^{\gd_{p,2}}$.
		It follows from \eqref{eqn:AMR-apriori} and the fact that $f$ is a $T^\mb{p}$-atom that
		\begin{equation*}
			\nm{F_1}_{L^2_{s + \gd}} \leq \nm{\mb{S}} |B|^{\gd_{p,2}}
			\simeq_{n,p} \nm{\mb{S}} |4B|^{\gd_{p,2}},
		\end{equation*}
		and so we can set $\gl_1 \simeq_{n,p} \nm{\mb{S}}$.
		
		\textbf{Estimate for the global parts.}
		Suppose $k \geq 2$.
		Since $F_k$ is supported in the tent $T(2^{k+1}B)$, we must estimate $\nm{F_k}_{L_{s+\gd}^2}$ in terms of $|2^{k+1} B|^{\gd_{p,2}}$.
		We use Minkowski's integral inequality to estimate
		\begin{align*}
			&\nm{F_k}_{L_{s+\gd}^2} \\
			&= \bigg( \int_0^{2^{k+1}r} \bigg\| t^{-(s + \gd)} \mb{1}_{B(c,2^{k+1} r - t)} \int_0^r m_a^b(t/\gt) S_{t,\gt} f(\gt) \, \frac{d\gt}{\gt} \bigg\|_2^2  \, \frac{dt}{t} \bigg)^{1/2} \\
			&\leq \bigg( \int_0^{2^{k+1}r} \bigg( \int_0^r \, t^{-(s + \gd)} m_a^b(t/\gt) \| \mb{1}_{A(c,2^k r - t, 2^{k+1} r - t)} S_{t,\gt} f(\gt) \|_2 \frac{d\gt}{\gt} \bigg)^2 \, \frac{dt}{t} \bigg)^{1/2} \\
			&\leq \int_0^r \bigg( \int_0^{2^{k+1}r} t^{-2(s + \gd)} m_a^b(t/\gt)^2 \| \mb{1}_{A(c,2^k r - t, 2^{k+1} r - t)} S_{t,\gt} f(\gt) \|_2^2 \, \frac{dt}{t} \bigg)^{1/2} \, \frac{d\gt}{\gt},
		\end{align*}
		where
                \begin{equation*}
                  A(c,2^k r - t, 2^{k-1}r - t) = B(c,2^{k-1}r - t) \setminus B(c,2^k r - t)
                \end{equation*}
                as defined in Section \ref{section:notation}.
		Note that $f(\gt)$ is supported in $B(c,r - \gt)$.
		We have
		\begin{align*}
			d(\supp f(\gt), A(c,2^k r - t, 2^{k+1}r - t)) &\geq d(B(c,r), \bbR^n \sm B(c,2^k r - t)) \\
			&= ((2^k - 1)r - t)_+,
		\end{align*}
		so we split the region of integration $(0,2^{k+1}r) \times (0,r)$ into three subregions,
		\begin{align*}
			R_1 &:= \{(t,\gt) : t < \gt < r\} \\
			R_2 &:= \left\{(t,\gt) : \gt < t < \frac{2^k-1}{2} r\right\} \\
			R_3 &:= \left\{(t,\gt) : t > \frac{2^k - 1}{2} r\right\},
		\end{align*}
		and denote the corresponding integrals by $\mb{I}_1$, $\mb{I}_2$, and $\mb{I}_3$.
		
		On $R_3$, where $t > \gt$ and where there is no spatial separation, we have
		\begin{align}
		\mb{I}_3 &\lesssim K\int_0^r \bigg( \int_{\frac{2^k-1}{2} r}^{2^{k+1}r} t^{-2(s + \gd)} (t/\gt)^{-2b} t^{2\gd}\nm{f(\gt)}_{2}^2 \, \frac{dt}{t} \bigg)^{1/2}  \, \frac{d\gt}{\gt} \nonumber \\
		&\lesssim_{b,s} K\int_0^r \nm{\gt^{-s} f(\gt)}_{2} \gt^{b + s} (2^k r)^{-(b+s)} \, \frac{d\gt}{\gt} \nonumber \\
		&\leq K2^{-k(b+s)} \bigg( \int_0^r \nm{\gt^{-s} f(\gt)}_{2}^2 \, \frac{d\gt}{\gt} \bigg)^{1/2} \bigg( \int_0^r ( \gt/r )^{2(b+s)} \, \frac{d\gt}{\gt} \bigg)^{1/2} \nonumber \\
		&\simeq_{b,s} K2^{-k(b+s)} \nm{f}_{L_{s}^2} \label{sb0} \\
		&\leq K2^{-k(b+s)} |B|^{\gd_{p,2}} \nonumber \\
		&= K2^{-k(b+s)} |2^{k+1}B|^{\gd_{p,2}} \bigg( \frac{|2^{k+1}B|}{|B|} \bigg)^{-\gd_{p,2}} \nonumber \\
		&\simeq K2^{-k(b+s+n\gd_{p,2})} |2^{k+1}B|^{\gd_{p,2}}, \nonumber
		\end{align}
		where we used $b + s > -n\gd_{p,2} > 0$ in \eqref{sb0}.
		
		On $R_1$, where $t < \gt$ and where the off-diagonal estimates for $(\gt^{-\gd}S_{t,\gt})_{\gt \in (t,\infty)}$ involve spatial separation, 	
		\begin{align*}
			\mb{I}_1 &\lesssim K
			\int_0^r \bigg( \int_0^\gt t^{-2(s + \gd)} (t/\gt)^{2a} \bigg( \bigg( \frac{(2^k - 1)r - t}{\gt} \bigg)^{-M} \nm{f(\gt)}_{2} \bigg)^2 \, \frac{dt}{t} \bigg)^{1/2} \, \gt^{\gd} \frac{d\gt}{\gt}  \\
			&\simeq_{a,s,\gd} K 2^{-kM} \int_0^r \left( \frac{\gt}{r} \right)^M  \nm{\gt^{-s} f(\gt)}_{2}  \, \frac{d\gt}{\gt}  \\ %\label{sa0}
			&\leq K 2^{-kM} \nm{f}_{L_{s}^2} \left( \int_0^r \left( \frac{\gt}{r} \right)^{2 M} \, \frac{d\gt}{\gt} \right)^{1/2}  \\
			&\lesssim_{M} K2^{-k(M+n\gd_{p,2})} |2^{k+1}B|^{\gd_{p,2}} %\label{be}
		\end{align*}
		 using that $a > s + \gd$ in the second line, and then arguing as for $\mb{I}_3$.
		
		On $R_2$, we have $t > \gt$ and the off-diagonal estimates for $(t^{-\gd} S_{t,\gt})_{t \in (\gt,\infty)}$ again involve spatial separation, and the restrictions on $t$ imply $(2^k - 1)r - t > \frac{2^k - 1}{2} r \gtrsim 2^k r$.
		Therefore
		\begin{align*}
			\mb{I}_2 &\lesssim K \int_0^r \bigg( \int_{\gt}^{\frac{2^k - 1}{2} r} t^{-2s} (t/\gt)^{-2b} \bigg( \frac{(2^k - 1)r - t}{t} \bigg)^{-2M} \nm{f(\gt)}_{2}^2 \, \frac{dt}{t} \bigg)^{1/2} \, \frac{d\gt}{\gt} \\
			&\lesssim K \int_0^r \nm{\gt^{-s} f(\gt)}_{2} \gt^{b+s} \bigg( \int_\gt^{\frac{2^k - 1}{2} r} t^{-2(b + s)} \bigg( \frac{2^k r}{t} \bigg)^{-2M} \, \frac{dt}{t} \bigg)^{1/2} \, \frac{d\gt}{\gt}  \\
			&\lesssim_{b,s,M} K2^{-kM} \int_0^r \nm{\gt^{-s} f(\gt)}_{2} \gt^{b+s} r^{-M} (2^k r)^{-(b+s-M)} \, \frac{d\gt}{\gt} \\ % \label{sbM0} 
			&\lesssim_{b,s} K2^{-k(b+s+n\gd_{p,2})} |2^{k+1}B|^{\gd_{p,2}} %\label{b3}
		\end{align*}
		using that $M > s +b$ and arguing as before.
		
		Summing up, we have
		\begin{align*}
			\nm{F_k}_{L_{s + \gd}^2}
			&\leq \mb{I}_1 + \mb{I}_2 + \mb{I}_3 \\
			&\lesssim_{a,b,M,p,s,\gd} K \big(2^{-k(M + n\gd_{p,2})} + 2^{-k(b+s+n\gd_{p,2})} \big) |2^{k+1}B|^{\gd_{p,2}},
		\end{align*}
		so for $k \geq 2$ we can set
		\begin{equation*}
			\gl_k \simeq K \big( 2^{-k(M + n\gd_{p,2})} + 2^{-k(b+s+n\gd_{p,2})} \big).
		\end{equation*}
		This yields
		\begin{align*}
			\nm{(\gl_k)}_{\ell^p(\bbN)}^p
			&\simeq \nm{\mb{S}}^p + K^p \sum_{k=2}^\infty \big( 2^{-k(M+n\gd_{p,2})} + 2^{-k(b+s+n\gd_{p,2})}\big)^p \\
			&\leq \nm{\mb{S}}^p + K^p \sum_{k=2}^\infty 2^{-kp(M+n\gd_{p,2})} + 2^{-kp(b+s+n\gd_{p,2})},
		\end{align*}
		which is finite because of the assumption \eqref{AMRparassn}.
		The implicit constants do not depend on the atom $f$.
		Therefore $\mb{S}f$ is in $T^p_{s + \gd}$, with quasinorm bounded independently of $f$ and controlled by a linear combination of $\nm{\mb{S}}$ and $K$.
	
	\textbf{Step 2: from compactly supported atoms to $T^p_s \cap L_c^2(\bbR_+ : L^2)$.}
	This final part of the argument exactly follows \cite[Proof of Theorem 4.9, Step 3]{AMR08}.
	One must show that every function in $T_{s}^p \cap L_c^2(\bbR_+ : L^2)$ may be decomposed into a sum of compactly supported atoms, and that such decompositions converge in both $T_{s}^p$ (which is automatic) and in $L_{s}^2$.
	We omit further details.
\end{proof}

\section{Extension and contraction operators}\label{sec:eco}

Throughout this section we assume that $A$ satisfies the Standard Assumptions (see Definition \ref{dfn:std-assns}).

\begin{dfn}
  For all $\gy \in H^\infty$ define the \emph{extension operator}\index{extension operator}
  \begin{equation*}
    \map{\bbQ_{\gy,A}}{\overline{\mc{R}(A)}}{L^\infty(\bbR_+: L^2)}
  \end{equation*}
  by $(\bbQ_{\gy,A} f)(t) := \gy_t(A)f$ for $f \in \overline{\mc{R}(A)}$ and $t \in \bbR_+$.
  If in addition $\gy \in \gY_+^+$, then $\bbQ_{\gy,A}$ is defined on all of $L^2$, and by Theorem \ref{thm:fc-quadest} we have boundedness $\map{\bbQ_{\gy,A}}{L^2}{L^2(\bbR^{1+n}_+) = L^2(\bbR_+ : L^2)}$ (recall that the measure $dt/t$ is always used on $\bbR_+$).
  For $\gf \in \gY_+^+$, this allows us to define the \emph{contraction operator}\index{contraction operator}
  \begin{equation*}
    \map{\bbS_{\gf,A} := (\bbQ_{\td{\gf},A^*})^*}{L^2(\bbR^{1+n}_+)}{\overline{\mc{R}(A)}}.
  \end{equation*}
  Note that $\bbQ_{\gy,A} = (\bbS_{\td{\gy},A^*})^*$ when $\gy \in \gY_+^+$.
\end{dfn}

A quick computation yields an integral representation of $\bbS_{\gf,A}$.

\begin{prop}\label{prop:Srep}
	Suppose $\gf \in \gY_+^+$.
	Then for all $f \in L_c^2(\bbR_+:L^2)$ we have
	\begin{equation}\label{eqn:S-intrep}
		\bbS_{\gf,A} f := \int_0^\infty \gf_t(A) f(t) \, \frac{dt}{t}.
	\end{equation}
\end{prop}

The integral in \eqref{eqn:S-intrep} converges absolutely since $f \in L^1(\bbR_+ : L^2)$. However, if one only assumes $f\in L^2(\bbR_+ : L^2)$ (dropping the requirement of compact support in $\bbR_+$), then the integral converges weakly in $L^2$.

Compositions of the form $\bbQ_{\gy,A} \gh(A) \bbS_{\gf,A}$, with $\gh$ a (possibly unbounded) holomorphic function, are important in what follows, so we begin investigation of their properties.
Fix $\gd \in \bbR$, and suppose $\gh \in \gY_{-\gd}^\gd$, $\gy \in H^\infty$, and $\gf \in \gY_{\gd}^{-\gd} \cap \gY_+^+$.
Then for all $f \in L_c^2(\bbR_+ : L^2)$ we have $\bbS_{\gf,A}f \in \mc{D}(\gh(A))$ and the integral representation
\begin{equation}\label{eqn:comp-repn}
	(\bbQ_{\gy,A} \gh(A) \bbS_{\gf,A} f)(t) = \int_0^\infty (\gy_t(A) \gh(A) \gf_\gt(A))f(\gt) \, \frac{d\gt}{\gt}.
\end{equation}
Therefore we can write
\begin{equation*}
	\bbQ_{\gy,A} \gh(A) \bbS_{\gf,A} \sim ((\gy_t\gh\gf_\gt)(A))_{t,\gt > 0}.
\end{equation*}
Our task now is to show when the results of Section \ref{sec:intops} apply to these operators.
We draw some conclusions even when $\gh$ is not bounded, provided that $\gy_t \gh \gf_\gt \in \gY_+^+$.
This will ultimately lead to Theorem \ref{thm:fcprops}.

\begin{lem}\label{lem:AMR-kernel-lem}
	Suppose $\gs + \gt \geq 0$ and $\gd \in \bbR$.
	Let $\gy \in \gY_{\gs}^{\gt}$, $\gf \in \gY_{\gt + \gd}^{\gs - \gd}$, and $\gh \in \gY_{-\gd}^{\gd}$, and define the operator
	\begin{equation}\label{eqn:kerlem-Str}
		\td{S}_{t,r} := m_\gs^{\gt + \gd}(t/r)^{-1} (\gy_t \gh \gf_r)(A).
	\end{equation}
	Then for all $t_0,r_0 > 0$ the operator families $(t^{-\gd} \td{S}_{t,r_0})_{t \in (r_0,\infty)}$ and $(r^{-\gd} \td{S}_{t_0,r})_{r \in (t_0,\infty)}$ satisfy   off-diagonal estimates of order $\gs + \gt$, uniformly in $r_0$ and $t_0$ respectively.
	The implicit constants in these off-diagonal estimates depend linearly on $\nm{\gh}_{\gY_{-\gd}^\gd}$.
\end{lem}

The proof is a variation of \cite[Lemma 3.7]{AMR08}.

\begin{proof}
	If $t_0 \leq r$ we can write
	\begin{align*}
		r^{-\gd} \td{S}_{t_0,r} &= r^{-\gd} (t_0 / r)^{-\gs} [\gy_{t_0}(z) \gh(z) \gf_r(z)](A) \\
		&= [(t_0 z)^{-\gs} \gy_{t_0}(z) \gh^\gd(z) (r z)^{\gs-\gd} \gf_\gt(z)](A)
	\end{align*}
	where $\gh^\gd \in H^\infty$ is defined by $\gh^\gd(z) := z^\gd \gh(z)$.
	Note that $\nm{\gh^\gd}_\infty = \nm{\gh}_{\gY_{-\gd}^{\gd}}$.
	Since $\gy \in \gY_\gs^\gt$ and $\gs + \gt \geq 0$, the function
	\begin{equation*}
	\gamma(t_0) : z \mapsto (t_0 z)^{-\gs} \gy_{t_0}(z) \gh^\gd(z)
	\end{equation*}
	is in $H^\infty$ with bound uniform in $t_0$, linear in $\nm{\gh}_{\gY_{-\gd}^\gd}$, and clearly independent of $r$.
	Furthermore, the function $\gq : z \mapsto z^{\gs-\gd} \gf(z)$ is in $\gY_{\gs + \gt}^{0}$, and so we can write
	\begin{equation*}
		r^{-\gd} \td{S}_{t_0,r} = \gamma(t_0)(A) \gq_{\gt}(A)
	\end{equation*}
	where $\gamma(t)$ is uniformly in $H^\infty$ and $\gq \in \gY_{\gs+\gt}^{0}$. 
	Theorem \ref{thm-fcode} then implies that the family $(\td{S}_{t_0,r})_{r \in (t_0,\infty)}$ satisfies off-diagonal estimates of order $\gs + \gt$ uniformly in $t_0 > 0$, with implicit constants linear in $\nm{\gh}_\infty$.
	
	Likewise, if $r_0 \leq t$ we can write
	\begin{align*}
		t^{-\gd} \td{S}_{t,r_0} &= t^{-\gd} (t / r_0)^{\gt + \gd} [\gy_t(z) \gh(z) \gf_{r_0}(z)](A) \\
		&= [(r_0 z)^{-(\gt + \gd)} \gf_{r_0}(z) \gh^\gd(z) (tz)^{\gt} \gy_t(z)](A)
	\end{align*}
	and proceed in the same way, the consequence being that $(t^{-\gd} \td{S}_{t,r_0})_{t \in (r_0,\infty)}$ satisfies off-diagonal estimates of order $\gs + \gt$ uniformly in $\gt_0 > 0$, with implicit constants linear in $\nm{\gh}_{\gY_{-\gd}^{\gd}}$.
\end{proof}

\begin{lem}\label{lem:schur-hyp-verifn}
  Fix $s,\gd \in \bbR$.
  Suppose $\gy \in \gY_{(s+\gd)+}^{ -(s+\gd)+}$, $\gf \in \gY_{-s+}^{s+}$, and $\gh \in \gY_{-\gd}^{\gd}$.
  Then the operator $S \sim ((\gy_t \gh \gf_r)(A))_{t,r > 0}$ extends to a bounded operator $L_{s}^2 \to L_{s + \gd}^2$.
  Moreover, this extension is the weak limit (in $L_{s + \gd}^2$) of the integral \eqref{eqn:comp-repn}.
\end{lem}
      
\begin{proof}
	Fix $\varepsilon > 0$ such that $\gy \in \gY_{\varepsilon + s + \gd }^{\varepsilon - (s + \gd)}$ and $\gf \in \gY_{\varepsilon - s}^{\varepsilon + s}$.
	First note that $\gy_t \gh \gf_r \in \gY_+^+$, so the operators $S_{t,r} := (\gy_t \gh \gf_r)(A)$ are all bounded and defined by the  integral \eqref{dunford} on $L^2$.
	We will make use of Lemma \ref{lem:schur-lem}, so we write $r = \gk t$ and begin by estimating
	\begin{align}
		\nm{r^{-\gd} S_{t,r}}_{\mc{L}(L^2(\bbR^n))}
		&\lesssim_{\gy,\gf} (\gk t)^{-\gd} \nm{\gh_{1/t}}_{\gY_{-\gd}^{\gd}} \int_0^\infty m_{\varepsilon + s}^{\varepsilon - s}(t\gl) m_{\varepsilon - s}^{\varepsilon + s} (\gk t \gl) \, \frac{d\gl}{\gl} \nonumber \\
		&\leq \gk^{-\gd} \nm{\gh}_{\gY_{-\gd}^{\gd}} \int_0^\infty m_{\varepsilon + s}^{\varepsilon - s}(\gl) m_{\varepsilon - s}^{\varepsilon + s} (\gk \gl) \, \frac{d\gl}{\gl} \label{eqn:dilat}
	\end{align}
	using Lemma \ref{lem:dilation-norm-control} to eliminate the powers of $t$ in \eqref{eqn:dilat}.
	If $\gk \leq 1$, we have
	\begin{align*}
		&\gk^{-\gd} \int_0^\infty m_{\varepsilon + s}^{\varepsilon - s}(\gl) m_{\varepsilon - s}^{\varepsilon + s}(\gk\gl) \, \frac{d\gl}{\gl} \\
		&= \gk^{-\gd} \bigg( \gk^{\varepsilon - s} \int_0^1 \gl^{2\varepsilon} \, \frac{d\gl}{\gl} + \gk^{\varepsilon - s} \int_1^{1/\gk} \, \frac{d\gl}{\gl} + \gk^{-\varepsilon - s} \int_{1/\gk}^\infty \gl^{-2\varepsilon} \, \frac{d\gl}{\gl} \bigg) \\
		&\lesssim  \gk^{\varepsilon - s - \gd}(2 + \log(1/\gk)).
	\end{align*}
	If $\gk \geq 1$, then by the same argument we have
	\begin{equation*}
		\gk^{-\gd} \int_0^\infty m_{\varepsilon + s}^{\varepsilon - s}(\gl) m_{\varepsilon - s}^{\varepsilon + s}(\gk\gl) \, \frac{d\gl}{\gl} \lesssim \gk^{-\varepsilon - s - \gd}(2 + \log(\gk)).
	\end{equation*}
	Since the function
	\begin{equation*}
		\gamma(\gk) := \left\{ \begin{array}{ll} \gk^{\varepsilon}(2 + \log(1/\gk)) & (\gk \leq 1) \\ \gk^{-\varepsilon}(2 + \log(\gk)) & (\gk \geq 1) \end{array} \right.
	\end{equation*}
	is in $L^1(\bbR_+)$, Lemma \ref{lem:schur-lem} completes the proof of boundedness on the dense subspace $L_c^2(\bbR_+ : L^2)$.
	The extension to all of $L_s^2$ as a weak integral follows by the comment after Proposition \ref{prop:Srep}.
\end{proof}

 The following theorem is the basis of Chapter \ref{chap:ahsb}.
From the viewpoint of applications, the most important part is the decay condition on $\gy$ at $0$.
We would particularly like to take $\gy = \sgp \in \gY_0^\infty$ and $\gd = 0$, which is possible provided that $i(\mb{p}) \leq 2$ and $\gq(\mb{p}) < 0$.

\begin{thm}[Boundedness of contraction/extension compositions]\label{thm:QS-compn-bddness}
	Suppose $\mb{p}$ is an exponent, $\gd \in \bbR$ and $\gh \in \gY_{-\gd}^{\gd}$.
	Suppose that either
	\begin{itemize}
		\item
		 $i(\mb{p}) \leq 2$ and 
		\begin{equation}\label{eqn:QS-thm-assn}
			\gy \in \gY_{(\gq(\mb{p})+\gd)+}^{(-(\gq(\mb{p})+\gd) + n|\frac{1}{2} - j(\mb{p})|)+} \cap H^\infty \quad \text{and} \quad \gf \in \gY_{(-\gq(\mb{p}) + n|\frac{1}{2} - j(\mb{p})|)+}^{\gq(\mb{p})+} \cap \gY_+^+,
		\end{equation}
		or
		\item
		$i(\mb{p}) \geq 2$ and
		\begin{equation*}
			\gy \in \gY_{(\gq(\mb{p}) + n|\frac{1}{2} - j(\mb{p})|)+}^{(-\gq(\mb{p}))+} \cap \gY_+^+ \quad \text{and} \quad \gf \in \gY_{(-\gq(\mb{p}) + \gd)+}^{(\gq(\mb{p}) - \gd + n|\frac{1}{2} - j(\mb{p})|)+} \cap \gY_+^+.
		\end{equation*}
		\end{itemize}
		then $\bbQ_{\gy,A} \gh(A) \bbS_{\gf,A}$ extends to a bounded operator $X^{\mb{p}} \to X^{\mb{p} + \gd}$ (by density when $i(\mb{p}) \leq 2$ and by duality when $i(\mb{p}) > 2$), with bounds linear in $\nm{\gh}_{\gY_{-\gd}^\gd}$.
\end{thm}

\begin{proof}
	We only prove the result for tent spaces.
	The $Z$-space result can be deduced by real interpolation, or alternatively by the dyadic characterisation of Proposition \ref{prop:Zpqs-seq-equivalence}.
	Furthermore, the result for $i(\mb{p}) \geq 2$ follows from that for $i(\mb{p}) \leq 2$ by duality, so we need only consider $i(\mb{p})  \leq 2$.
	Note that \eqref{eqn:comp-repn} and the assumptions on $\gy$ and $\gf$ imply that $\bbQ_{\gy,A} \gh(A) \bbS_{\gf,A}$ contains the integral operator with kernel $((\gy_t \gh \gf_\gt)(A))_{t,\gt > 0}$, so it suffices to work with this operator.
	Furthermore, the assumptions \eqref{eqn:QS-thm-assn} and Lemma \ref{lem:schur-hyp-verifn} imply that this operator (extended as a weak limit) is bounded from $T_{\gq(\mb{p})}^2$ to $T_{\gq(\mb{p}) + \gd}^2$, which yields the result when $i(\mb{p}) = 2$.
	
	\textbf{Step 1: $i(\mb{p}) \leq 1$.} 	
	The assumptions \eqref{eqn:QS-thm-assn} imply that there exists $\varepsilon > 0$ such that
	\begin{equation*}
		\gy \in \gY_{\gs + \varepsilon}^{\gt + \varepsilon} \quad \text{and} \quad \gf \in \gY_{(\gt + \varepsilon) + \gd}^{(\gs + \varepsilon) - \gd},
	\end{equation*}
	where $\gs := \gq(\mb{p}) + \gd$ and $\gt := -\gq(\mb{p}+\gd) + n|(1/2) - j(\mb{p})|$.
	Therefore by Lemma \ref{lem:AMR-kernel-lem}, the operator families $(t^{-\gd} \td{S}_{t,r_0})_{t \in (r_0, \infty)}$ and $(r^{-\gd} \td{S}_{t_0,r})_{r \in (t_0,\infty)}$, where $\td{S}_{t,r}$ is defined as in \eqref{eqn:kerlem-Str}, satisfy   off-diagonal estimates of order $n|(1/2) - j(\mb{p})| + 2\varepsilon$.
	Theorem \ref{thm:AMR} then applies with $a = \gs + \varepsilon$, $b = \gt + \varepsilon + \gd$, and $M = 2\varepsilon + n|(1/2) - j(\mb{p})|$, and we can conclude that $\bbQ_{\gy,A} \gh(A) \bbS_{\gf,A}$ is bounded from $T^{\mb{p}}$ to $T^{\mb{p} + \gd}$.
	
		\textbf{Step 2: $i(\mb{p}) \in (1,2)$.}
		The following argument originates from Stahlhut's thesis \cite[Proof of Lemma 3.2.6, Step 4]{sSthesis}.
		For $\gl \in \bbC$, define functions $\gy^\gl$ and $\gf^\gl$ by
		\begin{equation*}
			\gy^\gl(z) := \bigg( \frac{[z]}{1+[z]} \bigg)^\gl \gy(z), \quad \gf^\gl(z) := \bigg( \frac{1}{1 + [z]} \bigg)^\gl \gf(z).
		\end{equation*}
		If $\re \gl \geq n(1 - j(\mb{p}))$, then Step 1 applies with exponent $(1,\gq(\mb{p}))$, and we find that $\bbQ_{\gy^\gl, A} \gh(A) \bbS_{\gf^\gl,A}$ is bounded from $T^1_{\gq(\mb{p})}$ to $T^1_{\gq(\mb{p}) + \gd}$.
		Furthermore, if $\re \gl \geq -n|(1/2) - j(\mb{p})|$, then the discussion of the first paragraph of the proof applies, and we find that $\bbQ_{\gy^\gl, A} \gh(A) \bbS_{\gf^\gl,A}$ is bounded from $T^2_{\gq(\mb{p})}$ to $T^2_{\gq(\mb{p}) + \gd}$.
		By Stein interpolation in tent spaces (see \cite[Proof of Lemma 3.4]{AHM12}), when $\re \gl = 0$, we have that $\bbQ_{\gy^\gl,A} \gh(A) \bbS_{\gf^\gl,A}$ is bounded from $T^p_{\gq(\mb{p})}$ to $T^p_{\gq(\mb{p})+\gd}$ when $p \in (1,2)$ and $\gq \in (0,1)$ satisfy
		\begin{equation*}
			\frac{1}{p} = (1-\gq) + \frac{\gq}{2}, \quad 0 = (1-\gq)(1 - j(\mb{p})) + \gq\bigg(\frac{1}{2} - j(\mb{p})\bigg).
		\end{equation*}
		This occurs when $p = i(\mb{p})$.
		Applying this with $\gl = 0$ yields boundedness of $\bbQ_{\gy,A} \gh(A) \bbS_{\gf,A}$ from $T^\mb{p}$ to $T^{\mb{p} + \gd}$.
\end{proof}

	Finally, we present an abstract Calder\'on reproducing formula, which is ubiquitous in the study of abstract Hardy spaces.
	Whenever $\gy \in \gY_+^+$ and $\gf \in H^\infty$, we can define a bounded holomorphic function
	\begin{equation*}
		\gF_{\gy,\gf}(z) := \int_0^\infty \gy_t(z) \gf_t(z) \, \frac{dt}{t}, \quad z \in S_\gm.
	\end{equation*}
	The integral converges absolutely because $\gy\gf \in \gY_+^+$.
	It is not hard to show that
	\begin{equation*}
		\bbS_{\gy,A} \bbQ_{\gf,A} \subseteq \gF_{\gy,\gf}(A)
	\end{equation*}
	as operators on $\overline{\mc{R}(A)}$, with equality when $\gf \in \gY_+^+$.
	
	In \cite[Proposition 4.2]{AS16} it is shown that if $\gf \in H^\infty$ is nondegenerate, then there exists $\gy \in \gY_\infty^\infty$ such that $\gF_{\gf,\gy} \equiv 1$.
	This implies the following abstract Calder\'on reproducing formula.\index{Calder\'on reproducing formula}
	
	\begin{thm}\label{thm:CRF}
		Suppose $\gf \in H^\infty$ is nondegenerate.
		Then there exists $\gy \in \gY_\infty^\infty$ such that
		\begin{equation}\label{eqn:crf}
			\bbS_{\gy,A}\bbQ_{\gf,A} \subseteq I_{\overline{\mc{R}(A)}}
		\end{equation}
		as operators on $\overline{\mc{R}(A)}$, with equality when $\gf \in \gY_+^+$.
		Furthermore, if $\gf \in \gY_+^+$, then the operator $\bbQ_{\gy,A} \bbS_{\gf,A}$ is a projection from $L^2(\bbR^{1+n}_+)$ onto $\bbQ_{\gy,A} \overline{\mc{R}(A)}$.
	\end{thm}
	
	We refer to a pair $(\gf,\gy)$, with $\gf \in H^\infty$, $\gy \in \gY_+^+$ and satisfying \eqref{eqn:crf}, as \emph{Calder\'on siblings}.\index{Calder\'on siblings}
	Here is a simple example of the use of the abstract Calder\'on reproducing formula.
	
	\begin{cor}\label{cor:Q-injectivity}
		Suppose $\gf \in H^\infty$ is nondegenerate.
		Then the extension operator $\map{\bbQ_{\gf,A}}{\overline{\mc{R}(A)}}{L^\infty(\bbR_+ : \overline{\mc{R}(A)})}$ is injective.
	\end{cor}
	
	\begin{proof}
		Let $\gy \in \gY_+^+$ be a Calder\'on sibling of $\gf$, and suppose $f \in \overline{\mc{R}(A)}$ with $\bbQ_{\gf,A}f = 0$.
		Then $\bbQ_{\gf,A}f \in L^2(\bbR_+ : \overline{\mc{R}(A)})$, so $\bbS_{\gy,A} \bbQ_{\gf,A} f$ is defined, and by \eqref{eqn:crf} we have
		\begin{equation*}
			f = \bbS_{\gy,A} \bbQ_{\gf,A} f = 0,
		\end{equation*}
		proving injectivity of $\bbQ_{\gf,A}$.
	\end{proof}

\chapter{Adapted Besov--Hardy--Sobolev Spaces}\label{chap:ahsb}

Throughout this chapter we fix an operator $A$ satisfying the Standard Assumptions (see Definition \ref{dfn:std-assns}).
As in the previous chapter, we implicitly work with $\bbC^N$-valued functions without referencing this in the notation.

\section{Initial definitions, equivalent norms, and duality}\label{sec:hsbbasic}

The adapted Hardy--Sobolev and Besov spaces are defined, roughly speaking, by measuring extensions by $\bbQ_{\gy,A}$ in tent spaces and $Z$-spaces respectively.
This mimics the characterisation of classical Hardy--Sobolev and Besov spaces given in Theorem \ref{thm:BHS-Xspace}.
We will soon show that the resulting function spaces are independent of $\gy$, provided that $\gy$ has sufficient decay at $0$ and $\infty$.

\begin{dfn}
	Let $\gy \in H^\infty$ and let $\mb{p}$ be an exponent.
	We define the sets
	\begin{align*}
		{\bbH}^\mb{p}_{\gy,A} &:= \{f \in \overline{\mc{R}(A)} : \bbQ_{\gy,A} f \in T^\mb{p} \cap L^2(\bbR^{1+n}_+)\}, \\
		{\bbB}^\mb{p}_{\gy,A} &:= \{f \in \overline{\mc{R}(A)} : \bbQ_{\gy,A} f \in Z^\mb{p} \cap L^2(\bbR^{1+n}_+)\},
	\end{align*}
	equipped with what we will show to be quasinorms (Proposition \ref{prop:quasinorm-verification})
	\begin{equation*}
		\nm{f \mid {\bbH}^\mb{p}_{\gy,A}} := \nm{\bbQ_{\gy,A} f}_{T^\mb{p}}, \qquad
		\nm{f \mid {\bbB}^\mb{p}_{\gy,A}} := \nm{\bbQ_{\gy,A} f}_{Z^\mb{p}}.
	\end{equation*}
	We call these spaces \emph{pre-Hardy--Sobolev and pre-Besov spaces adapted to $A$} (respectively), and we call $\gy$ an \emph{auxiliary function}.
\end{dfn}

Generally we want to refer to the pre-Hardy--Sobolev and pre-Besov spaces simultaneously.
Thus we write
\begin{equation*}
	{\bbX}_{\gy,A}^\mb{p} := \{f \in \overline{\mc{R}(A)} : \bbQ_{\gy,A} f \in X^\mb{p} \cap L^2(\bbR^{1+n}_+)\},
\end{equation*}
where the pair $(X,\bbX)$ is either $(T,\bbH)$ or $(Z,\bbB)$, following the convention initiated in Section \ref{sec:unification}, and speak of \emph{pre-Besov--Hardy--Sobolev spaces},

\begin{rmk}\label{rmk:comp-definedness}
The condition $\bbQ_{\gy,A} f \in L^2(\bbR^{1+n}_+)$ is automatically satisfied when $\gy \in \gY_+^+$.
We impose it so that $\bbS_{\gf,A} \bbQ_{\gy,A} f$ is defined for all $f$ in ${\bbX}^\mb{p}_{\gy,A}$ when $\gy \in H^\infty \sm \gY_+^+$.
\end{rmk}

\begin{prop}\label{prop:quasinorm-verification}
	Let $\gy \in H^\infty$ and let $\mb{p}$ be an exponent.
	Then $\nm{\cdot \mid {\bbX}_{\gy,A}^\mb{p}}$ is a quasinorm on ${\bbX}_{\gy,A}^\mb{p}$.
\end{prop}

\begin{proof}
	The only quasinorm property which does not follow directly from linearity of $\bbQ_{\gy,A}$ and the corresponding quasinorm properties of $X^\mb{p}$ is positive definiteness.
	To show this, suppose $f \in {\bbX}_{\gy,A}^\mb{p}$ and $\nm{f \mid {\bbX}_{\gy,A}^\mb{p}} = 0$.
	Then $\bbQ_{\gy,A} f = 0$ in $X^\mb{p}$, and hence also in $L^\infty(\bbR_+ : \overline{\mc{R}(A)})$.
	By injectivity of $\map{\bbQ_{\gy,A}}{\overline{\mc{R}(A)}}{L^\infty(\bbR_+ : \overline{\mc{R}(A)})}$ (Corollary \ref{cor:Q-injectivity}), we conclude that $f = 0$.
\end{proof}

The following proposition quantifies the amount of decay needed on the auxiliary function $\gy$ in order to ensure that the $\bbX_{\gy,A}^\mb{p}$ quasinorm is equivalent to the $\bbX_{\gf,A}^\mb{p}$ quasinorm whenever $\gf$ has decay of arbitrarily high order at $0$ and $\infty$.

\begin{prop}[Independence on auxiliary function]\label{prop:equivalent-quasinorms}
	Let $\gf \in \gY_\infty^\infty$ and $\gy \in H^\infty$ be nondegenerate, let $\mb{p}$ be an exponent, and suppose that either
	\begin{itemize}
		\item $i(\mb{p}) \leq 2$ and $\gy \in \gY_{\gq(\mb{p})+}^{(-\gq(\mb{p}) + n|\frac{1}{2} - j(\mb{p})|)+}$, or
		\item $i(\mb{p}) \geq 2$ and $\gy \in \gY_{(\gq(\mb{p}) + n|\frac{1}{2} - j(\mb{p})|)+}^{-\gq(\mb{p})+} \cap \gY_+^+$.
	\end{itemize}
	Then ${\bbX}_{\gy,A}^\mb{p} = {\bbX}_{\gf,A}^\mb{p}$ with equivalent quasinorms.
\end{prop}

Observe that if $\gy \in \gY_\infty^\infty$ then the requirements above are satisfied for all $\mb{p}$.

\begin{proof}
	First let $\gn \in \gY_\infty^\infty$ be a Calder\'on sibling of $\gy$.
	Then for $f \in {\bbX}_{\gy,A}^\mb{p}$ we have
        \begin{equation*}
		\nm{f \mid {\bbX}_{\gf,A}^\mb{p}}
		= \nm{\bbQ_{\gf,A} f}_{X^\mb{p}}
		= \nm{\bbQ_{\gf,A} \bbS_{\gn,A} \bbQ_{\gy,A} f}_{X^\mb{p}} 
		\lesssim \nm{\bbQ_{\gy,A} f}_{X^\mb{p}}
		= \nm{f \mid {\bbX}_{\gy,A}^\mb{p}}.
	\end{equation*}
	In the second equality we use the Calder\'on reproducing formula \eqref{eqn:crf} along with the fact that $\bbS_{\gn,A} \bbQ_{\gy,A} f$ is defined for $f \in \bbX_{\gy,A}^\mb{p}$ (see Remark \ref{rmk:comp-definedness}).
	The quasinorm estimate follows from Theorem \ref{thm:QS-compn-bddness}, by the Standard Assumptions along with $\gf,\gn \in \gY_\infty^\infty$.
	
	Now let $\gn \in \gY_\infty^\infty$ be a Calder\'on sibling of $\gf$ and repeat the previous argument with the roles of $\gf$ and $\gy$ reversed, using the additional assumptions on $\gy$ to apply Theorem \ref{thm:QS-compn-bddness}.
	This leads to the reverse estimate
	\begin{equation*}
		\nm{f \mid {\bbX}_{\gy,A}^\mb{p}} \lesssim \nm{f \mid {\bbX}_{\gf,A}^\mb{p}},
	\end{equation*}
	which completes the proof.
\end{proof}

\begin{dfn}\index{space!adapted Besov--Hardy--Sobolev}
	For all exponents $\mb{p}$ we define
	\begin{equation*}
		{\bbX}_{A}^\mb{p} := {\bbX}_{\gy,A}^\mb{p},
	\end{equation*}
	where any auxiliary function $\gy \in \gY_\infty^\infty$ may be used to define the space and its corresponding quasinorm.
	We also define $\gY({\bbX}_{A}^\mb{p})$ to be the set of all nondegenerate $\gf \in H^\infty$ such that ${\bbX}_{\gf,A}^\mb{p} = {\bbX}_{A}^\mb{p}$ with equivalent quasinorms.
With this notation at hand, Proposition \ref{prop:equivalent-quasinorms} tells us that
\begin{align*}
	\gY_{\gq(\mb{p})+}^{(-\gq(\mb{p}) + n|\frac{1}{2} - j(\mb{p})|)+} &\cap H^\infty \subset \gY({\bbX}_{A}^\mb{p})  &(i(\mb{p}) \leq 2), \\
	\gY_{(\gq(\mb{p}) + n|\frac{1}{2} - j(\mb{p})|)+}^{-\gq(\mb{p})+} &\cap \gY_+^+ \subset \gY({\bbX}_{A}^\mb{p}) &(i(\mb{p}) > 2).
\end{align*}
\end{dfn}

In Section \ref{section:bisectorial-hfc} we considered the positive and negative spectral subspaces
\begin{equation*}
	\overline{\mc{R}(A)}^\pm := \gc^\pm(A) \overline{\mc{R}(A)}.
\end{equation*}
These may be used to define corresponding positive and negative spectral subspaces of $\bbX_A^\mb{p}$.

\begin{dfn}\index{space!adapted Besov--Hardy--Sobolev!positive and negative}
	Let $\mb{p}$ be an exponent.
	We define the \emph{positive and negative pre-Besov--Hardy--Sobolev spaces} by
	\begin{equation*}
		{\bbX}_{A}^{\mb{p},\pm} := {\bbX}_{A}^\mb{p} \cap \overline{\mc{R}(A)}^\pm,
	\end{equation*}
	equipped with any of the equivalent ${\bbX}_{A}^\mb{p}$ quasinorms.
	Often we just refer to these as \emph{spectral subspaces}.
        In Corollary \ref{cor:Xpm-charn} we characterise the spectral subspaces as images of the spectral projections $\gc^\pm(A)$.
\end{dfn}

The spaces ${\bbX}_{A}^\mb{p}$ may also be characterised in terms of the contraction maps $\bbS_{\gy,A}$ for any $\gy \in \gY_+^+$.
Recall that $X^2 = T^2 = Z^2 = L^2(\bbR^{1+n}_+)$.

\begin{prop}[Characterisation by contraction maps]\label{prop:Xss-charn}\index{space!adapted Besov--Hardy--Sobolev!characterisation by contraction maps}
	Let $\mb{p}$ be an exponent and let $\gy \in \gY_+^+$ be nondegenerate.
	Suppose furthermore that
	\begin{equation*}
		\gy \in \left\{ \begin{array}{ll} \gY_{(-\gq(\mb{p}) + n|\frac{1}{2} - j(\mb{p})|)+}^{\gq(\mb{p})+} & (i(\mb{p}) \leq 2) \\ \gY_{(-\gq(\mb{p}))+}^{(\gq(\mb{p}) + n| \frac{1}{2} - j(\mb{p})| )+} & (i(\mb{p}) \geq 2) \end{array} \right.
	\end{equation*}
	(note that $\gy \in \gY_\infty^\infty$ always satisfies this assumption).
	Then we have
	\begin{equation}\label{eqn:S-map-characterisation}
		{\bbX}_{A}^\mb{p} = \bbS_{\gy,A} ( X^2 \cap X^\mb{p} ),
	\end{equation}
	and the mapping
	\begin{equation*}
		f \mapsto \inf\{\nm{F}_{X^\mb{p}} : F \in X^2 \cap X^\mb{p}, \; \bbS_{\gy,A} F = f\}
	\end{equation*}
	is an equivalent quasinorm on ${\bbX}_{A}^\mb{p}$.
\end{prop}

\begin{proof}
	Fix a Calder\'on sibling $\gf \in \gY_\infty^\infty$ of $\gy$.
	First we show the equality \eqref{eqn:S-map-characterisation}.
	Suppose $f \in {\bbX}_{A}^\mb{p}$.
	Then $\bbQ_{\gf,A} f \in X^2 \cap X^\mb{p}$, and by Theorem \ref{eqn:crf} we have $f = \bbS_{\gy,A} (\bbQ_{\gf,A} f)$.
	Conversely, suppose that $f = \bbS_{\gy,A} F$ for some $F \in X^2 \cap X^\mb{p}$.
	Then $f \in \overline{\mc{R}(A)}$, and $\bbQ_{\gf,A} f = \bbQ_{\gf,A} \bbS_{\gy,A} F$.
	Theorem \ref{thm:QS-compn-bddness} then implies that $\bbQ_{\gf,A} f \in X^\mb{p}$, which shows that $f \in {\bbX}_{A}^\mb{p}$.
	This proves \eqref{eqn:S-map-characterisation}.
	
	Now for the quasinorm equivalence.
	Suppose $f \in {\bbX}_{A}^\mb{p}$.
	Then $f = \bbS_{\gy,A} \bbQ_{\gf,A} f$ with $\bbQ_{\gf,A} f \in X^2 \cap X^\mb{p}$, and so
	\begin{equation*}
		\inf\{ \nm{F}_{X^\mb{p}} : F \in X^2 \cap X^\mb{p}, \bbS_{\gy,A} F = f\}
		\leq \nm{\bbQ_{\gf,A} f}_{X^\mb{p}} \simeq \nm{f}_{{\bbX}_{A}^\mb{p}}.
	\end{equation*}
	Conversely, suppose $F \in X^2 \cap X^\mb{p}$ and $\bbS_{\gy,A} F = f$.
	Then
	\begin{equation*}
		\nm{F}_{X^\mb{p}} \gtrsim \nm{\bbQ_{\gf,A} \bbS_{\gy,A} F}_{X^\mb{p}} 
                = \nm{\bbQ_{\gf,A} f}_{X^\mb{p}} 
		\simeq \nm{f}_{{\bbX}_{A}^\mb{p}}
	\end{equation*}
	using Theorem \ref{thm:QS-compn-bddness} again in the first line, and thus completing the proof.
\end{proof}

\begin{rmk}
	In Propositions \ref{prop:equivalent-quasinorms} and \ref{prop:Xss-charn} we quantify precisely how much decay (depending on $\mb{p}$) is sufficient for an auxiliary function $\gf$ to characterise $\bbX_A^\mb{p}$ (either via $\bbQ_{\gf,A}$ or $\bbS_{\gf,A}$).
	In both cases having $\gf \in \gY_\infty^\infty$ is sufficient, so in deducing abstract properties of the spaces $\bbX_A^\mb{p}$ we need only consider such functions.
	However, for applications to boundary value problems we must consider specific auxiliary functions without such decay, so this degree of precision is necessary.
\end{rmk}

\begin{cor}[Density of intersections]\label{cor:adapted-space-int-dens}
	Let $\mb{p}$ and $\mb{q}$ be exponents, and suppose $\bbX_1,\bbX_2 \in \{\bbH,\bbB\}$.
	If $\mb{p}$ is finite then $(\bbX_1)_A^\mb{p} \cap (\bbX_2)_A^{\mb{q}}$ is dense in $(\bbX_1)_A^{\mb{p}}$.
	Otherwise, $(\bbX_1)_A^\mb{p} \cap (\bbX_2)_A^{\mb{q}}$ is weak-star dense in $(\bbX_1)_A^{\mb{p}}$.
\end{cor}

\begin{proof}
	We suppose that $\mb{p}$ is finite; the same argument works for infinite $\mb{p}$, replacing limits with weak-star limits and norms with appropriate duality pairings.
	
	Suppose $f \in (\bbX_1)_A^\mb{p}$, and fix $\gy \in \gY_\infty^\infty$.
	By Proposition \ref{prop:Xss-charn} we have $f = \bbS_{\gy,A} F$ for some $F \in T^2 \cap (X_1)^\mb{p}$, and by Proposition \ref{prop:mixed-x-density} we have $F = \lim_{k \to \infty} F_k$ (limit in $(X_1)^\mb{p}$) for some sequence $(F_k)_{k \in \bbN}$ in $T^2 \cap (X_1)^\mb{p} \cap (X_2)^\mb{q}$.
	For all $k \in \bbN$ define
	\begin{equation*}
		f_k := \bbS_{\gy,A} F_k \in (\bbX_1)_A^\mb{p} \cap (\bbX_2)_A^{\mb{q}}.
	\end{equation*}
	Then we have, again using Proposition \ref{prop:Xss-charn},
	\begin{equation*}
		\lim_{k \to \infty} \nm{f - f_k}_{(\bbX_1)_A^{\mb{p}}}
		\lesssim \lim_{k \to \infty} \nm{F - F_k}_{(X_1)^\mb{p}} = 0.
	\end{equation*}
	This proves the claimed density.
\end{proof}

The pre-Besov--Hardy--Sobolev spaces inherit a duality pairing from $\overline{\mc{R}(A)} \subset L^2(\bbR^n)$.
However, we cannot identify $\bbX_{A^*}^{\mb{p}^\prime}$ with the dual of $\bbX_{A}^{\mb{p}}$, because in general these spaces are incomplete, while the dual of a quasinormed space is always complete.
We deal with completions in Section \ref{sec:completions}.

\begin{prop}[Duality estimate]\index{duality!of adapted BHS spaces}
	Let $\mb{p}$ be an exponent.
	Then for all $f \in {\bbX}_{A}^\mb{p}$ and $g \in {\bbX}_{A^*}^{\mb{p}^\prime}$ we have
	\begin{equation}\label{eqn:pre-duality}
		|\langle f,g \rangle| \lesssim \nm{f}_{{\bbX}_{A}^\mb{p}} \nm{g}_{{\bbX}_{A^*}^{\mb{p}^\prime}},
	\end{equation}
	where $\langle \cdot,\cdot \rangle$ is the inner product on $L^2(\bbR^n)$.
\end{prop}

\begin{proof}
	Let $\gf,\gy \in \gY_\infty^\infty$ be nondegenerate and suppose $\varepsilon > 0$.
	By Proposition \ref{prop:Xss-charn} there exist $F \in X^2 \cap X^\mb{p}$ and $G \in X^2 \cap X^{\mb{p}^\prime}$ such that $\bbS_{\gf,A} F = f$ and $\bbS_{\gy,A^*} G = g$, with
	\begin{equation*}
		\nm{F}_{X^\mb{p}} \lesssim (1+\varepsilon) \nm{f}_{{\bbX}_{A}^\mb{p}} \quad \text{and} \quad \nm{G}_{X^{\mb{p}^\prime}} \lesssim (1+\varepsilon) \nm{g}_{{\bbX}_{A^*}^{\mb{p}^\prime}}.
	\end{equation*}
	Since $\bbS_{\gf,A}^* = \bbQ_{\td{\gf},A^*}$, and using that the $L^2(\bbR^{1+n}_+)$ inner product is a duality pairing for tent and $Z$-spaces, we thus have
	\begin{align}
		|\langle f,g \rangle|
		&= |\langle F, \bbQ_{\td{\gf},A^*} \bbS_{\gy,A^*} G\rangle_{L^2(\bbR^{1+n}_+)}| \nonumber\\
		&\lesssim \nm{F}_{X^\mb{p}} \nm{\bbQ_{\td{\gf},A^*} \bbS_{\gy,A^*} G}_{X^{\mb{p}^\prime}} \nonumber\\
		&\lesssim \nm{F}_{X^\mb{p}} \nm{G}_{X^{\mb{p}^\prime}} \label{line:QSuse} \\
		&\lesssim (1+\varepsilon)^2 \nm{f}_{{\bbX}_{A}^\mb{p}} \nm{g}_{{\bbX}_{A^*}^{\mb{p}^\prime}}, \nonumber
	\end{align}
	where \eqref{line:QSuse} follows from Theorem \ref{thm:QS-compn-bddness}.
	Since $\varepsilon > 0$ was arbitrary we obtain \eqref{eqn:pre-duality}.
\end{proof}

The tent space and $Z$-space embeddings of Chapter \ref{chap:fs-prelims} immediately yield corresponding embeddings of pre-Besov--Hardy--Sobolev spaces.

\begin{prop}[Mixed embeddings]\label{prop:adapted-space-embeddings}\index{embedding!of adapted BHS spaces}
	Let $\mb{p}$ and $\mb{q}$ be exponents with $\mb{p} \neq \mb{q}$ and $\mb{p} \hookrightarrow \mb{q}$.
	Then we have the embedding
	\begin{equation*}
		(\bbX_0)_{A}^{\mb{p}} \hookrightarrow (\bbX_1)_{A}^{\mb{q}},
	\end{equation*}
	where $\bbX_0,\bbX_1 \in \{\bbH,\bbB\}$, and the corresponding embedding holds for spectral subspaces.
\end{prop}

\begin{proof}
	This follows directly from the definition of the spaces ${\bbX}_{A}^\mb{p}$ and from Theorem \ref{thm:mixed-emb}.
	Spectral subspace versions follow by intersecting with $\overline{\mc{R}(A)}^\pm$.
\end{proof}

\begin{rmk}
	For $\mb{p} = (p,s)$ we sometimes write $\bbX_A^\mb{p} = \bbX_{s,A}^p$, and for $\mb{p} = (\infty,s;\ga)$ we may write $\bbX_A^\mb{p} = \bbX_{s;\ga,A}^\infty$.
	This notation is a bit heavy, so we avoid it whenever possible, except in the case of $\bbX_{0,A}^2$, which we abbreviate as $\bbX_A^2$ (as is standard).
      \end{rmk}

      \begin{rmk}
        There are other ways to define adapted Hardy--Sobolev (or more generally, Triebel--Lizorkin) and Besov spaces.
        See in particular the work of Kunstmann and Ullmann \cite{KU14} and Ullmann's thesis \cite{aUthesis}, in which Triebel--Lizorkin spaces adapted to `$\mc{R}_s$-sectorial' operators are constructed via a vertical $s$-power function norm (a vertical square function when $s=2$).
        We have not investigated possible links between our adapted function spaces and theirs.
      \end{rmk}

\section{Mapping properties of the holomorphic functional calculus}\label{sec:Xhfc-props}

Just as we proved independence on auxiliary functions in the previous section, we can prove various mapping properties (including boundedness) of the holomorphic functional calculus between pre-Besov--Hardy--Sobolev spaces.\index{functional calculus!on adapted BHS spaces}
The first result says heuristically that an operator of homogeneity $\gd$ decreases regularity by $\gd$. 

\begin{thm}\label{thm:fcprops}
	Let $\mb{p}$ be an exponent and $\gd \in \bbR$.
	Suppose $\gh \in \gY_{-\gd}^{\gd}$.
	Then for all $f \in \mc{D}(\gh(A)) \cap \bbX_{A}^\mb{p}$,
	\begin{equation*}
		\nm{\gh(A)f}_{\bbX_{A}^{\mb{p}+\gd}} \lesssim \nm{\gh}_{\gY_{-\gd}^{\gd}} \nm{f}_{\bbX_{A}^\mb{p}},
	\end{equation*}
	and likewise for spectral subspaces.
\end{thm}

\begin{proof}
	Let $\gf,\gy \in \gY_\infty^\infty$ and let $\gn \in \gY_\infty^\infty$ be a Calder\'on sibling of $\gy$.
	Then for all $f \in \mc{D}(\gh(A)) \cap {\bbX}_{A}^\mb{p}$ we have
        \begin{equation*}
		\nm{\gh(A)f}_{{\bbX}_{A}^{\mb{p}+\gd}}
		\simeq \nm{\bbQ_{\gf,A} \gh(A) \bbS_{\gn,A} \bbQ_{\gy,A} f}_{X^{\mb{p}+\gd}}
		\lesssim \nm{\gh}_{\gY_{-\gd}^{\gd}} \nm{\bbQ_{\gy}f}_{X^\mb{p}}
		\simeq \nm{\gh}_{\gY_{-\gd}^{\gd}} \nm{f}_{{\bbX}_{A}^\mb{p}}
	\end{equation*}
	where the quasinorm estimate follows from Theorem \ref{thm:QS-compn-bddness}.
	To incorporate spectral subspaces in this argument, write $f \in \mc{D}(\gh(A)) \cap {\bbX}_{A}^{\mb{p},\pm}$ as $f = \gc^\pm(A) f$ and
	\begin{equation*}
		\gh(A) f = \gh(A) \gc^\pm(A) f = \gc^\pm(A) \gh(A) f,
	\end{equation*}
	and note that this shows that $\gh(A)$ maps $\mc{D}(\gh(A)) \cap\overline{\mc{R}(A)}^\pm$ into $\overline{\mc{R}(A)}^\pm$.
\end{proof}

Because the spaces ${\bbX}_{A}^\mb{p}$ may be incomplete, we cannot extend the operators $\gh(A)$ by boundedness without introducing completions.
This is done in Section \ref{sec:completions}.
Of course, when $\gh \in H^\infty$ we have $\mc{D}(\gh(A)) = \overline{\mc{R}(A)}$, and so we obtain bounded holomorphic functional calculus in the following sense.

\begin{cor}\label{cor:Xbhfc}
	Let $\mb{p}$ be an exponent and $\gh \in H^\infty$.
	Then for all $f \in {\bbX}_{A}^\mb{p}$,
	\begin{equation*}
		\nm{\gh(A)f}_{{\bbX}_{A}^\mb{p}} \lesssim \nm{\gh}_\infty \nm{f}_{{\bbX}_{A}^\mb{p}},
	\end{equation*}
	and likewise for spectral subspaces.
\end{cor}

This allows us to characterise the spectral subspaces $\bbX_A^{\mb{p}, \pm}$ as images of spectral projections.

\begin{cor}\label{cor:Xpm-charn}\index{space!adapted Besov--Hardy--Sobolev!positive and negative}
	Let $\mb{p}$ be an exponent.
	Then we have
	\begin{equation*}
		{\bbX}_{A}^{\mb{p},\pm} = \gc^\pm(A) {\bbX}_{A}^\mb{p}.
	\end{equation*}
\end{cor}

\begin{proof}
	If $f \in {\bbX}_{A}^{\mb{p},\pm}$ then by definition we have $f = \gc^\pm(A) f \in \gc^\pm(A) {\bbX}_{A}^\mb{p}$.
	Conversely, if $f \in \gc^\pm(A) {\bbX}_{A}^\mb{p}$, then $f$ is in $\overline{\mc{R}(A)}^\pm$, and by Corollary \ref{cor:Xbhfc} we have $f \in {\bbX}_{A}^\mb{p}$.
	Therefore $f \in \overline{\mc{R}(A)}^\pm \cap {\bbX}_{A}^\mb{p} = {\bbX}_{A}^{\mb{p},\pm}$.
\end{proof}

The power functions $A^\gl := [z \mapsto z^\gl](A)$ for $\gl \in \bbR \sm \{0\}$ (see Section \ref{section:bisectorial-hfc}) are generally unbounded on $\overline{\mc{R}(A)}$, but they do map between our adapted spaces with a shift in regularity (when we intersect with the domain).
This is a direct consequence of Theorem \ref{thm:fcprops} since $[z \mapsto z^\gl] \in \gY_{\gl}^{-\gl}$; the norm equivalence is obtained by applying Theorem \ref{thm:fcprops} with both $\gl$ and $-\gl$.

\begin{cor}\label{cor:regbump}
	Let $\mb{p}$ be an exponent and $\gl \in \bbR \sm \{0\}$.
	Then for all $f \in \mc{D}(A^\gl) \cap {\bbX}_{A}^\mb{p}$,
	\begin{equation*}
		\nm{A^\gl f}_{{\bbX}_{A}^{\mb{p} - \gl}} \simeq \nm{f}_{{\bbX}_{A}^\mb{p}}.
	\end{equation*}
\end{cor}

Since the operators $A^\gl$ are all densely defined in $\overline{\mc{R}(A)}$, and since $A^{\gl_0} A^{\gl_1} = A^{\gl_0 + \gl_1}$ whenever this is meaningful, we have almost proven that $A^\gl$ is an isomorphism from ${\bbX}_{A}^\mb{p}$ to ${\bbX}_{A}^{\mb{p} - \gl}$.
We must extend everything by boundedness to make this rigorous.
As previously mentioned, this requires the introduction of completions.

\section{Completions, interpolation, and inclusions}\label{sec:completions}

\index{completion!of adapted BHS spaces}

The spaces ${\bbX}_{A}^\mb{p}$ defined in the previous section are called \emph{pre}-Besov--Hardy--Sobolev spaces because, with the exception of ${\bbX}_{A}^2 = \overline{\mc{R}(A)}$, they need not be complete.
One could try to solve this problem by taking arbitrary completions $({\bbX}_{A}^\mb{p})^c$ of ${\bbX}_{A}^\mb{p}$ and declaring these to be the Besov--Hardy--Sobolev spaces adapted to $A$.
However, if we take this approach, then for different exponents $\mb{p}_1$ and $\mb{p}_2$, there may not exist a natural topological vector space in which the completions $({\bbX}_{A}^{\mb{p}_i})^c$ both embed.
This presents difficulties when discussing interpolants of these completions.
The impact of this problem on abstract Hardy space theory seems to have first been discussed by the second author, McIntosh, and Morris \cite{AMM15}.
We avoid this issue by introducing certain \emph{canonical completions} within tent and $Z$-spaces (and hence within $L^0(\bbR^{1+n}_+)$).
If another completion is possible---for example, within the space $\mc{Z}^\prime(\bbR^n)$ of distributions modulo polynomials, in which the classical smoothness spaces are embedded---then we are free to identify this with our canonical completion.

By a \emph{completion} of a quasinormed space\index{completion!of a quasinormed space} $Q$ we mean a continuous injective map $\map{\gi}{Q}{\td{Q}}$, where $\td{Q}$ is a complete quasinormed space and $\gi(Q)$ is dense in $\td{Q}$.
By a \emph{weak-star completion}\index{completion!weak-star} of $Q$, we mean $\map{\gi}{Q}{\td{Q}}$ as above, where $\td{Q}$ is a dual space and where $\gi(Q)$ is weak-star dense in $\td{Q}$.
We may refer to $\td{Q}$ as the completion, with the inclusion $\gi$ being implicit.

In this section, whenever $\mb{p}$ is infinite, we interpret `completion' to mean `weak-star completion'.

\begin{dfn}\label{dfn:canonical-completion}
	For an exponent $\mb{p}$ and an auxiliary function $\gy \in \gY({\bbX}_{A}^\mb{p})$, define the \emph{canonical completion}\index{completion!canonical}\index{completion!of adapted BHS spaces}
	\begin{equation*}
		\gy{\mb{X}}_{A}^\mb{p} := \overline{\bbQ_{\gy,A} {\bbX}_{A}^\mb{p}} \subset X^\mb{p}
	\end{equation*}
	and likewise
	\begin{equation*}
		\gy{\mb{X}}_{A}^{\mb{p},\pm} := \overline{\bbQ_{\gy,A} {\bbX}_{A}^{\mb{p},\pm}} \subset \gy{\mb{X}}_{A}^\mb{p}
	\end{equation*}
	where the closures are taken in the $X^\mb{p}$ quasinorm when $\mb{p}$ is finite, and in the weak-star topology on $X^\mb{p}$ when $\mb{p}$ is infinite.
We equip $\gy\mb{X}_A^\mb{p}$ and $\gy\mb{X}_A^{\mb{p},\pm}$ with the $X^\mb{p}$ quasinorm, so that $\gy{\mb{X}}_{A}^\mb{p}$ and $\gy{\mb{X}}_{A}^{\mb{p},\pm}$ are quasi-Banach spaces.
\end{dfn}

\begin{rmk}
Heuristically, what we are doing here is shifting viewpoint from \emph{functions} to \emph{$A$-adapted $\gy$-expansions of functions}, with each choice of $\gy$ giving a different image of the space $\bbX_A^\mb{p}$.
This is like considering Fourier transforms of functions instead of functions themselves, except in this case we do not have a fixed `Fourier transform'; we only have a collection of `expanded Fourier multipliers' (the operators $\bbQ_{\gy,A}$).
This corresponds to the redundancy in the choice of auxiliary function used to define $\bbX_A^\mb{p}$.
The advantage of this approach is that the space $\bbQ_{\gy,A} \bbX_A^\mb{p}$ of $A$-adapted $\gy$-expansions of $\bbX_A^\mb{p}$-functions has a natural completion, while (at least it seems that) the space $\bbX_A^\mb{p}$ does not.
\end{rmk}

\begin{prop}\label{prop:completion}
	Fix $\mb{p}$ and $\gy$ as in Definition \ref{dfn:canonical-completion}.
	Then $\map{\bbQ_{\gy,A}}{{\bbX}_{A}^\mb{p}}{\gy{\mb{X}}_{A}^\mb{p}}$ and $\map{\bbQ_{\gy,A}}{{\bbX}_{A}^{\mb{p},\pm}}{\gy{\mb{X}}_{A}^{\mb{p},\pm}}$ are completions of ${\bbX}_{A}^\mb{p}$ and ${\bbX}_{A}^{\mb{p},\pm}$.
\end{prop}

\begin{proof}
	By construction the spaces $\gy{\mb{X}}_{s,A}^p$ and $\gy{\mb{X}}_{s,A}^{p,\pm}$ are complete and contain $\bbQ_{\gy,A} {\bbX}_{A}^\mb{p}$ and $\bbQ_{\gy,A} {\bbX}_{A}^{\mb{p},\pm}$ respectively as dense subspaces (in the weak-star topology when $\mb{p}$ is infinite).
	The map $\map{\bbQ_{\gy,A}}{{\bbX}_{A}^\mb{p}}{\gy{\mb{X}}_{A}^\mb{p}}$ is continuous since $\gy \in \gY({\bbX}_{A}^\mb{p})$, and injective by Corollary \ref{cor:Q-injectivity}.
	These properties automatically continue to hold for the restrictions of $\bbQ_{\gy,A}$ to the spectral subspaces ${\bbX}_{A}^{\mb{p},\pm}$.
\end{proof}

Completions are unique up to isomorphism, so any completion may be identified with any canonical completion.
It is useful to make this identification precise.

\begin{prop}[Identification of completions]\label{prop:completion-identification}
	Fix $\mb{p}$ and $\gy$ as in Definition \ref{dfn:canonical-completion} and suppose that $\map{\gi}{{\bbX}_{A}^\mb{p}}{\mb{X}}$ is a completion of ${\bbX}_{A}^\mb{p}$.
	Then there exists a unique map $\map{\mb{Q}_{\gy,A}^\gi}{\mb{X}}{X^\mb{p}}$ such that the triangle
	\begin{equation*}
		\xymatrix{
			\mb{X} \ar[dr]^{\mb{Q}_{\gy,A}^\gi} & \\
			{\bbX}_{A}^\mb{p} \ar[u]^{\gi} \ar[r]_{\bbQ_{\gy,A}} & X^\mb{p}
		}
	\end{equation*}
	commutes, and this map is an isomorphism between $\mb{X}$ and $\gy{\mb{X}}_{A}^\mb{p}$.
	Its inverse is the map $\map{\mb{S}_{\gy,A}^\gi}{X^\mb{p}}{\mb{X}}$ which is the unique continuous extension of the map $\map{\gi \circ \bbS_{\gn,A}}{X^2 \cap X^\mb{p}}{\mb{X}}$ (using the weak-star topology on $X^\mb{p}$ when $\mb{p}$ is infinite) for any $\gn \in \gY_\infty^\infty$ which is a Calder\'on sibling of $\gy$.
	The same results hold if we replace all spaces with corresponding spectral subspaces.
\end{prop}

\begin{proof}
	Since $\map{\bbQ_{\gy,A}}{{\bbX}_{A}^\mb{p}}{\gy{\mb{X}}_{A}^\mb{p}}$ is a completion of ${\bbX}_{A}^\mb{p}$, by the universal property of completions there exists a unique map $\map{\wtd{\mb{Q}}_{\gy,A}^\gi}{\mb{X}}{\gy{\mb{X}}_{A}^\mb{p}}$ such that the triangle
	\begin{equation*}
		\xymatrix{
			\mb{X} \ar[dr]^{\wtd{\mb{Q}}_{\gy,A}^\gi} & \\
			{\bbX}_{A}^\mb{p} \ar[u]^{\gi} \ar[r]_{\bbQ_{\gy,A}} & \gy{\mb{X}}_{A}^\mb{p}
		}
	\end{equation*}
	commutes.
	Hence we have a commutative diagram
	\begin{equation*}
		\xymatrix{
		\mb{X} \ar[dr]^{\wtd{\mb{Q}}_{\gy,A}^\gi} \ar[r]^{\mb{Q}_{\gy,A}^\gi} & X^\mb{p} \\
		{\bbX}_{A}^\mb{p} \ar[u]_{\gi} \ar[r]_{\bbQ_{\gy,A}} & \gy{\mb{X}}_{A}^\mb{p} \ar[u]_{\text{id}} .
		}
	\end{equation*}
	Since
	\begin{equation*}
		(\id \circ \wtd{\mb{Q}}_{\gy,A}^\gi) \circ \gi = \bbQ_{\gy,A} = \mb{Q}_{\gy,A}^\gi \circ \gi,
	\end{equation*}
	by uniqueness we must have $\mb{Q}_{\gy,A}^\gi = \text{id} \circ \wtd{\mb{Q}}_{\gy,A}^\gi = \wtd{\mb{Q}}_{\gy,A}^\gi$.
	Therefore it suffices to show that $\wtd{\mb{Q}}_{\gy,A}^\gi$ satisfies the desired properties.
	
	To show that $\mb{S}_{\gy,A}^\gi \wtd{\mb{Q}}_{\gy,A}^\gi = \id_{\mb{X}}$, observe that we have a commutative diagram
	\begin{equation}\label{eqn:big-diagram}
		\xymatrix{
		\mb{X} \ar[r]^(.4){\wtd{\mb{Q}}_{\gy,A}^\gi} & \gy{\mb{X}}_{A}^\mb{p} \ar@{^{(}->}[r] & X^\mb{p} \ar[r]^{\mb{S}_{\gy,A}^\gi } & \mb{X} \ar[r]^(.4){\wtd{\mb{Q}}_{\gy,A}^\gi} & \gy{\mb{X}}_{A}^\mb{p} \\
		{\bbX}_{A}^\mb{p} \ar[u]^{\gi} \ar[r]^(.4){\bbQ_{\gy,A}} \ar@/_2pc/[rrr]_{\id} & \gy({\bbX}_{A}^\mb{p}) \ar@{^{(}->}[u] \ar@{^{(}->}[r] \ar@/_2pc/[rrr]_{\id} & X^2 \cap X^\mb{p} \ar@{^{(}->}[u] \ar[r]^(.6){\bbS_{\gn,A}} & {\bbX}_{A}^\mb{p} \ar[u]_{\gi} \ar[r]^(.4){\bbQ_{\gy,A}} & \gy({\bbX}_{A}^\mb{p}). \ar@{^{(}->}[u]
		}
	\end{equation}
	Thus we have
	\begin{equation*}
	\mb{S}_{\gy,A}^\gi \wtd{\mb{Q}}_{\gy,A}^\gi \gi = \gi \circ \id = \gi
	\end{equation*}
	and
	\begin{equation*}
		\wtd{\mb{Q}}_{\gy,A}^\gi \mb{S}_{\gy,A}^\gi|_{\gy({\bbX}_{A}^\mb{p})} = \id, 
	\end{equation*}
	so by uniqueness of extensions we must have that $\wtd{\mb{Q}}_{\gy,A}^\gi$ and $\mb{S}_{\gy,A}^\gi$ are mutual inverses.
	The proof for spectral subspaces is identical.
\end{proof}

As a corollary of this argument we can show that $\gy \mb{X}_A^\mb{p}$ is a retract of $X^\mb{p}$.
This is crucial in identifying interpolants.

\begin{cor}\label{cor:QSproj}
	Fix $\mb{p}$, $\gy$, and $\gn$ as in Proposition \ref{prop:completion-identification}, and let $\map{\gi}{{\bbX}_{A}^\mb{p}}{\mb{X}}$ be a completion of ${\bbX}_{A}^\mb{p}$.
	Then the map $\map{\mb{Q}_{\gy,A}^\gi \mb{S}_{\gn,A}^\gi}{X^\mb{p}}{\gy{\mb{X}}_{A}^\mb{p}}$ is the extension of the projection $\map{\bbQ_{\gy, A} \bbS_{\gn, A} }{X^2 \cap X^\mb{p}}{\bbQ_{\gy,A} {\bbX}_{A}^\mb{p}}$ in the appropriate topology (hence independent of $\gi$), and it is a projection onto $\gy{\mb{X}}_{A}^\mb{p}$.
	The same statements hold for spectral subspaces.
\end{cor}

We write $\mb{Q}_{\gy,A} \mb{S}_{\gn,A} := \mb{Q}_{\gy,A}^\gi \mb{S}_{\gn,A}^\gi$ to denote the extension considered above.
Now that we have thought hard enough about completions, we can extend the duality and boundedness results of the previous sections.

\begin{prop}[Duality]\label{prop:completion-duality}\index{duality!of canonical completions}\index{duality!of adapted BHS spaces}
	Let $\mb{p}$ be a finite exponent, and let $\gy, \gn \in \gY_\infty^\infty$ be Calder\'on siblings.
	Then the $L^2$ duality pairing identifies $\td{\gn} {\mb{X}}_{A^*}^{\mb{p}^\prime}$ as the Banach space dual of $\gy {\mb{X}}_{A}^\mb{p}$, and also identifies $\td{\gn} {\mb{X}}_{A^*}^{\mb{p}^\prime,\pm}$ as the Banach space dual of $\gy {\mb{X}}_{A}^{\mb{p},\pm}$.
\end{prop}

\begin{proof}
	If $f \in \gy {\mb{X}}_{A}^\mb{p}$ and $g \in \td{\gn} {\mb{X}}_{A^*}^{\mb{p}^\prime}$, then we immediately have
	\begin{equation*}
		|\langle f,g \rangle| \leq \nm{f}_{X^\mb{p}} \nm{g}_{X^{\mb{p}^\prime}} = \nm{f}_{\gy {\mb{X}}_{A}^\mb{p}} \nm{g}_{\td{\gn} {\mb{X}}_{A^*}^{\mb{p}^\prime}},	
	\end{equation*}
	so every $g \in \td{\gn} {\mb{X}}_{A^*}^{\mb{p}^\prime}$ induces a bounded linear functional on ${\mb{X}}_{A}^\mb{p}$.
	Conversely, suppose $\gf \in (\gy{\mb{X}}_{A}^\mb{p})^\prime$.
	Then we can define a bounded linear functional $\gF \in (X^\mb{p})^\prime$ by
	\begin{equation*}
		\gF(F) := \gf(\mb{Q}_{\gy,A} \mb{S}_{\gn,A} F)
	\end{equation*}
	for all $F \in X^\mb{p}$.
	By $X$-space duality, there exists a function $G_\gF \in X^{\mb{p}^\prime}$ such that
	\begin{equation*}
		\langle F,G_\gF \rangle = \gF(F)
	\end{equation*}
	for all $F \in X^\mb{p}$, satisfying
	\begin{equation*}
		\nm{G_\gF}_{X^{\mb{p}^\prime}} \simeq \nm{\gF}_{(X^\mb{p})^\prime} \lesssim \nm{\gf}_{(\gy{\mb{X}}_{A}^\mb{p})^\prime}.
	\end{equation*}
	Hence for all $f \in \gy{\mb{X}}_{A}^\mb{p}$ we have
	\begin{equation*}
		\langle f,(\mb{Q}_{\gy,A} \mb{S}_{\gn,A})^* G_\gF \rangle
		= \langle f,G_\gF \rangle
		= \gF(f)
		= \gf(f)
	\end{equation*}
	since $f = \mb{Q}_{\gy,A} \mb{S}_{\gn,A} f$.
	Since $\mb{Q}_{\gy,A} \mb{S}_{\gn,A}$ is the continuous extension of $\bbQ_{\gy,A} \bbS_{\gn,A}$ from $X^2 \cap X^\mb{p}$ to $X^\mb{p}$, and since $(\bbQ_{\gy,A} \bbS_{\gn,A})^* = \bbQ_{\td{\gn},A^*} \bbS_{\td{\gy},A^*}$ on $X^2$, we find that $(\mb{Q}_{\gy,A} \mb{S}_{\gn,A})^* = \mb{Q}_{\td{\gn},A^*} \mb{S}_{\td{\gy},A^*}$.
	Therefore
	\begin{equation*}
		\gf(f) = \langle f, G_\gf \rangle
	\end{equation*}
	for all $f \in \gy {\mb{X}}_{A}^\mb{p}$, where $G_\gf = \mb{Q}_{\td{\gn},A^*} \mb{S}_{\td{\gy},A^*} G_\gF \in \td{\gn} {\mb{X}}_{A^*}^{\mb{p}^\prime}$.
	Furthermore we have
	\begin{equation*}
		\nm{G_\gf}_{\td{\gn}{\mb{X}}_{A^*}^{\mb{p}^\prime}}
		= \nm{\mb{Q}_{\td{\gn},A^*} \mb{S}_{\td{\gy},A^*} G_\gF}_{X^{\mb{p}^\prime}}
		\lesssim \nm{\gf}_{(\gy{\mb{X}}_{A}^\mb{p})^\prime}.
	\end{equation*}
	As with every other result in this section, the same proof works for spectral subspaces.
	\end{proof}

\begin{prop}[Boundedness of functional calculus]\label{prop:comp-fc}\index{functional calculus!on canonical completions}\index{functional calculus!on adapted BHS spaces}
	Let $\mb{p}$ be an exponent, $\gd \in \bbR$, and $\gh \in \gY_{-\gd}^{\gd}$.
	Suppose $\map{\gi_1}{\bbX_A^\mb{p}}{\mb{X}}$ and $\map{\gi_2}{\bbX_A^{\mb{p} + \gd}}{\mb{Y}}$ are completions.
	Then $\gh(A)$ extends to a bounded operator $\map{\wtd{\gh(A)}}{\mb{X}}{\mb{Y}}$, in the sense that the diagram
	\begin{equation*}
		\xymatrix{
			\mc{D}(\gh(A)) \cap \bbX_A^\mb{p} \ar[d]_{\gi_1} \ar[rr]^{\gh(A)} & &\bbX_A^{\mb{p} + \gd} \ar[d]^{\gi_2} \\
			\mb{X} \ar[rr]_{\wtd{\gh(A)}} & & \mb{Y} 
		}
	\end{equation*}
	commutes, and that
	\begin{equation}\label{eqn:completed-fc}
		\nm{\wtd{\gh(A)} f}_{\mb{Y}} \lesssim \nm{\gh}_{\gY_{-\gd}^\gd} \nm{f}_{\mb{X}}.
	\end{equation}
	for all $f \in \mb{X}$.
	Similar results hold for spectral subspaces.
\end{prop}

\begin{proof}
	Since $\mc{D}(\gh(A))$ is dense in $\bbX_A^2 = \overline{\mc{R}(A)}$ and since $\bbX_A^2 \cap \bbX_A^\mb{p}$ is dense in $\bbX_A^\mb{p}$ (Corollary \ref{cor:adapted-space-int-dens} for finite exponents, duality for infinite exponents using the weak-star topology), we have that $\mc{D}(\gh(A)) \cap \bbX_A^\mb{p}$ is dense in $\bbX_A^\mb{p}$.
	The result then follows from Theorem \ref{thm:fcprops} and the universal property of completions.
\end{proof}

\begin{rmk}
	Evidently `completed' versions of Corollaries \ref{cor:Xbhfc}, \ref{cor:Xpm-charn},  and \ref{cor:regbump} can be formulated.
\end{rmk}

\begin{rmk}
	In the situation of Proposition \ref{prop:comp-fc} we will write $\gh(A)$ to denote both the original operator $\mc{D}(\gh(A)) \cap \bbX_A^\mb{p} \to \bbX_A^{\mb{p}+\gd}$ and its extension to completions $\mb{X} \to \mb{Y}$.
	This will not cause any ambiguity, but one should be careful.
\end{rmk}

\begin{rmk}
	If $\mb{X}$ is a completion of $\bbX_A^\mb{p}$, then the spectral projections $\gc^\pm(A)$ extend to projections on $\mb{X}$ by Proposition \ref{prop:comp-fc}, so we get a spectral decomposition
	\begin{equation*}
		\mb{X} = \mb{X}^+ \oplus \mb{X}^-.
	\end{equation*}
	It is immediate that $\mb{X}^+$ and $\mb{X}^-$ are completions of $\bbX_A^{\mb{p},+}$ and $\bbX_A^{\mb{p},-}$ respectively.
	That is, every completion of $\bbX_A^\mb{p}$ induces natural completions of the spectral subspaces $\bbX_A^{\mb{p},\pm}$.
\end{rmk}

Finally we can present the interpolation theorem for canonical completions of pre-Besov--Hardy--Sobolev spaces.
Having established so much abstract theory, this is now a simple consequence of the interpolation results for tent spaces and $Z$-spaces.

\begin{thm}[Interpolation of completions]\label{thm:completed-interpolation}\index{interpolation!of canonical completions}\index{interpolation!of adapted BHS spaces}
	Fix $0 < \gh < 1$ and $\gy \in \gY_\infty^\infty$.
	Let $\mb{p}_0$ and $\mb{p}_1$ be exponents, and set $\mb{p}_\gh := [\mb{p}_0,\mb{p}_1]_\gh$.
	\begin{enumerate}[(i)]
		\item
		Suppose $j(\mb{p}_0), j(\mb{p}_1) \geq 0$, with equality for at most one exponent.
		Then we have the identification
		\begin{equation*}
			[\gy \mb{H}_A^{\mb{p}_0}, \gy \mb{H}_A^{\mb{p}_1}]_\gh = \gy\mb{H}_A^{\mb{p}_\gh}.
		\end{equation*}
		
		\item
		Suppose at least one of $\mb{p}_0$ and $\mb{p}_1$ is finite.
		Then we have the identification
		\begin{equation*}
			[\gy \mb{B}_A^{\mb{p}_0}, \gy \mb{B}_A^{\mb{p}_1}]_\gh = \gy\mb{B}_A^{\mb{p}_\gh}.
		\end{equation*}
		
		\item
		Suppose $\gq(\mb{p}_0) \neq \gq(\mb{p}_1)$.
		Then we have the identification
		\begin{equation*}
			(\gy \mb{X}_A^{\mb{p}_0}, \gy \mb{X}_A^{\mb{p}_1})_{\gh,i(\mb{p}_\gh)} = \gy\mb{B}_A^{\mb{p}_\gh}.
		\end{equation*}
	\end{enumerate}
	All of these statements have analogues for spectral subspaces.
\end{thm}

\begin{proof}
	Fix a Calder\'on sibling $\gn \in \gY_\infty^\infty$ of $\gy$.
	By Corollary \ref{cor:QSproj} the map $\bbQ_{\gy,A} \bbS_{\gn,A}$ extends to a map $\map{\mb{Q}_{\gy,A} \mb{S}_{\gn,A}}{X^{\mb{p}_0} + X^{\mb{p}_1}}{\gy \mb{X}_A^{\mb{p}_0} + \gy \mb{X}_A^{\mb{p}_1}}$ which restricts to projections $X^{\mb{p}_0} \to \gy \mb{X}_A^{\mb{p}_0}$ and $X^{\mb{p}_1} \to \gy \mb{X}_A^{\mb{p}_1}$.
	Therefore by the retraction/coretraction interpolation theorem (see \cite[\textsection 1.2.4]{hT78}),\footnote{This is only stated for Banach spaces in the given reference. The only property specific to Banach spaces which is needed is the validity of the closed graph theorem, which also holds for quasi-Banach spaces \cite[\textsection 2]{nK03}.} for all interpolation functors $\mc{F}$ we have
	\begin{equation*}
		\mc{F}(\mb{X}_A^{\mb{p}_0}, \mb{X}_A^{\mb{p}_1}) = \mb{Q}_{\gy,A} \mb{S}_{\gn,A} \mc{F}(X^{\mb{p}_0}, X^{\mb{p}_1}).
	\end{equation*}
	The results then follow from Corollary \ref{cor:QSproj} and the interpolation theorems for tent- and $Z$-spaces (Theorems \ref{thm-wts-c-interpolation} and \ref{thm:ts-rint-full}, and Propositions \ref{prop:Z-rint-full} and \ref{prop:Z-cint}).
\end{proof}

\begin{rmk}\label{rmk:interpolation-extn}
	The assumption $\gy \in \gY_\infty^\infty$ is only for convenience.
	The argument still holds provided that $\gy \in \gY(\bbX_A^\mb{p})$ for all the spaces $\bbX_A^\mb{p}$ involved.
	Furthermore, when interpolating between (e.g. positive) spectral subspaces, it suffices to have  $\nm{f}_{\bbX_A^{\mb{p},+}} \simeq \nm{\bbQ_{\gy,A} f}_{X^\mb{p}}$ for all $f \in \bbX_A^{\mb{p},+}$.
	We will use this observation in Section \ref{sec:solspace-interpolation}.
\end{rmk}

We now present an application of this interpolation theorem.
For applications to boundary value problems, it will be absolutely crucial to know when we have $\bbX_{DB}^\mb{p} = \bbX_D^\mb{p}$, where $D$ is the Dirac operator and $B$ is a multiplier as in Subsection \ref{ssec:fop}.
As a preliminary step, we consider inclusions of the form $\bbX_{A_0}^\mb{p} \hookrightarrow \bbX_{A_1}^\mb{p}$ for operators $A_0$ and $A_1$ satisfying the Standard Assumptions and with equal range closures (as is the case with $D$ and $DB$).

\begin{dfn}\label{dfn:inclusion-regions}
	Suppose $A_0$ and $A_1$ are operators satisfying the Standard Assumptions, and additionally assume that $\overline{\mc{R}(A_0)} = \overline{\mc{R}(A_1)} =: \overline{\mc{R}}$.
	Define the \emph{inclusion region}\index{region!inclusion}
	\begin{equation*}
		i(\mb{X},A_0,A_1) := \{\text{$\mb{p} \in \mb{E}_\text{fin} : \nm{f}_{\bbX_{A_1}^\mb{p}} \lesssim \nm{f}_{\bbX_{A_0}^\mb{p}}$ for all $f \in \overline{\mc{R}}$} \}.
	\end{equation*}
\end{dfn}

It is technically simpler for us to restrict to finite exponents in this definition, and this is good enough for our applications, but it is not essential.

\begin{lem}\label{lem:incl-QS}
	Suppose that $\mb{p}$ is a finite exponent, and let $A_0$ and $A_1$ be as in Definition \ref{dfn:inclusion-regions}.
	Then $\mb{p} \in i(\mb{X},A_0,A_1)$ if and only if for all Calder\'on siblings $\gf,\gy \in \gY_\infty^\infty$ the composition
	\begin{equation*}
		\bbQ_{\gf,A_0} \bbX_{A_0}^\mb{p} \stackrel{ \bbS_{\gy,A_0} }{ \longrightarrow } \overline{\mc{R}} \stackrel{ \bbQ_{\gf,A_1} }{ \longrightarrow} L^2(\bbR_+ : L^2)
	\end{equation*}
	is bounded from $\bbQ_{\gf,A_0} \bbX_{A_0}^\mb{p} $ to $\bbQ_{\gf,A_1} \bbX_{A_1}^\mb{p}$.
\end{lem}

\begin{proof}
	For any pair of Calder\'on siblings $\gf,\gy \in \gY_\infty^\infty$, we have
	\begin{equation*}
		\nm{f}_{\bbX_{A_1}^\mb{p}} \lesssim \nm{f}_{\bbX_{A_0}^{\mb{p}}} \qquad \text{for all $f \in \overline{\mc{R}}$}
	\end{equation*}
	if and only if
	\begin{equation}\label{eqn:controleqn}
		\nm{ \bbQ_{\gf,A_1} f }_{X^\mb{p}} \lesssim \nm{\bbQ_{\gf,A_0} f}_{X^\mb{p}} \qquad \text{for all $f \in \overline{\mc{R}}$}
	\end{equation}
	if and only if
	\begin{equation*}
		\nm {\bbQ_{\gf,A_1} \bbS_{\gf,A_0} F }_{X^\mb{p}} \lesssim \nm{ F }_{X^\mb{p}} \qquad \text{for all $F \in \bbQ_{\gf,A_0} \bbX_{A_0}^\mb{p}$,}
	\end{equation*}
	using that the right hand side of \eqref{eqn:controleqn} is finite if and only if $f = \bbS_{\gy,A_0} F$ for some $F \in \bbQ_{\gf,A_0} \bbX_{A_0}^\mb{p}$, and the Calder\'on reproducing formula (combining Proposition \ref{prop:Xss-charn} and Theorem \ref{thm:CRF}).
\end{proof}

\begin{thm}[Interpolation of inclusions of adapted spaces]\label{thm:inclusion-interpolation}\index{interpolation!of inclusions of adapted BHS spaces}
	Let $A_0$ and $A_1$ be as in Definition \ref{dfn:inclusion-regions}.
	Let $\mb{p}_0$ and $\mb{p}_1$ be finite exponents, let $\gh \in (0,1)$, and set $\mb{p}_\gh := [\mb{p}_0,\mb{p}_1]_\gh$.
	\begin{enumerate}[(i)]
		\item
			If $\mb{p}_0,\mb{p}_1 \in i(\mb{H},A_0,A_1)$, then $\mb{p}_\gh \in i(\mb{H},A_0,A_1)$.
		\item
			If $\mb{p}_0,\mb{p}_1 \in i(\mb{B},A_0,A_1)$, then $\mb{p}_\gh \in i(\mb{B},A_0,A_1)$.
		\item
			If $\gq(\mb{p}_0) \neq \gq(\mb{p}_1)$ and $\mb{p}_0,\mb{p}_1 \in i(\mb{X},A_0,A_1)$, then $\mb{p}_\gh \in i(\mb{B},A_0,A_1)$.
	\end{enumerate}
\end{thm}

\begin{proof}
	Fix Calder\'on siblings $\gf,\gy \in \gY_\infty^\infty$.
	By Lemma \ref{lem:incl-QS}, we have $\mb{p}_0,\mb{p}_1 \in i(\mb{X};A_0,A_1)$ if and only if $\bbQ_{\gf,A_1} \bbS_{\gy,A_0}$ is bounded from $\bbQ_{\gf,A_0} \bbX_{A_0}^{\mb{p}_i} $ to $\bbQ_{\gf,A_1} \bbX_{A_1}^{\mb{p}_i}$ for $i \in \{1,2\}$, in which case $\bbQ_{\gf,A_1} \bbS_{\gy,A_0}$ extends to a bounded operator $\gf\mb{X}_{A_0}^{\mb{p}_i} \to \gf\mb{X}_{A_1}^{\mb{p}_i}$ for $i \in \{1,2\}$.
	The theorem then follows from the identification of interpolation spaces from Theorem \ref{thm:completed-interpolation} and a second application of Lemma \ref{lem:incl-QS}.
\end{proof}

\section{The Cauchy operator on general adapted spaces}\label{sec:sgp-abstract}

Recall the Cauchy extension operator, defined in \eqref{eqn:cauchyoperator} by
\begin{equation*}
	C_A f(t) := e^{-tA}\gc^{\sgn(t)}(A)f = e^{-|t|[A]}\gc^{\sgn(t)}(A)f = \sgp_{|t|}(A) \gc^{\sgn(t)}(A) f.
\end{equation*}
By Proposition \ref{prop:comp-fc}, for any $\gf \in H^\infty$, the operator $\map{\gf(A)}{\bbX_A^\mb{p}}{\bbX_A^\mb{p}}$ extends to a bounded map on any completion $\mb{X}$ of $\bbX_A^\mb{p}$.
Thus we can extend the Cauchy operator to a map $\map{\mb{C}_A}{\mb{X}}{L^\infty(\bbR \sm \{0\} : \mb{X})}$.\index{Cauchy operator!on adapted BHS spaces}

\begin{prop}[Properties of Cauchy extensions]\label{prop:sgp-continuity}
	Let $\mb{p}$ be an exponent, and fix a completion $\mb{X}$ of ${\bbX}_A^{\mb{p}}$.\footnote{Recall that we mean a weak-star completion when $\mb{p}$ is infinite, and in this case we use the weak-star topology on $\mb{X}$.}
	Then for all $f \in \mb{X}$ the Cauchy extension $\mb{C}_A f$ is in $C^\infty(\bbR \sm \{0\} : \mb{X})$, and solves the Cauchy equation
	\begin{equation*}
		\partial_t \mb{C}_A f + A\mb{C}_A f = 0
	\end{equation*}
	strongly in $C^\infty(\bbR \sm \{0\} : \mb{X})$.
	Furthermore we have the limits
	\begin{equation}\label{eqn:sgp-limits}
		\lim_{t \to 0^\pm} \mb{C}_A f(t) = \gc^\pm(A) f \quad \text{and} \quad \lim_{t \to \pm\infty} \mb{C}_{A} f(t) = 0
	\end{equation}
	in $\mb{X}$.
\end{prop}

\begin{proof}
	First we prove the limit results, for which it suffices to consider $t > 0$ and the semigroup extension $e^{-|t|[A]}$ in place of $\mb{C}_A$.
	Furthermore, it suffices to consider finite exponents $\mb{p}$, as otherwise we can deduce the limits \eqref{eqn:sgp-limits} by testing against $\bbX_{A^*}^{\mb{p}^\prime}$.
	By density, it suffices to prove the limits
	\begin{equation*}
		\lim_{t \to 0} e^{-t[A]}f = f \quad \text{and} \quad \lim_{t \to \infty} e^{-t[A]}f = 0
	\end{equation*}
	for $f \in \bbX_A^\mb{p}$.
	
	For $f \in {\bbH}_A^{\mb{p}}$, these follow from arguments almost identical to those of \cite[Propositions 4.5 and 4.6]{AS16}, the only difference being the presence of the weight $\gk^{-\gq(\mb{p})}$, which does not fundamentally change the argument.
	Now fix an exponent $\mb{q} \neq \mb{p}$ such that $\mb{q} \hookrightarrow \mb{p}$, so that ${\bbH}_A^\mb{q} \hookrightarrow {\bbB}_A^\mb{p}$ (Proposition \ref{prop:adapted-space-embeddings}).
	For $f \in {\bbH}_A^{\mb{q}}$ we then have
	\begin{equation*}
		\lim_{t \to 0} \nm{e^{-t[A]}f - f}_{{\bbB}_A^{\mb{p}}} \lesssim \lim_{t \to 0} \nm{e^{-t[A]}f - f}_{{\bbH}_A^{\mb{q}}} = 0
	\end{equation*}
	and
	\begin{equation*}
		\lim_{t \to \infty} \nm{e^{-t[A]}f}_{{\bbB}_A^{\mb{p}}} \lesssim \lim_{t \to \infty} \nm{e^{-t[A]}f}_{{\bbH}_A^{\mb{q}}} = 0.
	\end{equation*}
	Since ${\bbH}_A^\mb{q}$ is dense in ${\bbB}_A^{\mb{p}}$ (Corollary \ref{cor:adapted-space-int-dens}), these limits hold for all $f \in {\bbB}_A^\mb{p}$.
	
	Now we prove the smoothness result, and establish the Cauchy equation.
	Again it suffices to consider $t > 0$, as the result for $t < 0$ uses the same argument.
	First observe that the function $\map{\gF}{\bbR_+}{H^\infty}$ defined by $\gF(t) = [z \mapsto e^{-tz}\gc^+(z)]$ is smooth with Fr\'echet derivative $D_t \gF \colon \bbR \to H^\infty$ given by $D_t \gF(\gt) = [z \mapsto -\gt z e^{-t z} \gc^+(z)]$.
	Next, note that the map $\gO_A \colon H^\infty \to \mc{L}(\mb{X})$ with $\gO_A(f) = f(A)$ is linear and bounded in the strong topology (Proposition \ref{prop:comp-fc}).
	By the chain rule, the composition of these maps is smooth, with Fr\'echet derivative
	\begin{equation*}
		D_t (\gO_A \circ \gF)(\gt) = \gO_A \circ D_t \gF(\gt) = -\gt A e^{-tA} \gc^+(A).
	\end{equation*}
	We can then write
	\begin{equation*}
		\partial_t \mb{C}_A f(t) = D_t(\gO_A \circ \gF)(1)f = -Ae^{-tA}\gc^+(A)f = -A\mb{C}_A f(t),
	\end{equation*}
	which completes the proof.
\end{proof}

Now we must address whether $C_A|_{\bbR_+}$ maps $\bbX_A^{\mb{p},+}$ into $X^\mb{p}$.
This would imply that one could construct $C^\infty(\bbR_+ : \bbX_A^\mb{p})$ solutions to Cauchy problems which are in $X^\mb{p}$ with initial data in $\bbX_A^{\mb{p},+}$.
It turns out that this is only reasonable when $\gq(\mb{p}) < 0$.
For $i(\mb{p}) \leq 2$ we already know everything we need to prove this, but for $i(\mb{p}) > 2$ we need more information (see Section \ref{sec:sgp-classical}).
In fact, for $i(\mb{p}) \leq 2$, we can work with the semigroup extension $\bbQ_{\sgp,A}$ instead of the Cauchy extension; that is, we need not incorporate spectral projections into the argument.

\begin{thm}[Semigroup characterisation, $i(\mb{p}) \leq 2$]\label{thm:sgpnorm-abstract}
	Let $\mb{p}$ be an exponent with $i(\mb{p}) \leq 2$ and $\gq(\mb{p}) < 0$.
	Then for all $f \in \overline{\mc{R}(A)}$,
	\begin{equation*}
		\nm{f}_{{\bbX}_A^\mb{p}} \simeq \nm{\bbQ_{\sgp,A} f}_{X^\mb{p}}.
	\end{equation*}
\end{thm}

\begin{proof}
	By Proposition \ref{prop:equivalent-quasinorms}, we have
	\begin{equation*}
		\sgp  \in \gY_0^\infty \subset \gY_{\gq(\mb{p})+}^{(-\gq(\mb{p}) + n|\frac{1}{2} - j(\mb{p})|)+}\cap H^\infty \subset \gY(\bbX_A^\mb{p}),
	\end{equation*}
	which yields the result. 
\end{proof}

\begin{rmk}\label{rmk:sgp-abstract}
	The estimate
	\begin{equation*}
		\nm{f}_{\bbX_A^\mb{p}} \lesssim \nm{\bbQ_{\sgp,A} f}_{X^\mb{p}}
	\end{equation*}
	holds for all $\mb{p}$, as can be shown by a Calder\'on reproducing argument as in the proof of Proposition \ref{prop:equivalent-quasinorms}.
	The reverse estimate need not hold in general.
\end{rmk}

\begin{cor}[Cauchy characterisation of positive adapted spaces, $i(\mb{p}) \leq 2$]\label{cor:cauchy-abstract}\index{space!adapted Besov--Hardy--Sobolev!characterisation by Cauchy operator}
	Let $\mb{p}$ be an exponent with $i(\mb{p}) \leq 2$ and $\gq(\mb{p}) < 0$.
	Then for all $f \in \overline{\mc{R}(A)}^+$,
	\begin{equation*}
		\nm{f}_{{\bbX}_A^\mb{p}} \simeq \nm{C_A f}_{X^\mb{p}}.
	\end{equation*}
\end{cor}

\begin{rmk}
We have only considered the positive spectral subspace here, but of course everything works just as well for the negative spectral subspace.
\end{rmk}

The following technical results will be needed in Section \ref{sec:class}.

\begin{lem}\label{lem:sgp-sector-repn}
	For every $M > 0$, there exist functions $\gf^+, \gf^- \in H^\infty$ such that $(\gf^\pm_s(A))_{s > 0}$ satisfies   off-diagonal estimates of order $M$, $\gf_s^\pm(A) = e^{-s[A]}$ on the corresponding spectral subspace $\bbX_A^{2,\pm}$, and $\lim_{s \to 0} \gf_s^\pm (A) = I$ in the $L^2$-strong operator topology.
\end{lem}

For a proof, see \cite[Lemma 15.1]{AM15}, noting that $\bbH_A^{2,\pm} = \bbB_A^{2,\pm}$.
Although this result is stated for $A \in \{DB,BD\}$ there, the proof only uses the Standard Assumptions.

\begin{cor}\label{cor:sgp-slice-bddness}
	Let $p \in (0,\infty]$.
	Suppose $f \in \bbX_A^{2,+} \cap E^p$.
	Then $C_A f(t) \in E^p$ for each $t \in \bbR_+$, and if $p < \infty$ then $\lim_{t \to 0^+} C_A f(t) = f$ in $E^p$.
	Similarly for $f \in \bbX_A^{2,-} \cap E^p$ and $t \in \bbR_-$.
\end{cor}

\begin{proof}
	Choose functions $\gf^+$ as in Lemma \ref{lem:sgp-sector-repn} such that $(\gf^+_s(A))_{s > 0}$ satisfies off-diagonal estimates of large order.
	By Proposition \ref{prop:slice-bddness}, the operators $\gf^+_t (A)$ are bounded on $E^p$.
	Since $\gf_t^+ (A) = e^{-tA}$ on $\bbX_A^{2,+}$, we have $C_A f(t) = \gf^+_t(A) f \in E^p$ for all $t \in \bbR_{+}$.
	The limit statement follows from Lemma \ref{lem:sgp-sector-repn} (which gives strong convergence in $L^2$) and Proposition \ref{prop:slice-convergence} (which improves this to $E^p$).
	The proof for the negative spectral subspace is identical.
\end{proof}

\chapter{Spaces Adapted to Perturbed Dirac Operators}\label{chap:diffops}

In this chapter we consider the Dirac operator\index{Dirac operator}\index{Dirac operator!perturbed}
\begin{equation*}
	D = \begin{bmatrix} 0 & \dv_\parallel \\ -\nabla_\parallel & 0 \end{bmatrix}
\end{equation*}
and its perturbations $DB$ and $BD$, where $B$ is a bounded elliptic multiplier (see Subsection \ref{ssec:fop}).
These operators act on $\bbC^{m(1+n)}$-valued functions, so from now on we take $N = m(1+n)$ without further mention.

The operator $D$, initially defined distributionally, may be defined as an unbounded operator on $L^2=L^2(\bbR^n : \bbC^{m(1+n)})$ with maximal domain
\begin{equation*}
	\mc{D}(D) = \{f \in L^2 : Df\in L^2\},
\end{equation*} 
and generally it is better to use this interpretation of $D$.
The nullspace and the closure of the range of $D$ may be simply characterised by means of the transversal/tangential notation: we have $f\in\mc{N}(D)$ if and only if $f_{\perp}=0$ and $\dv_\parallel f_{\parallel}=0$, and $f\in \overline{\mc{R}(D)}$ if and only if $\curl_\parallel f_{\parallel}=0$.
Note that $\nabla_\parallel f_\perp = 0$ implies that $f_\perp$ is constant, which then forces $f_\perp = 0$ since $f \in L^2$.

The perturbed Dirac operators $DB$ and $BD$ satisfy the Standard Assumptions (Theorem \ref{thm:DB-fundamental-props}), so pre-Besov--Hardy--Sobolev spaces $\bbX_{DB}^\mb{p}$ and $\bbX_{BD}^\mb{p}$ are defined.
In this chapter we investigate these spaces, and in particular their connection with classical smoothness spaces.

Recall from Proposition \ref{prop:bisectorialfacts} that every bisectorial operator $A$ on $L^2$ induces a topological splitting
\begin{equation*}
	L^2 = \mc{N}(A) \oplus \overline{\mc{R}(A)},
\end{equation*}
and that $\bbP_A$ denotes the orthogonal projection of $L^2$ onto $\overline{\mc{R}(A)}$ along $\mc{N}(A)$.
The operators $\bbP_D$, $\bbP_{DB}$, and $\bbP_{BD}$ will appear often in our analysis.

\begin{rmk}
  The operator $D$ may instead be one of the first order differential operators considered by Hyt\"onen, McIntosh, and Portal \cite{HMP08,HM10,HMP11}; see \cite[\textsection 2.3]{AS16}.
  Our results apply to such generalisations.
  However, we shall stick with this choice of $D$, as this operator in particular is relevant to elliptic equations.
\end{rmk}

\section{$D$-adapted spaces}\label{sec:Dadapted}\index{space!adapted Besov--Hardy--Sobolev!D-adapted@$D$-adapted}

The $D$-adapted pre-Besov--Hardy--Sobolev spaces $\bbX_D^\mb{p}$ play a central role in our theory. 
For all exponents $\mb{p}$, they may be identified as projections of classical smoothness spaces (Theorem \ref{thm:class-space-identn-new}).
Recall that we use the notation $\mb{X}^\mb{p}$ to denote classical smoothness spaces, as in Section \ref{sec:hss}.
These spaces are all embedded in $\mc{Z}^\prime = \mc{Z}^\prime(\bbR^n)$, the space of tempered distributions modulo polynomials.\footnote{We use the same notation to refer to $\bbC^N$-valued tempered distributions modulo polynomials for different values of $N$. No ambiguity is caused by this abuse of notation.}

The projection $\bbP_D$ extends boundedly to $\mb{X}^\mb{p}$ for all finite exponents $\mb{p}$ by virtue of being a Fourier multiplier within the scope of the Mikhlin multiplier theorem (see \cite[Theorem 5.2.2]{hT83} and \cite[Proposition 4.4]{HMP11}).
When $\mb{p}$ is infinite, $\bbP_D$ extends by duality to $\mb{X}^\mb{p}$.
Note also that we can write this projection as
\begin{equation*}
	\bbP_D = -\gD^{-1} D^2,
\end{equation*}
and that the transversal part of $\bbP_D$ is the identity operator.

\begin{dfn}\label{dfn:boldXD}
  For all exponents $\mb{p}$, define $\mb{X}_D^\mb{p} := \bbP_D \mb{X}^\mb{p}$.
  Note that $\mb{X}_D^\mb{p}$ is a closed subspace of $\mb{X}^\mb{p}$, hence also of $\mc{Z}^\prime$.
\end{dfn}

\begin{thm}[Identification of $D$-adapted spaces]\label{thm:class-space-identn-new}
	Let $\mb{p}$ be an exponent.
	Then $\bbX_D^\mb{p} = \bbP_D(\mb{X}^\mb{p} \cap L^2)$, with norm estimates
	\begin{align*}
		\nm{\bbP_D f}_{\bbX_D^\mb{p}} &\lesssim  \nm{f}_{\mb{X}^\mb{p}} \qquad (f \in \mb{X}^\mb{p} \cap L^2), \\
		\nm{h}_{\mb{X}^\mb{p}} &\simeq \nm{h}_{\bbX_D^\mb{p}} \qquad (h \in \bbX_D^\mb{p}).
	\end{align*}
	In particular, $\bbX_D^\mb{p} \hookrightarrow \mb{X}_D^\mb{p}$ with dense image.
\end{thm}

\begin{proof}
	Let $\gy \in \gY_\infty^\infty(S_\gm)$ be a nondegenerate even function, where $\gm \in (0,\gp/2)$.
	Then there exists a function $\gF \in \gY_\infty^\infty(S_{2\gm}^+)$ such that $\gy(z) = \gF(z^2)$ for all $z \in S_\gm$, where
	\begin{equation*}
		S_{2\gm}^+ := \{z \in \bbC \sm \{0\} : |\arg(z)| < 2\gm\}
	\end{equation*}
	is the sector of angle $2\gm$, and where $\gY_\infty^\infty(S_{2\gm}^+)$ is defined analogously to $\gY_\infty^\infty(S_\gm)$.
	We then have
	\begin{equation*}
		\gF(-t^2 \gD)f = t^{-n}\gf(t^{-1}\cdot) \ast f
	\end{equation*}
	for all $f \in L^2$, where $\gf = \hat{\gy}$.
	Note that $\gf \in \mc{Z}$, and that we may choose $\gy$ such that $\hat{\gf}(\gx) > 0$ when $1/2 \leq |\gx| \leq 2$, so we may apply Theorem \ref{thm:BHS-Xspace}, which yields
	\begin{equation}\label{eqn:Phicharn}
		\nm{f}_{\mb{X}^\mb{p}} \simeq \nm{t \mapsto \gF(-t^2 \gD)f}_{X^\mb{p}}.
	\end{equation}
	
	Now we prove the claimed properties of $\bbP_D$.
	First let $f \in \mb{X}^\mb{p} \cap L^2$.
	For all $t > 0$, note that
	\begin{equation*}
		\gy(tD)\bbP_D f = \gF(t^2 D^2) \bbP_D f = \gF(-t^2 \gD)\bbP_D f.
	\end{equation*}
	Thus
	\begin{equation*}
		\nm{\bbP_D f}_{\bbX_D^\mb{p}}
		\simeq \nm{\bbQ_{\gy,D} \bbP_D f}_{X^\mb{p}}
		= \nm{t \mapsto \gF(-t^2 \gD)\bbP_D f}_{X^\mb{p}}
		\simeq \nm{\bbP_D f}_{\mb{X}^\mb{p}}
		\lesssim \nm{f}_{\mb{X}^\mb{p}}
	\end{equation*}
	by \eqref{eqn:Phicharn} and boundedness of $\bbP_D$ on $\mb{X}^\mb{p}$.
	On the other hand, suppose $h \in \bbX_D^{\mb{p}}$.
	We will show that $h \in \mb{X}^\mb{p}$; since $h = \bbP_D h \in L^2$, this will complete the proof.
	As before, for every $t > 0$ we have
	\begin{equation*}
		\gF(-t^2 \gD)h = \gF(-t^2 \gD)\bbP_D h = \gF(t^2 D^2)\bbP_D h = \gy(tD)h,
	\end{equation*}
	and hence
	\begin{equation*}
		\nm{h}_{\mb{X}^\mb{p}} \simeq \nm{t \mapsto \gF(-t^2 \gD)h}_{X^\mb{p}} = \nm{\bbQ_{\gy,D}h}_{X^\mb{p}} \simeq \nm{h}_{\bbX_D^\mb{p}}.
	\end{equation*}
	Therefore $h \in \mb{X}^\mb{p}$.
\end{proof}

\begin{rmk}
	In \cite{AS16} this result is proven for Hardy spaces $\mb{H}^{(p,0)}$ with $p > n/(n+1)$.
	This restriction arises because the argument there uses the atomic Hardy space $\bbH_{D,\text{ato},1}^p$.
	To get the full result by an atomic argument one should use atoms with a higher order of $D$-cancellation (as with classical Hardy spaces: for $p$ small, the atomic decomposition requires atoms of higher cancellation).
\end{rmk}

It follows from Theorem \ref{thm:class-space-identn-new} that $\mb{X}_D^\mb{p}$ is a completion of $\bbX_D^\mb{p}$ for all exponents $\mb{p}$.\index{completion!of $D$-adapted spaces}
This is the most natural completion of $\bbX_D^\mb{p}$ (we also have the `canonical completions' $\gf \mb{X}_D^\mb{p}$ from Section \ref{sec:completions}, which are not so canonical in this case).

\section{Similarity of functional calculus: going between $DB$ and $BD$}\label{sec:similarity}\index{functional calculus!similarity}

The operators $DB$ and $BD$ are, unsurprisingly, related.
However, again unsurprisingly, moving between the two operators introduces technicalities.
The main fact to notice is that $\mc{R}(DB) = \mc{R}(D)$, and that the restrictions $DB|_{\overline{\mc{R}(DB)}}$ and $BD|_{\overline{\mc{R}(BD)}}$ are similar under conjugation by $B|_{\overline{\mc{R}(DB)}}$ \cite[Proposition 2.1]{AS14.1}.
Consequently, whenever $f \in \mc{D}(D) \cap \overline{\mc{R}(BD)}$ and $\gf \in H^\infty$ we have
\begin{equation*}
	D\gf(BD)f = \gf(DB)Df.
\end{equation*}
We refer to this principle as \emph{similarity of functional calculi}.
This will eventually imply that we have an identification like $D\mb{X}_{BD}^{\mb{p}+1} = \mb{X}_{DB}^{\mb{p}}$ (this is just an informal statement; it is made precise in Corollary \ref{cor:similarity}).

We need a local coercivity property of $B$, which is proven in \cite[Lemma 5.14]{AS16}.

\begin{lem}\label{lem:local-coercivity}
	For any $u \in L_\text{loc}^2$ with $Du \in L_\text{loc}^2$ and any ball $B(x,t) \in \bbR^n$,
	\begin{equation*}
		\int_{B(x,t)} |Du|^2  \lesssim_{B,n,m} \int_{B(x,2t)} |BDu|^2 + t^{-2} \int_{B(x,2t)} |u|^2.
	\end{equation*}
\end{lem}

\begin{prop}[Intertwining and regularity shift]\label{prop:D-mapping}
	Let $\mb{p}$ be an exponent. Then $\map{D}{ {\bbX}_{BD}^{\mb{p}+1} \cap \mc{D}(D)}{{\bbX}_{DB}^{\mb{p}}\cap \mc{R}(D)}$ is bijective, and for all $f\in  {\bbX}_{BD}^{\mb{p}+1} \cap \mc{D}(D)$,
	\begin{equation*}
		\nm{Df}_{{\bbX}_{DB}^{\mb{p}}} \simeq \nm{f}_{{\bbX}_{BD}^{\mb{p}+1}}.
	\end{equation*}
	
\end{prop}

\begin{proof}
	We only consider $T^\mb{p}$ with $\mb{p} = (p,s)$ finite; all other cases are proven by the same argument.
	Let $\gy \in \gY_\infty^\infty$ be nondegenerate and define $\td{\gy} \in \gY_\infty^\infty$ by $\td{\gy}(z) = z{\gy}(z)$.
	Then $\td{\gy}(DB)$ maps $\overline{\mc{R}(DB)}$ into $\mc{D}((DB)^{-1})$.
	Since $f \in \mc{D}(D)$ we have $Df \in \mc{R}(D) = \mc{R}(DB)$.
	Using similarity of functional calculi, write
	\begin{align*}
		\nm{Df}_{\bbH^{\mb{p}}_{DB}}
		&\simeq \nm{t \mapsto \gy(tDB)Df}_{T^p_{s}} \\ 
		&= \nm{t \mapsto D\gy(tBD) f}_{T^p_{s}} \\
		&= \nm{t \mapsto D(BD)^{-1} \td{\gy}(tBD) f}_{T^p_{s+1}}.		
	\end{align*}
	For all $t > 0$ we have
	\begin{equation*}
		(BD)^{-1} \td{\gy}(tBD)f \in L^2,  D(BD)^{-1} \td{\gy}(tBD)f = \gy(tDB)Df \in L^2,
	\end{equation*}
	so we can apply Lemma \ref{lem:local-coercivity} for each $t > 0$ with $u = (BD)^{-1} \td{\gy}(tBD)f$ as follows: for all $x \in \bbR^n$,
	\begin{align*}
		&\mc{A}(t \mapsto t^{-s} D(BD)^{-1} \td{\gy}(tBD) f)(x) \\
		&= \bigg( \int_0^\infty t^{-2s} \int_{B(x,t)} |(D(BD)^{-1} \td{\gy}(tBD)f)(y)|^2 \, dy \, \frac{dt}{t^{n+1}} \bigg)^{1/2} \\
		&\lesssim \bigg( \int_0^\infty t^{-2s} \bigg[ \int_{B(x,2t)} |(\td{\gy}(tBD)f)(y)|^2 \, dy  \\
		&\qquad \qquad \qquad \qquad + \int_{B(x,2t)} |(\gy(tBD)f)(y)|^2 \, dy \bigg] \, \frac{dt}{t^{n+1}} \bigg)^{1/2} \\
		&\lesssim \mc{A}(\gk^{-s} \bbQ_{\td{\gy},BD}f)(x) + \mc{A}(\gk^{-s} \bbQ_{\gy,BD} f)(x).
	\end{align*}
	Therefore
	\begin{equation*}
		\nm{Df}_{\bbH^{\mb{p}}_{DB}}
		\lesssim \nm{\bbQ_{\td{\gy},BD} f}_{T^{\mb{p}+1}} + \nm{\bbQ_{\gy,BD} f}_{T^{\mb{p}+1}} \simeq \nm{f}_{\bbH^{\mb{p}+1}_{BD}}.
	\end{equation*}
	To prove the reverse estimate, using that $f \in \mc{D}(D) = \mc{D}(BD)$ and $f \in \overline{\mc{R}(BD)}$, write
	\begin{align*}
		\nm{f}_{\bbH_{BD}^{\mb{p}+1}}
		&= \nm{(BD)^{-1} BD f}_{\bbH_{BD}^{\mb{p}+1}} \\
		&\lesssim \nm{BD f}_{\bbH_{BD}^{\mb{p} }} \\
		&\simeq\nm{t \mapsto \gy(tBD)BDf}_{T^{\mb{p} }} \\
		&= \nm{t \mapsto B\gy(tDB)Df}_{T^{\mb{p} }} \\
		&\lesssim \nm{\bbQ_{\gy,DB}Df}_{T^{\mb{p} }} \\
		&\simeq \nm{Df}_{\bbH_{DB}^{\mb{p}}}
	\end{align*}
	using \eqref{cor:regbump}, boundedness of $B$, and similarity of functional calculi. 
	We next prove surjectivity. Let $g\in {\bbH}_{DB}^{\mb{p}}\cap \mc{R}(D)$. Since $g \in \mc{R}(D)$ there exists $h \in \mc{D}(D)$ such that $g = Dh$.
	Set $f=\bbP_{BD}h$. By Lemma \ref{lem:proj}, we have $f\in \mc{D}(D) \cap \overline{\mc{R}(BD)}$ and $Df=D\bbP_{BD}h= Dh= g$. The calculation above shows that $f\in \bbH_{BD}^{\mb{p}+1}$. 
\end{proof}

\begin{rmk}\label{rmk:prop:D-mapping}
The argument above proves more: it says that if $f\in  \mc{D}(D) \cap \overline{\mc{R}(BD)}$ with $Df\in {\bbX}_{DB}^{\mb{p}}$, then $f\in \bbX_{BD}^{\mb{p}+1}$. 
This argument is in particular valid when $B=I$ (although it is simpler in this case). 
\end{rmk}

Let us continue with a density lemma for adapted spaces. 

\begin{lem}\label{lem:density} Let $\mb{p}$ be an exponent.
Then ${\bbX}_{BD}^\mb{p} \cap \mc{D}(D)$ and ${\bbX}_{DB}^\mb{p} \cap \mc{R}(D)$ are dense in 
$\bbX_{BD}^\mb{p}$ and 
$\bbX_{DB}^\mb{p}$ respectively.\footnote{As usual, we mean weak-star density for infinite exponents.}
\end{lem}

\begin{proof} The result for $\bbX_{BD}^{\mb{p}} \cap \mc{D}(D)$ already appears in the proof of Proposition \ref{prop:comp-fc}, since $\mc{D}(D) = \mc{D}(BD)$.
Let $f\in {\bbX}_{DB}^\mb{p}$. 
By the argument in Proposition \ref{prop:sgp-continuity}, we have that
\begin{equation*}
	f_{\varepsilon} := e^{-\varepsilon[DB]}f -  e^{-(1/\varepsilon)[DB]}f
\end{equation*}
 belongs to ${\bbX}_{DB}^\mb{p}$ and converges to $f$ there as $\varepsilon \to 0$.
Using that
\begin{equation*}
	f_{\varepsilon}= -\int_{\varepsilon}^{1/\varepsilon} [DB]e^{-s[DB]}f \, ds
\end{equation*}
 and using the spectral subspace decomposition, we see that $f_{\varepsilon}$ is in $\mc{R}(D)$. 
\end{proof}

The following is an immediate corollary of Proposition \ref{prop:D-mapping} and Lemma \ref{lem:density}.

\begin{cor}\label{cor:similarity} Assume that $\mb{p}$ is an exponent. Let $\mb{X}$ and $\mb{Y}$ be completions of $\bbX_{BD}^{\mb{p}+1}$ and $\bbX_{DB}^{\mb{p}}$ respectively.
Then $D$ extends to an isomorphism $\mb{X} \to \mb{Y}$ (which also holds for spectral subspaces), and 
\begin{equation}\label{eqn:simgeneral}
	D\gf(BD)= \gf(DB)D
\end{equation}
where we write $\gf(BD)$ and $\gf(DB)$ to denote the extended operators.
In particular, for Cauchy extensions we have
\begin{equation}\label{eqn:simcauchy}
	DC_{BD}= C_{DB}D.
\end{equation}
\end{cor} 

\section{Inclusions and identifications of $DB$- and $D$-adapted spaces}\label{sec:rel}

Now we discuss spaces adapted to $DB$, and the range of exponents for which they may be identified with spaces adapted to $D$.
As shown in the introduction, the following `identification region' plays a central role in the theorems of Section \ref{sec:class}.
The sets $i(\mb{X},A_0,A_1)$ were defined in Definition \ref{dfn:inclusion-regions}.

\begin{dfn}\index{region!identification}
	We define the \emph{identification region} for $DB$ as
	\begin{align}\label{eqn:I-dfn}
		I(\mb{X},DB) &:= i(\mb{X},D,DB) \cap i(\mb{X},DB,D) \\
		&=\big\{\mb{p} \in \mb{E}_\text{fin} : \text{$\nm{f}_{{\bbX}_{DB}^\mb{p}} \simeq \nm{f}_{{\bbX}_{D}^\mb{p}}$ for all $f \in \overline{\mc{R}(DB)} = \overline{\mc{R}(D)}$}\big\} \nonumber,
	\end{align}
	and for $s \in \bbR$,
	\begin{equation*}
		I_s(\mb{X},DB) := \big\{ i(\mb{p}) : \mb{p} \in I(\mb{X},DB)\  \mathrm{with}\  \gq(\mb{p}) = s \big\} \subset (0,\infty).
	\end{equation*}
\end{dfn}

Note that $I(\mb{X},DB)$ is defined to be a set of finite exponents.
We could include infinite exponents in this definition, but it is technically more convenient to restrict ourselves to finite exponents.
Whenever we deal with infinite exponents in applications, we always consider the space under consideration as a dual space, and the exponent corresponding to its predual will be in $I(\mb{X},DB^*)$ (see for example Theorem \ref{thm:mainthm-gtr2}).

We recall some key results of the second author and Stahlhut, which follow from \cite[Proposition 4.17, Proposition 4.18, Theorem 5.1 and Remark 5.2]{AS16}.
An alternative proof of the second result has been given by Frey, McIntosh, and Portal \cite[Corollary 3.2(3)]{FMP15}.

\begin{thm}\label{thm:AS-identn}\
	\begin{enumerate}[(i)]
	\item
		For $p \in (n/(n+1),2]$ we have $(p,0) \in i(\mb{H},DB,D)$, and for $p \in [2,\infty)$ we have $(p,0) \in i(\mb{H},D,DB)$.
	\item
          There exist $\varepsilon, \varepsilon^\prime > 0$ (depending on $B$) such that
          \begin{equation*}
            (2n/(n+2) - \varepsilon, 2+\varepsilon^\prime) \subset I_0(\mb{H},DB).
          \end{equation*}
	\item
		If $n=1$, then $I_0(\mb{H},DB) = (1/2, \infty)$.
	\item
		The set $I_0(\mb{H},DB)$ is an open interval.
	\end{enumerate}
\end{thm}

\begin{rmk}
	The number $2 + \varepsilon^\prime$ is the supremum of those $p$ for which $DB$ has bounded $H^\infty$ functional calculus in $L^p$, and $\frac{2n}{n+2} - \varepsilon$ is at most the infimum of $\frac{2p}{p+n}$ over all $p$ such that $DB$ has bounded $H^\infty$ functional calculus in $L^p$.
	The numbers $\varepsilon$ and $\varepsilon^\prime$ depend only on $n$, $m$, and the ellipticity constants.
\end{rmk}

Following Theorem \ref{thm:AS-identn}, we define the interval
\begin{equation*}
	I_0^{\text{min}}(\mb{H},DB) := \left\{ \begin{array}{ll} (\frac{2n}{n+2} - \varepsilon, 2+\varepsilon^\prime) & (n \geq 2) \\ (\frac12 ,\infty) & (n=1), \end{array} \right.
\end{equation*}
where $\varepsilon$ and $\varepsilon^\prime$ are the numbers appearing in Theorem \ref{thm:AS-identn}, and the region of exponents
\begin{equation*}
	\wtd{I}_0^{\text{min}}(\mb{H},DB) := \left\{ \mb{p} \in \mb{E} : \gq(\mb{p}) = 0, i(\mb{p}) \in I_0^{\text{min}}(\mb{H},DB) \right\}.
      \end{equation*}
      
We now extend these results to allow for more general exponents of order $\gq(\mb{p}) \in [-1,0]$, and also to incorporate Besov spaces.
The $\heartsuit$-duality operation on exponents provides a link between the inclusion regions $i(\mb{X},D,DB)$ and $i(\mb{X},DB^*,D)$ (Proposition \ref{prop:dual-interval-incl}), hence also between the identification regions $I(\mb{X},DB)$ and $I(\mb{X},DB^*)$ (Corollary \ref{cor:dual-interval-identn}).

\begin{prop}[$\heartsuit$-duality of inclusion regions]\label{prop:dual-interval-incl}\index{duality!of inclusion regions}
	If $\mb{p} \in i(\mb{X},DB,D)$, then $\nm{f}_{{\bbX}_{DB^*}^{\mb{p}^\heartsuit}} \lesssim \nm{f}_{{\bbX}_{D}^{\mb{p}^\heartsuit}}$ for all $f \in \overline{\mc{R}(D)}$.
	In particular, if $\mb{p}^\heartsuit$ is also finite, then $\mb{p}^\heartsuit \in i(\mb{X},D,DB^*)$.
	Likewise, if $\mb{p} \in i(\mb{X},D,DB)$, then $\nm{f}_{{\bbX}_{DB^*}^{\mb{p}^\heartsuit}} \gtrsim \nm{f}_{{\bbX}_{D}^{\mb{p}^\heartsuit}}$ for all $f \in \overline{\mc{R}(D)}$, and if $\mb{p}^\heartsuit$ is also finite then $\mb{p}^\heartsuit \in i(\mb{X},DB^*,D)$.

\end{prop}

\begin{proof}
	We only prove the first statement.
	The proof of the second statement is identical, with the $\lesssim$ replaced by a $\gtrsim$.
	Suppose $g \in \mc{D}(D) \cap \bbX_{B^* D}^{\mb{p}^\prime}$.
	Arguing by `$\heartsuit$-duality',
	\begin{align}
		\nm{Dg}_{\bbX_{DB^*}^{\mb{p}^\heartsuit}}
		&\simeq \sup_{h \in \mc{D}(D) \cap \bbX_{B D}^{(\mb{p}^\heartsuit)^\prime}} |\langle Dg, h \rangle| \nonumber \\
		&= \sup_{h \in \mc{D}(D) \cap \bbX_{B D}^{\mb{p} + 1}} |\langle g, Dh \rangle| \nonumber \\
		&\simeq \sup_{h \in \bbX_{DB}^{\mb{p}}} |\langle g, h \rangle| \label{line:i}\\
		&\lesssim \sup_{h \in \bbX_D^\mb{p}} |\langle g, h \rangle| \label{line:a} \\
		&\simeq \nm{g}_{\bbX_D^{\mb{p}^\prime}} \nonumber \\
		&\simeq \nm{Dg}_{\bbX_D^{\mb{p}^\heartsuit}} \label{line:d}
	\end{align}
	with all suprema taken over appropriately normalised elements.
	The equivalence \eqref{line:i} uses Proposition \ref{prop:D-mapping}, and then \eqref{line:a} uses that $\bbX_{DB}^\mb{p} \subset \bbX_D^\mb{p}$, i.e. that $\mb{p} \in i(\mb{X},DB,D)$.
	The final equivalence \eqref{line:d} uses Corollary \ref{cor:regbump}.
	Since $D(\mc{D}(D) \cap \bbX_{B^* D}^{\mb{p}^\prime})$ is dense in $\bbX_{DB^*}^{\mb{p}^\heartsuit}$ by Lemma \ref{lem:density} (weak-star dense when $\mb{p}^\prime$ is infinite), we are done.
\end{proof}

\begin{cor}\label{cor:dual-interval-identn}\index{duality!of identification regions}
	If $\mb{p} \in I(\mb{X},DB)$, then $\nm{f}_{{\bbX}_{DB^*}^{\mb{p}^\heartsuit}} \simeq \nm{f}_{{\bbX}_{D}^{\mb{p}^\heartsuit}}$ for all $f \in \overline{\mc{R}(D)}$.
	In particular, if $\mb{p}^\heartsuit$ is also finite, then $\mb{p}^\heartsuit \in I(\mb{X},DB^*)$.
\end{cor}

The following results then follow from Theorem \ref{thm:AS-identn}, Proposition \ref{prop:dual-interval-incl}, and Corollary \ref{cor:dual-interval-identn}.

\begin{cor}\label{cor:AScor}\
	\begin{enumerate}[(i)]
	\item
		For $p \in (1,2]$ we have $(p,-1) \in i(\mb{H},DB,D)$, for $p \in [2,\infty)$ we have $(p,-1) \in i(\mb{H},D,DB)$, and for $\ga \in [0,-1)$ we have $(\infty,-1;\ga) \in i(\mb{H},D,DB)$.
		
	\item
		There exist $\varepsilon, \varepsilon^\prime > 0$ such that
	\begin{equation*}
		I_{-1}(\mb{H},DB) \supset \left\{ \begin{array}{ll} (1,\infty) & (n=1) \\ (2-\varepsilon,\infty) & (n = 2) \\ {(2-\varepsilon,2n/(n-2) + \varepsilon^\prime)} & (n \geq 3). \end{array}
		\right.
	\end{equation*}
	\item
		$I_{-1}(\mb{H},DB)$ is an open interval.
	\end{enumerate}
\end{cor}

We can derive non-trivial regions of exponents within which we have embeddings $\bbX_{DB}^\mb{p} \hookrightarrow \bbX_{D}^\mb{p}$ (or the reverse embedding), or identifications $\bbX_{DB}^\mb{p} = \bbX_D^\mb{p}$, by combining Theorem \ref{thm:AS-identn} and Corollary \ref{cor:AScor} with Theorem \ref{thm:inclusion-interpolation}.

\begin{dfn}
  Define the exponent region $I_\text{max}$ \index{I-max@$I_{\text{max}}$} by 
  \begin{equation*}
    I_{\text{max}} := \bigg\{\mb{p} \in \mb{E} : \gq(\mb{p}) \in [-1,0], j(\mb{p}) \in \bigg(\frac{-\gq(\mb{p})}{n},\frac{n+1-\gq(\mb{p})}{n} \bigg) \bigg\}.
  \end{equation*}
  The region $I_\text{max}$ is pictured in Figure \ref{fig:Imax-diagram}.
  Note that $I_\text{max}$ is the convex hull of the union of lines
  \begin{equation*}
    \left\{\mb{p} \in \mb{E} : \gq(\mb{p}) = 0, j(\mb{p}) \in \left(0,\frac{n+1}{n} \right) \right\} \cup \left\{\mb{p} \in \mb{E} : \gq(\mb{p}) = -1, j(\mb{p}) \in \left(\frac{-1}{n},1 \right) \right\}
  \end{equation*}
  in the $(j,\gq)$-plane.
  We also define a region $I_\text{min}$ pictorially in Figure \ref{fig:Imin}.
  A formal definition can be made in terms of the numbers $\varepsilon, \varepsilon^\prime$ from Theorem \ref{thm:AS-identn} and Corollary \ref{cor:AScor}, but it is a bit awkward to do it this way. This diagram shows the case $n \geq 3$; the region $I_{-1}(\mb{H},DB)$ is generally larger when $n \leq 2$, as shown in Corollary \ref{cor:AScor}. \index{I-min@$I_{\text{min}}$}
\end{dfn}

\begin{figure}
\caption{The region $I_\text{max}$.}\label{fig:Imax-diagram}
\begin{center}
\begin{tikzpicture}[scale=2.5]
	%\draw [help lines] (0,0) grid (5,2);

	\newcommand*{\reg}{0.6}%intermediate regularity s_0 to show (modulus)

	%DB interval endpoints, can be tweaked to personal preference
	\coordinate (P00) at (1,2);
	\coordinate (P01) at (3.35,2);
	\coordinate (P10) at (0.65,0);
	\coordinate (P11) at (3,0);
	
	%I_s interval endpoints
	\coordinate (S0) at (${1-\reg}*(P00) + \reg*(P10)$);
	\coordinate (S1) at (${1-\reg}*(P01) + \reg*(P11)$);  %compute as convex combinations

	\draw [thin] (0.5,2) -- (3.5,2); %theta=0 axis
	\draw [thin] (0.5,0) -- (3.5,0); %theta=-1 axis
	
	%j axis with labels
	\draw [thick,->] (0.5,-0.5) -- (3.5,-0.5);
	\draw [fill=black] (2,-0.5) circle [radius = .5pt];
	\node [below] at (2,-0.5) {$\frac{1}{2}$};
	\draw [fill=black] (3,-0.5) circle [radius = .5pt];
	\node [below] at (3,-0.5) {$1$};
	\draw [fill=black] (3.35,-0.5) circle [radius = .5pt];
	\node [below] at (3.35,-0.5) {$\frac{n+1}{n}$};
	\node [right] at (3.5,-0.5) {$j(\mb{p})$};
	
	%alpha axis with labels
	%\draw [thick,|->] (1,-0.5) -- (0,-0.5); 
	%\node [left] at (0,-0.5) {$\alpha$};
	\draw [fill=black] (0.65,-0.5) circle [radius = .5pt];
	\draw [fill=black] (1,-0.5) circle [radius = .5pt];
	\node [below] at (1,-0.5) {$0$};
	\node [below] at (0.65,-0.5) {$\frac{-1}{n}$};
	
	%theta axis with labels
	\draw [thick,->] (0,-0.2) -- (0,2.2);
	\node [above] at (0,2.2) {$\gq(\mb{p})$};
	\draw [fill=black] (0,2) circle [radius = .5pt];
	\node [left] at (0,2) {$0$};
	%label at intermediate point if desired
	%\draw [fill=black] (${1-\reg}*(-0.5,2) + \reg*(-0.5,0)$) circle [radius = .5pt];
	%\node [left] at (${1-\reg}*(-0.5,2) + \reg*(-0.5,0)$) {$s_0$};
	\draw [fill=black] (0,0) circle [radius = .5pt];
	\node [left] at (0,0) {$-1$};

	%p axis vertical lines
	\draw [thin,dotted] (0.65,2.1) -- (0.65,-0.1); %j=-1/n line
	\draw [thin,dotted] (1,2.1) -- (1,-0.1); %j=0 line
	\draw [thin,dotted] (2,2.1) -- (2,-0.1); %j=1/2 line
	\draw [thin,dotted] (3,2.1) -- (3,-0.1); %j=1 line
	\draw [thin,dotted] (3.35,2.1) -- (3.35,-0.1); %j = n/(n+1) line

	\draw [dotted] (P00) -- (P10);
	\draw [dotted] (P01) -- (P11);
	
	%shade interior
	\path [fill=lightgray, opacity = 0.5] (P00)--(P10)--(1,0)--(P00);
	\path [fill=lightgray, opacity = 0.5] (P00)--(3,2)--(P11)--(1,0)--(P00);
	\path [fill=lightgray, opacity = 0.5] (3,2)--(P11)--(P01)--(3,2);
	
	%theta = 0 interval
	\draw [thick] (P00) -- (P01);
	\draw [fill=white] (P00) circle [radius = .75pt];
	\draw [fill=white] (P01) circle [radius = .75pt];
	%\node [above] at ($1/2*(P00) + 1/2*(P01)$) {$I_{DB,0}$};
	
	%theta = -1 interval
	\draw [thick] (P10) -- (P11);
	\draw [fill=white] (P10) circle [radius = .75pt];
	\draw [fill=white] (P11) circle [radius = .75pt];
	%\node [above] at ($1/2*(P10) + 1/2*(P11)$) {$I_{DB,-1}$};
	
	%draw intermediate interval
	%\draw [thick] (S0) -- (S1);
	%\draw [fill=white] (S0) circle [radius = .75pt];
	%\draw [fill=white] (S1) circle [radius = .75pt];
	%\node [above] at ($1/2*(S0) + 1/2*(S1)$) {$I_{DB,s_0}$};
	
\end{tikzpicture}
\end{center}
\end{figure}
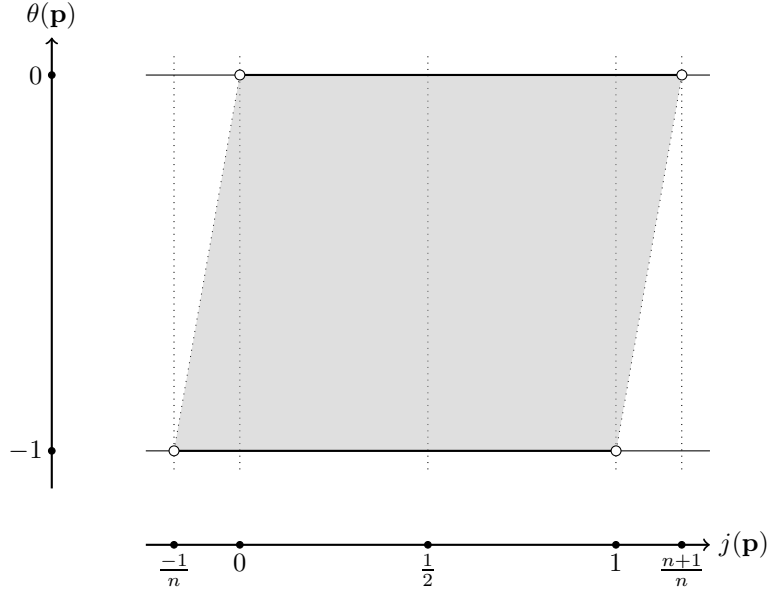
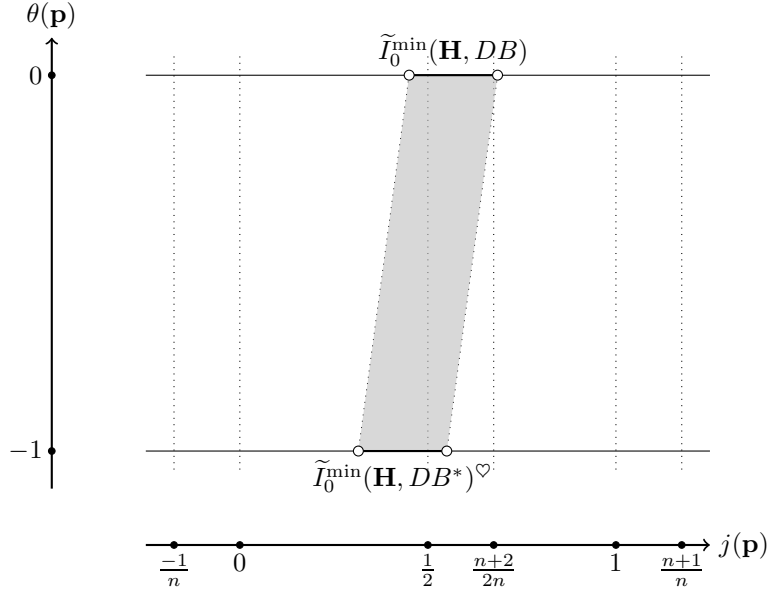
\begin{figure}
\caption{The region $I_\text{min} \subset I(\mb{H},DB)$.}\label{fig:Imin}
\begin{center}
\begin{tikzpicture}[scale=2.5]
	%\draw [help lines] (0,0) grid (5,2);

	\newcommand*{\reg}{0.6}%intermediate regularity s_0 to show (modulus)

	%DB interval endpoints, can be tweaked to personal preference
	\coordinate (P00) at (1.9,2);
	\coordinate (P01) at (2.37,2);
	\coordinate (P10) at (1.63,0);
	\coordinate (P11) at (2.1,0);
	
	%I_s interval endpoints
	\coordinate (S0) at (${1-\reg}*(P00) + \reg*(P10)$);
	\coordinate (S1) at (${1-\reg}*(P01) + \reg*(P11)$);  %compute as convex combinations

	\draw [thin] (0.5,2) -- (3.5,2); %theta=0 axis
	\draw [thin] (0.5,0) -- (3.5,0); %theta=-1 axis
	
	%j axis with labels
	\draw [thick,->] (0.5,-0.5) -- (3.5,-0.5);
	\draw [fill=black] (2,-0.5) circle [radius = .5pt];
	\node [below] at (2,-0.5) {$\frac{1}{2}$};
	\draw [fill=black] (2.35,-0.5) circle [radius = .5pt];
	\node [below] at (2.35,-0.5) {$\frac{n+2}{2n}$};
	\draw [fill=black] (3,-0.5) circle [radius = .5pt];
	\node [below] at (3,-0.5) {$1$};
	\draw [fill=black] (3.35,-0.5) circle [radius = .5pt];
	\node [below] at (3.35,-0.5) {$\frac{n+1}{n}$};
	\node [right] at (3.5,-0.5) {$j(\mb{p})$};
	
	%alpha axis with labels
	%\draw [thick,|->] (1,-0.5) -- (0,-0.5); 
	%\node [left] at (0,-0.5) {$\alpha$};
	\draw [fill=black] (0.65,-0.5) circle [radius = .5pt];
	\draw [fill=black] (1,-0.5) circle [radius = .5pt];
	\node [below] at (1,-0.5) {$0$};
	\node [below] at (0.65,-0.5) {$\frac{-1}{n}$};
	
	%theta axis with labels
	\draw [thick,->] (0,-0.2) -- (0,2.2);
	\node [above] at (0,2.2) {$\gq(\mb{p})$};
	\draw [fill=black] (0,2) circle [radius = .5pt];
	\node [left] at (0,2) {$0$};
	%label at intermediate point if desired
	%\draw [fill=black] (${1-\reg}*(-0.5,2) + \reg*(-0.5,0)$) circle [radius = .5pt];
	%\node [left] at (${1-\reg}*(-0.5,2) + \reg*(-0.5,0)$) {$s_0$};
	\draw [fill=black] (0,0) circle [radius = .5pt];
	\node [left] at (0,0) {$-1$};

	%p axis vertical lines
	\draw [thin,dotted] (0.65,2.1) -- (0.65,-0.1); %j=-1/n line
	\draw [thin,dotted] (1,2.1) -- (1,-0.1); %j=0 line
	\draw [thin,dotted] (2,2.1) -- (2,-0.1); %j=1/2 line
	\draw [thin,dotted] (2.35,2.1) -- (2.35,-0.1); %j = (n+2)/2n line
	\draw [thin,dotted] (3,2.1) -- (3,-0.1); %j=1 line
	\draw [thin,dotted] (3.35,2.1) -- (3.35,-0.1); %j = n/(n+1) line

	\draw [dotted] (P00) -- (P10);
	\draw [dotted] (P01) -- (P11);
	
	%shade interior
	\path [fill=lightgray, opacity = 0.6] (P00)--(P01)--(P11)--(P10)--(P00);
	
	%theta = 0 interval
	\draw [thick] (P00) -- (P01);
	\draw [fill=white] (P00) circle [radius = .75pt];
	\draw [fill=white] (P01) circle [radius = .75pt];
	\node [above] at ($1/2*(P00) + 1/2*(P01)$) {$\wtd{I}_0^{\text{min}}(\mb{H},DB)$};
	
	%theta = -1 interval
	\draw [thick] (P10) -- (P11);
	\draw [fill=white] (P10) circle [radius = .75pt];
	\draw [fill=white] (P11) circle [radius = .75pt];
	\node [below] at ($1/2*(P10) + 1/2*(P11)$) {$\wtd{I}_{0}^{\text{min}}(\mb{H},DB^*)^\heartsuit$};
	
	%draw intermediate interval
	%\draw [thick] (S0) -- (S1);
	%\draw [fill=white] (S0) circle [radius = .75pt];
	%\draw [fill=white] (S1) circle [radius = .75pt];
	%\node [above] at ($1/2*(S0) + 1/2*(S1)$) {$I_{DB,s_0}$};
	
\end{tikzpicture}
\end{center}
\end{figure}

\begin{thm}\label{thm:identn-etc}\
	\begin{enumerate}[(i)]
	\item
	Suppose that $\mb{p} \in I_\text{max}$.
	If $i(\mb{p}) \leq 2$, then $\bbH_{DB}^\mb{p} \hookrightarrow \bbH_D^\mb{p}$, and if $i(\mb{p}) \geq 2$, then $\bbH_D^\mb{p} \hookrightarrow \bbH_{DB}^\mb{p}$.
	Furthermore, if $\gq(\mb{p}) \in (-1,0)$, then the same results hold with $\mb{H}$ replaced by $\mb{B}$.
	\item
	The region $I_\text{min}$ is contained in $I(\mb{H},DB)$, and its interior is contained in $I(\mb{B},DB)$.
	\end{enumerate}
\end{thm}

Heuristically, $I_{\max}$ is the largest region of exponents $\mb{p}$ for which we expect to have $\mb{p} \in I(\mb{X},DB)$ (it is certainly the largest region one can obtain by interpolation between the results of the second author and Stahlhut).
Of course, we may have $\mb{p} \in I(\mb{X},DB)$ for some $\mb{p} \notin I_\text{max}$, as in the case $B = I$.
The second part of Theorem \ref{thm:identn-etc} tells us that there is a \emph{fixed} open region, $I_\text{min}$, of exponents $\mb{p}$ for which $\bbX_{DB}^\mb{p} = \bbX_D^\mb{p}$ for all $B$.
Generally this will be true for a larger range of exponents depending on $B$.

We can exploit the first part of Theorem \ref{thm:identn-etc} and the extrapolation theorem of {\v{S}}ne{\u\i}berg \cite{iS74} to prove an openness result.
{\v{S}}ne{\u\i}berg's extrapolation theorem was extended to quasi-Banach spaces by Kalton and Mitrea \cite[Theorem 2.7]{KM98}, and elaborated upon by Kalton, Mayboroda, and Mitrea \cite[Theorem 8.1]{KMM07}.

\begin{thm}\label{thm:identn-openness}\index{region!identification!openness}\index{extrapolation!of identification region}
	The exponent regions $I(\mb{X},DB) \cap \{\mb{p} \in I_{\text{max}} : \gq(\mb{p}) \in (-1,0)\}$ are open.
\end{thm}

\begin{proof}
	We only prove the result for $I(\mb{H},DB)$, as that for $I(\mb{B},DB)$ follows by the same argument.
	Suppose $\mb{p} \in I(\mb{H},DB) \cap \{\mb{p} \in I_{\text{max}} : \gq(\mb{p}) \in (-1,0)\}$.
	Since $I_\text{min} \subset I(\mb{H},DB)$ is an open neighbourhood of $\{\mb{p} \in I_\text{max} : i(\mb{p}) = 2\}$ (relatively to $I_\text{max}$), we can assume $i(\mb{p}) \neq 2$.
	We will further assume $i(\mb{p}) < 2$, as the argument for $i(\mb{p}) > 2$ is the same.
	
	There exists a ball $B_\mb{p}$ in the $(j,\gq)$-plane such that for all $\mb{q} \in B_\mb{p}$ we have $i(\mb{q}),i([\mb{p},\mb{q}]_{-1}) < 2$, $\gq(\mb{p}) \in (-1,0)$, and $\mb{q} \in I_{\text{max}}$.
	For all such $\mb{q}$ write $\mb{q}_0 := \mb{q}$ and $\mb{q}_1 := [\mb{p},\mb{q}]_{-1}$; we then have $\mb{p} = [\mb{q}_0,\mb{q}_1]_{1/2}$.
	
	Fix $\mb{q} \in B_\mb{p}$.
	By part (i) of Theorem \ref{thm:identn-etc}, we have $\mb{q}_0, \mb{q}_1 \in I(\mb{H},DB,D)$.
	Let $\gf,\gy \in \gY_\infty^\infty$ be Calder\'on siblings.
	Then by Lemma \ref{lem:incl-QS} the composition $\bbQ_{\gf,D}\bbS_{\gy,DB}$ is bounded from $\bbQ_{\gf,DB} \bbH_{DB}^{\mb{q}_i}$ to $\bbQ_{\gf,D} \bbH_D^{\mb{q}_i}$ for $i \in \{0,1\}$, and thus its extension, which we call $\mathbf{Q}_{\gf,D}\mathbf{S}_{\gy,DB}$, is bounded between the completions $\gf \mb{H}_{DB}^{\mb{q}_i} \to \gf \mb{H}_{D}^{\mb{q}_i}$.
	Furthermore, the assumption on $\mb{p}$ implies that $\mathbf{Q}_{\gf,D} \mathbf{S}_{\gy,DB}$ is an isomorphism from $\gf \mb{H}_{DB}^{\mb{p}}$ to $\gf \mb{H}_{D}^{\mb{p}}$.
	For $\gn \in (-1/2,1/2)$ define $\mb{p}_\gn := [\mb{q}_0, \mb{q}_1]_{1/2 + \gn}$.
	Since
	\begin{equation*}
		\gf\mb{H}_A^{\mb{p}_\gn} = [\gf\mb{H}_A^{\mb{q}_0}, \gf\mb{H}_A^{\mb{q}_1}]_{1/2 + \gn}
	\end{equation*}
	(Theorem \ref{thm:completed-interpolation}), {\v{S}}ne{\u\i}berg's extrapolation theorem implies that there exists $\varepsilon > 0$ independent of $\mb{q} \in B_\mb{p}$ (see Remark \ref{rmk:eps-indeps} below) such that for all $\gn \in (-\varepsilon,\varepsilon)$,
	\begin{equation*}
		\map{\mathbf{Q}_{\gf,D} \mathbf{S}_{\gy,DB}}{\gf \mb{H}_{DB}^{\mb{p}_\gn}}{\gf \mb{H}_{D}^{\mb{p}_\gn}}
	\end{equation*}
	is an isomorphism.
	Since for each $\gn \in (-\varepsilon,\varepsilon)$ the map $\mathbf{Q}_{\gf,D} \mathbf{S}_{\gy,DB}$
	 extends $\bbQ_{\gf,D} \bbS_{\gy,DB}$ and since by the Calder\'on reproducing formula (Theorem \ref{thm:CRF}) its inverse (if it exists) extends $\bbQ_{\gf,DB} \bbS_{\gy,D}$, this implies that $\mb{p}_\gn \in I(\mb{H},DB)$, and thus there is a neighbourhood $O_{\mb{p}}$ of $\mb{p}$ such that
	 \begin{equation*}
	 	O_{\mb{p}} \subset B_{\mb{p}} \subset I(\mb{H},DB) \cap \{\mb{p} \in I_{\text{max}} : \gq(\mb{p}) \in (-1,0)\}.
	\end{equation*}
	This completes the proof.
      \end{proof}

      \begin{rmk}\label{rmk:eps-indeps}
        The fact that $\varepsilon$ is independent of $\mb{q}$ in the proof above is not obvious. Inspection of the Kalton--Mitrea proof of {\v{S}}ne{\u\i}berg's theorem \cite[Theorem 2.7]{KM98} shows that $\varepsilon$ depends only on the norms of $\map{\mb{Q}_{\gf,D} \mb{S}_{\gy,DB}}{\gf \mb{H}_{DB}^{\mb{q}_i}}{\gf \mb{H}_{D}^{\mb{q}_i}}$ and the constants in the norm equivalences $\gf\mb{H}_A^{[\mb{q}_0,\mb{q}_1]_\gq} = [\gf\mb{H}_A^{\mb{q}_0}, \gf\mb{H}_A^{\mb{q}_1}]_{\gq}$. Having fixed $\gf$ and $\gy$, these constants are bounded independently of $\mb{q}$, provided that we restrict $\mb{q}$ to a small ball $B_\mb{p}$, as we have done.
      \end{rmk}

\begin{rmk}
If $\mb{p} \in I(\mb{X},DB)$, then we identify $\mb{X}^\mb{p}_D = \bbP_D \mb{X}^\mb{p}$ (see Definition \ref{dfn:boldXD}) as a completion of ${\bbX}_{DB}^\mb{p}$ via the extension of the identity map ${\bbX}_{DB}^\mb{p} \to {\bbX}_{D}^\mb{p}$. 
If $\mb{p}$ is infinite and $\mb{p}^\heartsuit \in I(\mb{X},DB^*)$, then by Proposition \ref{prop:dual-interval-incl} we may identify $\mb{X}_D^\mb{p}$ as a weak-star completion of $\bbX_{DB}^\mb{p}$.
We write $\mb{X}_{DB}^\mb{p}$ for these completions.

Note that we do not have equality of $\mb{X}_{DB}^{\mb{p},+}$ and $\mb{X}_{D}^{\mb{p},+}$.
The first of these spaces is defined via the spectral projection $\gc^+(DB)$, while the second is defined via $\gc^+(D)$.
However, we do of course have $\mb{X}_{DB}^{\mb{p},+} \subset \mb{X}_D^\mb{p}$.
This will be important in applications to boundary value problems.
\end{rmk}

\begin{rmk}\label{rmk:Bduals}
	For a coefficient matrix $A$ as in the introduction, if $B = \hat{A}$, then $\widehat{A^*} = \td{B} := NB^* N$, where
	\begin{equation*}
		N := \begin{bmatrix} I & 0 \\ 0 & -I \end{bmatrix}.
	\end{equation*}
	Since $DN = -ND$ and $N$ acts on $\overline{\mc{R}(D)}$, the operators $DB^*$ and $-D\td{B}$ are similar on $\overline{\mc{R}(D)} = \overline{\mc{R}(DB^*)} = \overline{\mc{R}(D\td{B})}$.
	Thus all functional calculus properties of $DB^*$ can be transferred to $D\td{B}$, and vice versa.
	This gives natural isomorphisms between $\bbX_{DB^*}^\mb{p}$ and $\bbX_{D\td{B}}^\mb{p}$, and in particular we have $I(\mb{X},DB^*) = I(\mb{X},D\td{B})$.
	However, note that this isomorphism interchanges spectral subspaces: $\bbX_{DB^*}^{\mb{p},\pm} = \bbX_{D\td{B}}^{\mb{p},\mp}$.
	For further details see \cite[\textsection 12.2]{AS16}.
\end{rmk}

\section{Projections and $BD$-adapted spaces}\label{sec:proj-BD}

In the previous section we discussed exponents $\mb{p}$ for which $\bbX_{DB}^\mb{p} = \bbX_D^\mb{p}$.
For such exponents, we show that the projections $\bbP_{D}$ and $\bbP_{BD}$ can be used to move between $\bbX_D^{\mb{p}+1}$ and the $BD$-adapted space $\bbX_{BD}^{\mb{p}+1}$.
Note that we must now use the exponent $\mb{p} + 1$.
We begin with a basic lemma.

\begin{lem}\label{lem:proj}\
\begin{enumerate}[(i)]
	\item For all $h\in L^2$,  we have $D\bbP_{BD}h = Dh$ in the distributional sense.
	Consequently the projection $\bbP_{BD}$ preserves $\mc{D}(D)$, and $\nm{D\bbP_{BD}h}_{2}= \nm{Dh}_{2}$ (note that this holds in particular for $B=I$). 

	\item The restrictions $\map{\bbP_{BD}}{\overline{\mc{R}(D)}}{ \overline{\mc{R}(BD)}}$ and $\map{\bbP_{D}}{\overline{\mc{R}(BD)}}{ \overline{\mc{R}(D)}}$ are bounded isomorphisms and mutually inverse. 

	\item The restrictions
          \begin{align*}
            \map{\bbP_{BD}}{&\mc{D}(D) \cap \overline{\mc{R}(D)}}{ \mc{D}(D) \cap \overline{\mc{R}(BD)}} \\
            \map{\bbP_{D}}{&\mc{D}(D) \cap \overline{\mc{R}(BD)}}{ \mc{D}(D) \cap \overline{\mc{R}(D)}}
          \end{align*}
          are mutually inverse.
        \end{enumerate}
\end{lem}

\begin{proof} To see the first statement, use that $\bbP_{BD}$ is a projection along $\mc{N}(BD) = \mc{N}(D)$ to write $D(h-\bbP_{BD}h)=0$. 
As $Dh$ and $D\bbP_{BD}h$ can be computed distributionally, the equality follows.
The second statement is classical, and holds because the projections $\bbP_{BD}$ and $\bbP_D$ have same nullspace.
The third follows by combining the first and second.
\end{proof}

\begin{prop}\label{prop:commdiag}
  Assume that $\mb{p}$ is an exponent for which ${\bbX}_{DB}^{\mb{p}} = {\bbX}_{D}^{\mb{p}}$.
  Then we have commutative diagrams of bounded bijective maps
\begin{equation*}
		\xymatrix{
		 {\bbX}_{BD}^{\mb{p}+1} \cap \mc{D}(D) \ar[d]_{\bbP_{D}} \ar[r]^{D}   & {\bbX}_{DB}^{\mb{p}}\cap \mc{R}(D)  \ar[d]^{\id}
		 \\
		{\bbX}_{D}^{\mb{p}+1} \cap \mc{D}(D)  \ar[r]_{D}  & {\bbX}_{D}^{\mb{p}}\cap \mc{R}(D) 		} \qquad
		 \xymatrix{
		 {\bbX}_{BD}^{\mb{p}+1} \cap \mc{D}(D)  \ar[r]^{D}   & {\bbX}_{DB}^{\mb{p}}\cap \mc{R}(D)  
		 \\
		{\bbX}_{D}^{\mb{p}+1} \cap \mc{D}(D) \ar[u]^{\bbP_{BD}} \ar[r]_{D}  & {\bbX}_{D}^{\mb{p}}\cap \mc{R}(D). \ar[u]_{\id}  
		}
\end{equation*}
In particular, the projections are mutually inverse. 
\end{prop}

\begin{proof}
Commutativity of the diagrams comes from Lemma \ref{lem:proj}.
All we need to show is boundedness of both projections. 
First consider $f\in {\bbX}_{D}^{\mb{p}+1} \cap \mc{D}(D)$.
Then $\bbP_{BD}f\in \overline{\mc{R}(BD)}$ and   $D\bbP_{BD}f=Df$, hence
\begin{equation*}
\nm{D\bbP_{BD}f}_{{\bbX}_{DB}^{\mb{p}}}= \nm{Df}_{{\bbX}_{DB}^{\mb{p}}} \simeq  \nm{Df}_{{\bbX}_{D}^{\mb{p}}} \simeq  \nm{f}_{{\bbX}_{D}^{\mb{p}+1}}
\end{equation*}
where we used the assumption in the first equivalence and then Proposition \ref{prop:D-mapping} for $D$ in the second. This shows that $\bbP_{BD}f\in {\bbX}_{BD}^{\mb{p}+1}$ with the desired estimate.
Now consider $f\in {\bbX}_{BD}^{\mb{p}+1} \cap \mc{D}(D)$.
Then we have $D\bbP_{D}f=Df$, so
\begin{equation*}
\nm{D\bbP_{D}f}_{{\bbX}_{D}^{\mb{p}}}= \nm{Df}_{{\bbX}_{D}^{\mb{p}}} \simeq  \nm{Df}_{{\bbX}_{DB}^{\mb{p}}} \simeq  \nm{f}_{{\bbX}_{BD}^{\mb{p}+1}}
\end{equation*}
by the same argument as before. 
\end{proof}

The following corollaries are then immediate.

\begin{cor}\label{cor:similarity2} Assume that $\mb{p}$ is an exponent for which ${\bbX}_{DB}^{\mb{p}}= {\bbX}_{D}^{\mb{p}}$.
Then for $\gf \in H^\infty$ we have 
\begin{equation}\label{eqn:simgeneralpbd}
	D\gf(BD)\bbP_{BD} = \gf(DB)D
\end{equation}
where $\gf(BD)$ and $\gf(DB)$ are extensions to (any) completions.
In particular, for Cauchy extensions, 
\begin{equation}\label{eqn:simcauchy2}
	DC_{BD}\bbP_{BD}= C_{DB}D. 
\end{equation}
\end{cor} 

\begin{cor}\label{cor:perpequal} Let $\mb{p}$ be an exponent for which ${\bbX}_{DB}^{\mb{p}} = {\bbX}_{D}^{\mb{p}}$. 
Then for any completion $\mb{Y}$ of $\bbX_{BD}^{\mb{p}+1}$ we have 
$\bbP_{D}\mb{Y}= \mb{X}_{D}^{\mb{p}+1}$ and $\bbP_{BD}\mb{X}_{D}^{\mb{p}+1}= \mb{Y}$. In particular, taking the transversal component, we have
\begin{equation}
\label{eq:rect}
(\bbP_{D}\mb{Y})_{\perp} = (\mb{X}_{D}^{\mb{p}+1})_{\perp} = \mb{X}_\perp^{\mb{p}+1}.
\end{equation} 
\end{cor}

\begin{proof}
	The first equality in \eqref{eq:rect} follows from the equalities in the statement.
	The second is because the projection $\bbP_D$ is the identity on the $\perp$ component.
\end{proof}

\section{The Cauchy operator on $DB$-adapted spaces}\label{sec:sgp-classical}\index{Cauchy operator!on $DB$-adapted BHS spaces}

This section is devoted to the proof of the following theorem.

\begin{thm}[Cauchy characterisation of adapted Besov--Hardy--Sobolev spaces, $i(\mb{p}) > 2$]\label{thm:sgpnorm-concrete}\index{space!adapted Besov--Hardy--Sobolev!characterisation by Cauchy operator}
	Let $\mb{p}$ be such that $i(\mb{p}) > 2$, $\gq(\mb{p}) \in (-1,0)$ and $\mb{p}^\heartsuit \in I(\mb{X},DB^*)$.
	Then for all $f \in \overline{\mc{R}(DB)}$,
	\begin{equation*}
		 \nm{C_{DB}^+ f}_{X^\mb{p}} \lesssim \nm{f}_{{\mb{X}}^\mb{p}_{D}}.
	\end{equation*}
\end{thm}

\begin{rmk}
	The condition $\mb{p}^\heartsuit \in I(\mb{X},DB^*)$ is equivalent to $\mb{p} \in I(\mb{X},DB)$ when $\mb{p}$ is finite (Proposition \ref{prop:dual-interval-incl}).
	Note that the norm on the right hand side is that of $\mb{X}_D^\mb{p}$, which is a completion of $\bbX_{DB}^\mb{p}$ from the hypothesis and Corollary \ref{cor:dual-interval-identn} (changing $B$ to $B^*$ and $\mb{p}$ to $\mb{p}^\heartsuit$).
\end{rmk}

\begin{rmk}\label{rmk:cauchy-reverse}
The reverse estimate
\begin{equation*}
	\nm{f}_{\mb{X}_A^\mb{p}} \lesssim \nm{C_A^+ f}_{X^\mb{p}}
\end{equation*}
holds for all operators $A$ satisfying the Standard Assumptions, all $\mb{p}$, and all $f \in \mb{X}_A^{\mb{p},+}$ (see Remark \ref{rmk:sgp-abstract}).
However, we do not know whether Theorem \ref{thm:sgpnorm-concrete} holds with $DB$ replaced by $A$ and without the assumption on $\mb{p}^\heartsuit$.
\end{rmk}

In contrast with Theorem \ref{thm:sgpnorm-abstract}, the proof of this theorem is quite long, and requires concrete arguments.
First we establish a technical lemma.
Recall that
\begin{equation*}
  A(x,r_0,r_1) := B(x,r_1) \setminus B(x,r_0)
\end{equation*}
was defined in Section \ref{section:notation}.

\begin{lem}\label{lem:C-D-lem}
	Suppose $\gq \in (-1,0)$, $g \in \mc{D}(D)$, and $f = Dg$.
	Then for all $\gx \in \bbR^n$, $\gt > 0$, and $M \in \bbN$, we have
	\begin{align*}
		&\dint_{T(B(\gx,\gt))} |t^{-\gq} (I + itDB)^{-2} f(x)|^2 \, \frac{dx \, dt}{t} \\
		&\lesssim_M \int_{B(\gx,4\gt)} \bigg( \int_{B(\gx, 4\gt)} + \sum_{j=2}^\infty 2^{-2j(M - \frac{n}{2} - (1+\gq))} \int_{A(\gx, 2^{j-1} \gt, 2^{j+2} \gt)} \bigg) G(x,y) \, dx \, dy,
	\end{align*}
	where
	\begin{equation*}
		G(x,y) := \frac{|g(x)-g(y)|^2}{|x-y|^{n + 2(1+\gq)}}.
	\end{equation*}
      \end{lem}

\begin{proof}
	Fix $\gc_1, \gc \in C_c^\infty(\bbR^n)$ such that
	\begin{align*}
		\supp \gc_1 &\subset B(\gx,4\gt), &\gc_1|_{B(\gx,2\gt)} \equiv \text{const}, \\
		\supp \gc &\subset A(\gx,\gt/2,4\gt) &\gc|_{A(\gx,\gt,2\gt)} \equiv \text{const.}
	\end{align*}
	For all $j \geq 2$ define $\gc_j(x) := \gc(2^{-j} x)$, so that $\supp \gc_j \subset A(\gx,2^{j-1}\gt, 2^{j+2}\gt)$ and $\gc_j|_{A(\gx,2^j \gt, 2^{j+1} \gt)} \equiv 1$.
	We can choose the functions $\gc_1$ and $\gc$ such that $\sum_{j=1}^\infty \gc_j \equiv 1$.
	Let
	\begin{equation*}
		c := \barint_{B(\gx,\gt/2)}^{} g.
	\end{equation*}
	Then we have
	\begin{equation*}
		f = D(g-c) = \sum_{j=1}^\infty D(g-c)\gc_j.
	\end{equation*}
	Here, and throughout the proof, we write $D(g-c)\gc_j$ to mean $D[(g-c)\gc_j]$.
	
	First we prove that
	\begin{align}
		\dint_{T(B(\gx,\gt))} &|t^{-\gq}(I+itDB)^{-2} D(g-c)\gc_1(x)|^2 \, \frac{dx \, dt}{t} \nonumber \\
		&\lesssim \int_{B(\gx,4\gt)} \int_{B(\gx,4\gt)} G(x,y) \, dx \, dy \label{eqn:local-est-res}.
	\end{align}
	Since $\gY_0^2 \in \gY_{\gq+}^{-\gq} \cap H^\infty \subset \gY(\bbX^{(2,\gq)}_{DB})$ and since $(2,\gq) \in I(\mb{H},DB)$, we have
	\begin{align*}
		\dint_{T(B(\gx,\gt))} &|t^{-\gq}(I+itDB)^{-2} D(g-c)\gc_1(x)|^2 \, \frac{dx \, dt}{t} \\
		&\leq \dint_{\bbR^{1+n}_+} |t^{-\gq}(I+itDB)^{-2} D(g-c)\gc_1(x)|^2 \, \frac{dx \, dt}{t} \\
		&\simeq \nm{D(g-c)\gc_1}_{\bbH^{(2,\gq)}_{DB}}^2 \\
		&\simeq \nm{(g-c)\gc_1}_{\dot{H}^2_{\gq + 1}}^2 \\
		&\simeq \int_{\bbR^n} |\mc{D}_{\gq + 1}^2(g-c)\gc_1(x)|^2 \, dx
	\end{align*}
	using Theorem \ref{thm:stein} (which is valid since $2 > 2n/(n+1+\gq)$) in the last line.
	
	We claim that
	\begin{align}
          \int_{\bbR^n} \int_{\bbR^n} &\frac{|(g-c)\gc_1(z) - (g-c)\gc_1(y)|^2}{|z-y|^{n + 2(\gq + 1)}} \, dz \, dy \nonumber \\
          &\lesssim \dint_{B(\gx,4\gt)^2} \frac{|g(z) - g(y)|^2}{|z-y|^{n+2(\gq+1)}} \, dz \, dy, \label{eqn:difference-claim}
	\end{align}
	from which estimate \eqref{eqn:local-est-res} will follow.
	First observe that if $y \in B(\gx,\gt)$ and $z \in B(\gx,\gt/2)$, then $\gc_1(z) = \gc_1(y) = 1$ and the estimate \eqref{eqn:difference-claim} (restricted to such $y$ and $z$) follows immediately.
	Next, we can estimate
	\begin{align*}
		&\int_{B(\gx,\gt)^c} \int_{B(\gx,\gt/2)} \frac{|(g-c)\gc_1(z) - (g-c)\gc_1(x)|^2}{|z-x|^{n + 2(\gq+1)}} \, dz \, dx \\
		&\lesssim \int_{A(\gx,\gt,4\gt)} \int_{B(\gx,\gt/2)} \frac{|(g-c)(z)(1 - \gc_1(x))|^2}{|z-x|^{n+2(\gq+1)}} \, dz \, dx + A \\
		&\lesssim_\gc \int_{A(\gx,\gt,4\gt)} \int_{B(\gx,\gt/2)} |z-x|^{-n-2(\gq+1)} \bigg( \barint_{B(\gx,\gt/2)}^{} |g(z) - g(y)| \, dy \bigg)^2 \, dz \, dx + A \\
		&\lesssim r^n \int_{B(\gx,\gt/2)} |z-y|^{-n-2(\gq+1)} \bigg( \barint_{B(\gx,\gt/2)}^{} |g(z) - g(y)| \, dy \bigg)^2 \, dz + A \\
		&\lesssim r^n \barint_{B(\gx,\gt/2)}^{} \int_{B(\gx,\gt/2)} \frac{|g(z) - g(y)|^2}{|z-y|^{n + 2(\gq + 1)}} \, dz \, dy + A\\
		&\lesssim \dint_{B(\gx,4\gt)^2} \frac{|g(z) - g(y)|^2}{|z-y|^{n+2(\gq+1)}} \, dz \, dy,
	\end{align*}
	where
	\begin{equation*}
		A \lesssim \dint_{B(\gx,4\gt)^2} \frac{|g(z) - g(y)|^2}{|z-y|^{n+2(\gq+1)}} \, dz \, dy,
	\end{equation*}
	and where we used $|z-x| \gtrsim |z-y|$ on the region of integration.

	Finally, we estimate
	\begin{align*}
		&\int_{\bbR^n} \int_{B(\gx,\gt/2)^c} \frac{|(g-c)\gc_1(z) - (g-c)\gc_1(x)|^2}{|z-x|^{n + 2(\gq+1)}} \, dz \, dx \\
		&\lesssim \int_{B(\gx,4\gt)} \int_{A(\gx,\gt/2,4\gt)} \frac{|(g-c)(x)(\gc_1 - 1)(z) - (g-c)(x)(\gc_1 - 1)(x)|^2}{|z-x|^{n+2(\gq+1)}} \, dz \, dx + A \\
		&\lesssim_\gc \int_{B(\gx,4\gt)} \int_{A(\gx,\gt/2,4\gt)} |z-x|^{-n-2\gq} \bigg( \barint_{B(\gx,\gt/2)}^{} |g(x) - g(y)| \, dy \bigg)^2 \, dz \, dx + A \\
		&\lesssim \int_{B(\gx,4\gt)} r^{-2\gq} \bigg( \barint_{B(\gx,\gt/2)}^{} |g(x) - g(y)| \, dy \bigg)^2 \, dx + A \\
		&\lesssim \dint_{B(\gx,4\gt)^2} \frac{|g(x) - g(y)|^2}{|x-y|^{n+2(\gq+1)}} \, dx \, dy  
	\end{align*}
	with $A$ as before, using $r \gtrsim |x-y|$ and $|x-y|^{n + 2\gq} \gtrsim |x-y|^{n + 2(\gq + 1)}$ on the region of integration.
	
	Now we handle the remaining $\gc_j$ terms.
	For $j \geq 2$, by local coercivity (Lemma \ref{lem:local-coercivity}), the equality $BD(I + itBD)^{-1} = (I - (I+itBD)^{-1})/it$, and off-diagonal estimates of the families $(I + itBD)^{-1}$ and $(I + itBD)^{-2}$ of order $M$, we can estimate
	\begin{align*}
		\dint_{T(B(\gx,\gt))} &|t^{-\gq} (I+itDB)^{-2} D(g-c)\gc_j(x)|^2 \, \frac{dx \, dt}{t} \\
		&\lesssim \int_0^\gt t^{-2\gq} \int_{B(\gx,\gt)} |D(I + itBD)^{-2}(g-c)\gc_j(x)|^2 \, dx \, \frac{dt}{t} \\
		&\lesssim \int_0^\gt t^{-2\gq} \bigg[ \int_{B(\gx,2\gt)} |(BD(I+itBD)^{-2}(g-c)\gc_j(x)|^2 \, dx \\
		&\qquad \qquad + \gt^{-2} \int_{B(\gx,2\gt)} |(I + itBD)^{-2}(g-c)\gc_j(x)|^2 \, dx \bigg] \, \frac{dt}{t} \\
		&\lesssim \int_0^\gt t^{-2\gq} \bigg( \frac{2^j \gt}{t} \bigg)^{-2M} ( t^{-2} + \gt^{-2} ) \nm{(g-c)\gc_j}_2^2  \, \frac{dt}{t} \\
		&\lesssim \frac{2^{-2jM}}{\gt^{2(1+\gq)}} \nm{(g-c)\gc_j}_2^2.
	\end{align*}
	Furthermore, for each $j \geq 2$ we have
	\begin{align*}
		\nm{(g-c)\gc_j}_2^2
		&\leq \int_{A(\gx,2^{j-1} \gt, 2^{j+2} \gt)} \bigg( \barint_{B(\gx,\gt/2)}^{} |g(x) - g(y)| \, dy \bigg)^2 \, dx \\
		&\leq \gt^{-n} \int_{B(\gx,\gt/2)} \int_{A(\gx,2^{j-1}\gt, 2^{j+2}\gt)} |g(x) - g(y)|^2 \, dx \, dy \\
		&\lesssim \gt^{2(1-\gq)} 2^{j(n+2(1+\gq))} \int_{B(\gx,4\gt)} \int_{A(\gx,2^{j-1}\gt,2^{j+2}\gt)} \frac{|g(x) - g(y)|^2}{|x-y|^{n+2(1+\gq)}} \, dx \, dy.
	\end{align*}
	Putting these estimates together completes the proof of the lemma.
\end{proof}

\begin{proof}[Proof of Theorem \ref{thm:sgpnorm-concrete}]\

\textbf{Step 1: Reduction to a resolvent estimate.}

As stated in \cite[Proof of Lemma 15.1]{AM15}, there exists $\gr \in H^\infty$ of the form
	\begin{equation*}
		\gr(z) = \sum_{m=1}^N c_m (1 + imz)^{-2}
	\end{equation*}
	for some scalars $c_1,\ldots,c_N \in \bbC$, and $\gy \in \gY_N^2$ nondegenerate, such that
	\begin{equation*}
		e^{-z} = \gr(z) + \gy(z) \quad \text{for all $z \in S_\gm^+$}.
	\end{equation*}
	Therefore we have
	\begin{equation}\label{eqn:sgp-red}
		\nm{C_{DB}^+ f}_{X^\mb{p}}
		\lesssim_N \nm{t \mapsto (I + itDB)^{-2}\gc^+(DB)f}_{X^\mb{p}} + \nm{\bbQ_{\gy,DB}\gc^+(DB)f}_{X^\mb{p}}.  
	\end{equation}
	For $N$ sufficiently large we have
	\begin{equation*}
		\gy \in \gY_N^2 \subset \gY_{(\gq(\mb{p}) + n|\frac{1}{2} - j(\mb{p})|)+}^{-\gq(\mb{p})+} \cap \gY_+^+ \subset \gY(\bbX_{DB}^\mb{p}),
	\end{equation*}
	and so
	\begin{equation*}
		\nm{\bbQ_{\gy,DB} \gc^+(DB)f}_{X^\mb{p}} \lesssim \nm{f}_{\bbX_{DB}^\mb{p}} \simeq \nm{f}_{\bbX_D^\mb{p}}
	\end{equation*}
	by Proposition \ref{prop:dual-interval-incl} and using that $\mb{p}^\heartsuit \in I(\mb{X},DB^*)$.
	Therefore it suffices to prove the estimate
	\begin{equation}\label{eqn:resgoal}
		\nm{t \mapsto (I + itDB)^{-2} f}_{X^\mb{p}} \lesssim \nm{f}_{\bbX_{D}^\mb{p}}
	\end{equation}
	for all $f \in \overline{\mc{R}(DB)}$.
	Applying this inequality to $\gc^+(DB)f$ and invoking the boundedness of $\gc^+(DB)$ on $\bbX_{DB}^\mb{p}=\bbX_{D}^\mb{p}$ will yield
	\begin{equation*}
		\nm{C_{DB}^+ f}_{X^\mb{p}}
		\lesssim \nm{ \gc^+(DB) }_{\bbX_{D}^\mb{p}}
		\simeq \nm{ f }_{\bbX_D^\mb{p}}.
	\end{equation*}
	To prove \eqref{eqn:resgoal}, by density (Corollary \ref{cor:adapted-space-int-dens} and density of $\mc{R}(D)$ in $\bbX_D^2$),\footnote{When $\mb{p}$ is infinite, we use weak-star density.} it suffices to consider $f = Dg$ for $g \in \mc{D}(D) \cap \overline{\mc{R}(D)}$ such that
\begin{equation*}
	\nm{f}_{\bbX_D^\mb{p}} \simeq \nm{f}_{\mb{X}^\mb{p}} \simeq \nm{g}_{\mb{X}^{\mb{p}+1}}.
\end{equation*}
	
	We will now prove \eqref{eqn:resgoal} by splitting cases depending on the function space under consideration.

\textbf{Step 2a: Completing the proof for Hardy--Sobolev spaces.}

Suppose $i(\mb{p}) < \infty$ and $(X,\bbX) = (T,\bbH)$.
Lemma \ref{lem:C-D-lem} and a crude estimate give
\begin{align*}
	\dint_{T(B(\gx,\gt))} &|t^{-\gq(\mb{p})} (I + itDB)^{-2} f(x)|^2 \, \frac{dx \, dt}{t} \\
	&\lesssim_M \bigg( 1 + \sum_{j=2}^\infty 2^{-2j(M-\frac{n}{2} - (1+\gq(\mb{p})))} \bigg) \int_{B(\gx,4\gt)} |\mc{D}^2_{1+\gq(\mb{p})} g(x)|^2 \, dx,
\end{align*}
 so by taking $M > \frac{n}{2} - (1+\gq(\mb{p}))$ we get
\begin{equation*}
	\dint_{T(B(\gx,\gt))} |t^{-\gq(\mb{p})} (I + itDB)^{-2} f(x)|^2 \, \frac{dx \, dt}{t} \lesssim \int_{B(\gx,4\gt)} |\mc{D}_{1+\gq(\mb{p})}^2 g(x)|^2 \, dx.
\end{equation*}
Hence for all $\gx \in \bbR^n$ we have
\begin{equation*}
	\mc{C}_0(t \mapsto t^{-\gq(\mb{p})} (I + itDB)^{-2} f)(\gx)^2 \lesssim \sup_{\gt > 0} \barint_{B(\gx,4\gt)}^{} |\mc{D}_{1 + \gq(\mb{p})}^2 g|^2 = \mc{M}_2(\mc{D}^2_{1 + \gq(\mb{p})} g)(\gx)^2,
      \end{equation*}
      where $\mc{M}_2(f) := \mc{M}(|f|^2)^{1/2}$, with $\mc{M}$ the Hardy--Littlewood maximal operator.
	By Theorem \ref{thm:tent-AC}, boundedness of $\mc{M}_2$ on $L^{i(\mb{p})}$ (since $i(\mb{p}) > 2$), Theorem \ref{thm:stein} (using $1+\gq(\mb{p}) \in (0,1)$), $Dg = f$, and $\mb{p} \in I(\mb{H},DB)$, we get
	\begin{align*}
		\nm{t \mapsto (I + itDB)^{-2} f}_{T^{\mb{p}}}
		%\lesssim \nm{\mc{M}_2(\mc{D}^2_{1+\gq(\mb{p})} g)}_{L^{i(\mb{p})}} 
		\lesssim \nm{\mc{D}^2_{1+\gq(\mb{p})} g}_{L^{i(\mb{p})}} 
		\simeq \nm{g}_{\dot{H}^{\mb{p}+1}} 
		\simeq \nm{f}_{\bbH_D^\mb{p}},
	\end{align*}
	which completes the proof in the Hardy--Sobolev case.

\textbf{Step 2b: Completing the proof for $BMO$-Sobolev spaces.}
Suppose $\mb{p} = (\infty,\gq;0)$ and $(X,\bbX) = (T,\bbH)$.
For all $(t,x) \in \bbR^{1+n}_+$, Lemma \ref{lem:C-D-lem} and Strichartz's characterisation of $\dot{BMO}_{1+\gq}$ (Theorem \ref{thm:strichartz-bmo}) yield
\begin{align*}
	&\bigg( t^{-n} \dint_{T(B(x,t))} |\gt^{-\gq} (I + i\gt DB)^{-2} f(\gx)|^2 \, d\gx \, \frac{d\gt}{\gt} \bigg)^{1/2} \\
	&\lesssim_M t^{-n/2} \bigg( t^n \nm{g}_{\dot{\BMO}_{1+\gq}}^2 + \sum_{j=2}^\infty 2^{-2j(M - \frac{n}{2} - (1+\gq))} (2^{j+2} t)^n \nm{g}_{\dot{\BMO}_{1+\gq}}^2 \bigg)^{1/2} \\
	&\simeq \nm{g}_{\dot{\BMO}_{1+\gq}} \bigg( 1 + \sum_{j=2}^\infty 2^{-2j(M - n - (1+\gq))} \bigg)^{1/2} \\
	&\simeq \nm{g}_{\dot{\BMO}_{1+\gq}}
\end{align*}
provided that $M$ is sufficiently large.
Therefore as in the previous step we have
\begin{equation*}
	\nm{\gt \mapsto (I + i \gt DB)^{-2} f}_{T^\mb{p}}
	\lesssim {\nm{g}_{\mb{H}^{\mb{p} + 1}}} \\
	\simeq \nm{f}_{\bbH_{D}^\mb{p}},
\end{equation*}
which completes the proof in the $\BMO$-Sobolev case.

\textbf{Step 2c: Completing the proof for H\"older spaces.}
Let $\mb{p} = (\infty,\gq;\ga)$.
First we prove the result for $X = T$.
By the definition of the H\"older norm we have
\begin{equation*}
	G(x,y) \leq \nm{g}_{\dot{\gL}_{1+\gq+\ga}}^2 |x-y|^{2\ga - n},
\end{equation*}
so by Lemma \ref{lem:C-D-lem},
\begin{align*}
	&t^{-\ga} \bigg( t^{-n} \dint_{T(B(x,t))} |\gt^{-\gq} (I + i\gt DB)^{-2} f(\gx)|^2 \, d\gx \, \frac{d\gt}{\gt} \bigg)^{1/2} \\
	&\lesssim_M t^{-\ga - \frac{n}{2}} \nm{g}_{\dot{\gL}_{1 + \gq + \ga}} \bigg( \dint_{B(x,4t)^2} \, \frac{d\gx \, d\gh}{|\gx-\gh|^{n-2\ga}} \\
	&\qquad +\sum_{j=2}^\infty 2^{-2j(M - \frac{n}{2} - (1+\gq))} \int_{B(x,4t)} \int_{A(x,2^{j-1} t, 2^{j+2}t)} \, \frac{d\gx \, d\gh}{|\gx-\gh|^{n-2\ga}} \bigg)^{1/2} \\
	&\lesssim t^{-\ga-\frac{n}{2}} \nm{g}_{\dot{\gL}_{1 + \gq + \ga}} \bigg( t^{n+2\ga} + \sum_{j=2}^\infty 2^{-2j(M-\frac{n}{2} - (1+\gq))} 2^{-j(n-2\ga)} t^{n+2\ga} \bigg)^{1/2} \\
	&= \nm{g}_{\dot{\gL}_{1 + \gq + \ga}}
\end{align*}
for $M$ sufficiently large.
Therefore, by the same concluding argument as in the previous steps,
\begin{equation*}
	\nm{\gt \mapsto (I + itDB)^{-2} f}_{T^\mb{p}} \lesssim \nm{f}_{\dot{\gL}^{\gq + \ga}}.
\end{equation*}

In the case that $X = Z$, since $\mb{p}$ is infinite, Lemma \ref{lem:TZ-fixedexp-emb} yields $T^\mb{p} \hookrightarrow Z^\mb{p}$, and so by previous estimate we have \begin{equation*}
	\nm{\gt \mapsto (I + itDB)^{-2} f}_{Z^\mb{p}} \lesssim \nm{\gt \mapsto (I + itDB)^{-2} f}_{T^\mb{p}} \lesssim \nm{f}_{\dot{\gL}^{\gq + \ga}}.
\end{equation*}
This completes the proof in the H\"older space case.

\textbf{Step 2d: Completing the proof for Besov spaces.}

Let $\mb{p} = (p,\gq)$.
We use a slightly different argument here.
Fix cutoff functions $\gc_1, \gc \in C_c^\infty(\bbR^n)$ with
	\begin{align*}
		\supp \gc_1 &\subset B(0,4), &\gc_1|_{B(0,2)} \equiv \text{const}, \\
		\supp \gc &\subset A(0,1/2,4) &\gc|_{A(0,1,2)} \equiv \text{const},
	\end{align*}
for all integers $j \geq 2$ define $\gc_j(x) := \gc(2^{-j} x)$, and for all $j \geq 1$ define
\begin{equation*}
	\gh_j(t,x,\gx) := \gc_j\left(\frac{x-\gx}{t}\right) \qquad ((t,x) \in \bbR^{1+n}_+, \gx \in \bbR^n).
\end{equation*}
As before, these functions can be chosen such that $\sum_{j=1}^\infty \gh_j = 1$.
Also define
\begin{equation*}
	\td{g}(t,x,\gx) := g(\gx) - \barint_{B(x,t)}^{} g(\gz) \, d\gz \qquad ((t,x) \in \bbR^{1+n}_+, \gx \in \bbR^n)).
\end{equation*}
By the triangle inequality we have
\begin{align}
	&\nm{ t \mapsto (I + itDB)^{-2} f }_{Z^p_\gq}^p \nonumber \\
	&\lesssim \sum_{j=1}^\infty \dint_{\bbR^{1+n}_+} \left( \bariint_{\gO(t,x)}^{} |\gt^{-\gq} (I + i\gt DB)^{-2}D(\td{g}\gh_j)(t,x,\gx) |^2 \, d\gx \, d\gt \right)^{p/2} \, dx \, \frac{dt}{t}, \label{eqn:besov-terms}
\end{align}
where the operators involving $D$ and $B$ act in the $\gx$ variable.
By using local coercivity (Lemma \ref{lem:local-coercivity}) as in the proof of Lemma \ref{lem:C-D-lem}, the $j$-th term in \eqref{eqn:besov-terms} can be estimated by
\begin{align*}
	&\dint_{\bbR^{1+n}_+} \bigg( \bariint_{\gO(t,x)}^{} |\gt^{-\gq} (I + i\gt DB)^{-2}D(\td{g}\gh_j)(t,x,\gx) |^2 \, d\gx \, d\gt \bigg)^{p/2} \, dx \, \frac{dt}{t} \\
	&\lesssim \dint_{\bbR^{1+n}_+} \bigg( \bariint_{\gO_{c}(t,x)}^{} |\gt^{-\gq-1} (I + i\gt BD)^{-1} (\td{g}\gh_j)(t,x,\gx)|^2 \, d\gx \, d\gt \bigg)^{p/2} \, dx \, \frac{dt}{t} \\
	&+ \dint_{\bbR^{1+n}_+} \bigg( \bariint_{\gO_c(t,x)}^{} |\gt^{-\gq-1}(I + i\gt BD)^{-2} (\td{g}\gh_j)(t,x,\gx)|^2 \, d\gx \, d\gt \bigg)^{p/2} \, dx \, \frac{dt}{t}
\end{align*}
with Whitney parameter $c = (2,2)$.
The two terms in this sum differ only in the power of the resolvent.
The resolvent families $(I + i\gt DB)^{-1}$ and $(I + i\gt DB)^{-2}$ both satisfy off-diagonal estimates of arbitrarily large order $M$ (as off-diagonal estimates may be composed); we will use this to estimate the terms above, making reference only to $(I + i\gt DB)^{-1}$.

For $j=1$ we estimate
\begin{align*}
	&\dint_{\bbR^{1+n}_+} \bigg( \bariint_{\gO_{c}(t,x)}^{} |\gt^{-\gq-1} (I + i\gt BD)^{-1} (\td{g}\gh_1)(t,x,\gx)|^2 \, d\gx \, d\gt \bigg)^{p/2} \, dx \, \frac{dt}{t} \\
	&\lesssim \dint_{\bbR^{1+n}_+} \bigg( t^{-2\gq - 2 - n} \int_{B(x,4t)} |\td{g}(t,x,\gx)|^2 \, d\gx \bigg)^{p/2} \, dx \, \frac{dt}{t} \\
	&\lesssim \dint_{\bbR^{1+n}_+} \bigg( \barint_{B(x,4t)}^{} \barint_{B(x,t)}^{} \bigg| \frac{g(\gx) - g(\gz)}{t^{\gq + 1}} \bigg|^2 \, d\gz \, d\gx \bigg)^{p/2} \, dx \, \frac{dt}{t} \\
	&\leq \dint_{\bbR^{1+n}_+} \barint_{B(x,4t)}^{} \barint_{B(x,t)}^{} \bigg| \frac{g(\gx) - g(\gz)}{t^{\gq + 1}} \bigg|^p \, d\gz \, d\gx \, dx \, \frac{dt}{t} \\
	&= \int_{\bbR^n} \int_{\bbR^n} \int_0^\infty \int_{B(\gx,4t) \cap B(\gz,t)} \, dx \, \frac{1}{t^{2n+p(\gq+1)}} \, \frac{dt}{t} |g(\gx) - g(\gz)|^p \, d\gz \, d\gx \\
	&\leq \int_{\bbR^n} \int_{\bbR^n} \int_{|\gz - \gx|/5}^\infty \frac{1}{t^{n+p(\gq+1)}} \, \frac{dt}{t} \, |g(\gx) - g(\gz)|^p \, d\gz \, d\gx \\
	&\simeq \int_{\bbR^n} \int_{\bbR^n} \frac{|g(\gx) - g(\gz)|^p}{|\gz - \gx|^{n + p(\gq+1)}} \, d\gz \, d\gx \\
	&\simeq \nm{g}_{\dot{B}^{p,p}_{\gq+1}} \\
	&\simeq \nm{f}_{\dot{B}^{p,p}_\gq},
\end{align*}
using that $\gh_1(t,x,\cdot)$ is supported in $B(x,4t)$, that $p/2 > 1$,  that $B(\gx,4t) \cap B(\gz,t)$ is nonempty only if $t > |\gz-\gx|/5$, and the Besov norm characterisation from Theorem \ref{thm:besov-difference}.

For $j \geq 2$ we have, using off-diagonal estimates,
\begin{align*}
	&\dint_{\bbR^{1+n}_+} \bigg( \bariint_{\gO_{c}(t,x)}^{} |\gt^{-\gq-1} (I + i\gt BD)^{-1} (\td{g}\gh_j)(t,x,\gx)|^2 \, d\gx \, d\gt \bigg)^{p/2} \, dx \, \frac{dt}{t} \\
	&\lesssim 2^{-j(Mp-(np)/2)} \\
	&\qquad \cdot \dint_{\bbR^{1+n}_+} \bigg( t^{-2\gq - 2} \barint_{A(x,2^{j-1}t, 2^{j+2}t)}^{} \barint_{B(x,t)}^{} |g(\gx) - g(\gz)|^2 \, d\gz \, d\gx \bigg)^{p/2} \, dx \, \frac{dt}{t} \\
	&\leq 2^{-j(Mp-(np)/2)} \dint_{\bbR^{1+n}_+} \barint_{A(x,2^{j-1}t, 2^{j+2}t)}^{} \barint_{B(x,t)}^{} \bigg| \frac{g(\gx) - g(\gz)}{t^{\gq+1}} \bigg|^p \, d\gz \, d\gx \, dx \, \frac{dt}{t} \\
	&= 2^{-j(Mp - (np)/2 + n)} \\
	&\qquad \cdot \int_{\bbR^n} \int_{\bbR^n} \int_0^\infty \frac{1}{t^{2n + p(\gq+1)}} \int_{B(\gz,t) \cap A(\gx,2^{j-1}t, 2^{j+1}t)} \, dx \, \frac{dt}{t} \, |g(\gx) - g(\gz)|^p \, d\gx \, d\gz \\
	&\lesssim 2^{-j(Mp - (np)/2 + 2n)} \int_{\bbR^n} \int_{\bbR^n} \int_{2^{-j}|\gz - \gx|}^\infty \frac{1}{t^{n + p(\gq+1)}} \, \frac{dt}{t} |g(\gx) - g(\gz)|^p \, d\gx \, d\gz \\
	&\simeq 2^{-j(pM - (np)/2 + n - p(\gq+1))} \nm{f}_{\dot{B}^{p,p}_\gq}
\end{align*}
arguing similarly to before.
For $M$ sufficiently large, we can thus estimate \eqref{eqn:besov-terms} by summing a geometric series, yielding
\begin{equation*}
	\nm{ t \mapsto (I + itDB)^{-2} f }_{Z^p_\gq} \lesssim \nm{ f }_{\dot{B}^{p,p}_\gq}
\end{equation*}
as required.
This completes the proof.
\end{proof}

\section{Technicalities involving extensions and projections}\label{sec:further}

In previous sections we were careful to keep track of which completions (of spaces and operators) we used.
However, in previous work---in particular in \cite{AS16}---such care was not taken, and so there are some ambiguities which we must address.
Here we tie up these loose ends.
This section may be skipped on first and subsequent readings.

Here, for various holomorphic functions $\gf$, we will define operators $\wtd{\gf(BD)}$, $\overline{\gf(BD)}$, and $\gf(\bbP_D BD)$, which are (related to) extensions of $\gf(BD)$ to various completions.
In this section we take an arbitrary completion $\dot{\mc{D}}(D)$ of $\mc{D}(D)$ with respect to the seminorm $\nm{Dh}_2 \simeq \nm{BD h}_2$.
Note that $\bbH_{BD}^{(2,1)} = \mc{D}(D) \cap \overline{\mc{R}(BD)}$, and thus we can pick a completion $\mb{H}_{BD}^{(2,1)}$ of $\bbH_{BD}^{(2,1)}$ within $\dot{\mc{D}}(D)$.
This applies in particular when $B = I$.
By Lemma \ref{lem:proj}, $\bbP_D$ extends to an isomorphism from $\mb{H}_{BD}^{(2,1)}$ into $\mb{H}_D^{(2,1)}$ whose inverse is the extension of $\bbP_{BD} f$.
 
\begin{prop}
Let $\gf \in H^\infty$. 
Then there exists a bounded operator
\begin{equation*}
	\map{\wtd{\gf(BD)}}{\dot {\mc{D}}(D)}{\mb{H}_{BD}^{(2,1)}}
\end{equation*}
such that $D\wtd{\gf(BD)}= {\gf(BD)}D$.
Moreover, for $h\in \mc{D}(D)\cap \overline{\mc{R}(BD)}$, we have $\wtd{\gf(BD)}h= {\gf(BD)}h$ and for $h\in \mc{D}(D) \cap \overline{\mc{R}(D)}$, we have $ \wtd{\gf(BD)}h= {\gf(BD)}\bbP_{BD}h$.
\end{prop} 

\begin{proof}
We argue by density.
Consider $h\in \mc{D}(D)$.
As $\gf(DB)Dh \in \overline{\mc{R}(D)}$, there exists $g\in \mb{H}_{D}^{(2,1)}$ such that $Dg=\gf(DB)Dh$.
Set $\wtd{\gf(BD)}h= \bbP_{BD}g \in \mb{H}_{BD}^{(2,1)}$.
We then have $D\wtd{\gf(BD)}h=D\bbP_{BD}g=Dg= \gf(DB)Dh$.
Taking extensions shows the first claim. 

Next, let $h\in \mc{D}(D)\cap \overline{\mc{R}(BD)}$.
Then
\begin{equation*}
	D\wtd{\gf(BD)}h = \gf(DB)Dh = D \gf(BD)h,
\end{equation*}
where the second equality follows from Corollary \ref{cor:similarity}.
As both $\wtd{\gf(BD)}h$ and $\gf(DB)h$ belong to $\mb{H}_{BD}^{(2,1)}$, we deduce their equality. 
 
Finally, let $h\in \mc{D}(D) \cap \overline{\mc{R}(D)}$.
Then
\begin{equation*}
	D\wtd{\gf(BD)}h = \gf(DB)Dh = \gf(DB)D\bbP_{BD}h = D\gf(BD)\bbP_{BD}h,
\end{equation*}
where the last equality is from the one above applied to $\bbP_{BD}h$.
Again we deduce $\wtd{\gf(BD)}h=\gf(BD)\bbP_{BD}h$. 
\end{proof}
 
In the next proposition we consider an extension of $\gf(BD)$ from $\overline{\mc{R}(BD)}$ to all of $L^2$ given by imposing $\gf(BD)|_{\mc{N}(BD)} = 0$, which is reasonable for certain $\gf$.
 
\begin{prop}\label{prop:secondext} Let $\gf \in H^\infty$ be such that the limit $\lim_{z\to 0, z\in S_{\theta}} \gf(z)$ either vanishes or does not exist.
In either case, set $\gf(0)=0$, and define
$\map{\overline{\gf(BD)}}{L^2}{ \overline{\mc{R}(BD)}}$ by
\begin{equation*}
	\overline{\gf(BD)}h= {\gf(BD)}\bbP_{BD}h.
\end{equation*}
Then $\overline{\gf(BD)}$ is bounded and its restriction to $\mc{D}(D)$ agrees with 
$\wtd{\gf(BD)}$, and thus extends boundedly from $\mb{H}_{D}^{(2,1)}$ into $\mb{H}_{BD}^{(2,1)}$.
\end{prop}
  
\begin{proof}
Well-definedness and boundedness of $\overline{\gf(BD)}$ on $L^2$ follow from the splitting $L^2= \overline{\mc{R}(BD)}\oplus {\mc{N}(BD)}$.
The construction in the previous result shows that $\overline{\gf(BD)}$ agrees with $\wtd{\gf(BD)}$ on $\mc{D}(D)$.
\end{proof}

\begin{rmk}
The assumption on $\gf$ is necessary to avoid functions such as the resolvent $(1+iz)^{-1}$, which naturally extends to 1 by continuity at 0.
In that case we would not have agreement with $\wtd{\gf(BD)}$ on the domain of $D$.
This proposition is adapted to characteristic functions $\gf=\gc^\pm$ and products involving characteristic functions.
\end{rmk}

One can consider the operator $\bbP_{D}BD$ on $\overline{\mc{R}(D)}$ and then establish an adapted $L^2$-Sobolev space theory.
This was done in the work of the second author with McIntosh and Mourgoglou \cite{AMM13} and also pursued in \cite[Section 11]{AS16}.

\begin{prop}
Consider the operator $\bbP_{D}BD$ defined on $\overline{\mc{R}(D)}$ with domain $\mc{D}(D)\cap \overline{\mc{R}(D)}$.
Then for all $\gf\in H^\infty$ we have the similarity of operator extensions
\begin{equation*}
\gf(\bbP_{D}BD)= \bbP_{D}\gf(BD) \bbP_{BD}
\end{equation*}
on $\mb{H}_D^{(2,1)}$.
\end{prop}

\begin{proof} 
First note that $\bbP_D BD$ coincides with $\bbP_D BD \bbP_{BD}$.
Thus $\bbP_D BD$ is similar to $BD$, restricted to $\mc{D}(D) \cap \overline{\mc{R}(BD)}$.
Lemma \ref{lem:proj} yields the equality $\bbP_{D}BD=\bbP_{D}BD\bbP_{BD}$ on $\mc{D}(D)\cap \overline{\mc{R}(D)}$, as $\bbP_{BD}$ restricted to $\mc{D}(D)\cap \overline{\mc{R}(D)}$ is the inverse of $\bbP_{D}$ restricted to $\mc{D}(D)\cap \overline{\mc{R}(BD)}$.
The extension in the statement follows.
\end{proof}

We remark that as $\bbP_D BD$ does not satisfy the Standard Assumptions (because of the presence of $\bbP_D$) there is no available Besov--Hardy--Sobolev theory and only the adapted Sobolev spaces $\mc{H}^s_{\bbP_D BD}$ of \cite[\textsection 11]{AS16} can be developed.
These are similar to the spaces $\mb{H}_{BD}^{(2,s)}$ when $0 \leq s \leq 1$.
Thus the operator $\bbP_D BD$ is of limited interest in our context.

The conclusion of this section is that the operators $\gf(BD)$ of \cite[\textsection 12.3]{AS16} should be replaced by $\gf(BD) \bbP_{BD}$ or even by $\bbP_D \gf(BD) \bbP_{BD}$ to be fully defined on various extensions.
But as long as they are applied to $L^2$ functions, this ambiguity is resolved by using $\overline{\gf(BD)}$ instead.
All results in \cite[\textsection 12.3]{AS16} are nevertheless correct.

\chapter[Classification of Solutions to CR systems and Elliptic Equations]{Classification of Solutions to Cauchy--Riemann Systems and Elliptic Equations}\label{chap:BVPCR}

In this chapter we implicitly work with a fixed $m \in \bbN$, meaning that we consider the equation $L_A u = 0$ where $u$ is a $\bbC^m$-valued function.
All of our arguments are independent of $m$.
Most of the function spaces in this chapter are spaces of $\bbC^{m(1+n)}$-valued functions, but we do not reference this in the notation (as the notation is complicated enough already).
Therefore we write $L^2(\bbR^n)$ in place of $L^2(\bbR^n : \bbC^{m(1+n)})$, and so on.
As in the previous chapter, we fix the Dirac operator $D$ and multipliers $B$ from Subsection \ref{ssec:fop}.

\section{Basic properties of solutions}\index{elliptic equation!basic properties of solutions}

We use the following properties of solutions to $L_A u = 0$.

\begin{prop}\label{prop:wksolnprops}
	Suppose that $u$ solves $L_A u = 0$.
	Then the following are true.
	\begin{enumerate}[(1)]
		\item\label{itm:td} The transversal derivative $\partial_t u$ solves $L_A(\partial_t u) = 0$.
		
		\item\label{itm:slicelem} The function $t \mapsto \nabla_A u(t,\cdot)$ is in $C^\infty(\bbR_+ : L_\text{loc}^2(\bbR^n))$, and for all Whitney parameters $c = (c_0,c_1)$ and $t \in \bbR_+$ we have
		\begin{equation*}
			\barint_{B(x,c_0 t)}^{} |\nabla_A u(t,x)|^2 \, dx \lesssim \bariint_{\gO_c(t,x)}^{} |\nabla_A u(s,y)|^2 \, ds \, dy.
		\end{equation*}
		
		\item\label{itm:derivests} For all exponents $\mb{p}$, all $k \in \bbN$, and all $C \geq 1$ we have
		\begin{align*}
			\sup_{\substack{t,t^\prime \in \bbR_+ \\ C^{-1} \leq t/t^\prime \leq C}} \nm{\partial_t^k \nabla_A u(t,\cdot)}_{E^{\mb{p}-k}(t^\prime)} &\lesssim_C \nm{\partial_t^k \nabla_A u}_{X^{\mb{p} - k}} \\
			&= \nm{ \nabla_A \partial_t^k u}_{X^{\mb{p} - k}} \\
			&\lesssim \nm{\nabla_A u}_{X^{\mb{p}}}.
		\end{align*}
		In particular, if $\nabla_A u$ is in $X^\mb{p}$, then the function $t \mapsto \nabla_A u(t,\cdot)$ is in $C^\infty(\bbR_+ : E^\mb{p})$
	\end{enumerate}
\end{prop}

\begin{proof}
	Property \eqref{itm:td} follows from $t$-independence of the coefficients.
	The remaining statements are consequences of the Caccioppoli inequality, and are proven in \cite[\textsection 5]{AM15} for tent spaces.
	The corresponding $Z$-space statements are proven in the same way.
\end{proof}

\begin{rmk}\label{rmk:Fell}
	By Theorem \ref{thm:AAM}, if $F$ is a solution to a Cauchy--Riemann system $(\CR)_{DB}$, then parts (2) and (3) of Proposition \ref{prop:wksolnprops} hold with $\nabla_A u$ replaced by $F$.
	
	Furthermore, suppose that $G$ solves the anti-Cauchy--Riemann system
\begin{equation*}\label{eqn:antiCR}\index{Cauchy--Riemann system!anti-}
	(a\CR)_{DB} : \left\{ \begin{array}{rl} \partial_t G - DB G = 0 & \text{in $\bbR^{1+n}_+$,} \\ \curl_\parallel G_\parallel = 0 & \text{in $\bbR^{1+n}_+$}\end{array} \right.
\end{equation*}
defined analogously to $(\CR)_{DB}$ but with a sign change.
Then the reflection $F(t) := G(-t)$ solves $(\CR)_{DB}$ on the lower half-space $\bbR^{1+n}_-$.
	By using $X$-spaces associated with the lower half-space rather than the upper half-space, parts (2) and (3) of Proposition \ref{prop:wksolnprops} hold with $\nabla_A u$ replaced by $F$ and with $\bbR_+$ replaced by $\bbR_-$.
	A simple reflection argument then shows that parts (2) and (3) of Proposition \ref{prop:wksolnprops} hold for $G$.
\end{rmk}

As a consequence we obtain the following technical lemma, which is analogous to \cite[Lemma 10.2]{AM15}.

\begin{lem}\label{lem:sgp-in-slice}
	Fix $\mb{p}$ with $i(\mb{p}) < 2$ and $\gq(\mb{p}) < 0$, suppose $M \in \bbN$, and let $f \in {\bbX}_{DB}^{\mb{p}}$.
	Then for all $t > 0$ we have that $(tDB)^M e^{-t[DB]}\gc^{\pm}(DB)f \in E^{\mb{p}}$, with
	\begin{equation*}
		\sup_{t > 0} \nm{(tDB)^M e^{-t[DB]}\gc^{\pm}(DB)f}_{E^\mb{p}(t)} \lesssim \nm{f}_{{\bbX}_{DB}^\mb{p}}.
	\end{equation*}
\end{lem}

\begin{proof}
	We estimate
	\begin{align*}
		\sup_{t > 0} \nm{(tDB)^M e^{-t[DB]} \gc^{\pm}(DB)f}_{E^\mb{p}(t)} &\lesssim \nm{t \mapsto (tDB)^M e^{-t[DB]} \gc^{\pm}(DB)f}_{X^\mb{p}} \\
		&\simeq \nm{f}_{{\bbX}_{DB}^\mb{p}}.
	\end{align*}
	The first line comes from Proposition \ref{prop:wksolnprops}, using that $(DB)^M e^{-t[DB]}\gc^{\pm}(DB)f$ solves either $(\CR)_{DB}$ or $(a\CR)_{DB}$.
	The second line is due to the fact that $[z \mapsto z^M e^{-[z]}] \in \gY(\bbX_{DB}^\mb{p})$ when $i(\mb{p}) < 2$ and $\gq(\mb{p}) < 0$.
\end{proof}

\section{Decay of solutions at infinity}\label{sec:dos}\index{elliptic equation!decay of solutions}

In the boundary value problems introduced in Subsection \ref{ssec:fbvp}, we impose the decay condition
\begin{equation}\label{eqn:decaygradu}
	\lim_{t \to \infty} \nabla_\parallel u(t,\cdot) = 0 \qquad \text{in $\mc{Z}^\prime(\bbR^n : \bbC^{mn})$}
\end{equation}
for a solution $u$ to $L_A u$ with $\nabla u$ in $X^\mb{p}$.
This is equivalent to 
\begin{equation}\label{eqn:decayu}
	\lim_{t \to \infty} \langle u(t,\cdot), \gf\rangle = 0 \qquad \text{for all  $\gf \in \mc{S}(\bbR^n : \bbC^{m})$ with $\int x^\alpha \gf(x)\, dx=0$ when $\alpha\ne 0$}
\end{equation}
because the divergence operator is surjective from the set of test functions in \eqref{eqn:decayu} onto $\mc{Z}(\bbR^n : \bbC^{mn})$.
This is a convenient generalisation of the usual pointwise vanishing limit at $\infty$, since solutions to the elliptic equations that we consider may not have pointwise values.
In this section we show that this decay condition is redundant for certain exponents $\mb{p}$ (quantified in terms of $A$).
In fact, our results give decay not only in $\mc{Z}^\prime$, but also in the slice space $E^\infty$ (Lemma \ref{lem:decaylem}) or in $L^2$ (in Lemma \ref{lem:solution-decay}).
Demanding decay in $\mc{Z}^\prime$ is really just an artefact of having identified the classical smoothness spaces $\mb{X}^\mb{p}$ as subspaces of $\mc{Z}^\prime$.

By the `hole filling technique' attributed to Widman (see \cite[page 23]{mG83}) there exists a number $\gl(A) \in (0,n+1)$ such that for all $\gl \in [0,\gl(A))$, for all $(t_0,x_0) \in \bbR^{n+1}_+$ and $0<r<R<\infty$, and for all weak solutions $u$ to $L_A u = 0$, we have
\begin{equation}\label{eqn:DG}
	\dint_{B((t_0,x_0),r)} |\nabla u(t,x)|^2 \, dx \, dt  \lesssim_\gl \left( \frac{r}{R} \right)^\gl \dint_{B((t_0,x_0),R)} |\nabla u(t,x)|^2 \, dx \, dt,
\end{equation}
where $B((t_0,x_0),r)$ and $B((t_0,x_0),R)$ denote open balls in $\bbR^{1+n}$, and where $R$ is such that $B((t_0,x_0),R) \subset \bbR^{1+n}_+$.
These balls can be taken with respect to any norm on $\bbR^{1+n}$, keeping in mind that the implicit constant in \eqref{eqn:DG} will depend on the chosen norm.
By ellipticity we may replace the gradient $\nabla$ with the conormal gradient $\nabla_A$ in \eqref{eqn:DG}.

\begin{lem}\label{lem:decaylem}
	Suppose that the exponent $\mb{p}$ lies in the shaded region pictured in Figure \ref{fig:decaylem}, which depends on $\gl(A)$.
	Let $u$ be a solution to $L_A u = 0$ on $\bbR^{1+n}_+$ such that $\nabla_A u \in X^\mb{p}$.
	Then $\lim_{t \to \infty} \nabla_A u(t,\cdot) = 0$ in $E^\infty$ (and therefore also in $\mc{Z}^\prime$).
\end{lem}

\begin{rmk}
	The shaded region in Figure \ref{fig:decaylem} is the open half-plane determined by the equation $j(\mb{p}) > \frac{\gq(\mb{p})}{n} - \frac{n+1-\gl(A)}{2n}$.
	Note that $\frac{n+1-\gl(A)}{2n} \geq \frac{1}{2}$ if and only if $\gl(A) \leq 1$; we have pictured the case $\lambda(A) > 1$.
	In Lemma \ref{lem:solution-decay} we handle exponents $\mb{p}$ with $i(\mb{p}) \leq 2$ and $\gq(\mb{p}) < 0$ independently of $\gl(A)$, so this distinction is not important.
\end{rmk}

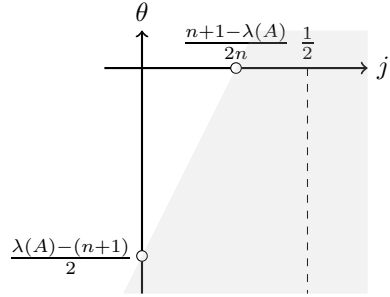
\begin{figure}
\caption{The exponent region in Lemma \ref{lem:decaylem}.}\label{fig:decaylem}
\begin{center}
\begin{tikzpicture}[scale=1]
	%\draw [help lines] (-1,-2) grid (4,2);

	%\draw [thin] (0,2) -- (4.5,2); %theta=0 axis
	%\draw [thin] (0,0) -- (4.5,0); %thetya=-1 axis
	
	%axes
	\draw [thick,->] (-0.5,1) -- (3,1); %j axis
	\draw [thick,->] (0,-2) -- (0,1.5); %theta axis
	
	%guidelines
	\draw [thin,dashed] (2.2,1) -- (2.2,-2); 
	%\draw [thin,dotted] (2,1.5) -- (2,-2); %j=1 guideline
	%\draw [thin,dotted] (-1.25,-2) -- (3,-2); %theta = -1 line
	%\draw [thin,dotted] (0,1) -- (3,1); %theta = 1/2 guideline
	%\draw [thin,dotted] (-0.666,-2) -- (0,0); %inclusion slope
	%\draw [thin,dotted] (-0.666,-2) -- (-0.666,0); %j = -1/n line
	
	%axis labels
	\node [above] at (0,1.5) {$\theta$};
	\node [left] at (0,-1.5) {$\frac{\gl(A) - (n+1)}{2}$};
	\draw [fill=white] (0,-1.5) circle [radius = 2pt];
	\node [above] at (2.2,1) {$\frac12$};
	\node [above] at (1.25,1) {$\frac{n+1-\gl(A)}{2n}$};
	\draw [fill=white] (1.25,1) circle [radius = 2pt];
	
	%\node [above] at (-0.666,0) {$-\frac{1}{n}$} ;
	%\node [below] at (1.1,0) {$\frac{1}{2}$};
	%\node [below] at (2.1,0) {$1$};
	%\node [below] at (2.66,0) {$\frac{n+1}{n}$};
	\node [right] at (3,1) {$j$};
	
	%filled regions
	\path [fill=lightgray, opacity = 0.2] (-0.25,-2)--(1.5,1.5)--(3,1.5)--(3,-2)--(-0.25,-2);
	
\end{tikzpicture}
\end{center}
\end{figure}

\begin{proof}[Proof of Lemma \ref{lem:decaylem}]
	The region pictured in Figure \ref{fig:decaylem} is precisely the set of exponents $\mb{p}$ such that there exists an infinite exponent $\mb{q}$ with $\mb{p} \hookrightarrow \mb{q}$ and $r(\mb{q}) < \frac{\gl(A) - (n+1)}{2}$.
	Fix such a $\mb{q}$.
	For all $\gl < \gl(A)$ we can estimate
	\begin{align}
		\nm{ \nabla_A u(t,\cdot) }_{E^\infty_{0}(1)}
		&\simeq \sup_{x \in \bbR^n} \bigg( \int_{B(x,1)} |\nabla_A u(t,y)|^2 \, dy \bigg)^{1/2} \nonumber \\
		&\lesssim \sup_{x \in \bbR^n} \bigg( \int_{t-\frac{1}{2}}^{t+\frac{1}{2}} \int_{B(x,1)} |\nabla_A u(s,y)|^2 \, dy \, ds \bigg)^{1/2} \label{line:propsuse2} \\
		&\lesssim \sup_{x \in \bbR^n} \bigg( t^{-\gl} \int_{t - \frac{t}{2}}^{t + \frac{t}{2}} \int_{B(x,t)} |\nabla_A u(s,y)|^2 \, dy \, ds \bigg)^{1/2} \,  \label{line:DGuse} \\
		&\lesssim t^{\frac{(n+1)-\gl}{2} + r(\mb{q})} \bigg( \barint_{t - \frac{t}{2}}^{t+\frac{t}{2}} \nm{ \nabla_A u(s,\cdot) }_{E^\mb{q}(s)}^2 \, ds \bigg)^{1/2} \nonumber \\
		&\lesssim t^{\frac{(n+1)-\gl}{2} + r(\mb{q})} \nm{ \nabla_A u }_{X^\mb{p}}, \nonumber
	\end{align}
	where \eqref{line:propsuse2} follows from Proposition \ref{prop:wksolnprops}, \eqref{line:DGuse} follows from \eqref{eqn:DG},\footnote{Here we use the balls $B((t,x),r) := (t - r/2, t + r/2) \times B(x,r)$.} and the last line follows from the embedding $E^\mb{p}(s) \hookrightarrow E^\mb{q}(s)$ and another application of Proposition \ref{prop:wksolnprops}.
	For $\gl$ sufficiently close to $\gl(A)$ we have $(n+1-\gl)/2 + r(\mb{q}) < 0$, so we find that $\lim_{t \to \infty} \nabla_A u(t,\cdot) = 0$ in $E^\infty$.
\end{proof}

\begin{rmk}
	It is known that $\gl(A) > n-1$ if and only if $A$ satisfies the De Giorgi--Nash--Moser condition \eqref{eqn:DGNM} of all exponents less than $\ga = (\gl(A) - (n-1))/2$.\index{De Giorgi--Nash--Moser condition}
	In this case we have $\frac{\gl(A) - (n+1)}{2} = \ga - 1$ and $\frac{n+1-\gl(A)}{2n} = \frac{1-\ga}{n}$, and Lemma \ref{lem:decaylem} then holds for the shaded region pictured in Figure \ref{fig:decaylem2}.
	Evidently this region increases as the {De Giorgi}--Nash--Moser exponent $\ga$ increases.
        
        As a concrete example, consider $A$ satisfying the De Giorgi--Nash--Moser condition of some exponent $\alpha$ close to $1$, and let $p \in (1,\infty)$ and $\theta \in (-1,0)$ with $\theta$ close to $0$.
        Then whenever $u$ solves the elliptic equation $L_A u = 0$ on $\bbR^n$ and $\nabla_A u \in T^{(p,\theta)}$, we have that $\lim_{t \to \infty} \nabla_A u(t,\cdot) = 0$ in $E^\infty$ provided that $p$ is not too large.
        However, for $p$ sufficiently large (depending on $\theta$ and $\alpha$), this decay is not guaranteed. 
\end{rmk}

\begin{figure}
\caption{The region in Lemma \ref{lem:decaylem} in the case that $A$ satisfies the De Giorgi--Nash--Moser condition with exponent $\ga$.}\label{fig:decaylem2}
\begin{center}
\begin{tikzpicture}[scale=1]
	%\draw [help lines] (-1,-2) grid (4,2);

	%\draw [thin] (0,2) -- (4.5,2); %theta=0 axis
	%\draw [thin] (0,0) -- (4.5,0); %thetya=-1 axis
	
	%axes
	\draw [thick,->] (-0.5,1) -- (3,1); %j axis
	\draw [thick,->] (0,-2) -- (0,1.5); %theta axis
	
	%guidelines
	\draw [thin,dashed] (2.2,1) -- (2.2,-2); 
	\draw [thin,dashed] (-0.5,-2) -- (3,-2); %j=1 guideline
	%\draw [thin,dotted] (-1.25,-2) -- (3,-2); %theta = -1 line
	%\draw [thin,dotted] (0,1) -- (3,1); %theta = 1/2 guideline
	%\draw [thin,dotted] (-0.666,-2) -- (0,0); %inclusion slope
	%\draw [thin,dotted] (-0.666,-2) -- (-0.666,0); %j = -1/n line
	
	%axis labels
	\node [above] at (0,1.5) {$\theta$};
	\node [left] at (0,-1.5) {$\ga - 1$};
	\draw [fill=white] (0,-1.5) circle [radius = 2pt];
	\node [above] at (2.2,1) {$\frac12$};
	\node [above] at (1.25,1) {$\frac{1-\ga}{n}$};
	\draw [fill=white] (1.25,1) circle [radius = 2pt];
	
	\node [left] at (-0.5,-2) {$-1$} ;
	%\node [below] at (1.1,0) {$\frac{1}{2}$};
	%\node [below] at (2.1,0) {$1$};
	%\node [below] at (2.66,0) {$\frac{n+1}{n}$};
	\node [right] at (3,1) {$j$};
	
	%filled regions
	\path [fill=lightgray, opacity = 0.2] (-0.25,-2)--(1.5,1.5)--(3,1.5)--(3,-2)--(-0.25,-2);
	
\end{tikzpicture}
\end{center}
\end{figure}
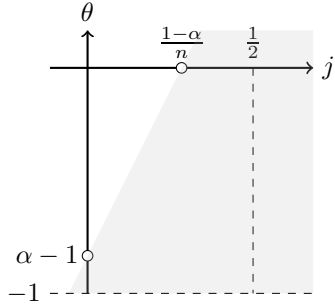

An easier argument can be used to deduce decay in $L^2$ for exponents $\mb{p}$ with $i(\mb{p}) \leq 2$ and $\gq(\mb{p}) < 0$.

\begin{lem}\label{lem:solution-decay} 
	Let $\mb{p} = (p,s)$ with $i(\mb{p}) \leq (0,2]$ and $\gq(\mb{p}) < 0$, and suppose $F \in X^\mb{p}$ solves $(\CR)_{DB}$ or $(a\CR)_{DB}$.
	Then $\lim_{t \to 0} F(t) = 0$ in $L^2$. 
\end{lem}

\begin{proof}
	By Proposition \ref{prop:wksolnprops}, for all $t \in \bbR_+$ we have $\nm{F(t)}_{E^\mb{p}(t)} \lesssim \nm{F}_{X^\mb{p}}$, 
	and so
	\begin{equation*}
		\nm{F(t)}_{2} \lesssim \nm{F(t)}_{E_0^{i(\mb{p})}(t)} 
		= t^{\gq(\mb{p})} \nm{F(t)}_{E^\mb{p}(t)} \lesssim t^{\gq(\mb{p})} \nm{F}_{X^{\mb{p}}}
	\end{equation*}
	using the embedding $E^{i(\mb{p})}_0(t) \hookrightarrow E^2_0(t) = L^2$ (the equality $E_0^2(t) = L^2$ is a consequence of Fubini's theorem).
	Since $\gq(\mb{p}) < 0$, we have
	\begin{equation*}
		\lim_{t \to \infty} F(t) = 0
	\end{equation*}
	in $L^2$.
\end{proof}

\section{Classification and representation theorems}\label{sec:class}

The following classification theorems (Theorems \ref{thm:mainthm-leq2} and \ref{thm:mainthm-gtr2}) for solutions to $(\CR)_{DB}$ (as formulated in Subsection \ref{ssec:fop}) will be proven in the next section.
For the convenience of the reader we state them both here.
We remind the reader that more `concrete' versions of these theorems are stated in Subsection \ref{ssec:csCR} of the introduction, and that a table detailing the possible interpretations of $(X^\mb{p}, \mb{X}^\mb{p})$ is in Section \ref{sec:exponent-space-table}.

\begin{thm}[Classification of solutions to $(\CR)_{DB}$, $i(\mb{p}) \leq 2$]\label{thm:mainthm-leq2}\index{Cauchy--Riemann system!classification of solutions}
	Let $\mb{p}$ be an exponent with $i(\mb{p}) \leq 2$ and $\gq(\mb{p}) < 0$, and fix a completion $\mb{X}$ of ${\bbX}_{DB}^\mb{p}$.
	\begin{enumerate}[(i)]
	\item
	For all $F_0 \in \mb{X}^{+}$, $\mb{C}_{DB}^+ F_0$ solves $(\CR)_{DB}$, and $\nm{ \mb{C}_{DB}^+ F_0 }_{X^\mb{p}} \lesssim \nm{ F_0 }_{\mb{X}}$.
	\item
	Conversely, if $F \in X^\mb{p}$ solves $(\CR)_{DB}$, then there exists a unique $F_0 \in \mb{X}^+$ such that $F = \mb{C}_{DB}^+ F_0$. Furthermore, $\nm{ F_0 }_{\mb{X}} \lesssim \nm{ F }_{X^\mb{p}}$.
	\end{enumerate}
\end{thm}

The boundary limit $F_0$ is necessarily in a completion of $\bbX_{DB}^{\mb{p},+}$.
When we identify such a completion with a classical function space (as for example when $\mb{p} \in I(\mb{X},DB)$), the boundary limit exists in a classical sense.

When $i(\mb{p}) > 2$ the argument is much more complicated.
In this case we must restrict attention to exponents $\mb{p}$ with $\gq(\mb{p}) \in (-1,0)$, and such that the adapted space ${\bbX}_{DB^*}^{\mb{p}^\heartsuit}$ may be identified with the classical space ${\bbX}_{D}^{\mb{p}^\heartsuit}$, plus we need an additional decay condition on $F$.
Recall that for such $\mb{p}$ we have identified $\mb{X}_{DB}^{\mb{p},+}$ as a subspace of $\mb{X}_D^\mb{p}$, and that  if $\mb{p}$ is finite, then $\mb{p}^\heartsuit \in I(\mb{X},DB^*)$ if and only if $\mb{p} \in I(\mb{X},DB)$ (Proposition \ref{prop:dual-interval-incl}).

\begin{thm}[Classification of solutions to $(\CR)_{DB}$, $i(\mb{p}) > 2$]\label{thm:mainthm-gtr2}
	Let $\mb{p}$ be an exponent with $i(\mb{p}) > 2$, $\gq(\mb{p}) \in (-1,0)$, and $\mb{p}^\heartsuit \in I(\mb{X},DB^*)$.
	Furthermore, if $\mb{p}$ is infinite, suppose that $r(\mb{p}) < 0$.
	
	\begin{enumerate}[(i)]
	\item
	If $F_0 \in \mb{X}_{DB}^{\mb{p},+}$, then $\mb{C}_{DB}^+ F_0$ solves $(\CR)_{DB}$, $\lim_{t \to \infty} \mb{C}_{DB}^+ F_0(t)_\parallel = 0$ in $\mc{Z}^\prime(\bbR^n : \bbC^{nm})$, and $\nm{ \mb{C}_{DB}^+ F_0 }_{X^\mb{p}} \lesssim \nm{ F_0 }_{\mb{X}_{DB}^{\mb{p}}} \simeq \nm{ F_0 }_{\mb{X}_{D}^{\mb{p}}}$.
	\item
	Conversely, if $F \in X^\mb{p}$ solves $(\CR)_{DB}$ and $\lim_{t \to \infty} F(t)_\parallel = 0$ in $\mc{Z}^\prime(\bbR^n : \bbC^{nm})$, then there exists a unique $F_0 \in {\mb{X}}_{DB}^{\mb{p},+}$ such that $F = \mb{C}_{DB}^+ F_0$.
	Furthermore, $\nm{ F_0 }_{\mb{X}_{D}^{\mb{p}}} \simeq \nm{ F_0 }_{\mb{X}_{DB}^\mb{p}} \lesssim \nm{ F }_{X^\mb{p}}$.

	\end{enumerate}
\end{thm}  

Note that the conditions on $\mb{p}$ in Theorem \ref{thm:mainthm-gtr2} imply that $\mb{p} \in I_{\text{max}}$.
The conditions in Theorems \ref{thm:mainthm-leq2} and \ref{thm:mainthm-gtr2}, along with those in \cite[Theorems 1.1 and 1.3]{AM15}, suggest the definition of the following \emph{classification region}, which we will use in stating our results in the following sections.

\begin{dfn}\label{dfn:J} We define the \emph{classification region} for $DB$ as\index{region!classification}
	\begin{align*}
		J(\mb{X},DB) := \{\mb{p} \in I_{\text{max}} : &\text{[$i(\mb{p}) \leq 2$ and $\mb{p} \in I(\mb{X},DB)$]} \\
		&\text{or [$i(\mb{p}) > 2$ and $\mb{p}^\heartsuit \in I(\mb{X},DB^*)$]}\}.
	\end{align*}
\end{dfn}

It is not strictly necessary to impose $\mb{p} \in I_{\text{max}}$ in this definition, but it is technically convenient and in our results we lose nothing from doing so.
Unlike the identification region $I(\mb{X},DB)$, we do not restrict to finite exponents in this definition.
Note however that
	\begin{equation*}
		\{ \mb{p} \in J(\mb{X},DB) \cap \mb{E}_\text{fin} : \gq(\mb{p}) \in [-1,0]\} = I(\mb{X},DB)
	\end{equation*}
follows from Proposition \ref{prop:dual-interval-incl}.
	
For the moment let us assume Theorems \ref{thm:mainthm-leq2} and \ref{thm:mainthm-gtr2}.
Although we have proven more in the case $i(\mb{p}) \leq 2$, the following result is a good summary of what we can deduce for solutions of the elliptic equation.  

\begin{thm}[First representation theorem]\label{thm:firstrepresentation}\index{elliptic equation!classification of solutions}  Let $\mb{p}$ be an exponent in the classification region  $J(\mb{X},DB)$ with $B=\hat A$.  Let $u$ be a solution to $L_{A}u=0$ in $\bbR_+^{1+n}$ with $\nabla_{A}u \in X^\mb{p}$ and $\lim_{t \to \infty}\nabla_{\parallel} u(t) = 0$ in $\mc{Z}^\prime(\bbR^n : \bbC^{mn})$. Then there exists a unique $F_{0}\in  \mb{X}_{DB}^{\mb{p},+}$ such that $\nabla_{A}u= C_{DB}^+(F_0)$ and $F_0 = \lim_{t \to 0} \nabla_A u(t,\cdot)$ in $\mb{X}^\mb{p}$.  \end{thm}

We often refer to the element $F_0$ obtained in this theorem as $\nabla_A u|_{t = 0}$.

\begin{proof} The results for $\gq(\mb{p}) = 0$ and $\gq(\mb{p}) = -1$ correspond to \cite[Theorems 1.1 and 1.3]{AM15}, so we need only consider $\gq(\mb{p}) \in (-1,0)$.
By Theorem \ref{thm:AAM}, solutions $u$ to $L_A u = 0$ are in bijective correspondence (up to an additive constant) to solutions $F$ to $(\CR)_{DB}$, with $F = \nabla_A u$.
	By Theorems \ref{thm:mainthm-leq2} and \ref{thm:mainthm-gtr2} and by our assumptions, $F = \mb{C}_{DB}^+ F_0$ for a unique $F_0 \in \mb{X}_{DB}^{\mb{p},+}$, %(and so $F(t) \in \mb{X}_{DB}^{\mb{p},+}$ for all $t$ with limit 0 at infinity in this space), every such $F_0$ determines a solution $F$,  %%this isn't needed
	and by Proposition \ref{prop:sgp-continuity},
	\begin{equation*}
		F_0 = \lim_{t \to 0} \nabla_A u(t,\cdot)
	\end{equation*}
	with limit in $\mb{X}^\mb{p}$. 
 \end{proof}

\begin{rmk} We stress again that that since the exponent $\mb{p}$ is in the classification region, the space of elements $F_0$ obtained in the previous theorem is a subspace of $\mb{X}^{\mb{p}}_{D}$, which is a space of distributions modulo polynomials.
We obtain \textit{a posteriori}  that $\lim_{t \to \infty} \nabla_{A}u(t) = 0$ in $\mb{X}^{\mb{p}}$, which encodes behaviour at infinity of the conormal derivative.
Note also that if $\mb{p}$ is in the region given by Lemma \ref{lem:decaylem} then the decay condition on $\nabla_{\parallel}u$  is redundant.
In particular the condition may be eliminated from the requirements when $i(\mb{p}) \leq 2$, which is in agreement with the statement of Theorem \ref{thm:mainthm-leq2}.
\end{rmk}

Through Theorem \ref{thm:firstrepresentation} we may obtain a representation for solutions themselves, rather than for their conormal gradients.

\begin{thm}[Second representation theorem]\label{thm:secondrepresentation}\index{elliptic equation!classification of solutions} Let $\mb{p}$ be an exponent in $J(\mb{X}, DB)$ with $B=\hat A$. 
Let $u$ be a solution to $L_{A}u=0$ in $\bbR_+^{1+n}$ with $\nabla_{A}u \in X^\mb{p}$ and $\lim_{t \to \infty} \nabla_{\parallel} u(t) = 0$ in $\mc{Z}^\prime(\bbR^n : \bbC^{mn})$.
Then there exists a unique $\wtd F_{0}\in  \bbP_{D}\mb{X}_{BD}^{\mb{p}+1,+}$ such that 
\begin{equation}\label{eqn:2rep}
	u= -(\bbP_{D}C_{BD}^+ \bbP_{BD}\wtd F_{0})_{\perp}
\end{equation}
up to an additive constant.
Moreover, $D\wtd F_{0}=\nabla_{A}u|_{t=0}$ and $(\wtd F_{0})_{\perp}= -u|_{t=0}$ up to an additive constant in $(\mb{X}_{D}^{\mb{p}+1})_{\perp}$, where $\nabla_A u|_{t=0}$ is the trace obtained in the first representation theorem.
All limits are taken in the weak-star topology when $\mb{p}$ is infinite.
\end{thm}

\begin{proof}
By Corollary \ref{cor:similarity2}, there exist unique $\wtd F_{0} \in \mb{X}_{D}^{\mb{p}+1}$ and $F_{0}^\sharp \in \mb{X}_{BD}^{\mb{p}+1}$ such that 
$D\wtd F_{0}=  DF_{0}^\sharp = \nabla_A u|_{t = 0}$.
Moreover, $\wtd F_{0}= \bbP_{D}F_{0}^\sharp$ and $F_{0}^\sharp= \bbP_{BD}\wtd F_{0} \in \mb{X}_{BD}^{\mb{p}+1,+}$ by Corollary \ref{cor:similarity}.
We will show \eqref{eqn:2rep} for this choice of $\wtd{F_0}$.

For simplicity assume  $\nabla_A u|_{t = 0} \in \bbX^{p,+}_{DB}\cap{\mc{R}(D)}$ (a dense class) so that $\wtd F_{0}$ and $F_{0}^\sharp$ belong to $\mc{D}(D)$, and likewise $(C_{BD}^+\bbP_{BD}\wtd F_{0})(t) \in \mc{D}(D)$ for all $t>0$ (see Proposition \ref{prop:commdiag}). 
Set $v=-(\bbP_D C_{BD}^+\bbP_{BD}\wtd F_{0})_{\perp}$. By \eqref{eqn:simcauchy2} we have
\begin{equation*}
\partial_{t}v= (BD\bbP_D C_{BD}^+\bbP_{BD}\wtd F_{0})_{\perp}
= (BD C_{BD}^+\bbP_{BD}\wtd F_{0})_{\perp}
= (BC_{DB}^+ (\nabla_A u|_{t = 0}))_{\perp}
\end{equation*}
and 
\begin{equation*}
\nabla_{x}v= (D \bbP_D C_{BD}^+\bbP_{BD}\wtd F_{0})_{\parallel}
=(D C_{BD}^+\bbP_{BD}\wtd F_{0})_{\parallel}
=(C_{DB}^+ (\nabla_A u|_{t=0}))_{\parallel}.
\end{equation*}
By the relation $B=\hat A$, these two equalities are equivalent to $\nabla_{A}v=C_{DB}^+ (\nabla_A u|_{t=0})$, hence $\nabla_{A}v=\nabla_{A}u$ on $\bbR^{1+n}_{+}$ and $v=u$ up to an additive constant. 
The general case follows by a limiting argument.
\end{proof}

\begin{rmk}
We obtain from the proof that $\lim_{t \to \infty} u(t) = 0$ in $(\mb{X}_{D}^{\mb{p}+1})_{\perp}$. 
As $(\mb{X}_D^{\mb{p}+1})_\perp$ can be embedded in the space of distributions modulo constants, this is stronger than decay in $\mc{Z}^\prime$. 
\end{rmk}

We can consider $-C_{BD}^+\bbP_{BD}\wtd F_{0}$ as a conjugate system to $u$, because it solves the equation $\partial_{t}F+BDF=0$.
Its transversal component, morally speaking, should be $u$ (up to a constant); this point of view was taken by the second author and Stahlhut \cite{AS16}.
However this is a purely abstract object (unless $\wtd F_0 \in L^2$), and as such does not have a well-defined transversal component.
Nevertheless things can be made concrete: $-\bbP_{D}C_{BD}^+\bbP_{BD}\wtd F_{0}(t)$ is an element of $\mb{X}_D^{\mb{p}+1}$ (hence a tempered distribution modulo polynomials) for all $t > 0$, its tranversal component is exactly $u(t)$  (up to an additive constant), and its tangential component can be used as a conjugate vector.
This was done by Hofmann, Kenig, Mayboroda, and Pipher \cite{HKMP15.2} in a different manner and in a different formulation. 

\section{Proofs of classification theorems}

Now we prove Theorems \ref{thm:mainthm-leq2} and \ref{thm:mainthm-gtr2}.
We deal with both theorems simultaneously.

\subsection{Construction of solutions via Cauchy extension}\label{ssec:solnconstruction}

We begin with the proof of part (i) of both theorems.

Let $F_0 \in \mb{X}_{DB}^{\mb{p},+}$.
Then the estimate $\nm{ \mb{C}_{DB}^+ F_0 }_{X^\mb{p}} \lesssim \nm{ F_0 }_{\mb{X}_{D}^\mb{p}}$ follows from either Theorem \ref{thm:sgpnorm-abstract} or Theorem \ref{thm:sgpnorm-concrete}.

In Proposition \ref{prop:sgp-continuity} we showed that $\mb{C}^+_{DB} F_0$ solves $(\CR)_{DB}$ \emph{strongly in $\mb{X}_{DB}^{\mb{p},+}$}.
Generally $\mb{X}_{DB}^{\mb{p},+}$ need not be contained in $L_\text{loc}^2(\bbR^n)$, and so these two solution concepts need not coincide.
We must argue differently here.
If $F_0 \in \overline{\mc{R}(DB)}$, then Proposition \ref{prop:sgp-cauchy} implies that $C_{DB}^+ F_0$ solves $(\CR)_{DB}$ strongly in $C^\infty(\bbR_+ : L^2)$, and this implies that $C_{DB}^+ F_0$ solves $(\CR)_{DB}$.
It remains to deal with $F_0 \in \mb{X}^\mb{p}_{DB} \sm \bbX^{\mb{p}}_{DB}$.
For such an $F_0$, let $(F_0^k)_{k \in \bbN}$ be a sequence in $\bbX_{DB}^\mb{p}$ which converges to $F_0$ as $k \to \infty$ (in the weak-star topology when $\mb{p}$ is infinite).
Then, again using either Theorem \ref{thm:sgpnorm-abstract} or Theorem \ref{thm:sgpnorm-concrete}, we have
\begin{equation*}
	\lim_{k \to \infty} C_{DB}^+ F_0^k = \mb{C}_{DB}^+ F_0 \qquad \text{in $X^\mb{p}$},
\end{equation*}
and hence also in $L_\text{loc}^2(\bbR^{1+n}_+)$.
It follows that $\mb{C}_{DB}^+ F_0$ solves $(\CR)_{DB}$.

It remains to show that $\lim_{t \to \infty} \mb{C}_{DB}^+ F_0(t)_\parallel = 0$ in $\mc{Z}^\prime(\bbR^n : \bbC^{nm})$ when $i(\mb{p}) > 2$.
This follows from Proposition \ref{prop:sgp-continuity}, since we have $\lim_{t \to \infty} \mb{C}_{DB}^+ F_0(t) = 0$ in $\mb{X}_{DB}^\mb{p} \hookrightarrow \mc{Z}^\prime$.

\subsection{Initial limiting arguments}

We now begin preparation for the proof of part (ii) of Theorems \ref{thm:mainthm-leq2} and \ref{thm:mainthm-gtr2}.
This section is a rephrasing of the start of \cite[\textsection 8]{AM15}.
There are no fundamentally new ideas, but the notation and the flow of ideas are simplified.

For $t_0 \in \bbR_+$ we write $\bbR_{t_0} := \bbR \sm \{t_0\}$ and $\bbR_{+,t_0} := \bbR_+ \sm \{t_0\}$.

\begin{dfn}
	For $t_0 \in \bbR_+$ and $\gf \in L^2(\bbR^n)$, we define the test function $\mb{G}_{t_0,\gf} \in C^\infty(\bbR_{+,t_0} : \mc{D}(B^* D))$ by
	\begin{align*}
		\mb{G}_{t_0,\gf}(t) &:= \sgn(t_0-t) e^{-[(t_0-t)B^* D]} \gc^{\sgn(t_0-t)} (B^* D) \bbP_{B^* D} \gf \\
		&= \left\{ \begin{array}{rl} e^{-(t_0 - t)B^* D} \gc^+(B^* D) \bbP_{B^* D} \gf & \text{if $t < t_0$} \\ 
		-e^{(t - t_0)B^* D} \gc^-(B^* D) \bbP_{B^* D} \gf & \text{if $t > t_0$} \end{array} \right.
	\end{align*}
	for all $t \in \bbR_{+,t_0}$.
\end{dfn}

Note that $\partial_t \mb{G}_{t_0,\gf} = B^* D \mb{G}_{t_0,\gf}$.
Also observe that since $D$ annihilates the nullspace $\mc{N}(B^* D)$ and since $L^2(\bbR^n) = \mc{N}(B^* D) \oplus \overline{\mc{R}(B^* D)}$, whenever $\gf \in \mc{D}(D)$,
\begin{equation}\label{eqn:DGident}
	D\mb{G}_{t_0,\gf}(t) = \sgn(t_0 - t) e^{-[(t_0 - t)DB^*]} \gc^{\sgn(t_0-t)} (DB^*) D\gf.
\end{equation}

The following lemma is a rewording of \cite[Lemma 7.4]{AM15}.

\begin{lem}\label{lem:AMtesting}
	Let $F$ solve $(\CR)_{DB}$.
	Fix $\gf \in L^2(\bbR^n)$, $t_0 \in \bbR_+$, and let $\gh \in \Lip(\bbR_+ : \bbR)$ and $\gc \in \Lip(\bbR^n : \bbR)$ be compactly supported in $\bbR_{+,t_0}$ and $\bbR^n$ respectively.
	Then we have, with absolutely convergent integrals,
	\begin{align}
		\dint_{\bbR^{1+n}_+} & \left\langle \gh^\prime(t) \gc(x) B^* D \mb{G}_{t_0,\gf} (t,x), F(t,x) \right\rangle\, dx \, dt \nonumber \\
		&= \dint_{\bbR^{1+n}_+} \left\langle \gh(t) B^* [D,m_\gc] \, \partial_t \mb{G}_{t_0,\gf}(t,x) , F(t,x) \right\rangle \, dx \, dt \label{eqn:AMtesting-conclusion}
	\end{align}
	where $m_\gc$ denotes the multiplication operator on $L^2(\bbR^n)$ with symbol $\gc$.
\end{lem}

As a corollary, under an integrability condition involving $F$ and $\gf$, we can obtain the following.

\begin{cor}\label{cor:spatial-limit}
Let $F$, $\gf$, and $t_0$ be as in the statement of Lemma \ref{lem:AMtesting}.
Suppose also that for all compact $K \subset \bbR_{+,t_0}$ we have
\begin{equation}
\mb{1}_{K}(t) |B^* D \mb{G}_{t_0,\gf}(t,x)||F(t,x)| \in L^1(\bbR^{1+n}_+). \label{eqn:spatial-limit-condn}
\end{equation}
Then for all $\gh \in \Lip(\bbR_+ : \bbR)$ compactly supported in $\bbR_{+,t_0}$, we have the absolutely convergent integral
\begin{equation}\label{eqn:spatial-limit-concln}
	\dint_{\bbR^{1+n}_+} \left\langle \gh^\prime(t) B^* D\mb{G}_{t_0,\gf}(t,x) , F(t,x) \right\rangle \, dt \, dx = 0.
\end{equation}
\end{cor}

\begin{proof}
	Fix $\gc \in \Lip(\bbR^n : \bbR)$ with $\gc(x) = 1$ for all $x \in B(0,1)$, and for $R > 0$ define $\gc_R(x) := \gc(x/R)$.
	Then $\gc_R \to 1$ and $[D,m_{\gc_R}] \to 0$ pointwise as $R \to \infty$,\footnote{More precisely, $[D,m_{\gc_R}]$ is given by multiplication with a function that tends to $0$ pointwise.} since $\nm{[D,m_{\gc_R}]}_\infty \lesssim R^{-1} \nm{\nabla \gc}_\infty$.
	Condition \eqref{eqn:spatial-limit-condn} applied with $K = \supp \gh$, the fact that $\partial_t \mb{G}_{t_0,\gf} = B^* D \mb{G}_{t_0,\gf}$, and boundedness of $\gh$ and $\gh^\prime$ imply
	\begin{align*}
		|\gh^\prime(t) B^* D \mb{G}_{t_0,\gf}(t,x)||F(t,x)| &\in L^1(\bbR^{1+n}_+) \quad \text{and} \\
		|\gh(t) \partial_t \mb{G}_{t_0,\gf}(t,x)||F(t,x)| &\in L^1(\bbR^{1+n}_+).
	\end{align*}
	This allows us to deduce \eqref{eqn:spatial-limit-concln} from the equality of Lebesgue integrals \eqref{eqn:AMtesting-conclusion} and dominated convergence.
\end{proof}

Now, assuming that \eqref{eqn:spatial-limit-condn} holds, we can conclude the following.

\begin{cor}\label{cor:limit-in-time}
	Let $F$, $\gf$, and $t_0$ be as in the statement of Lemma \ref{lem:AMtesting}.
	Assume also that condition \eqref{eqn:spatial-limit-condn} is satisfied.
	Then for sufficiently small $\varepsilon > 0$ we have
	\begin{align}
		\barint_{t_0 + \varepsilon}^{t_0 + 2\varepsilon} & \int_{\bbR^n} \left\langle B^* D \mb{G}_{t_0,\gf}(t,x) , F(t,x) \right\rangle \, dx \, dt \nonumber \\
		&= \barint_{t_0 + (2\varepsilon)^{-1}}^{t_0 + \varepsilon^{-1}} \int_{\bbR^n}  \left\langle B^*D \mb{G}_{t_0,\gf}(t,x), F(t,x)\right\rangle  \, dx \, dt\label{eqn:limit-in-time-1}
	\end{align}
	and
	\begin{align}
		-\barint_{\varepsilon}^{2\varepsilon} & \int_{\bbR^n} \left\langle B^* D \mb{G}_{t_0,\gf}(t_0 - t,x), F(t_0 - t,x) \right\rangle \, dx \, dt \nonumber \\
		&= \barint_{\varepsilon}^{2\varepsilon} \int_{\bbR^n}\left\langle B^* D \mb{G}_{t_0,\gf}(t,x) , F(t,x) \right\rangle \, dx \, dt. \label{eqn:limit-in-time-2}
	\end{align}
	These are all absolutely convergent integrals.
\end{cor}

\begin{proof}
	As in \cite[\textsection 8, Step 1b]{AM15} this follows from applying Corollary \ref{cor:spatial-limit} with the piecewise linear functions $\gh_1, \gh_2 \in \Lip(\bbR_+ : \bbR)$ drawn in Figure \ref{fig:etas}, where we impose $\varepsilon < \min(t_0/4, 1/4, 1/t_0)$ (we have carried out a change of variables in the left hand side of \eqref{eqn:limit-in-time-2}).
\end{proof}

\begin{figure}
\caption{The functions $\gh_1$ and $\gh_2$.}\label{fig:etas}
\begin{center}

\begin{tikzpicture}[scale = 1.5, every node/.style={font=\footnotesize}]

	%FIRST GRAPH
	\draw [thin] (0,0) -- (4.3,0);	%x-axis with label
	\draw [thin,->] (4.3,0) -- (4.5,0);
	\node [right] at (4.5,0) {$t$} ;
	
	\draw [thin] (0,0) -- (0,0.7);	%y-axis with label
	\draw [thin,->] (0,0.7) -- (0,0.9);
	\node [above] at (0,0.9) {$\gh_1(t)$};
	
	\draw [fill=black] (0,0) circle [radius = 0.5pt];
	\node [left] at (0,0) {$0$};
	\draw [fill=black] (0,0.7) circle [radius = 0.5pt];
	\node [left] at (0,0.7) {$1$};
	
	\draw [fill=black] (0.2,0) circle [radius = 0.5pt];
	\node [below] at (0.2,0) {$t_0+\varepsilon$};
	\draw [fill=black] (0.9,0) circle [radius = 0.5pt];
	\node [below] at (0.9,0) {$t_0 + 2\varepsilon$};
	\draw [fill=black] (2.1,0) circle [radius = 0.5pt];
	\node [below] at (2.3,0) {$t_0 + (2\varepsilon)^{-1}$};
	\draw [fill=black] (3.8,0) circle [radius = 0.5pt];
	\node [below] at (3.8,0) {$t_0 + \varepsilon^{-1}$};
		
	\draw [thick] (0,0) -- (0.2,0) -- (0.9,0.7) -- (2.1,0.7) -- (3.8,0) -- (4.5,0);
	
	%SECOND GRAPH  *move everything (-6, -2) (can change 2nd...)
	\draw [thin] (0,-2) -- (4.3,-2);	%x-axis with label
	\draw [thin,->] (3.3,-2) -- (4.5,-2);
	\node [right] at (4.5,-2) {$t$} ;
	
	\draw [thin] (0,-2) -- (0,-1.3);	%y-axis with label
	\draw [thin,->] (0,-1.3) -- (0,-1.1);
	\node [above] at (0,-1.1) {$\gh_2(t)$};
	
	\draw [fill=black] (0,-2) circle [radius = 0.5pt];
	\node [left] at (0,-2) {$0$};
	\draw [fill=black] (0,-1.3) circle [radius = 0.5pt];
	\node [left] at (0,-1.3) {$1$};
	
	\draw [fill=black] (0.4,-2) circle [radius = 0.5pt];
	\node [below] at (0.4,-2) {$\varepsilon$};
	\draw [fill=black] (0.8,-2) circle [radius = 0.5pt];
	\node [below] at (0.8,-2) {$2\varepsilon$};
	\draw [fill=black] (2.1,-2) circle [radius = 0.5pt];
	\node [below] at (2.0,-2) {$t_0 - 2\varepsilon$};
	\draw [fill=black] (2.5,-2) circle [radius = 0.5pt];
	\node [below] at (2.8,-2) {$t_0 - \varepsilon$};
	
	\draw [thick] (0,-2) -- (0.4,-2) -- (0.8,-1.3) -- (2.1,-1.3) -- (2.5,-2) -- (4.5,-2);
\end{tikzpicture}
\end{center}
\end{figure}
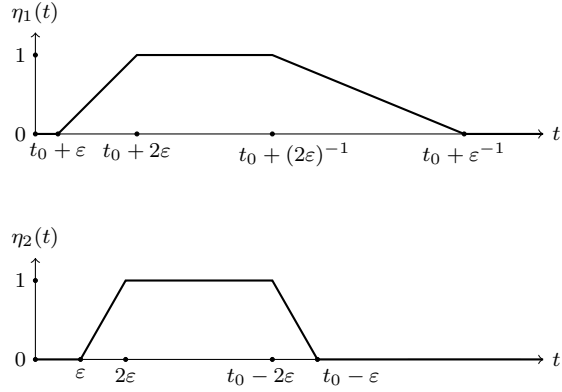

\subsection{Proof of Theorem \ref{thm:mainthm-leq2}}

Recall that part (i) has already been proven in Subsection \ref{ssec:solnconstruction}; here we prove part (ii).

All of the results in this section are valid for $\mb{p} = (p,s)$ such that $p \leq 2$ and $s < 0$.
We do not `fix' such a $\mb{p}$, however, because in the final step we will invoke prior results with a different choice of $\mb{p}$.

\textbf{Step 1: Verification and application of initial limiting arguments.}

\begin{lem}\label{lem:smallp-limit}
	Let $\gf \in L^2(\bbR^n)$.
	Then $\mb{1}_{K \times \bbR^n} B^* D\mb{G}_{t_0,\gf} \in X^{\mb{p}^\prime}$ for all compact $K \subset \bbR_{+,t_0}$, with
	\begin{equation}\label{eqn:smallp-limit}
		\nm{\mb{1}_{K \times \bbR^n} B^* D\mb{G}_{t_0,\gf}}_{X^{\mb{p}^\prime}} \lesssim \nm{\gf}_2 \dist(K,t_0)^{-1} K_-^{s+n\gd_{p,2}}
	\end{equation}
	where $K_- = \inf(K)$.
\end{lem}

\begin{proof}
	First note that the estimate
	\begin{equation*}
		\nm{\mb{1}_{K \times \bbR^n} B^* D \mb{G}_{t_0,\gf}}_{X_{-s-n\gd_{p,2}}^2} \lesssim \nm{\gf}_2 \dist(K,t_0)^{-1} K_-^{s+n\gd_{p,2}}
	\end{equation*}
	can be shown by writing
	\begin{align}
		\nm{\mb{1}_{K \times \bbR^n} B^* D \mb{G}_{t_0,\gf}}_{X_{-s-n\gd_{p,2}}^2}
		&= \bigg( \int_{K_-}^{K_+} \nm{t^{s+n\gd_{p,2}} B^* D \mb{G}_{t_0,\gf}}_2^2 \, \frac{dt}{t} \bigg)^{1/2} \nonumber\\
		&\lesssim \nm{\gf}_2 \bigg( \int_{K_-}^{K_+} t^{2(s + n\gd_{p,2})} \dist(K,t_0)^{-2} \, \frac{dt}{t} \bigg)^{1/2} \label{line:fc-dist} \\
		&\lesssim \nm{\gf}_2 \dist(K,t_0)^{-1} K_-^{s+n\gd_{p,2}} \label{line:negativity}.
	\end{align}
	The estimate \eqref{line:fc-dist} follows by writing
	\begin{equation*}
		B^* D \mb{G}_{t_0,\gf} = \frac{\sgn(t_0-t)}{t_0 - t} (t_0-t)B^* D e^{-[(t_0-t)B^* D]} \gc^{\sgn(t_0-t)}(B^* D) \bbP_{B^* D} \gf
	\end{equation*}
	and noting that the operator
	\begin{equation*}
	(t_0 - t)B^* D e^{-[(t_0-t)B^* D]} \gc^{\sgn(t_0-t)}(B^* D) \bbP_{B^* D}
	\end{equation*}
	is bounded on $L^2(\bbR^n)$ uniformly in $t \in \bbR_{+,t_0}$, and that $|(t_0 - t)^{-1}| \lesssim \dist(K,t_0)^{-1}$ for $t \in K$.
	Then \eqref{line:negativity} follows because $s+n\gd_{p,2}$ is negative whenever $s < 0$ and $p < 2$.
	Now use the $X$-space embeddings to write
	\begin{equation*}
		X_{-s-n\gd_{p,2}}^2 \hookrightarrow X^{\mb{p}^\prime},
	\end{equation*}
	from which follows \eqref{eqn:smallp-limit}.
\end{proof}

\begin{cor}\label{cor:intlim-0}
	Let $\gf \in L^2(\bbR^n)$ and suppose that $F \in X^\mb{p}$ solves $(\CR)_{DB}$.
	Then
	\begin{equation}\label{eqn:limit-1-2}
		\lim_{\varepsilon \to 0} \barint_{t_0 + \varepsilon}^{t_0 + 2\varepsilon} \int_{\bbR^n} \left\langle B^* D \mb{G}_{t_0,\gf}(t,x) , F(t,x) \right\rangle \, dx \, dt = 0.
	\end{equation}
\end{cor}

\begin{proof}
	For $\varepsilon > 0$ small the previous lemma yields 
	\begin{equation*}
		\nm{\mb{1}_{[t_0 + (2\varepsilon)^{-1},t_0 + \varepsilon^{-1}] \times \bbR^n} B^* D\mb{G}_{t_0,\gf}}_{X^{\mb{p}^\prime}} \lesssim \nm{\gf}_2 (2\varepsilon) (t_0 + (2\varepsilon)^{-1})^{s+n\gd_{p,2}},
	\end{equation*}
	which decays as $\varepsilon \to 0$ since $s + n\gd_{p,2}$ is negative when $s < 0$ and $p \leq 2$.
	Therefore in particular, by $X$-space duality, condition \eqref{eqn:spatial-limit-condn} is satisfied, and by boundedness of the above quasinorms as $\varepsilon \to 0$ we can take the $\varepsilon \to 0$ limit in \eqref{eqn:limit-in-time-1} to obtain \eqref{eqn:limit-1-2}.
\end{proof}

\textbf{Step 2: Semigroup property of $F$.}

\begin{lem}\label{lem:smallpsoln}
	Suppose $F \in X^{\mb{p}}$ solves $(\CR)_{DB}$.
	Then $F \in C^\infty(\bbR_+ : \bbH_{DB}^2)$, $F(t) \in \mc{D}(DB)$ for all $t > 0$, and $\partial_t F + DBF = 0$ holds strongly in $C^\infty(\bbR_+ : \bbH_{DB}^2)$.
\end{lem}

\begin{proof}
	We already have that $F \in C^\infty(\bbR_+ : L^2_\text{loc}(\bbR^n))$ from Proposition \ref{prop:wksolnprops}, and furthermore that $\partial_t F \in X^{\mb{p}-1}$.
	Hence we have $F(t_0),(\partial_t F)(t_0) \in E^{\mb{p}}$ for all $t_0 \in \bbR_+$, and therefore by the slice space containments of Proposition \ref{prop:slice-embeddings} we obtain $F(t_0), (\partial_t F)(t_0) \in L^2$ for all $t_0 \in \bbR_+$.
	Therefore $F(t_0) \in \mc{D}(DB)$ for all $t_0 \in \bbR_+$, and $\partial_t F + DBF = 0$ holds in $L^2$.
	We can iterate this argument by reapplying $\partial_t$, as this preserves the property of solving $(\CR)_{DB}$ as well as the previously stated $L^2$ containments, so we obtain $F \in C^\infty(\bbR_+ : L^2)$.
	
	Now since $\lim_{t_0 \to \infty} F(t_0) = 0$ in $L^2$ (Lemma \ref{lem:solution-decay}), we can write
	\begin{equation*}
		F(t_0) = -\int_{t_0}^\infty (\partial_t F)(\gt) \, d\gt
		= -\int_{t_0}^\infty DB (F(\gt)) \, d\gt \in \overline{\mc{R}(DB)}
	\end{equation*}
	by the fundamental theorem of calculus.
	Therefore $F(t_0) \in {\bbH}_{DB}^2$ for all $t_0$, and since the $\bbH_{DB}^2$-norm is equivalent to the $L^2$-norm when restricted to $\overline{\mc{R}(DB)}$, this completes the proof.
\end{proof}

\begin{lem}\label{lem:psmallplus}
	Suppose that $F \in X^\mb{p}$ solves $(\CR)_{DB}$.
	Then for all $t_0 > 0$ and $\gt \geq 0$ we have $F(t_0) \in \bbH_{DB}^{2,+} = \overline{R(DB)}^+$ and
	\begin{equation}\label{eqn:derivsgp}
		F(t_0 + \gt) = e^{-\gt DB} (F(t_0))
	\end{equation}
	(recall that $e^{-\gt DB}$ is defined on the positive spectral subspace).
\end{lem}

\begin{proof}
	For all $\gf \in L^2(\bbR^n)$, the function $t \mapsto B^* D \mb{G}_{t_0,\gf}(t)$ is smooth in $t \in \bbR_{+,t_0}$ with values in $\bbH^2_{B^* D}$ and with
	\begin{equation*}
		\lim_{t \downarrow t_0} B^* D \mb{G}_{t_0,\gf}(t) = -B^* D \gc^-(B^* D) \bbP_{B^*D} \gf
	\end{equation*}
	in $\bbH_{B^* D}^2$.
	Furthermore, by Lemma \ref{lem:smallpsoln}, $t \mapsto F(t)$ is smooth in $t \in \bbR_+$ with values in $\bbH_{DB}^2$.
	Therefore we may write for all $\gf \in L^2(\bbR^n)$, using \eqref{eqn:limit-1-2} from Corollary \ref{cor:intlim-0},
	\begin{align}
		0 &= \lim_{\varepsilon \to 0} \barint_{t_0 + \varepsilon}^{t_0 + 2\varepsilon} \int_{\bbR^n} \langle B^* D \mb{G}_{t_0,\gf}(t,x), F(t,x) \rangle \, dx \, dt \nonumber\\
		&= \lim_{\varepsilon \to 0} \barint_{t_0 + \varepsilon}^{t_0 + 2\varepsilon} \langle B^* D \mb{G}_{t_0,\gf}(t), F(t) \rangle_{\bbH_{B^* D}^2} \, dt \nonumber\\
		&= -\langle B^* D \gc^- (B^* D) \bbP_{\overline{\mc{R}(B^* D)}} \gf, F(t_0) \rangle_{\bbH_{B^* D}^2}. \label{eqn:2testing}
	\end{align}
	Hence for all $\gfv \in \overline{\mc{R}(B^*D)}$ and all $\gd > 0$, since $e^{-\gd[B^* D]}$ maps $\bbH_{B^* D}^{2,-}$ into itself, applying \eqref{eqn:2testing} to $\gf = e^{-\gd[B^* D]} \gfv$ yields
	\begin{equation}\label{eqn:test:negkill}
		\langle B^* D e^{\gd B^* D} \gc^-(B^* D) \gfv, F(t_0) \rangle_{\bbH_{B^* D}^2} = 0.
	\end{equation}
	The subspace
	\begin{equation*}
		\{B^* D e^{\gd B^* D} \gc^-(B^* D) \gfv : \gfv \in \overline{\mc{R}(B^* D)}\} \subset L^2(\bbR^n)
	\end{equation*}
	is dense in $\bbH_{B^* D}^{2,-}$ (see \cite[p. 28]{AM15}), so \eqref{eqn:test:negkill} and the decomposition $\bbH_{DB}^2 = \bbH_{DB}^{2,+} \oplus \bbH_{DB}^{2,-}$ imply that $F(t_0) \in \bbH_{DB}^{2,+}$.
	
	Now we derive the semigroup equation \eqref{eqn:derivsgp}.
	For all $\gd \geq 0$ and $\gf \in \bbH_{B^* D}^2$, define
	\begin{equation*}
		\gf_\gd := e^{-\gd[B^* D]} \gf
	\end{equation*}
	and
	\begin{equation*}
		I_{t_0,\gf}^{\varepsilon,\gd} := \barint_{\varepsilon}^{2\varepsilon} \langle B^* D \mb{G}_{t_0,\gf_\gd}(t), F(t) \rangle_{\bbH_{B^*D}^2} \, dt.
	\end{equation*}
	By \eqref{eqn:limit-in-time-2}, using the same argument as before to write everything in terms of $\bbH_{B^*D}^2$-duality, we have
	\begin{align*}
		\lim_{\varepsilon \to 0} I_{t_0,\gf}^{\varepsilon,\gd}
		&= -\lim_{\varepsilon \to 0} \barint_{\varepsilon}^{2\varepsilon} \langle B^* D e^{-tB^* D} e^{-\gd B^* D} \gc^+ (B^* D) \gf, F(t_0 - t) \rangle_{\bbH_{B^* D}^2} \, dt \\
		&= -\langle B^* D e^{-\gd B^* D} \gc^+(B^* D)\gf, F(t_0) \rangle_{\bbH_{B^* D}^2}.
	\end{align*}
	Therefore for all $\gt \geq 0$, $\gd \geq 0$, and $\gf \in \bbH_{B^* D}^2$, using $I_{t_0,\gf}^{\varepsilon,\gd+\gt} = I_{t_0 + \gd, \gf}^{\varepsilon,\gt}$, we have
	\begin{align*}
		&\langle B^* D e^{-\gd B^* D} \gc^+(B^* D) \gf, e^{-\gt DB} (F(t_0)) \rangle_{\bbH_{B^* D}^2} \\
		&= \langle B^* D e^{-(\gd+\gt) B^* D} \gc^+(B^* D)\gf, F(t_0) \rangle_{\bbH_{B^* D}^2} \\
		&= -\lim_{\varepsilon \to 0} I_{t_0,\gf}^{\varepsilon,\gd+\gt} \\
		&= -\lim_{\varepsilon \to 0} I_{t_0 + \gt, \gf}^{\varepsilon,\gd} \\
		&= \langle B^* D e^{-\gd B^* D} \gc^+(B^* D) \gf, F(t_0 + \gt) \rangle_{\bbH_{B^* D}^2}.
	\end{align*}
	As before, the subspace $\{B^* D e^{-\gd B^* D} \gc^+(B^* D)\gf : \gf \in \bbH_{B^* D}^2\}$ is dense in $\bbH_{B^* D}^{2,+}$, so by duality we have $F(t_0 + \gt) = e^{-\gt DB} F(t_0)$ in $\bbH_{DB}^{2,+}$ for all $t_0 > 0$ and all $\gt \geq 0$.
	\end{proof}

\textbf{Step 3: Completing the proof.}

\begin{prop}[Existence of boundary trace]\label{prop:leq2-trace}
	Suppose that $F \in X^\mb{p}$ solves $(\CR)_{DB}$, and let $\mb{X}$ be a completion of ${\bbX}_{DB}^{\mb{p}}$.
	Then there exists a unique $F_0 \in \mb{X}^+$ such that $F = \mb{C}_{DB}^+ F_0$.
	Furthermore, $\nm{F_0}_{\mb{X}} \lesssim \nm{F}_{X^\mb{p}}$.
\end{prop}

\begin{proof}
	Fix an exponent $\td{\mb{p}}$ with $i(\td{\mb{p}}) \in (1,2]$ and $\gq(\td{\mb{p}}) < 0$ such that $\mb{p} \hookrightarrow \td{\mb{p}}$ (when $p > 1$ we may take $\td{\mb{p}} = \mb{p}$).
	By Lemma \ref{lem:psmallplus} we have $F(t_0) \in \bbH_{DB}^{2,+} \cap \mc{D}(DB)$ for all $t_0 > 0$.
	We can then estimate
	\begin{align}
		\nm{F(t_0)}_{{\bbX}_{DB}^{\td{\mb{p}}}}
		&\simeq \nm{DB F(t_0)}_{{\bbX}_{DB}^{\td{\mb{p}}-1}} \label{line:regbump} \\
		&= \nm{\gt \mapsto e^{-\gt DB} (DB F)(t_0)}_{X^{\td{\mb{p}}-1}} \label{line:semigpnorm} \\
		&= \nm{\gt \mapsto DB F(t_0 + \gt)}_{X^{\td{\mb{p}}-1}} \label{line:semigrpeqn} \\
		&= \nm{S_{t_0} DB F}_{X^{\td{\mb{p}}-1}} \nonumber \\
		&\lesssim \nm{DBF}_{X^{\td{\mb{p}}-1}} \label{line:shiftusage} \\
		&= \nm{\partial_t F}_{X^{\td{\mb{p}}-1}} \nonumber \\
		&\lesssim \nm{F}_{X^{\td{\mb{p}}}} \label{line:propsuse} \\
		&\lesssim \nm{F}_{X^{\mb{p}}}. \label{line:embusage}
	\end{align}
	The first line \eqref{line:regbump} is from Corollary \ref{cor:regbump}.
	Line \eqref{line:semigpnorm} comes from Theorem \ref{thm:sgpnorm-abstract}.
	Line \eqref{line:semigrpeqn} comes from Lemma \ref{lem:psmallplus}, \eqref{line:shiftusage} comes from Proposition \ref{prop:s12shift} because $i(\td{\mb{p}} - 1) \leq 2$ and $\gq(\td{\mb{p}}-1) \leq -1/2$, \eqref{line:propsuse} comes from Proposition \ref{prop:wksolnprops}, and finally \eqref{line:embusage} follows from $X$-space embeddings by $\mb{p} \hookrightarrow \td{\mb{p}}$.
	Therefore $F(t_0) \in {\bbX}_{DB}^{\td{\mb{p}},+}$ uniformly in $t_0 > 0$.
	
	Let $\td{\mb{X}}$ be a completion of $\bbX_{DB}^{\td{\mb{p}}}$, so that $\mb{X} \hookrightarrow \td{\mb{X}}$.
	Since $\td{\mb{X}}^+$ can be identified with the dual of $\td{\mb{Y}}^+$ for any completion $\td{\mb{Y}}$ of ${\bbX}_{B^* D}^{\td{\mb{p}}^\prime}$, there exists a sequence $t_k \downarrow 0$ and an $F_0 \in \td{\mb{X}}^+$ such that $F(t_k)$ converges weakly to $F_0$ in $\td{\mb{X}}^+$ as $k \to \infty$.
	We thus have for all $\gf \in {\bbX}_{B^* D}^{\td{\mb{p}}^\prime,+}$ and for all $\gt > 0$,
	\begin{align}
		\langle \gf, e^{-\gt DB} F_0 \rangle_{\td{\mb{Y}}}
		&= \langle e^{-\gt B^* D} \gf, F_0 \rangle_{\td{\mb{Y}}} \nonumber \\
		&= \lim_{k \to \infty} \langle e^{-\gt B^* D} \gf, F(t_k) \rangle_{\td{\mb{Y}}} \nonumber\\
		&= \lim_{k \to \infty} \langle \gf, e^{-\gt DB} F(t_k) \rangle_{\td{\mb{Y}}} \nonumber\\
		&= \lim_{k \to \infty} \langle \gf, F(t_k + \gt) \rangle_{\td{\mb{Y}}} \label{line:sgpagain} \\
		&= \langle \gf, F(\gt) \rangle_{\td{\mb{Y}}},\nonumber
	\end{align}
	using Lemma \ref{lem:psmallplus} in \eqref{line:sgpagain}.
	Therefore by density we have $\mb{C}_{DB}^+ F_0 = F$.
	
	It only remains to show that $F_0$ is in $\mb{X}^+$, with the correct quasinorm estimate, and uniquely determined.
	Recall that $\mb{C}_{DB}^+$ is equal to the restriction of $\mb{Q}_{\sgp,DB}$ to the positive spectral subspace ($\mb{Q}_{\sgp,DB}$ is the extension of $\bbQ_{\sgp,DB}$ to the completion $\mb{X}$).
	Let $\gf \in \gY_\infty^\infty$ be a Calder\'on sibling of $\sgp$.
	Then $F_0 = \mb{S}_{\gf,DB} \mb{Q}_{\sgp,DB} F_0 = \mb{S}_{\gf,DB} F$, so by Proposition \ref{prop:completion-identification} we have $F_0 \in \mb{X}$ with $\nm{F_0}_{\mb{X}} \lesssim \nm{F}_{X^\mb{p}}$.
	In fact, since $F(t_0) \in \bbH_{DB}^{2,+}$ for all $t_0 > 0$, we find that $F$ is in the positive spectral subspace $\mb{X}^+$.
	Uniqueness follows by injectivity of $\mb{Q}_{\sgp,DB}$ (Proposition \ref{prop:completion-identification}).
	\end{proof}

This completes the proof of Theorem \ref{thm:mainthm-leq2}.

\subsection{Proof of Theorem \ref{thm:mainthm-gtr2}}

Recall that part (i) has already been proven in Subsection \ref{ssec:solnconstruction}; here we prove part (ii).
Our proof roughly follows that of \cite[Theorem 1.3]{AM15}, arguing via a series of rather technical lemmas.
In this section we continually assume that $\mb{p}$ satisfies the assumptions of Theorem \ref{thm:mainthm-gtr2}.	Most of the lemmas work without assuming $\mb{p}^\heartsuit \in I(\mb{X},DB^*)$, but we gain nothing from dropping this assumption.

\textbf{Step 1: Establishing a good class of test functions.}

We define the following class of test functions for $\bbX_{DB}^\mb{p}$:
\begin{equation*}
	\bbD^\mb{p}(X) := \big\{\gf \in \mc{D}(D) : D\gf \in {\bbX}_{DB^*}^{\mb{p}^\heartsuit}, \; \gc^{\pm}(DB^*) D\gf \in E^{\mb{p}^\heartsuit} \big\}.
\end{equation*}
This is large enough to contain the Schwartz functions and to be stable under the action of various operators, yet it is restrictive enough to let us exploit slice space containments.

\begin{lem}\label{lem:test-schwartz-containment}
	The Schwartz class $\mc{S}(\bbR^n : \bbC^{m(1+n)})$ is contained in $\bbD^\mb{p}(X)$.
      \end{lem}

      This is a modification of the argument of \cite[Lemma 8.10]{AM15}.

\begin{proof}
	Suppose $\gf \in \mc{S}$.
	Then $\gf \in \mc{D}(D)$ and $D\gf \in {\bbX}_{D}^{\mb{p}^\heartsuit} = {\bbX}_{DB^*}^{\mb{p}^\heartsuit}$ by the assumption on $\mb{p}$.
	It remains to show that $\gc^\pm(DB^*) D\gf \in E^{\mb{p}^\heartsuit}$, and this takes some work.
	Since $D\gf \in \mc{S} \subset E^{\mb{p}^\heartsuit}$ (Proposition \ref{prop:slice-distribution}) and since
	\begin{equation*}
		D\gf = \gc^+(DB^*) D\gf + \gc^-(DB^*) D\gf
	\end{equation*}
	it suffices to show that $\gc^+(DB^*) D\gf$ is in $E^{\mb{p}^\heartsuit}$.
	
	Define $\gy \in \gY_N^\infty$, with $N$ large to be chosen later, by
	\begin{equation*}
		\gy(z) := \frac{[z]^N e^{-[z]}}{N!}.
	\end{equation*}
	Then for all $t \in \bbR \sm \{0\}$ we have
	\begin{equation*}
		\int_0^\infty \gy(st) \, \frac{ds}{s}
		= \frac{1}{N!} \int_0^\infty s^N e^{-s} \, \frac{ds}{s} = 1,
	\end{equation*}
	so by holomorphy we have
	\begin{equation*}
		\int_0^\infty \gy(sz) \, \frac{ds}{s} = 1
	\end{equation*}
	for all $z \in S_\gm$.
	By the same argument, along with integration by parts and induction on $N$, for all $z \in S_\gm$ we have
	\begin{equation*}
		\int_1^\infty \gy(sz) \, \frac{ds}{s} = P([z]) e^{-[z]}
	\end{equation*}
	where $P$ is a real polynomial of degree $N-1$.
	Therefore by functional calculus on $\overline{\mc{R}(DB^*)}$ we may write
	\begin{equation*}
		\gc^+(DB^*)D\gf = \int_0^1 (\gy\gc^+)(sDB^*) D\gf \, \frac{ds}{s} + P(DB^*) e^{-DB^*} \gc^+(DB^*) D\gf.
	\end{equation*}
	By Lemma \ref{lem:sgp-in-slice} (using $i(\mb{p}^\heartsuit) < 2$ and $\gq(\mb{p}^\heartsuit) < 0$) we have 
	\begin{equation*}
		P(DB^*) e^{-DB^*} \gc^+(DB^*) D\gf \in E^{\mb{p}^\heartsuit},
	\end{equation*}
	so it suffices to show that
	\begin{equation*}
		\int_0^1 (\gy\gc^+)(sDB^*) D\gf \, \frac{ds}{s} \in E^{\mb{p}^\heartsuit}.
	\end{equation*}
	
	For $f \in L^2(\bbR^n)$ write
	\begin{equation*}
		G(f) := \int_0^1 (\gy\gc^+)(sDB^*) f \, \frac{ds}{s};
	\end{equation*}
	since $\gy\gc^+ \in \gY_+^\infty$ this is defined for all $f \in L^2(\bbR^n)$ (not just $f \in \overline{\mc{R}(DB^*)})$.
	Note that the family $((\gy\gc^+)(sDB^*))_{s > 0}$ satisfies   off-diagonal estimates of order $N$ (Theorem \ref{thm-fcode}).
	For $Q,R \in \mc{D}_0$ (here $\mc{D}_0$ is the set of standard dyadic cubes in $\bbR^n$ with sidelength $1$) with $d(Q,R) \geq 1$ we can estimate
	\begin{align*}
		\nm{G(\mb{1}_R D\gf)}_{L^2(Q)}
		&= \bigg( \int_Q \left| \int_0^1 (\gy \gc^+)(sDB^*) \mb{1}_R D\gf(x) \, \frac{ds}{s} \right|^2 \, dx \bigg)^{1/2} \\
		&\leq \int_0^1 \nm{(\gy\gc^+)(sDB^*) \mb{1}_R D\gf}_{L^2(Q)} \, \frac{ds}{s} \\
		&\lesssim \int_0^1 \bigg( \frac{d(Q,R)}{s} \bigg)^{-N} \, \frac{ds}{s} \nm{\mb{1}_R D\gf}_2 \\
		&\simeq d(Q,R)^{-N} \nm{\mb{1}_R D\gf}_2.
	\end{align*}
	For all other $Q,R \in \mc{D}_0$ we have instead
	\begin{equation*}
		\nm{G(\mb{1}_R D\gf)}_{L^2(Q)} \lesssim \nm{\mb{1}_R D\gf}_2.
	\end{equation*}
	Therefore we have by the discrete characterisation of slice spaces (Proposition \ref{prop:slice-discrete}), writing $R \sim Q$ to mean that $\dist(R,Q) = 0$ and noting that $\dist(R,Q) \geq 1$ if $R \not\sim Q$,
	\begin{align*}
		\nm{G(D\gf)}_{E^{\mb{p}^\heartsuit}}
		&\simeq \bigg( \sum_{Q \in \mc{D}_1} \nm{G(D\gf)}_{L^2(Q)}^{i(\mb{p}^\heartsuit)} \bigg)^{1/i(\mb{p}^\heartsuit)} \\
		&\lesssim \bigg( \sum_{Q \in \mc{D}_1} \bigg( \sum_{R \sim Q} + \sum_{R \not\sim Q} \bigg) \nm{G(\mb{1}_R D\gf)}_{L^2(Q)}^{i(\mb{p}^\heartsuit)} \bigg)^{1/i(\mb{p}^\heartsuit)} \\
		&\lesssim \bigg( \sum_{\substack{Q \in \mc{D}_1 \\ R \sim Q}} \nm{D\gf}_{L^2(R)}^{i(\mb{p}^\heartsuit)} \bigg)^{1/i(\mb{p}^\heartsuit)} + \\
		&\qquad \bigg( \sum_{\substack{Q \in \mc{D}_1 \\ R \not\sim Q}} d(Q,R)^{-Ni(\mb{p}^\heartsuit)} \nm{D\gf}_{L^2(R)}^{i(\mb{p}^\heartsuit)} \bigg)^{1/i(\mb{p}^\heartsuit)} \\
		&=: \mb{I}_1 + \mb{I}_2.
	\end{align*}
	Since the number of cubes $R \in \mc{D}_0$ such that $R \sim Q$ is uniform in $Q$, we have
	\begin{equation*}
		\mb{I}_1 \simeq \bigg( \sum_{R \in \mc{D}_1} \nm{D\gf}_{L^2(R)}^{i(\mb{p}^\heartsuit)} \bigg)^{1/i(\mb{p}^\heartsuit)} \simeq \nm{D\gf}_{E^{\mb{p}^\heartsuit}}.
	\end{equation*}
	To handle $\mb{I}_2$ write
	\begin{align}
		\mb{I}_2 &= \bigg(\sum_{R \in \mc{D}_1} \nm{D\gf}_{L^2(R)}^{i(\mb{p}^\heartsuit)} \sum_{k = 1}^\infty k^{-Ni(\mb{p}^\heartsuit)} |\{Q \in \mc{D}_1 : d(Q,R) = k\}| \bigg)^{1/i(\mb{p}^\heartsuit)} \label{line:cubecount} \\
		&\simeq_{N,\mb{p}^\heartsuit,n} \nm{D\gf}_{E^{\mb{p}^\heartsuit}} \nonumber
	\end{align}
	using that the innermost sum in \eqref{line:cubecount} is independent of $R$ and convergent for $N$ sufficiently large.
	Therefore
	\begin{equation*}
		\nm{G(D\gf)}_{E^{\mb{p}^\heartsuit}} \lesssim \nm{D\gf}_{E^{\mb{p}^\heartsuit}} < \infty
	\end{equation*}
	which shows that $\gc^+(DB^*) D\gf \in E^{\mb{p}^\heartsuit}$ and completes the proof.
\end{proof}

\begin{lem}\label{lem:test-sgp-closure}
	We have the following stability properties of $\bbD^\mb{p}(X)$:
	\begin{enumerate}[(i)]
	\item for all $\gd > 0$, $e^{-\gd[B^* D]} \bbD^\mb{p}(X) \subset \bbD^\mb{p}(X)$,
	\item $\gc^{\pm}(B^* D) \bbD^\mb{p}(X) \subset \bbD^\mb{p}(X)$.
	\end{enumerate}
\end{lem}

\begin{proof}
	The function $[z \mapsto e^{-\gd[z]}]$ is in $H^\infty$ and has a polynomial limit at $0$, so $e^{-\gd[B^*D]} \gf$ may be defined for all $\gf \in \mc{D}(D)$ (not just those in $\overline{\mc{R}(B^* D)}$)\footnote{Although we did not discuss this in Subsection \ref{section:bisectorial-hfc}, this is a standard procedure. The representation \eqref{eqn:sgpinter} is all we need.} 
	as 
	\begin{equation*}
	e^{-\gd[B^* D]} \gf = e^{-\gd[B^* D]} \bbP_{B^* D} \gf + (I-\bbP_{B^* D})  \gf.
	\end{equation*}
	For all such $\gf$, by using the similarity of functional calculi along with the equalities $D(I - \bbP_{B^* D}) \gf=0$ and $D\bbP_{B^* D} \gf=D\gf$, we can write
	\begin{equation}\label{eqn:sgpinter}
		D(e^{-\gd[B^* D]} \gf) = e^{-\gd[DB^*]} D\gf.
	\end{equation}
	Since $D\gf$ is in ${\bbX}_{DB^*}^{\mb{p}^\heartsuit}$, so is $D(e^{-\gd[B^* D]} \gf)$.
	To see the slice space containments of spectral projections, write
	\begin{equation*}
		\gc^\pm (DB^*) D(e^{-\gd[B^* D]} \gf) = e^{-\gd[DB^*]} \gc^\pm (DB^*) D\gf.
	\end{equation*}
	By assumption $\gc^\pm (DB^*) D\gf$ is in $E^{\mb{p}^\heartsuit} \cap {\bbX}_{DB^*}^{\mb{p}^\heartsuit,\pm} \subset E^{\mb{p}^\heartsuit} \cap \bbX_{DB^*}^{2,\pm}$, and by Corollary \ref{cor:sgp-slice-bddness}, $e^{-\gd[DB^*]} \gc^\pm (DB^*) D\gf$ is in $E^{\mb{p}^\heartsuit}$. 
		
	The second part is proven similarly: we have $\gc^\pm(B^* D)\mc{D}(D) \subset \mc{D}(B^* D)$, and by similarity of functional calculi
	\begin{equation*}
		D \gc^\pm (B^* D) \gf_0 = \gc^\pm (DB^*) D\gf_0 \in {\bbX}_{DB^*}^{\mb{p}^\heartsuit}
	\end{equation*}
	and
	\begin{align*}
		\gc^\pm (DB^*) D \gc^\pm (B^* D)\gf &= \gc^\pm (DB^*) D\gf_0 \in E^{\mb{p}^\heartsuit} \\
		\gc^\mp (DB^*) D \gc^\pm (B^* D)\gf &= 0 \in E^{\mb{p}^\heartsuit},
	\end{align*}
        completing the proof.
\end{proof}

\begin{lem}\label{lem:int-slice}
	Suppose that $\gf \in {\bbX}_{B^* D}^{\mb{p}^\prime} \cap \mc{D}(B^* D)$.
	Then for all $t > 0$, $\gc^\pm (DB^*) e^{-t[DB^*]/2} D\gf$ is defined and contained in $E^{\mb{p}^\prime}$.
	Furthermore, $e^{-t[B^* D]/2} \gf \in \bbD^\mb{p}(X)$.
\end{lem}

\begin{proof}
	Note that $\mc{D}(B^* D) = \mc{D}(D)$.
	Since $D\gf \in {\bbX}_{DB^*}^{\mb{p}^\heartsuit}$ (Proposition \ref{prop:D-mapping}), by Lemma \ref{lem:sgp-in-slice} we find that
	\begin{equation}\label{eqn:sgp-slice}
		\gc^\pm (DB^*) e^{-t[DB^*]/2} D\gf = e^{-t[DB^*]/2} \gc^\pm (DB^*) D\gf \in E^{\mb{p}^\heartsuit} = E^{\mb{p}^\prime}.
	\end{equation}
	
	To see that $e^{-t[B^* D]/2} \gf$ is in $\bbD^\mb{p}(X)$, note that
	\begin{equation*}
		e^{-t[B^* D]/2} \gf \in \mc{D}(B^* D) = \mc{D}(D),
	\end{equation*}
	that
	\begin{equation*}
		D e^{-t[B^* D]/2} \gf  = e^{-t[DB^*]/2} D\gf \in {\bbX}_{DB^*}^{\mb{p}^\heartsuit},
	\end{equation*}
	and that
	\begin{equation*}
		\gc^{\pm}(DB^*) D e^{-t[B^* D]/2} \gf = e^{-t[DB^*]/2} \gc^\pm (DB^*) D\gf \in E^{\mb{p}^\heartsuit}
	\end{equation*}
	by \eqref{eqn:sgp-slice}.
\end{proof}

\textbf{Step 2: Verification and application of the initial limiting arguments.}

\begin{lem}\label{lem:badaperture1}
	Define the operator
	\begin{equation*}
	\td{\mb{G}}_{t_0} : \gf \mapsto \sgn(t_0 - t)e^{-[(t_0 - t)DB^*]} \gc^{\sgn(t_0-t)}(DB^*)\gf.
	\end{equation*}
	Let $K \subset \bbR_{+,t_0}$ be compact.
	Then for all $k \in \bbN$, $\mb{1}_{K \times \bbR^n} \td{\mb{G}}_{t_0}$ is bounded from ${\bbX}_{DB^*}^{\mb{p}^\heartsuit}$ to $X^{\mb{p}^\heartsuit + k}$, and this boundedness is uniform in $K$ provided $K_- > t_0 + 1$.
\end{lem}

\begin{proof}
	We will prove the result for tent spaces; the $Z$-space result then follows by real interpolation because the assumption on $\mb{p}$ is open in $(j(\mb{p}),\theta(\mb{p}))$.
	
	Suppose $\gf \in {\bbX}_{DB^*}^{\mb{p}^\heartsuit}$ and write $K = K_0 \cup K_\infty$, where $K_0 \subset (0,t_0)$ and $K_\infty \subset (t_0,\infty)$. 
	For all $x \in \bbR^n$,
	\begin{align*}
		&\mc{A}^2(\gk^{\gq(\mb{p})+1-k} \mb{1}_{K \times \bbR^n} \td{\mb{G}}_{t_0}(\gf))(x)^2 \\
		&= \bigg( \int_{K_0} \int_{B(x,t)} + \int_{K_\infty} \int_{B(x,t)} \bigg) |t^{\gq(\mb{p})+1-k} \td{\mb{G}}_{t_0}(\gf)(t,y)|^2 \, \frac{dy \, dt}{t^{1+n}} \\
		&=: \mb{I}_0 + \mb{I}_\infty.
	\end{align*}
	There exists $\ga > 0$ (depending on $K_0$) such that if $(t_0 - \gt, y) \in (K_0 \times \bbR^n) \cap \gG(x)$, then $(\gt,y) \in \gG^\ga(x)$ (see Figure \ref{fig:badaperture}).
	Thus, using \eqref{eqn:DGident} and that $t_0 - \gt \simeq_K \gt$ when $t_0 - \gt \in K_0$,
	\begin{align*}
		\mb{I}_0 &\leq \dint_{\gG^\ga (x)} \mb{1}_{K_0}(t_0 - \gt) \big| (t_0 - \gt)^{\gq(\mb{p})+1-k} \td{\mb{G}}_{t_0}(\gf)(t_0 - \gt,y) \big|^2 \, \frac{dy \, d\gt}{(t_0 - \gt)^{1+n}} \\
		&\lesssim_{K,k} \dint_{\gG^\ga(x)} \big| \gt^{\gq(\mb{p})+1} e^{-\gt DB^*} \gc^+ (DB^*) \gf(y)\big|^2 \, \frac{dy \, d\gt}{\gt^{1+n}}.
	\end{align*}
	Similarly, there exists $\gb > 0$ such that if $(t_0 + \gs,y) \in (K_1 \times \bbR^n) \cap \gG(x)$, then $(\gs,y) \in \gG^\gb(x)$, and using $2(\gq(\mb{p})+1) - n - 1 < 0$ we have
	\begin{align*}
		\mb{I}_\infty &\leq \dint_{\gG^\gb (x)} \mb{1}_{K_\infty}(t_0 + \gs) \big|(t_0 + \gs)^{\gq(\mb{p})+1-k} e^{\gs DB^*} \gc^- (DB^*) \gf(y) \big|^2 \, \frac{dy \, d\gs}{(t_0+\gs)^{1+n}} \\
		&\leq (K_\infty)_-^{-2k} \dint_{\gG^\gb(x)} (t_0 + \gs)^{2(\gq(\mb{p})+1) - n - 1} \big| e^{\gs DB^*} \gc^- (DB^*) \gf(y) \big|^2 \, dy \, d\gs \\
		&\leq (K_\infty)_-^{-2k} \dint_{\gG^\gb(x)} \gs^{2(\gq(\mb{p})+1) - n - 1} \big| e^{\gs DB^*} \gc^- (DB^*) \gf(y) \big|^2 \, dy \, d\gs \\
		&= (K_\infty)_-^{-2k} \dint_{\gG^\gb(x)} \big|\gs^{\gq(\mb{p})+1} e^{\gs DB^*} \gc^- (DB^*) \gf(y) \big|^2 \, \frac{dy \, d\gs}{\gs^{1+n}}.
	\end{align*}
	Therefore we can estimate
	\begin{align*}
		&\nm{\mb{1}_{K \times \bbR^n} \td{\mb{G}}_{t_0}(\gf)}_{T^{\mb{p}^\heartsuit + k}} \\
		&\lesssim C(K,k)\nm{e^{-tDB^*} \gc^+ (DB^*) \gf}_{T^{\mb{p}^\heartsuit}}
		+ (K_\infty)_-^{-2k}\nm{e^{tDB^*} \gc^- (DB^*) \gf}_{T^{\mb{p}^\heartsuit}} \\
		&\lesssim \nm{\gf}_{{\bbH}_{DB^*}^{\mb{p}^\heartsuit}} \\
		&< \infty
	\end{align*}
	using the semigroup characterisation of the ${\bbH}^{\mb{p}^\heartsuit,\pm}_{DB^*}$ quasinorm (Theorem \ref{thm:sgpnorm-abstract}), which is valid since $i(\mb{p}^\heartsuit) < 2$ and $\gq(\mb{p}^\heartsuit) < 0$.
	Note that if $K_- > t_0 + 1$ then $\mb{I}_0 = 0$, and that the aperture $\gb$ can remain fixed in this argument, which implies the claimed uniformity in $K$ since $K_-^{-2k}$ is bounded in $K_- > t_0 + 1$.
\end{proof}

	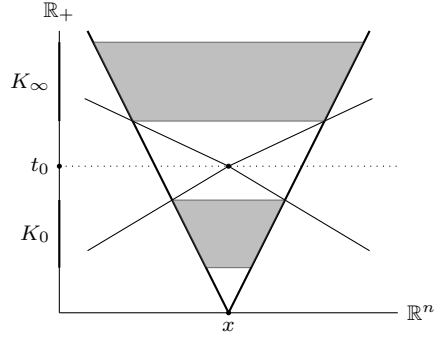
\begin{figure}
\caption{Cones of large aperture, used in Lemma \ref{lem:badaperture1}.}\label{fig:badaperture}
\begin{center}

\begin{tikzpicture}[scale = 1.5, every node/.style={font=\footnotesize}]
	
	%DATA
	\newcommand*{\basept}{1.5};
	\newcommand*{\aperture}{0.5};
	
	\newcommand*{\Kmm}{0.4};   %  make these an increasing sequence!
	\newcommand*{\KmM}{1};   %
	\newcommand*{\To}{1.3};    %
	\newcommand*{\KMm}{1.7};   %
	\newcommand*{\KMM}{2.4};   %
	
	%x-axis with label
	\draw [thin] (0,0) -- (3,0);	
	\node [right] at (3,0) {$\bbR^n$};
	
	%y-axis with label
	\draw [thin] (0,0) -- (0,2.5);	
	\node [above] at (0,2.5) {$\bbR_+$};

	%cone basepoint
	\draw [fill=black] (\basept,0) circle [radius = 0.5pt];
	\node [below] at (\basept,0) {$x$};   
	
	%point t_0
	\draw [fill=black] (0,\To) circle [radius = 0.5pt];
	\node [left] at (0,\To) {$t_0$};
	
	%point (x,t_0)
	\draw [fill = black] (\basept,\To) circle [radius = 0.5pt];
	
	%K_0
	\draw [thick] (0,\Kmm) -- (0,\KmM);
	\node [left] at (${1/2}*(0,\Kmm) + {1/2}*(0,\KmM)$) {$K_0$};
	
	%K_\infty
	\draw [thick] (0,\KMm) -- (0,\KMM);
	\node [left] at (${1/2}*(0,\KMm) + {1/2}*(0,\KMM)$) {$K_\infty$};

	%cone
	\draw [thick] (1.5,0) -- ({\basept - (2.5 * \aperture)}, 2.5);
	\draw [thick] (1.5,0) -- ({\basept + (2.5 * \aperture)}, 2.5);
	
	%lower big cone
	\draw [thin] (${1 - 2.5}*(\basept, \To) + {2.5}*(\basept - \KmM * \aperture,\KmM)$) -- (\basept, \To) -- (${1 - 2.5}*(\basept, \To) + {2.5}*(\basept + \KmM * \aperture,\KmM)$);
	
	%upper big cone
	\draw [thin] (${1 - 1.5}*(\basept, \To) + {1.5}*(\basept - \KMm * \aperture,\KMm)$) -- (\basept, \To) -- (${1 - 1.5}*(\basept, \To) + {1.5}*(\basept + \KMm * \aperture,\KMm)$);
	
	%intersections
	\draw[fill=gray,opacity=0.5] (\basept - \Kmm * \aperture,\Kmm) -- (\basept + \Kmm * \aperture,\Kmm) -- (\basept + \KmM * \aperture,\KmM) -- (\basept - \KmM * \aperture,\KmM);
	
	\draw[fill=gray,opacity=0.5] (\basept - \KMm * \aperture,\KMm) -- (\basept + \KMm * \aperture,\KMm) -- (\basept + \KMM * \aperture,\KMM) -- (\basept - \KMM * \aperture,\KMM);

	\draw [dotted] (0,1.3) -- (3,1.3);

	\end{tikzpicture}
\end{center}
\end{figure}
	
	\begin{cor}\label{cor:badaperture}
		Let $\gf \in \bbD^\mb{p}(X)$ and $k \in \bbN$.
		Then $\mb{1}_{K \times \bbR^n} D\mb{G}_{t_0,\gf} \in X^{\mb{p}^\heartsuit + k}$ for all compact $K \subset \bbR_{+,t_0}$, with uniform boundedness in $K$ provided $K_- > t_0 + 1$.
	\end{cor}
	
	\begin{proof}
		For $\gf \in \bbD^\mb{p}(X)$ we have $D\gf \in {\bbX}_{DB^*}^{\mb{p}^\heartsuit}$ and $D\mb{G}_{t_0,\gf} = \td{\mb{G}}_{t_0}(D\gf)$, so this follows from Lemma \ref{lem:badaperture1}.
	\end{proof}
	
	For $k \in \bbN$, whenever $F \in X^{\mb{p} - k}$ solves $(\CR)_{DB}$ we can invoke Corollary \ref{cor:limit-in-time} when $\gf \in \bbD^\mb{p}(X)$, yielding the equalities \eqref{eqn:limit-in-time-1} and \eqref{eqn:limit-in-time-2} for sufficiently small $\varepsilon > 0$.

\begin{cor}
	Let $\gf \in \bbD^\mb{p}(X)$ and $k \in \bbN$, and suppose that $F \in X^{\mb{p}-k}$ solves $(\CR)_{DB}$.
	Then 
	\begin{equation}\label{eqn:limit-1}
		\lim_{\varepsilon \to 0} \barint_{t_0 + \varepsilon}^{t_0 + 2\varepsilon} \int_{\bbR^n} \left\langle B^* D \mb{G}_{t_0,\gf}(t,x) , F(t,x) \right\rangle \, dx \, dt = 0.
	\end{equation}
\end{cor}

\begin{proof}
	For $\varepsilon < 1/2$ we have $t_0+(2\varepsilon)^{-1} > t_0 + 1$, so by Corollary \ref{cor:badaperture} we have
	\begin{equation*}
		\mb{1}_{[t_0 + (2\varepsilon)^{-1}, t_0 + \varepsilon^{-1}] \times \bbR^n} D\mb{G}_{t_0,\gf} \in X^{\mb{p}^\heartsuit+k+1}
	\end{equation*}
	with uniformly bounded quasinorms.
	Since $F \in X^{\mb{p} - k}$, and since $(\mb{p} - k)^\prime = \mb{p}^\heartsuit + k + 1$, absolute convergence of the $X$-space duality integrals implies that condition \eqref{eqn:spatial-limit-condn} is satisfied, and also that 
	\begin{equation*}
		\int_{t_0 + (2\varepsilon)^{-1}}^{t_0 + \varepsilon^{-1}} \int_{\bbR^n} |\langle B^* D \mb{G}_{t_0,\gf}(t,x), F(t,x) \rangle| \, dx \, dt \lesssim 1
	\end{equation*}
	for all $\varepsilon < 1/2$.
	Therefore we can take the limit as $\varepsilon \to 0$ of both sides of \eqref{eqn:limit-in-time-1} using dominated convergence to conclude that the right hand side vanishes.
\end{proof}

\textbf{Step 3: Weak semigroup properties of solutions.}

\begin{lem}\label{lem:dual-semigp-test}
	Suppose that $F \in X^{\mb{p} - k}$ solves $(\CR)_{DB}$ for some $k \in \bbN$.
	When $t_0 > 0$, $\gt \geq 0$, and $\gf \in \bbD^\mb{p}(X)$, we have
	\begin{equation}\label{eqn:dual-semigp-eqn}
		\langle B^* D \gf, F(t_0 + \gt) \rangle_{E^{\mb{p}^\heartsuit}} = \langle B^* e^{-\gt DB^*} \gc^+(DB^*) D\gf, F(t_0) \rangle_{E^{\mb{p}^\heartsuit}}.
	\end{equation}
\end{lem}

\begin{proof}
	We need to rewrite the integrals in \eqref{eqn:limit-1} and \eqref{eqn:limit-in-time-2} in terms of duality of slice spaces.
	By Proposition \ref{prop:wksolnprops}, $F(t)$ is in $E^\mb{p}$ for each $t \in \bbR_+$.
	By Lemma \ref{lem:sgp-in-slice}, since $D\gf \in {\bbX}_{DB^*}^{\mb{p}^\heartsuit}$, we have that $B^* D \mb{G}_{t_0,\gf}(t)$ is in $E^{\mb{p}^\heartsuit}$.
	Hence
	\begin{equation*}
		\int_{\bbR^n} \langle B^* D \mb{G}_{t_0,\gf} (t,x), F(t,x) \rangle \, dx = \langle B^* D \mb{G}_{t_0,\gf}(t), F(t) \rangle_{E^{\mb{p}^\heartsuit}}
	\end{equation*}
	by the slice space duality identification of Proposition \ref{prop:slice-duality}.
	Therefore \eqref{eqn:limit-1} and \eqref{eqn:limit-in-time-2} can be rewritten as
	\begin{equation}\label{eqn:limit-1-slicedual}
		\lim_{\varepsilon \to 0} \barint_{t_0 + \varepsilon}^{t_0 + 2\varepsilon} \langle B^* D \mb{G}_{t_0,\gf}(t), F(t) \rangle_{E^{\mb{p}^\heartsuit}} \, dt = 0 
	\end{equation}
	and
	\begin{equation}\label{eqn:limit-in-time-2-slicedual}
		-\lim_{\varepsilon \to 0} \barint_{\varepsilon}^{2\varepsilon} \langle B^* D \mb{G}_{t_0,\gf}(t_0 - t), F(t_0 - t) \rangle_{E^{\mb{p}^\heartsuit}} \, dt = \lim_{\varepsilon \to 0} \barint_{\varepsilon}^{2\varepsilon} \langle B^* D \mb{G}_{t_0,\gf}(t), F(t) \rangle_{E^{\mb{p}^\heartsuit}} \, dt.
	\end{equation}
	
	We evaluate these limits by continuity of the integrands.
	By Proposition \ref{prop:wksolnprops} we have $F \in C^\infty(\bbR_+ : E^\mb{p})$.
	By the definition of $\bbD^\mb{p}(X)$, we have that $B^* D \mb{G}_{t_0,\gf}(t) \in E^{\mb{p}^\heartsuit}$ for all $t \in \bbR_{+,t_0}$, with 
	\begin{equation*}
		\lim_{t \searrow t_0} B^* D \mb{G}_{t_0,\gf}(t) = -B^* D\gc^- (B^* D) \bbP_{B^* D} \gf 
		=  -B^* \gc^- (DB^*) D \gf
	\end{equation*}
	in $E^{\mb{p}^\heartsuit}$ by Corollary \ref{cor:sgp-slice-bddness}.
	Therefore \eqref{eqn:limit-1-slicedual} becomes
	\begin{equation}
		\langle B^* \gc^- (DB^*) D \gf , F(t_0) \rangle_{E^{\mb{p}^\heartsuit}} = 0. \label{eqn:dual-expn-1}
	\end{equation}
	
	Next we will prove
	\begin{equation}
	\langle B^* e^{-\gt DB^*} \gc^+ (DB^*) D \gf, F(t_0) \rangle_{E^{\mb{p}^\heartsuit}}
		= \langle B^* \gc^+ (DB^*) D \gf, F(t_0 + \gt) \rangle_{E^{\mb{p}^\heartsuit}} \label{eqn:dual-expn-2}
	\end{equation}
	by taking the limit of the left hand side of \eqref{eqn:limit-in-time-2-slicedual} and exploiting an algebraic property of the right hand side.	
	Summing \eqref{eqn:dual-expn-1} (at $t_0 + \gt$) and \eqref{eqn:dual-expn-2} will yield \eqref{eqn:dual-semigp-eqn} and complete the proof.
	
	For $\gf \in \bbD^\mb{p}(X)$ and $\gd \geq 0$, define
	\begin{equation*}
		I_{t_0,\gf}^{\varepsilon,\gd} := \barint_\varepsilon^{2\varepsilon} \langle B^* D \mb{G}_{t_0, \gf_\gd}(t), F(t) \rangle_{E^{\mb{p}^\heartsuit}} \, dt
	\end{equation*}
	where $\gf_\gd := e^{-\gd[B^* D]} \gf$.
	By Lemma \ref{lem:test-sgp-closure}, $\gf_\gd$ is in $\bbD^\mb{p}(X)$, and so we can apply \eqref{eqn:limit-in-time-2-slicedual} to get
	\begin{align*}
		\lim_{\varepsilon \to 0} I_{t_0,\gf}^{\varepsilon,\gd}
		&= -\lim_{\varepsilon \to 0} \barint_\varepsilon^{2\varepsilon} \langle B^* D \mb{G}_{t_0,\gf_\gd}(t_0 - t), F(t_0 - t) \rangle_{E^{\mb{p}^\heartsuit}} \, dt \\
		&= -\lim_{\varepsilon \to 0} \barint_\varepsilon^{2\varepsilon} \langle B^* e^{-tDB^*}  e^{-\gd DB^*} \gc^+(DB^*)D\gf (t), F(t) \rangle_{E^{\mb{p}^\heartsuit}} \, dt \\
		&= -\langle B^* e^{-\gd DB^*} \gc^+(DB^*)  D\gf(t_0), F(t_0) \rangle_{E^{\mb{p}^\heartsuit}}
	\end{align*}
	using the same argument as in the previous paragraph to establish the final equality.
	A simple computation shows that $I_{t_0, \gf}^{\varepsilon,\gd} = I_{t_0 + \gd, \gf}^{\varepsilon, 0}$, so we can conclude that
	\begin{align*}
		\langle B^* e^{-\gt DB^*} \gc^+ (DB^*) D \gf, F(t_0) \rangle_{E^{\mb{p}^\heartsuit}}
		&= -\lim_{\varepsilon \to 0} I_{t_0,\gf}^{\varepsilon, \gt} \nonumber \\
		&= - \lim_{\varepsilon \to 0} I_{t_0 + \gt,\gf}^{\varepsilon,0} \nonumber \\
		&= \langle B^* \gc^+ (DB^*) D \gf, F(t_0 + \gt) \rangle_{E^{\mb{p}^\heartsuit}},
	\end{align*}
	completing the proof.
\end{proof}

We can use this lemma, using that $E^\mb{p} \subset \mc{S}^\prime$, to see what happens when we test against Schwartz functions.

\begin{cor}\label{cor:schwartz-testing}
	Let $F$, $t_0$, and $\gt$ be as in Lemma \ref{lem:dual-semigp-test}, and suppose $\gf \in \mc{S}$.
	Then
	\begin{equation}\label{eqn:schwartz-testing}
		-\langle \gf, (\partial_t F)(t_0 + \gt) \rangle_{\mc{S}} = \langle B^* e^{-\gt DB^*} \gc^+(DB^*) D\gf, F(t_0) \rangle_{E^{\mb{p}^\heartsuit}}.
	\end{equation}
\end{cor}

\begin{proof}
	By Lemma \ref{lem:test-schwartz-containment}, $\mc{S} \subset \bbD^\mb{p}(X)$, so we can apply Lemma \ref{lem:dual-semigp-test} to $\gf$.
	Since $F(t_0 + \gt)$ and $(\partial_t F)(t_0 + \gt)$ are in $E^\mb{p}$, and $\gf$ and $B^* D \gf$ are in $E^{\mb{p}^\heartsuit}$, we can apply integration by parts in slice spaces (Proposition \ref{prop:slice-IBP}) to derive \eqref{eqn:schwartz-testing}.
\end{proof}

\textbf{Step 4: A reproducing formula for $(\partial_t F)(t_0)$ in terms of higher $t$-derivatives.}

\begin{lem}\label{lem:FHs-smoothness}
	Let $t_0 > 0$, $k \in \bbN_+$, and suppose that $F \in X^\mb{p}$ solves $(\CR)_{DB}$.
	Then $(\partial_t^k F)(t_0) \in {\mb{X}}_{D}^\mb{p}$, with
	\begin{equation}\label{eqn:deriv-H-est}
		\nm{(\partial_t^k F)(t_0)}_{{\mb{X}}_{D}^\mb{p}} \lesssim t_0^{-k} \nm{F}_{X^\mb{p}}.
	\end{equation}
\end{lem}

\begin{proof}
	Suppose $\gf \in \mc{S}$.
	First note that since $\partial_t^{k-1} F$ solves $(\CR)_{DB}$, and is in $X^{\mb{p} - (k-1)}$ by Proposition \ref{prop:wksolnprops}, Corollary \ref{cor:schwartz-testing} yields
	\begin{equation}\label{eqn:schwartz-testing-diff}
		-\langle \gf, (\partial_t^k F)(t_0/2 + \gt) \rangle_{\mc{S}} = \langle B^* e^{-\gt DB^*} \gc^+(DB^*) D\gf, (\partial_t^{k-1} F)(t_0/2) \rangle_{E^{\mb{p}^\heartsuit}}.
	\end{equation}
	Applying this with $\gt = t_0/2$ and using the slice space estimates of Lemma \ref{lem:sgp-in-slice} and Proposition \ref{prop:wksolnprops}, slice space duality, and $\mb{p}^\heartsuit \in I(\mb{X},DB^*)$, we have
	\begin{align*}
		&\left| \int_{\bbR^n} \langle \gf(x), \partial_t^k F(t_0)(x) \rangle \, dx \right| \\
		&\leq \nm{B^* e^{-t_0 DB^* /2} \gc^+(DB^*) D\gf}_{E^{\mb{p}^\heartsuit}(t_0/2)} \nm{(\partial_t^{k-1}F)(t_0/2)}_{E^{\mb{p}+1}(t_0/2)} \\
		&\lesssim \nm{D\gf}_{{\bbX}^{\mb{p}^\heartsuit}_{DB^*}} t_0^{-k} \nm{(\partial_t^{k-1} F)(t_0/2)}_{E^{\mb{p} - (k-1)}(t_0)} \\
		&\lesssim \nm{D\gf}_{{\mb{X}}^{\mb{p}^\heartsuit}_{D}} t_0^{-k} \nm{\partial_t^{k-1} F}_{X^{\mb{p} -(k-1)}} \\
		&\lesssim \nm{\gf}_{{\mb{X}}^{\mb{p}^\prime}} t_0^{-k} \nm{F}_{X^\mb{p}}.
	\end{align*}
	Since $\gf$ was arbitrary, this implies that $(\partial_t^k F)(t_0) \in ({\mb{X}}^{\mb{p}^\prime})^\prime = {\mb{X}}^\mb{p}$ with the norm estimate \eqref{eqn:deriv-H-est}.
	Furthermore, since $\partial_t^k F$ solves $(\CR)_{DB}$, each $(\partial_t^k F)(t_0)$ is in $\overline{\mc{R}(DB)} = \overline{\mc{R}(D)}$, which implies membership in ${\mb{X}}^\mb{p}_D$.
\end{proof}

We recall the following elementary lemma.

\begin{lem}\label{lem:reconstruction}
	Suppose $k \in \bbN_+$ and $g \in C^k(\bbR^+:\bbC)$, with $t^j g^{(j)}(t) \to 0$ as $t \to \infty$ for all integers $0 \leq j \leq k-1$.
	Then for all $t > 0$ we have
	\begin{equation*}
		g(t) = \frac{(-1)^k}{(k-1)!} \int_t^\infty g^{(k)}(\gt)(\gt-t)^{k-1} \, d\gt.
	\end{equation*}
\end{lem}

\begin{cor}\label{cor:F-reproduction}
	Suppose that $F \in X^\mb{p}$ solves $(\CR)_{DB}$.
	Then for all $t_0 > 0$ and $\gf \in \mc{S}$ we have
	\begin{equation*}
		\langle \gf, (\partial_t F)(t_0) \rangle_{\mc{S}} = \frac{(-1)^k}{(k-1)!} \int_{t_0}^\infty \langle \gf, (\partial_t^{k+1} F)(t) \rangle_{E^{\mb{p}^\prime}} (t-t_0)^{k-1} \, dt.
	\end{equation*}
\end{cor}

\begin{proof}
	By Lemma \ref{lem:FHs-smoothness}, the function $t_0 \mapsto (\partial_{t} F)(t_0)$ is in $C^\infty(\bbR_+ : {\mb{X}}_{D}^\mb{p})$.
	Therefore for all $\gf \in \mc{S}$ the function $g_\gf$ defined by
	\begin{equation*}
		g_\gf(t_0) := \langle \gf, (\partial_t F)(t_0) \rangle_{\mc{S}}
	\end{equation*}
	is in $C^\infty(\bbR_+ : \bbC)$,
	and for $k \in \bbN_+$ we have
	\begin{equation*}
		g_\gf^{(k)}(t_0) = \langle \gf, (\partial_{t}^{k+1} F)(t_0) \rangle_{\mc{S}} = \langle \gf, (\partial_{t}^{k+1} F)(t_0) \rangle_{E^{\mb{p}^\prime}}.
	\end{equation*}
	Furthermore, by the same lemma, we have
	\begin{equation*}
		|t_0^k g_\gf(t_0)|
		= t_0^k \big|\langle \gf,(\partial_{t} F)(t_0) \rangle_{{\mb{X}}_{D}^{p^\prime}}\big| 
		\lesssim_{\gf,F} t_0^{-1},
	\end{equation*}
	so the hypotheses of Lemma \ref{lem:reconstruction} are satisfied, and the result follows.
\end{proof}

\textbf{Step 5: Construction of associated `nice' solutions.}

Since $\mb{p}^\heartsuit \in I(\mb{X},DB^*)$, we have identified ${\mb{X}}_{DB}^\mb{p} := {\mb{X}}_{D}^\mb{p}$ as a completion of ${\bbX}_{DB}^\mb{p}$.
In this step of the proof, given a solution $F \in X^\mb{p}$ of $(\CR)_{DB}$, we will construct distributions modulo polynomials $\td{F}(t_0) \in {\mb{X}}_{D}^\mb{p}$ which satisfy the properties we want to show for $F(t_0)$.
In the remaining steps we will show that $\td{F}(t_0) = F(t_0)$, which will complete the proof.

\begin{lem}\label{lem:derivative-shift-containment}
	Suppose $F \in X^\mb{p}$ solves $(\CR)_{DB}$.
	Then for all $t_0 \in [0,\infty)$ and for sufficiently large $N \in \bbN$ we have
	\begin{equation*}
		\nm{(t,y) \mapsto t^N(\partial_t^N F)(t_0 + t,y)}_{X^{\mb{p}}} \lesssim \nm{F}_{X^\mb{p}}.
	\end{equation*}
\end{lem}

\begin{proof}
	This is an immediate corollary of Propositions \ref{prop:s12shift} and \ref{prop:wksolnprops}.
\end{proof}

Let $F \in X^\mb{p}$ solve $(\CR)_{DB}$.
For $N \in \bbN$ large enough that Lemma \ref{lem:derivative-shift-containment} applies, define $\gz \in \gY_1^\infty$ by
\begin{equation*}
	\gz(z) := c_N z e^{-[z]/2}
\end{equation*}
where $c_N = (-1)^{N+1}/N!$.
For $k \in \bbN$ define $\gc_k := \mb{1}_{[k^{-1},k] \times B(0,k)}$, and for all $t_0 \geq 0$ define
\begin{equation}\label{eqn:nice-soln}
	\td{F}_k(t_0) := \bbS_{\gz,DB}\bigg[t \mapsto \gc_k t^N (\partial_t^N F)\bigg(t_0 + \frac{t}{2}\bigg)\bigg].
\end{equation}
This is well-defined since the function $t \mapsto \gc_k t^N (\partial_t^N F)(t_0 + (t/2))$ is in $X^2$.
Lemma \ref{lem:derivative-shift-containment} and Proposition \ref{prop:Xss-charn} (using $\gz \in \gY_1^\infty \subset \gY_{-\gq(\mb{p})+}^{(\gq(\mb{p}) + n|\frac{1}{2} - j(\mb{p}))+}$, which requires $\gq(\mb{p}) > -1$) imply that each $\td{F}_k(t_0)$ is in ${\bbX}_{DB}^\mb{p}$, with
\begin{align*}
	\nm{\td{F}_k(t_0)}_{{\bbX}_{DB}^\mb{p}} &\lesssim \bigg\lVert t \mapsto \gc_k t^N(\partial_t^N F)\bigg(t_0 + \frac{t}{2} \bigg) \bigg\rVert_{X^\mb{p}} \ \\
	&\lesssim \nm{F}_{X^\mb{p}}.
\end{align*}
The functions $[t \mapsto \gc_k t^N(\partial_t^N F)\left(t_0 + \frac{t}{2}\right)]$ converge to $[t \mapsto t^N(\partial_t^N F)\left(t_0 + \frac{t}{2} \right)]$ in $X^{\mb{p}}$ as $k \to \infty$, so we can define $\td{F}(t_0) \in \mb{X}_{DB}^\mb{p}$ by
\begin{equation*}
	\td{F}(t_0) := \mb{S}_{\gz,DB}\bigg[t \mapsto t^N (\partial_t^N F)\bigg(t_0 + \frac{t}{2}\bigg)\bigg].
\end{equation*}
This satisfies the norm estimate
\begin{equation}\label{eqn:good-norm-est}
	\nm{\td{F}(t_0)}_{{\mb{X}}_{DB}^\mb{p}} \lesssim \nm{F}_{X^\mb{p}}.
\end{equation}

\begin{lem}\label{lem:test-dual-identn}
		Let $t_0 \geq 0$.
		Suppose $F \in X^\mb{p}$ solves $(\CR)_{DB}$ and define $\td{F}(t_0) \in {\mb{X}}_{DB}^\mb{p}$ as in the previous paragraphs.
		Suppose also that $\gfv \in {\bbX}_{B^* D}^{\mb{p}^\prime} \cap \mc{D}(B^* D)$.
	Then
	\begin{align}
          \langle \gfv, &\td{F}(t_0) \rangle_{{\mb{X}}_{B^*D}^{\mb{p}^\prime}} \nonumber \\       
		&= -c_N \int_0^\infty \left\langle t B^* e^{-\frac{t}{2} [DB^*]} D\gfv, t^N (\partial_t^N F)\left(t_0 + \frac{t}{2}\right)\right\rangle_{E^{\mb{p}^\prime}} \, \frac{dt}{t}. \label{eqn:test-dual-identn}
	\end{align}
\end{lem}

\begin{proof}
	First we show that the $E^{\mb{p}^\prime}$ duality pairing \eqref{eqn:test-dual-identn} makes sense.
	Since $\gfv \in {\bbX}_{B^* D}^{\mb{p}^\prime} \cap \mc{D}(B^* D)$, Lemma \ref{lem:int-slice} yields $e^{-t[DB^*]/2} D\gfv \in E^{\mb{p}^\prime}$.
	Since each $t B^*$ is a bounded operator on $E^{\mb{p}^\prime}$ we have $t B^* e^{-t[DB^*]/2} D\gfv \in E^{\mb{p}^\prime}$.
	On the other hand, since $t\mapsto (\partial_t ^N F)(t_0 + t/2)$ solves $(\CR)_{DB}$, by Proposition \ref{prop:wksolnprops} and Lemma \ref{lem:derivative-shift-containment} we have
	\begin{equation*}
		\nm{(\partial_t^N F)(t_0 + t/2)}_{E^{\mb{p}-N}(t)}
		\lesssim \nm{t \mapsto (\partial_t^N F)(t_0 + t/2)}_{X^{\mb{p} - N}} 
		\lesssim \nm{F}_{X^\mb{p}}
	\end{equation*}
	for all $t > 0$.
	Therefore the slice space dual pairing in \eqref{eqn:test-dual-identn} is meaningful.
	
	Now write
	\begin{align}
		&\langle \gfv, \td{F}(t_0) \rangle_{{\mb{X}}_{B^*D}^{\mb{p}^\prime}} \\
		&= \lim_{k \to \infty} \langle \gfv, \td{F}_k(t_0) \rangle_{{\bbX}_{B^* D}^{\mb{p}^\prime}} \nonumber\\
		&= \lim_{k \to \infty} \langle \bbQ_{\td{\gz}, B^* D} \gfv, [t \mapsto \gc_k t^N (\partial_t^N F)(t_0 + t/2)] \rangle_{X^\mb{p}} \nonumber\\
		&= \langle \bbQ_{\wtd{\gz}, B^* D} \gfv, [t \mapsto t^N (\partial_t^N F)(t_0 + t/2)] \rangle_{X^\mb{p}}\nonumber\\
		&= -c_N \dint_{\bbR^{1+n}_+} \big(t (B^*D e^{-t[B^* D]/2} \gfv)(x) , t^N (\partial_t^N F)(t_0 + t/2,x)\big) \, dx \, \frac{dt}{t} \nonumber\\
		&= -c_N \int_0^\infty \langle t B^* e^{-t[DB^*]/2} D\gfv, t^N (\partial_t^N F)(t_0 + t/2) \rangle_{E^{\mb{p}^\prime}} \, \frac{dt}{t} \nonumber
	\end{align}
	using $\wtd{\gz} = \gz$ and the slice space containments from the previous paragraph.
\end{proof}

Now we show that the distributions modulo polynomials $(\td{F}(t_0))_{t_0 \geq 0}$ are in fact given by the Cauchy operator applied to $\td{F}(0)$.
\begin{prop}\label{prop:f-semigp-equation}
	Let $F \in X^\mb{p}$ solve $(\CR)_{DB}$ and define $\td{F}$ as above.
	Then for all $t_0 \geq 0$ we have
	\begin{equation*}
		\td{F}(t_0) = e^{-t_0 [DB]} \gc^+(DB) \td{F}(0).
	\end{equation*}
	In particular, $\td{F}(0) \in {\mb{X}}_{DB}^{\mb{p},+}$, and so $\td{F} = \mb{C}_{DB}^+ (\td{F}(0))$.
\end{prop}

\begin{proof}
	Since $\td{F}(t_0) \in {\mb{X}}_{DB}^\mb{p}$ and since ${\bbX}_{B^* D}^{\mb{p}^\prime} \cap \mc{D}(B^* D)$ is dense in ${\mb{X}}_{B^* D}^{\mb{p}^\prime}$ (Corollary \ref{cor:adapted-space-int-dens} and density of $\mc{D}(B^* D)$ in $\bbX_{B^* D}^2$), it suffices to test against $\gfv \in {\bbX}_{B^* D}^{\mb{p}^\prime} \cap \mc{D}(B^* D)$.
	For all such $\gfv$ write
	\begin{align*}
		&\langle \gfv, e^{-t_0 [DB]} \gc^+(DB) \td{F}(0) \rangle_{{\mb{X}}_{B^* D}^{\mb{p}^\prime}} \\
		&= \langle e^{-t_0 [B^* D]} \gc^+(B^* D) \gfv, \td{F}(0) \rangle_{{\mb{X}}_{B^*D}^{\mb{p}^\prime}} \\
		&= -c_N \int_0^\infty \big \langle t B^* e^{-\frac{t}{2}[DB^*]} D\big( e^{-t_0 [B^* D]} \gc^+(B^* D)\gfv \big), t^N (\partial_t^N F)(t/2)\big\rangle_{E^{\mb{p}^\prime}} \, \frac{dt}{t} \\
		&= -c_N \int_0^\infty \big\langle t B^* e^{-t_0 [DB^*]} \gc^+(DB^*) D\big( e^{-\frac{t}{2}[B^* D]}\gfv\big), t^N (\partial_t^N F)(t/2)\big\rangle_{E^{\mb{p}^\prime}} \, \frac{dt}{t} \\
		&= -c_N \int_0^\infty \big\langle t B^*D \big( e^{-\frac{t}{2} [B^*D]}\gfv\big), t^N (\partial_t^N F)(t_0 + t/2) \big\rangle_{E^{\mb{p}^\prime}} \, \frac{dt}{t} \\
		&= -c_N \int_0^\infty \big\langle t B^* e^{-\frac{t}{2}[DB^*]}D\gfv, t^N (\partial_t^N F)(t_0 + t/2 )\big\rangle_{E^{\mb{p}^\prime}} \, \frac{dt}{t} \\
		&= \langle \gfv, \td{F}(t_0) \rangle_{{\mb{X}}_{B^*D}^{\mb{p}^\prime}}.
	\end{align*}
	In the third line we used that $e^{-t_0 [B^*D]} \gc^+(B^*D)$ maps ${\bbX}_{B^* D}^{\mb{p}^\prime} \cap \mc{D}(B^* D)$ into itself, and the representation \eqref{eqn:test-dual-identn}.
	In the fifth line we used Lemma \ref{lem:dual-semigp-test}, which is valid since $e^{-t[B^*D]/2}\gfv \in \bbD^\mb{p}(X)$ (Lemma \ref{lem:int-slice}) and since $[t \mapsto (\partial_t F)(t/2)] \in X^{\mb{p} - 1}$ solves $(\CR)_{DB}$.
	We use the representation \eqref{eqn:test-dual-identn} once more in the last line.
\end{proof}

This immediately implies the following corollary.

\begin{cor}\label{cor:f-semigp-equation-I}
	Let $F \in X^\mb{p}$ solve $(\CR)_{DB}$.
	Then $\td{F}(0) \in {\mb{X}}_{D}^{\mb{p},+}$, $\td{F} = \mb{C}_{DB}^+(\td{F}(0))$, and $\td{F} \in X^\mb{p}$.
\end{cor}

\begin{proof}
	All we need to show is that $\td{F}$ is in $X^\mb{p}$.
	This follows from Theorem \ref{thm:sgpnorm-concrete}.
\end{proof}

\textbf{Step 6: Equality of $\partial_t F$ and $\partial_t \td{F}$.}

By Corollary \ref{cor:f-semigp-equation-I} and Proposition \ref{prop:sgp-continuity}, for $F \in X^\mb{p}$ which solves $(\CR)_{DB}$, the function $t_0 \mapsto \td{F}(t_0)$ is in $C^\infty(\bbR_+ : {\mb{X}}_{D}^\mb{p})$.
Therefore we can consider $(\partial_t \td{F})(t_0) \in {\mb{X}}_{D}^\mb{p}$ as a distribution modulo polynomials.

\begin{lem}\label{lem:schwartz-dual-identn}
	Let $F \in X^\mb{p}$ solve $(\CR)_{DB}$.
	Then for all $t_0 > 0$ we have $(\partial_t F)(t_0) = (\partial_t \td{F})(t_0)$ in $\mc{Z}^\prime$.
\end{lem}

\begin{proof} 
	Fix $\gf \in \mc{Z}$.
	For all $k \in \bbN$ we have already computed (using that everything is in $L^2$)
	\begin{align}
		&\langle \gf, \td{F}_k(t_0) \rangle_{\mc{Z}} \nonumber \\
		&= -c_N \dint_{\bbR^{1+n}_+} \big(t (B^* e^{-t[DB^*]/2} D\gf)(x) , \gc_k t^N (\partial_t^N F)(t_0 + t/2,x)\big) \, dx \, \frac{dt}{t}. \label{eqn:schwartz-comp}
	\end{align}
	Since $\gf \in \mc{Z}$ we have $D\gf \in {\mb{X}}_{D}^{\mb{p}^\prime}$, so for each $t > 0$ we may apply the (extended operator) $e^{-t[DB^*]/2}$ to $D\gf$.
	We then have
	\begin{align}
		&\nm{t \mapsto t B^* e^{-t[DB^*]/2} D\gf}_{X^{\mb{p}^\prime}} \nonumber\\
		&= \nm{t \mapsto B^* e^{-t[DB^*]/2} D\gf}_{X^{\mb{p}^\heartsuit}} \nonumber\\
		&\lesssim \nm{D\gf}_{{\mb{X}}_{D}^{\mb{p}^\heartsuit}} \label{eqn:rmk-control} \\
		&\simeq \nm{\gf}_{{\mb{X}}^{\mb{p}^\prime}} < \infty, \nonumber
	\end{align}
	where \eqref{eqn:rmk-control} follows from Proposition \ref{prop:equivalent-quasinorms} since $[z \mapsto z e^{-[z]/2}] \in \gY({\bbX}_{D}^{\mb{p}^\heartsuit})$ (here we use $i(\mb{p}^\heartsuit) < 2$ and $\gq(\mb{p}^\heartsuit) < 0$).
	Since $[t \mapsto t^N \partial_t^N F(t_0 + t/2)] \in X^\mb{p}$ (Lemma \ref{lem:derivative-shift-containment}), the integral \eqref{eqn:schwartz-comp} is uniformly bounded in $k$ and so we can take the limit
	\begin{align*}
		\langle \gf, \td{F}(t_0) \rangle_{\mc{Z}}
		&= \lim_{k \to \infty} \langle \gf, \td{F}_k(t_0) \rangle_{\mc{Z}} \\
		&= -c_N \dint_{\bbR^{1+n}_+} \big(t(B^* e^{-t[DB^*]/2} D\gf)(x) , t^N (\partial_t^N F)(t_0 + t/2,x)\big) \, dx \, \frac{dt}{t}
	\end{align*}
	by dominated convergence.
	Using dominated convergence again, we can take the derivative:
	\begin{align*}
		&\langle \gf,(\partial_{t}\td{F})(t_0) \rangle_{\mc{Z}} \\
		&= \partial_{t} \langle \gf,\td{F}(t_0) \rangle_{\mc{Z}} \\
		&= -c_N \dint_{\bbR^{1+n}_+} \big(t (B^* e^{-t[DB^*]/2} D\gf)(x) , t^N (\partial_t^{N+1} F)(t_0 + t/2,x)\big) \, dx \, \frac{dt}{t} \\
		&= -c_N \int_0^\infty \big\langle t B^* e^{-t[DB^*]/2} D\gf, t^N (\partial_t^{N+1} F)(t_0 + t/2) \big\rangle_{E^{\mb{p}^\prime}} \, \frac{dt}{t}
	\end{align*}
	using that $\gf \in \bbD^\mb{p}(X)$ (Lemma \ref{lem:test-schwartz-containment}) to conclude that the slice space duality pairing is meaningful as in the proof of Lemma \ref{lem:test-dual-identn}.
	
	Now we rearrange:
	\begin{align}
		&\big\langle t B^* e^{-t[DB^*]/2} D\gf, t^N (\partial_t^{N+1} F)(t_0 + t/2) \big\rangle_{E^{\mb{p}^\prime}} \nonumber\\
		&=\big\langle t B^*D \big( e^{-t[B^*D]/2} \gf \big), t^N (\partial_t^{N+1} F)(t_0 + t/2) \big\rangle_{E^{\mb{p}^\prime}} \label{line:1}\\
		&=\big\langle t B^* \gc^+(DB^*) D e^{-t[B^*D]/2} \gf, t^N (\partial_t^{N+1} F)(t_0 + t/2) \big\rangle_{E^{\mb{p}^\prime}} \label{line:2}\\
		&= t^{N+1} \big\langle B^* e^{-t[DB^*]/2} \gc^+(DB^*) D\gf, (\partial_{t}^{N+1} F)(t_0 + t/2) \big\rangle_{E^{\mb{p}^\prime}} \label{line:3}\\
		&= t^{N+1} \big\langle B^* D\gf, (\partial_{t}^{N+1} F)(t_0 + t)\big\rangle_{E^{\mb{p}^\prime}} \label{line:5}\\
		&= -t^{N+1} \big\langle \gf,(\partial_{t}^{N+2} F)(t_0 + t)\big\rangle_{E^{\mb{p}^\prime}}. \label{line:6}
	\end{align}
	The first line \eqref{line:1} uses that $\gf \in \mc{D}(D) = \mc{D}(B^* D)$, \eqref{line:2} uses \eqref{eqn:dual-expn-1} and the fact that $e^{-t[B^* D]/2} \gf$ is in $\mc{D}^\mb{p}(X)$ (Lemma \ref{lem:test-sgp-closure}), \eqref{line:3} is just similarity of functional calculi and rearrangement, \eqref{line:5} uses the weak semigroup property \eqref{eqn:dual-semigp-eqn}, and \eqref{line:6} finishes with integration by parts in slice spaces (Proposition \ref{prop:slice-IBP}) and $(\CR)_{DB}$.
	
	Therefore we have
	\begin{align*}
		\langle \gf, (\partial_{t}\td{F})(t_0) \rangle_{\mc{Z}}
		&= c_N \int_0^\infty \left\langle \gf,(\partial_{t}^{N+2} F)(t_0 + t)\right\rangle_{E^{\mb{p}^\prime}} t^{N+1} \, \frac{dt}{t} \\
		&= \frac{(-1)^{N+1}}{N!} \int_{t_0}^\infty \left\langle \gf, (\partial_t^{N+2} F)(t) \right\rangle_{E^{\mb{p}^\prime}} (t-t_0)^N \, dt.
	\end{align*}
	
	Finally, applying Corollary \ref{cor:F-reproduction} with $k = N+1$, we get
	\begin{equation*}
		\langle \gf, (\partial_t F)(t_0) \rangle_\mc{Z} = \langle \gf, (\partial_t \td{F})(t_0) \rangle_{\mc{Z}}
	\end{equation*}
	for all $\gf \in \mc{Z}$ and all $t_0 > 0$.
	Therefore we have $(\partial_t F)(t_0) = (\partial_t \td{F})(t_0)$ in $\mc{Z}^\prime$ for all $t_0 > 0$ as claimed.
\end{proof}

\textbf{Step 7: Completing the proof.}

\begin{lem}\label{lem:final-lem}
	Let $F \in X^\mb{p}$ solve $(\CR)_{DB}$ with $\lim_{t \to \infty} F(t)_\parallel = 0$ in $\mc{Z}^\prime(\bbR^n : \bbC^{nm})$.
	Then $F = \td{F}$.
\end{lem}

\begin{proof}
	By Lemma \ref{lem:schwartz-dual-identn} we have $\partial_t F = \partial_t \td{F}$ in $\mc{Z}^\prime$, so there exists $G \in \mc{Z}^\prime$ such that $F(t_0) = G + \td{F}(t_0)$ for all $t_0 \in \bbR_+$.
	Since $\lim_{t_0 \to \infty} \td{F}(t_0) = 0$ in $\mb{X}_D^{\mb{p}}$ (Proposition \ref{prop:sgp-continuity}, using the weak-star topology when $\mb{p}$ is infinite) and hence also in $\mc{Z}^\prime$, we find that $G_\parallel = 0$.
	Following the argument of \cite[Step 5, page 50]{AM15}, we find that $G = \gb a$ modulo polynomials,
	where $a$ is invertible in $L^\infty$ and $\gb \in \bbC^m$.
	To complete the proof it suffices to show that $\gb = 0$.

	Note that the constant function $[t \mapsto G = F(t) - \td{F}(t)]$ is in $X^\mb{p}$.
	If $\mb{p}$ is finite, then $G \in E^\mb{p}$ (since $[t \mapsto G]$ solves $(\CR)_{DB}$), and this forces $\gb = 0$.
	If $\mb{p}$ is infinite, then first we note that if $\gb \neq 0$ then $[t \mapsto G] \notin T^\infty_{-1;\td{\ga}}$ for all $\td{\ga} \in [0,1)$: this follows from estimating
	\begin{align*}
		\nm{t \mapsto G}_{T^\infty_{-1;\td{\ga}}} &\gtrsim \frac{1}{R^{\td{\ga}}} \bigg( \frac{1}{R^n} \int_{B_R} \int_0^R |t^{-1} G(y,t)|^2 \, \frac{dt}{t} \, dy \bigg)^{1/2} \\
		&\simeq R^{-\td{\ga} - \frac{n}{2}} \bigg( \int_{B_R} R^{-2} |\gb|^2 |a(y)|^2 \, dy \bigg)^{1/2} \\
		&\geq R^{-\td{\ga} - 1} |\gb| \nm{a^{-1}}_\infty^{-1}
	\end{align*}
	for all balls $B_R$ of radius $R > 0$ in $\bbR^n$, and then taking $R \to 0$.
	Since $\gq(\mb{p}) > -1$ we have
	\begin{equation*}
		G \in X^\mb{p} \hookrightarrow T^\infty_{-1;1 + r(\mb{p})},
	\end{equation*}
	and since $r(\mb{p}) < 0$ (this is the only time we use this hypothesis), we must have $\gb = 0$.
	This completes the proof.
\end{proof}

Therefore, by Corollary \ref{cor:f-semigp-equation-I}, under the assumptions of Theorem \ref{thm:mainthm-gtr2}, we have that $F = \td{F} = \mb{C}_{DB}^+ (\td{F}(0))$, with $\td{F}(0) \in {\mb{X}}_{D}^{\mb{p},+}$ such that $\nm{ \td{F}(0) }_{\mb{X}_D^{\mb{p},+}} \lesssim \nm{ F }_{X^\mb{p}}$ (by \eqref{eqn:good-norm-est}).
Furthermore, if $f \in \mb{X}_{DB}^{\mb{p},+}$ and $F = \mb{C}_{DB}^+ f$, then by Proposition \ref{prop:sgp-continuity} we have
\begin{equation*}
	f = \lim_{t \to 0} \mb{C}_{DB}^+ F(t) = \td{F}(0)
\end{equation*}
with limit in $\mb{X}_{DB}^{\mb{p}}$.
This completes the proof of Theorem \ref{thm:mainthm-gtr2}.

\section{Interpolation of solution spaces}\label{sec:solspace-interpolation}\index{interpolation!of solution spaces}

As an interesting aside, we can now show that certain solution spaces for the equation $L_A u = 0$ form interpolation scales.
In the context of differential operators with constant coefficients on bounded Lipschitz domains, similar results were proven by Kalton, Mayboroda, and Mitrea \cite[Theorem 1.5]{KMM07}.
One major difference is that their `ambient spaces' are Triebel--Lizorkin or Besov spaces rather than tent or $Z$-spaces, which is possible since their differential operators have constant (hence very regular) coefficients.

For $\mb{p} \in I_{\text{max}}$, define\footnote{This definition is not totally rigorous, since $F(t)$ need not be defined for all $t > 0$ for arbitrary $F \in X^\mb{p}$. However we will only consider functions for which $F(t)$ is always defined. In particular the sets $S_{A,X}^\mb{p}$ defined below are well-defined.}\index{space!solution}
\begin{equation*}
	\wtd{X^{\mb{p}}} := \left\{ 
	\begin{array}{ll} 
		\left\{F \in X^\mb{p} : \text{$\lim_{t \to \infty} F(t)_\parallel = 0$ in $\mc{Z}^\prime(\bbR^n)$} \right\} & (\gq(\mb{p}) < 0), \\
		\left\{ F : N_*(F) \in L^{i(\mb{p})} \right\} & (\gq(\mb{p}) = 0).
	\end{array} \right.
\end{equation*}
where $N_*(F)$ is defined in \eqref{eqn:KPMF}.
We set $\nm{ F }_{\wtd{X^\mb{p}}}$ to be $\nm{ F }_{X^\mb{p}}$ or $\nm{ N_* F }_{L^{i(\mb{p})}}$ accordingly, and define 
\begin{equation*}\index{elliptic equation!solution space}
	S_{A,X}^\mb{p} := \left\{u \in W_{1,\text{loc}}^2(\bbR^n) : L_A u = 0, \; \nabla_A u \in \wtd{X^\mb{p}}\right\},
\end{equation*}
with $\nm{u}_{S_{A,X}^\mb{p}} := \nm {\nabla_A u}_{\wtd{X^\mb{p}}}$.

Recall that the classification regions $J(\mb{X},DB)$ were defined in Definition \ref{dfn:J}.

\begin{thm}[Interpolation of solution spaces]\label{thm:solspace-interpolation} 
	Let $\mb{p}_0$ and $\mb{p}_1$ be exponents, fix $\gh \in (0,1)$, and set $\mb{p}_\gh := [\mb{p}_0,\mb{p}_1]_\gh$.
	Let $B = \hat{A}$.
	\begin{enumerate}[(i)]
		\item
		Suppose $\mb{p}_0,\mb{p}_1 \in J(\mb{H},DB)$, with $\mb{p}_0$ finite and $j(\mb{p}_1) \geq 0$.
		Then
		\begin{equation*}
			[S_{A,T}^{\mb{p}_0}, S_{A,T}^{\mb{p}_1}]_\gh = S_{A,T}^{\mb{p}_\gh}.
		\end{equation*}
				
		\item
		Suppose $\mb{p}_0,\mb{p}_1 \in J(\mb{B},DB)$, with $\mb{p}_0$ finite.
		Then
		\begin{equation*}
			[S_{A,Z}^{\mb{p}_0}, S_{A,Z}^{\mb{p}_1}]_\gh = S_{A,Z}^{\mb{p}_\gh}.
		\end{equation*}
		
		\item
		Suppose $\mb{p}_0,\mb{p}_1 \in J(\mb{X},DB)$, with $\gq(\mb{p}_0) \neq \gq(\mb{p}_1)$.
		Then
		\begin{equation*}
			(S_{A,X}^{\mb{p}_0}, S_{A,X}^{\mb{p}_1})_{\gh,i(\mb{p}_\gh)} = S_{A,Z}^{\mb{p}_\gh}.
		\end{equation*}
	\end{enumerate}
\end{thm}

The conditions in part (i) exclude the possibility that $S_{A,T}^{\mb{p}_\eta}$ corresponds to a $BMO$-type space.

\begin{proof}
	For $\mb{p} \in J(\mb{X},DB)$, by Theorem \ref{thm:AAM} and Theorems \ref{thm:mainthm-leq2} and \ref{thm:mainthm-gtr2} (or the corresponding theorems of the second author and Mourgoglou when $\gq(\mb{p}) \in \{-1,0\}$ \cite[Theorems 1.1 and 1.3]{AM15}), and taking into account the decay results of Section \ref{sec:dos}, we have that the operator $\mb{Q}_{\sgp,DB}$ is bounded and invertible from $\mb{X}_{DB}^{\mb{p},+}$ into $S_{A,X}^\mb{p}$.
	By Theorems \ref{thm:sgpnorm-abstract} and \ref{thm:sgpnorm-concrete} (see also Remarks \ref{rmk:interpolation-extn} and \ref{rmk:cauchy-reverse}) we have
	\begin{equation*}
		\nm{f}_{\bbX_{DB}^{\mb{p},+}} \simeq \nm{\mb{Q}_{\sgp,DB} f}_{X^{\mb{p}}}
	\end{equation*}
	for all $f \in \bbX_{DB}^{\mb{p},+}$.
	Since the regions $J(\mb{X},DB)$ are closed under interpolation given the restrictions on $\mb{p}_0$ and $\mb{p}_1$ that we have made (this can be shown by combining Theorem \ref{thm:inclusion-interpolation} and Proposition \ref{prop:dual-interval-incl}), the theorem follows from the interpolation theorem for the canonical completions $\sgp \mb{X}_{DB}^{\mb{p},+}$ (Theorem \ref{thm:completed-interpolation}; see Remark \ref{rmk:interpolation-extn} to justify use of the auxiliary function $\sgp$ in this theorem), along with the usual retraction/coretraction interpolation theorem \cite[\textsection 1.2.4]{hT78}.
\end{proof}

We should illustrate this result in a concrete situation, as the current level of abstraction is quite high.
Let $\mb{p}_0=(2,0)$ and $\mb{p}_1=(\infty, s, 0)$, and consider case $(i)$ of the theorem.
The corresponding solution space norms are
\begin{align*}
\nm{f}_{S_{A,T}}^{\mb{p}_0} &= \nm{N_*(f)}_{L^2}, \\
\nm{f}_{S_{A,T}}^{\mb{p}_1} &= \nm{f}_{T^\infty_{s;0}} = \nm{\mc{C}_0 (\gk^{-s} f)}_{L^\infty}, \qquad \text{and} \\
\nm{f}_{S_{A,T}}^{\mb{p}_\gq} &= \nm{f}_{T^p_{\gq s}} = \|\mc{A} (\gk^{-\gq s} f)\|_{p}.
\end{align*}
with $f=\nabla_A u$ and $p = 2/(1-\gq)$.
The interpolation for generic $f$ looks unfamiliar, and is even more surprising when $f$ is the conormal gradient (or equivalently, the gradient) of a solution.
This can be understood by seeing that the interpolation theorems for generic function spaces $X^\mb{p}$ carry over to the solution spaces $S_{A,X}^\mb{p}$ with exponents in the classification region. 

\section{Boundary behaviour of solutions}\label{sec:Wavglims}\index{elliptic equation!boundary behaviour of solutions}

In this section we establish boundary behaviour of solutions to $L_A u = 0$.
What we are looking for is regularity up to the boundary in a classical sense, if possible.
Otherwise, we look for almost everywhere non-tangential convergence at the boundary.
Since solutions to $L_A u = 0$ need not have pointwise values, we use averages on Whitney regions approaching the boundary.
We prove results only for exponents in the classification region of Section \ref{sec:class}, although this restriction is not always necessary (see for example \cite{KP93}).
This is a natural restriction when using our Cauchy--Riemann approach, and also simplifies part of the argument.

Recall that the spaces $\wtd{X^\mb{p}}$ were defined in the previous section.

\begin{thm}\label{thm:Wavgnt}
	Let $B = \hat{A}$ and let $\mb{p} \in J(\mb{X},DB)$ with $r(\mb{p}) \in (-1,0)$.
	Let $u$ solve $L_A u = 0$, with $\nabla_A u \in \wtd{X^\mb{p}}$.
	\begin{enumerate}[(i)]
		\item
		Suppose $\mb{p}$ is finite.
		Then there exists $v \in \mb{X}^{\mb{p}+1}$ such that
		\begin{equation}\label{eqn:desired-limit}
			\lim_{t \to 0} \bariint_{\gO(t,x)}^{} u(\gt,\gx) \, d\gx \, d\gt  = v(x) \qquad \text{(a.e. $x \in \bbR^n$)}
		\end{equation}
		with $\lim_{t \to 0} \nabla_\parallel u(t,\cdot) = \nabla_\parallel v$ in $\mc{Z}^\prime$.		
		\item
		Suppose $\mb{p}$ is infinite.
		Then $u \in \dot{\gL}_{1 + r(\mb{p})}(\overline{\bbR^{1+n}_+})$ (i.e. $u$ has a $(1+r(\mb{p}))$-H\"older continuous extension to the closed upper half-space $\overline{\bbR^{1+n}_+}$).
	\end{enumerate}
\end{thm}

The proof is best presented as a series of lemmas.

\begin{lem}\label{lem:poincare-cons}
	Suppose that $t > 0$, $\gd \in (0,1)$, and let $c$ be a Whitney parameter.
	Then there is a Whitney parameter $\td{c}$ such that for all $u$ with $\nabla u \in L^2_\text{loc}(\bbR^{1+n}_+)$ and for all $x \in \bbR^n$ we have
	\begin{equation*}
		\bigg| \bariint_{\gO_c(t,x)}^{} u - \bariint_{\gO_c(\gd t,x)}^{} u \bigg| \lesssim_\gd t \bigg( \bariint_{\gO_{\td{c}}(t,x)}^{} |\nabla u|^2 \bigg)^{1/2}
	\end{equation*}
	with implicit constant independent of $t$. 
\end{lem}

\begin{proof}
	A straightforward computation shows that, writing $c = (c_0,c_1)$, if we define $\td{c} := (c_0,c_1/\gd)$ then we have
	\begin{equation*}
		\gO_c(t,x) \cup \gO_c(\gd t,x) \subset \gO_{\td{c}}(t,x)
	\end{equation*}
	for all $(t,x) \in \bbR^{1+n}_+$.
	We then have
	\begin{align}
		\bigg| \bariint_{\gO_c(t,x)}^{} u - \bariint_{\gO_c(\gd t,x)}^{} u   \bigg|
		&\leq \bigg| \bariint_{\gO_c(t,x)}^{} u - \bariint_{\gO_{\td{c}} (t,x)}^{} u \bigg| + \bigg| \bariint_{\gO_{\td{c}}(t,x)}^{} u - \bariint_{\gO_c(\gd t,x)}^{} u \bigg| \nonumber \\
		&\lesssim_\gd \bariint_{\gO_{\td{c}}(t,x)}^{} \bigg| u - \bariint_{\gO_{\td{c}}(t,x)}^{} u \bigg| \label{line:containments} \\
		&\leq \bigg( \bariint_{\gO_{\td{c}}(t,x)}^{} \bigg| u - \bariint_{\gO_{\td{c}}(t,x)}^{} u \bigg|^2 \bigg)^{1/2} \nonumber \\
		&\lesssim t \bigg( \bariint_{\gO_{\td{c}}(t,x)}^{} |\nabla u|^2 \bigg)^{1/2}, \label{line:poincare}
	\end{align}
	using in \eqref{line:containments} that
	\begin{equation*}
		\bigg| \barint_A^{} u - \barint_B^{} u \bigg| \leq \frac{|B|}{|A|} \barint_B^{} \bigg| u - \barint_B^{} u \bigg|
	\end{equation*}
	when $A \subseteq B$, and the Poincar\'e inequality in \eqref{line:poincare}.
\end{proof}

Let $c$ be a Whitney parameter.
For all $f \in L^0(\bbR^{1+n}_+)$ we define a `vertical Whitney maximal function' $V^\mb{p}_c(f)$ on $\bbR^n$ by
\begin{equation*}
	V_c(f)(x) := \sup_{t > 0} \mc{W}_c f(t,x).
\end{equation*}

\begin{lem}\label{lem:whit-mf-control}
	Let $\mb{p}$ be a finite exponent and suppose $f \in L^0(\bbR^{1+n}_+)$.
	Let $c$ be a Whitney parameter.
	Then
	\begin{equation}\label{eqn:maxest}
		\nm{V_c(\gk^{-\gq(\mb{p})} f)}_{L^{i(\mb{p})}(\bbR^n)} \lesssim \nm{f}_{X^\mb{p}}.
	\end{equation}
\end{lem}

\begin{proof}
	First note that $\gG_{c_0 c_1}(x) \supset \gO_c(t,x)$ for all $(t,x) \in \bbR^{1+n}_+$, so we have
	\begin{equation*}
		\bigg(\bariint_{\gO_c(t,x)}^{} |\gt^{-\gq(\mb{p})} f(\gt,\gx)|^2 \, d\gt \, d\gx \bigg)^{1/2} \lesssim \bigg( \dint_{\gG_{c_0 c_1}(x)} |\gt^{-\gq(\mb{p})} f(\gt,\gx)|^2 \, \frac{d\gt \, d\gx}{\gt^{n+1}} \bigg)^{1/2}.
	\end{equation*}
	Taking the supremum in $t > 0$ and then integrating in $x$ yields \eqref{eqn:maxest} with $X = T$.
	
	Next, for all $\varepsilon \in (0,1)$ and for $\td{c}_1 \geq c_1 \max(1+\varepsilon, (1-\varepsilon)^{-1})$, we have
	\begin{equation*}
		\gO_c(t,x) \subset \gO_{\td{c}}(t^\prime,x)
	\end{equation*}
	for all $t^\prime \in [(1-\varepsilon)t, (1+\varepsilon)t]$.
	Therefore
	\begin{align*}
		\bigg( &\bariint_{\gO_c(t,x)}^{} |\gt^{-\gq(\mb{p})} f(\gt,\gx)|^2 \, d\gt \, d\gx \bigg)^{i(\mb{p})/2} \\
		&\lesssim \barint_{(1-\varepsilon)t}^{(1+\varepsilon)t} \bigg( \bariint_{\gO_{\td{c}}(t^\prime, x)}^{} |\gt^{-\gq(\mb{p})} f(\gt,\gx)|^2 \, d\gt \, d\gx \bigg)^{i(\mb{p})/2} \, dt^\prime \\
		&\lesssim \int_0^\infty \bigg( \bariint_{\gO_{\td{c}}(t^\prime,x)}^{} |\gt^{-\gq(\mb{p})} |f(\gt,\gx)|^2 \, d\gt \, d\gx \bigg)^{i(\mb{p})/2} \, \frac{dt^\prime}{t^\prime}.
	\end{align*}
	Integrating in $x$ then yields \eqref{eqn:maxest} with $X = Z$.
\end{proof}

The following lemma is not too difficult to show, so we omit the proof.

\begin{lem}\label{lem:limit}
	Suppose $\map{h}{(0,\infty)}{\bbR}$ is a function satisfying
	\begin{equation*}
		|h(t) - h(s)| \lesssim_\gd |t-s|^\ga
	\end{equation*}
	whenever $\gd \in (0,1)$ and $|t-s| < \gd t$.
	Then the limit $\lim_{t \to 0} h(t)$ exists, and
	\begin{equation}\label{eqn:limit-rate}
		|h(t) - \lim_{t \to 0} h(t)| \lesssim |t|^\ga.
	\end{equation}
\end{lem}

We can apply the lemmas above to prove the following limit result.
Note that we do not need to assume $u$ solves any elliptic equation to establish this.

\begin{prop}\label{prop:X-grad-lims}
	Let $\mb{p}$ be a finite exponent with $\gq(\mb{p}) > -1$, and with $\mb{p} \in I_{\text{max}}$ if $i(\mb{p}) \leq 1$.
	If $u \in L^2_\text{loc}(\bbR^{1+n}_+)$ and $\nabla u \in X^\mb{p}$, then there exists a unique $v \in L^0(\bbR^n) \cap \mc{S}^\prime(\bbR^n)$ such that
	\begin{equation}\label{eqn:u-bdy-lim}
		\lim_{t \to 0} \bariint_{\gO_c(t,x)}^{} u = v(x) \qquad \text{a.e. $x \in \bbR^n$}.
	\end{equation}
	Moreover,
	\begin{equation}\label{eqn:u-bdy-sch}
		\lim_{\gt \to 0} \barint_\gt^{2\gt} u(t,\cdot) \, dt = v \qquad \text{in $\mc{S}^\prime$}
	\end{equation}
	and likewise
	\begin{equation}\label{eqn:u-bdy-sch-grad}
		\lim_{\gt \to 0} \barint_\gt^{2\gt} \nabla_\parallel u(t,\cdot) \, dt = \nabla_\parallel v \quad \text{in $\mc{S}^\prime$}.
	\end{equation}
\end{prop}

\begin{proof}
	Write
	\begin{equation*}
		h_x(t) := \bariint_{\gO_c(t,x)}^{} u.
	\end{equation*}
	By Lemma \ref{lem:poincare-cons}, for all $(t,x) \in \bbR^{1+n}_+$ and for all $\gd \in (0,1)$ we have
	\begin{align*}
		| h_x(t) - h_x(\gd t) | &\lesssim_\gd t\bigg( \bariint_{\gO_{\td{c}}(t,x)}^{} |\nabla u|^2 \bigg)^{1/2} \\
		&\simeq_\gd (t - \gd t)^{1 + \gq(\mb{p})} \bigg( \bariint_{\gO_{\td{c}}(t,x)}^{} | \gt^{-\gq(\mb{p})} \nabla u(\gt,\gx) |^2 \, d\gt \, d\gx \bigg)^{1/2} \\
		&\leq (t - \gd t)^{1 + \gq(\mb{p})} g(x)
	\end{align*}
	with $\td{c}$ coming from Lemma \ref{lem:poincare-cons} and with $g := V_{\td{c}}(\gk^{-\gq(\mb{p})} \nabla u)$.
	By Lemma \ref{lem:whit-mf-control} we find that $g \in L^{i(\mb{p})}(\bbR^n)$, and therefore $g(x)$ is finite for almost every $x \in \bbR^n$.
	Since $1 + \gq(\mb{p}) > 0$, for almost every $x$ we can apply Lemma \ref{lem:limit} to conclude that $\lim_{t \to 0} h_x(t)$ exists.
	Since the function $x \mapsto h_x(t)$ is measurable for each $t > 0$, the almost-everywhere defined limit
	\begin{equation*}
		v \colon x \mapsto \lim_{t \to 0} h_x(t)
	\end{equation*}
	is measurable, and by construction $v$ is the unique function satisfying \eqref{eqn:u-bdy-lim}.
	
	Next we will show that $v \in \mc{S}^\prime(\bbR^n)$.
	This will follow from the estimate
	\begin{equation}\label{eqn:v-g-est}
		|v(x) - v(y)| \lesssim |x-y|^{1 + \gq(\mb{p})} (g(x) + \overline{g}(x) + g(y))
	\end{equation}
	for all $x,y \in \bbR^n$ with $\overline{g}$ defined in the same way as $g$ but with a different choice of Whitney parameter.
	Indeed, assuming \eqref{eqn:v-g-est}, for all $x \in \bbR^n$ and all $y \in \bbR^n$ such that $|v(y)| < \infty$,
	\begin{equation*}
		|v(x)| \lesssim |v(y)| + |x-y|^{1 + \gq(\mb{p})} (g(x) + \overline{g}(x) + g(y)),
	\end{equation*}
	and if $i(\mb{p}) > 1$ then for all $\gf \in \mc{S}(\bbR^n)$ we will have
	\begin{equation*}
		\int_{\bbR^n} |v(x)||\gf(x)| \, dx < \infty
	\end{equation*}
	since $g, \overline{g} \in L^{i(\mb{p})}(\bbR^n)$ and $\gf \in L^{i(\mb{p})^\prime}(\bbR^n)$.
	If $i(\mb{p}) \leq 1$ then since $\mb{p} \in I_{\text{max}}$ (by assumption) there exists $\mb{q}$ with $\gq(\mb{q}) > -1$ and $i(\mb{q}) > 1$ such that $\mb{p} \hookrightarrow \mb{q}$, so that $\nabla u \in X^\mb{q}$, and we can then argue with $\mb{q}$ replacing $\mb{p}$.
	
	We now prove \eqref{eqn:v-g-est}.
	Let $R = |x-y|$ and let $c^\prime = (c_0+1,c_1)$.
	Then we have $\gO_c(x,R) \cup \gO_c(y,R) \subset \gO_{c^\prime}(x,R)$, and so arguing as in the proof of Lemma \ref{lem:poincare-cons} we have
	\begin{align*}
		|v(x) - v(y)| &\leq \bigg| v(x) - \bariint_{\gO_c(x,R)}^{} u \bigg| + \bigg| \bariint_{\gO_c(x,R)}^{} u - \bariint_{\gO_c(y,R)}^{} u \bigg| + \bigg| \bariint_{\gO_c(y,R)}^{} u - v(y) \bigg| \\
		&\lesssim R^{1 + \gq(\mb{p})}(g(x) + \overline{g}(x) + g(y))
	\end{align*}
	where $\overline{g}$ is defined using the Whitney parameter $c^\prime$ in place of $\td{c}$.
	
	It remains to prove the limits \eqref{eqn:u-bdy-sch} and \eqref{eqn:u-bdy-sch-grad}.
	By the same embedding argument as before, it suffices to prove this for $\mb{p}$ such that $i(\mb{p}) > 1$.
	Let $\gf \in \mc{S}^\prime(\bbR^n)$ and $\gt > 0$.
	Then we have
	\begin{align*}
		&\barint_{\gt}^{2\gt} |(u(t,\cdot) - v(\cdot), \gf)| \, dt \\
		&\leq \int_{\bbR^n} \barint_{\gt}^{2\gt} \barint_{B(x,c_0 \gt)}^{} |u(t,y) - v(y)| |\gf(y)| \, dy \, dt \, dx \\
		&\leq \int_{\bbR^n} \barint_{\gt}^{2\gt} \barint_{B(x,c_0 \gt)}^{} \bigg| u(t,y) - \bariint_{\gO_c(\gt,x)}^{} u(s,z) \, ds \, dz \bigg| |\gf(y)| \, dy \, dt \, dx \\
		&\quad + \int_{\bbR^n} \barint_{\gt}^{2\gt} \barint_{B(x,c_0 \gt)}^{} |v(x) - v(y)| |\gf(y)| \, dy \, dt \, dx \\
		&\quad + \int_{\bbR^n} \barint_{\gt}^{2\gt} \barint_{B(x,c_0 \gt)}^{} \bigg| \bariint_{\gO_c(\gt,x)}^{} u(s,z) \, ds \, dz - v(x) \bigg| |\gf(y)| \, dy \, dt \, dx \\
		&=: \mb{I}_1 + \mb{I}_2 + \mb{I}_3. 
	\end{align*}
	We estimate these three terms individually.
	First, for $\gt < 1$, with $c = (c_0,2)$,
	\begin{align*}
		\mb{I}_1 &\lesssim \int_{\bbR^n} \bigg( \bariint_{\gO_c(\gt,x)}^{} \bigg| u - \bariint_{\gO_c(\gt,x)}^{} u \bigg| \bigg) \sup_{y \in B(x,c_0 \gt)} |\gf(y)| \, dx \\
		&\lesssim \gt^{1 + \gq(\mb{p})} \int_{\bbR^n}  g(x) \sup_{y \in B(x,c_0)} |\gf(y)| \, dx \\
		&\lesssim_\gf \gt^{1+ \gq(\mb{p})} \nm{g}_{L^{i(\mb{p})}}
	\end{align*}
	since the function $x \mapsto \sup_{y \in B(x,c_0)} |\gf(y)|$ is in $L^{i(\mb{p})^\prime}$.
	In the same way we can prove
	\begin{equation*}
		\mb{I}_3 \lesssim  \gt^{1 + \gq(\mb{p})} \nm{g}_{L^{i(\mb{p})}}
	\end{equation*}
	for $\gt < 1$ by using \eqref{eqn:limit-rate}.
	Finally, using \eqref{eqn:v-g-est}, for $\gt < 1$ we have
	\begin{align*}
		\mb{I}_2 &\leq \int_{\bbR^n} \barint_{B(x,c_0 \gt)}^{} |v(x) - v(y)|  \, dy \, \sup_{z \in B(x,c)} |\gf(z)| dx \\
		&\lesssim \gt^{1 + \gq(\mb{p})} \int_{\bbR^n} \barint_{B(x,c_0 \gt)}^{} g(x) + g(y) \, dy \, \sup_{z \in B(x,c)} |\gf(z)| \, dx \\
		&\lesssim \gt^{1 + \gq(\mb{p})} \int_{\bbR^n} \left( g(x) + \mc{M}g(x) \right) \sup_{z \in B(x,c)} |\gf(z)| \, dx \\
		&\lesssim_\gf \gt^{1 + \gq(\mb{p})} \nm{g}_{L^{i(\mb{p})}} 
	\end{align*}
	where $\mc{M}$ is the Hardy--Littlewood maximal operator (which is bounded on $L^{i(\mb{p})}$).
	Since $1 + \gq(\mb{p}) > 0$ we therefore have
	\begin{equation*}
		\lim_{\gt \to 0} \barint_{\gt}^{2\gt} |(u(t,\cdot) - v(\cdot), \gf)| \, dt = 0,
	\end{equation*}
	which proves \eqref{eqn:u-bdy-sch}.
	The limit \eqref{eqn:u-bdy-sch-grad} is then proven by taking $\gy \in \mc{S}(\bbR^n : \bbC^n)$ and applying the previous argument to $\gf = \dv \gy$.
	\end{proof}
	
	These are all the results we need for the case of finite exponents.
	We will prove Theorem \ref{thm:Wavgnt} after presenting the results required for infinite exponents (which are simpler).
	
	\begin{prop}\label{prop:Z-trace}
		Let $\mb{p}$ be infinite. Then for all $u \in L_{\loc}^2(\bbR^{1+n}_+)$ we have
		\begin{equation}\label{eqn:Z-inf}
			\bariint_{\gO_c(x,R)}^{} \bigg| u - \bariint_{\gO_c(x,R)}^{} u \bigg|^2 \lesssim R^{2 + 2r(\mb{p})} \nm{\nabla u}_{Z^\mb{p}}^2.
		\end{equation}
		Moreover, if $r(\mb{p}) \in (-1,0]$, then for all $x \in \bbR^n$ the limit
		\begin{equation*}
			v(x) := \lim_{R \to 0} \bariint_{\gO_c(x,R)}^{} u
		\end{equation*}
		exists and satisfies
		\begin{equation}\label{eqn:v-holder}
			|v(x) - v(y)| \lesssim |x-y|^{1 + r(\mb{p})}.
		\end{equation}
	\end{prop}

	\begin{proof}
		The estimate \eqref{eqn:Z-inf} follows from the Poincar\'e inequality and the supremum definition of the $Z^\mb{p}$ norm (see Corollary \ref{cor:Zinfty-sup}).
		The remaining statements are then proven exactly as in Proposition \ref{prop:X-grad-lims} (using Lemma \ref{lem:limit}, without the additional complication of the function $g$).
	\end{proof}
	
	\begin{rmk}
		Proposition \ref{prop:Z-trace} shows that having a boundary trace in the H\"older space $\dot{\gL}^{1 + r(\mb{p})}(\bbR^n)$ is a generic property\footnote{That is, it holds for all $u$ in this space, without requiring that $u$ solves an elliptic equation.} of the space $\{u \in L^2_\text{loc} : \nabla u \in Z^\mb{p}\}$ when $\mb{p}$ is infinite and $r(\mb{p}) \in (-1,0]$.
		This was already observed by Barton and Mayboroda \cite[Theorem 6.3]{BM16}.
	\end{rmk}
	
	\begin{proof}[Proof of Theorem \ref{thm:Wavgnt}]
		First suppose $\mb{p}$ is finite.
		Proposition \ref{prop:X-grad-lims} furnishes a distribution $v \in L^0 \cap \mc{Z}^\prime$ satisfying \eqref{eqn:desired-limit}.
		Furthermore, Theorems \ref{thm:AAM}, \ref{thm:mainthm-leq2} and \ref{thm:mainthm-gtr2} imply that there exists $f =: \nabla_A u|_{t = 0} \in \mb{X}_D^\mb{p} \subset \mc{Z}^\prime$ such that $\nabla_A u(t,\cdot) \to \nabla_A u|_{t = 0}$ as $t \to 0$ in $\mb{X}_{DB}^{\mb{p},+} \subset \mb{X}_D^{\mb{p}}$, hence also in $\mc{Z}^\prime$.
		Along with the limit \eqref{eqn:u-bdy-sch-grad}, this yields
		\begin{equation*}
			\nabla_\parallel u|_{t = 0} = (\nabla_A u|_{t = 0})_\parallel = \nabla_\parallel v
		\end{equation*}
		which proves that $\nabla_\parallel v \in \mb{X}^\mb{p}$ (and therefore $v \in \mb{X}^{\mb{p} + 1}$), and $\lim_{t \to 0} \nabla_\parallel u(t,\cdot) = \nabla_\parallel v$ in $\mc{Z}^\prime$.
		
		Now suppose $\mb{p}$ is infinite. 
		We will show that for every cylinder $C(x,a,R)$ in $\bbR^{1+n}_+$ of the form $(a,a+R) \times B(x,R)$, we have the estimate
		\begin{equation}\label{eqn:campanato}
			\bariint_{C(x,a,R)}^{} \bigg| u(\gt,y) - \bariint_{C(x,a,R)}^{} u \bigg|^2 \, d\gt \, dy \lesssim R^{2(1 + r(\mb{p}))}.
		\end{equation}
		By Campanato's characterisation of the H\"older space $\dot{\gL}_{1 + r(\mb{p})}(\overline{\bbR^{1+n}_+})$ \cite[Teorema I.2]{sC63} (which uses that $1 + r(\mb{p}) \in (0,1)$), this will complete the proof of Theorem \ref{thm:Wavgnt} when $\mb{p}$ is infinite.
		
		By Theorems \ref{thm:AAM} and \ref{thm:mainthm-gtr2}, we know that for all $t \geq 0$, $t_0 > 0$ the function $t \mapsto \nabla_A u(t + t_0,\cdot)$ is in $Z^\mb{p}$ (using the embedding $T^\mb{p} \hookrightarrow Z^\mb{p}$ for infinite exponents from Lemma \ref{lem:TZ-fixedexp-emb}) with
		\begin{equation*}
			\nm{ t \mapsto \nabla_A u(t + t_0) }_{Z^\mb{p}} \lesssim \nm{ \nabla_A u(t_0) }_{\mb{X}_D^\mb{p}} \lesssim 1.
		\end{equation*}
		Therefore, using Proposition \ref{prop:Z-trace} for $u(\cdot + t_0)$ and translating, we get
		\begin{equation}\label{eqn:precampanato}
			\bariint_{(t_0,0) + \gO_c(x,R)}^{} \bigg| u - \bariint_{(t_0,0) + \gO_c(x,R)}^{} u \bigg|^2 \lesssim R^{2(1+r(\mb{p}))}.
		\end{equation}
		
		Now consider a cylinder $C(x,a,R)$ as above.
		If $a \geq R$, then $C(x,a,R) \subset (R-a,0) \times \gO_c(x,R)$, and in this case \eqref{eqn:campanato} follows from \eqref{eqn:precampanato}.
		To prove \eqref{eqn:campanato} for $a < R$, it suffices to take $a = 0$.
		Let $v$ be the boundary trace of $u$ given in Proposition \ref{prop:Z-trace}.
		Then we have
		\begin{align*}
			\bariint_{C(x,0,R)}^{} \bigg| u - \bariint_{C(x,0,R)}^{} u \bigg|^2
			&\lesssim \bariint_{C(x,0,R)}^{} | u - v(x) |^2 + \bariint_{C(x,0,R)}^{} \bigg| v(x) - \bariint_{C(x,0,R)}^{} u \bigg| \\
			&\lesssim \bariint_{C(x,0,R)}^{} | u - v(x) |^2 + R^{2(1 + r(\mb{p}))}
		\end{align*}
		using \eqref{eqn:limit-rate}.
		To handle the remaining summand, estimate
		\begin{align}
			&\bariint_{C(x,0,R)}^{} \left| u(\gt,y) - v(x) \right|^2 \, d\gt \, dy \nonumber \\
			&\lesssim \bariint_{C(x,0,2R)}^{} \bariint_{\gO_c(\gs,z)}^{} \left| u(\gt,y) - v(x) \right|^2 \, d\gt \, dy \, d\gs \, dz \nonumber \\
			&\lesssim \bariint_{C(x,0,2R)}^{} \bigg( \bariint_{\gO_c(\gs,z)}^{} | u(\gt,y) - v(z) |^2 \, d\gt \, dy \bigg) + \big( | v(z) - v(x) |^2 \big) \, d\gs \, dz \nonumber \\
			&\lesssim \bariint_{C(x,0,2R)}^{} \gs^{2(1+r(\mb{p}))} + |z-x|^{2(1+r(\mb{p}))} \, dz \, d\gs \label{line:limit-rate-use} \\
			&\lesssim R^{2(1 + r(\mb{p}))}
		\end{align}
		using \eqref{eqn:limit-rate} and \eqref{eqn:v-holder} in \eqref{line:limit-rate-use}, and in the last line using that $1 + r(\mb{p}) \geq 0$. 
		This proves \eqref{eqn:campanato}, and thus completes the proof of Theorem \ref{thm:Wavgnt}.
	\end{proof}

\chapter{Applications to Boundary Value Problems}\label{chap:wp}

\section{Characterisation of well-posedness and corollaries}\label{sec:wpc}

First let us put the boundary value problems given in the introduction (Subsection \ref{ssec:fbvp}) in a more convenient form.
Fix $m \in \bbN$ and let $\mb{p}$ be an exponent.
By making use of the transversal/tangential splitting $\bbC^{m(1+n)} = \bbC^{m} \oplus \bbC^{mn}$, we can write
\begin{align*}
	\mb{X}^\mb{p}_D(\bbR^n : \bbC^{m(1+n)}) &= \bbP_D (\mb{X}^\mb{p}(\bbR^n : \bbC^{m(1+n)})) \\
	&= (\mb{X}^\mb{p} \cap D\mc{Z}^\prime)(\bbR^n:\bbC^{m(1+n)}) \\
	&= \mb{X}^\mb{p}(\bbR^n : \bbC^m) \oplus (\mb{X}^\mb{p}(\bbR^n : \bbC^{mn}) \cap \nabla_\parallel \mc{Z}^\prime(\bbR^n : \bbC^{mn})) \\
	&=: \mb{X}_\perp^\mb{p} \oplus \mb{X}_\parallel^\mb{p}.
\end{align*}

\begin{dfn}\label{dfn:bvps}\index{boundary value problem}
For all exponents $\mb{p} \in I_{\text{max}}$ we define the \emph{Regularity problem}\index{boundary value problem!Regularity}
\begin{equation*}
	(R_{\mb{X}})_{A}^{\mb{p}} : \left\{ \begin{array}{l} 	L_A u = 0 \quad \text{in $\bbR_+^{1+n}$}, \\
											\lim_{t \to 0} \nabla_\parallel u(t,\cdot) = f \in \mb{X}^{\mb{p}}_\parallel, \\
											\nm{ \nabla u }_{\wtd{X^\mb{p}}} \lesssim \nm{ f }_{\mb{X}^{\mb{p}}}, \\
	\end{array} \right.
\end{equation*}
and the \emph{Neumann problem}\index{boundary value problem!Neumann}
\begin{equation*}
	(N_{\mb{X}})_{A}^{\mb{p}} : \left\{ \begin{array}{l} 	L_A u = 0 \quad \text{in $\bbR_+^{1+n}$}, \\
											\lim_{t \to 0} \partial_{\gn_A} u(t,\cdot) = f \in \mb{X}^{\mb{p}}_\perp, \\
											\nm{ \nabla u }_{\wtd{X^\mb{p}}} \lesssim \nm{ f }_{\mb{X}^{\mb{p}}}. \\
	\end{array} \right.
\end{equation*}
	The spaces $\wtd{X^\mb{p}}$ (which are just the spaces $X^\mb{p}$ with an additional decay condition at infinity) were defined in Section \ref{sec:solspace-interpolation}. 
	By $\lim_{t \to 0} \nabla_\parallel u(t,\cdot) = f \in \mb{X}^{\mb{p}}_\parallel$ we mean that $f \in \mb{X}^\mb{p}_\parallel$ and that the limit is in the $\mb{X}_\parallel^\mb{p}$ topology, and likewise for the limit in the Neumann problem.
        These limits are imposed in the weak-star topology when $\mb{p}$ is infinite.
	We say that such a problem is \emph{well-posed}\index{well-posedness} if for all boundary data $f$ there exists a unique $u$ (up to additive constant) satisfying the conditions of the problem.
\end{dfn}

We denote these problems simultaneously by $(P_{\mb{X}})_A^\mb{p}$, with $P$ standing for either $R$ or $N$.
We define the \emph{well-posedness region}\index{region!well-posedness}
\begin{equation*}
	\wellp(P_\mb{X})_A := \{\mb{p} \in I_{\text{max}} : \text{$(P_\mb{X})_A^\mb{p}$ is well-posed}\}.
\end{equation*}
\begin{rmk}
The problems $(R_{\mb{X}})_A^\mb{p}$ and $(N_{\mb{X}})_A^\mb{p}$ include all Regularity and Neumann problems introduced in Subsection \ref{ssec:fbvp}.
The definition above is much more concise (but, initially, much less clear).
\end{rmk}

\begin{rmk}
We stress that in the case of Hardy--Sobolev spaces, there is no \emph{a priori} connection between the interior control and the boundary condition.
In contrast, in the case of Besov spaces, they are connected by the trace theorem of Barton and Mayboroda \cite[Theorem 3.9]{BM16}.
\end{rmk}

\begin{rmk}
The boundary condition in $(R_{\mb{X}})_A^\mb{p}$ is equivalent to the Dirichlet condition
\begin{equation*}
	\lim_{t \to 0} u(t,\cdot) = g \in \mb{X}_\perp^{\mb{p}+1}\\ 
\end{equation*}
where $\nabla_\parallel g = f$, since $\nabla_\parallel$ is an isomorphism from $\mb{X}_\perp^{\mb{p}+1}$ onto $\mb{X}_\parallel^\mb{p}$.
Therefore $(R_{\mb{X}})_A^\mb{p}$ could be thought of as a Dirichlet problem $(D_{\mb{X}})_A^{\mb{p}+1}$.\index{boundary value problem!Dirichlet}
The limit above is taken in the topology of $\mb{X}_\perp^{\mb{p}+1}$, which we defined as a subspace of $\mc{Z}^\prime$, but which also embeds in the space of Schwartz distributions modulo constants when $\gq(\mb{p}) \in [-1,0]$.
When $\mb{p}$ is finite, we can be more precise.   
When $\theta(\mb{p})=-1$, $ \mb{X}_\perp^{\mb{p}+1}$ also embeds in the space of Schwartz distributions as the Lebesgue space $L^{i(\mb{p})}$.
In this case, there is also convergence in $L^{i(\mb{p})}$ having fixed the constant (recall that $\nabla u \in X^{\mb{p}}$ so $u$ is determined up to a constant), see \cite[Corollary 1.4]{AM15}.
The case $\theta(\mb{p})=0$ is also treated in \cite[Corollary 1.2]{AM15}.
In the intermediate case, we can proceed similarly.
First, Theorems \ref{thm:stein} and \ref{thm:besov-difference} show that when $\mb{p}$ is finite these spaces can be realised as contained (not embedded) in $L^1_{\loc}$.
We can then impose the limit in $L^1_{\loc}$, which fixes the floating constant in the determination of $u$ in terms of $g$. 
It can be shown that this limit holds \emph{a priori} in $L^1_{\loc}$.
We omit further details.
\end{rmk}

We use the first representation theorem (Theorem \ref{thm:firstrepresentation}) to characterise the well-posedness of $(R_{\mb{X}})_A^\mb{p}$ and $(N_{\mb{X}})_A^\mb{p}$ for $\mb{p}$ in the classification region $J(\mb{X},DB)$ (see Definition \ref{dfn:J}).
Let $N_\perp$ and $N_\parallel$ denote the projections from $\mb{X}_D^\mb{p}(\bbR^n : \bbC^{m(1+n)})$ onto $\mb{X}_\perp^\mb{p}$ and $\mb{X}_\parallel^\mb{p}$ respectively.
For $\mb{p} \in J(\mb{X},DB)$ we have defined $\mb{X}^{\mb{p},+}_{DB}(\bbR^n : \bbC^{m(1+n)})$ as a subspace of $\mb{X}^\mb{p}_D(\bbR^n : \bbC^{m(1+n)})$, and through this identification we define 
\begin{equation*}
\map{N_{\mb{X},DB,\parallel}^\mb{p}}{\mb{X}^{\mb{p},+}_{DB}}{\mb{X}_{\parallel}^\mb{p}} \qquad \text{and} \qquad \map{N_{\mb{X},DB,\perp}^\mb{p}}{\mb{X}^{\mb{p},+}_{DB}}{\mb{X}_{\perp}^\mb{p}}
\end{equation*}
as the restrictions of $N_\perp$ and $N_\parallel$ respectively.

\begin{thm}[Characterisation of well-posedness]\label{thm:WP-char}\index{well-posedness!characterisation}
	Let $B = \hat{A}$, and suppose that $\mb{p} \in J(\mb{X},DB)$.
	Then $(R_{\mb{X}})_A^\mb{p}$ (resp. $(N_{\mb{X}})_A^\mb{p}$) is well-posed if and only if $N_{\mb{X},DB,\parallel}^\mb{p}$ (resp. $N_{\mb{X},DB,\perp}^\mb{p}$) is an isomorphism.
\end{thm}

\begin{proof}
	The results for $\gq(\mb{p}) = 0$ and $\gq(\mb{p}) = -1$ correspond to \cite[Theorems 1.5 and 1.6]{AM15}, so we need only consider $\gq(\mb{p}) \in (-1,0)$.
	Theorem \ref{thm:firstrepresentation} tells us that $\nabla_{A}u= C_{DB}^+F_{0}$ for a unique $F_{0}\in  \mb{X}_{DB}^{\mb{p},+}$, with $F_0 = \lim_{t \to 0} \nabla_A u(t,\cdot)$ in $\mb{X}^\mb{p}$ (here is where we use that $\mb{p} \in J(\mb{X},DB)$).
	Since $\nabla_A = \left[ \partial_{\gn_A}, \nabla_\parallel \right]$, the boundary conditions for $(R_{\mb{X}})_A^\mb{p}$ and $(N_{\mb{X}})_A^\mb{p}$ can be rewritten as
	\begin{align*}
		N_{\parallel} (F_0) &= f \in \mb{X}_{\parallel}^{\mb{p}} \quad \text{and} \\
		N_{\perp} (F_0) &= f \in \mb{X}_{\perp}^{\mb{p}}
	\end{align*}
	respectively.
	The result follows.
\end{proof}

Define the \emph{energy exponent} $\mb{e} = (2,-1/2)$.\index{exponent!energy}
For all $A$, the Lax--Milgram theorem guarantees well-posedness of the problems $(R_{\mb{X}})_A^\mb{e}$ and $(N_{\mb{X}})_A^\mb{e}$ (see \cite[Theorems 3.2 and 3.3]{AMM13}).

\begin{dfn}
	Suppose $\mb{p}$ and $\mb{q}$ are exponents with $\gq(\mb{p}),\gq(\mb{q}) \in [-1,0]$.
	We say that the boundary value problems $(P_\mb{X})_A^\mb{p}, (P_\mb{X})_A^\mb{q}$ are \emph{mutually well-posed}\index{well-posedness!mutual} if they are both well-posed, and if for all boundary data $f \in \mb{X}_\bullet^\mb{p} \cap \mb{X}_\bullet^\mb{q}$ (where $\bullet$ is either $\parallel$ or $\perp$ depending on the choice of boundary condition) the solutions to $(P_\mb{X})_A^\mb{p}$ and $(P_\mb{X})_A^\mb{q}$ with boundary data $f$ are equal.
	We say that $(P_\mb{X})_A^\mb{p}$ is \emph{compatibly well-posed}\index{well-posedness!compatible} if $(P_\mb{X})_A^\mb{p}$ and $(P_\mb{X})_A^\mb{e}$ (where $\mb{e}$ is the energy exponent defined above) are mutually well-posed.
\end{dfn}

If $B = \hat{A}$ and $\mb{p},\mb{q} \in J(\mb{X},DB)$, then Theorem \ref{thm:WP-char} says that $(P_\mb{X})_A^\mb{p}$ and $(P_\mb{X})_A^\mb{q}$ are mutually well-posed if and only if $N_{\mb{X},DB,\bullet}^\mb{p}$ and $N_{\mb{X},DB,\bullet}^\mb{q}$ are isomorphisms and $(N_{\mb{X},DB,\bullet}^\mb{p})^{-1} = (N_{\mb{X},DB,\bullet}^\mb{q})^{-1}$ on $\mb{X}_\bullet^\mb{p} \cap \mb{X}_\bullet^\mb{q}$.
Evidently mutual well-posedness is transitive: if $(P_\mb{X})_A^\mb{p}$ and $(P_\mb{X})_A^\mb{q}$ are mutually well-posed, and if $(P_\mb{X})_A^\mb{q}$ and $(P_\mb{X})_A^\mb{r}$ are mutually well-posed, then $(P_\mb{X})_A^\mb{p}$ and $(P_\mb{X})_A^\mb{r}$ are mutually well-posed.

For finite exponents in the classification region we can interpolate mutual well-posedness; `mutuality' is required in order to interpolate invertibility.

\begin{thm}[Interpolation of mutual well-posedness]\label{thm:cwp-int}\index{interpolation!of mutual well-posedness}
	Let $B = \hat{A}$, and suppose $\mb{p}, \mb{q} \in J(\mb{X},DB) \cap \mb{E}_{\text{fin}}$.
	If $(P_\mb{X})_A^\mb{p}$ and $(P_\mb{X})_A^\mb{q}$ are mutually well-posed, then the problems $((P_\mb{X})_A^{[\mb{p},\mb{q}]_\gh})_{\gh \in [0,1]}$ are pairwise mutually well-posed.
	Furthermore, if $\gq(\mb{p}) \neq \gq(\mb{q})$ and $\mb{X} = \mb{H}$, then the problems $((P_\mb{B})_A^{[\mb{p},\mb{q}]_\gh})_{\gh \in (0,1)}$, $(P_\mb{H})_A^\mb{p}$ and $(P_\mb{H})_A^\mb{q}$ are pairwise mutually well-posed.
\end{thm}

\begin{proof}
	We use the interpolation result for smoothness spaces, Theorem \ref{thm:int-clas}.
	By the previous discussion, we have
	\begin{equation*}
	(N_{\mb{X},DB,\bullet}^\mb{p})^{-1} = (N_{\mb{X},DB,\bullet}^\mb{q})^{-1}
	\end{equation*}
	on the intersection $\mb{X}_\bullet^\mb{p} \cap \mb{X}_\bullet^\mb{q}$.
	Since this intersection is dense in $\mb{X}_\bullet^\mb{p}$ and $\mb{X}_\bullet^\mb{q}$ (here is where we use finiteness of $\mb{p}$ and $\mb{q}$), we have a well-defined operator
	\begin{equation*}
		\map{\wtd{N}}{\mb{X}_\bullet^\mb{p} + \mb{X}_\bullet^\mb{q}}{\mb{X}_{DB}^{\mb{p},+} + \mb{X}_{DB}^{\mb{q},+}}
	\end{equation*}
	which restricts to $(N_{\mb{X},DB,\bullet}^\mb{p})^{-1}$ and $(N_{\mb{X},DB,\bullet}^{\mb{q}})^{-1}$ on $\mb{X}_\bullet^\mb{p}$ and $\mb{X}_\bullet^\mb{q}$ respectively.
	By complex interpolation, for each $\gh \in [0,1]$, $\wtd{N}$ restricts to a bounded operator $\map{\wtd{N}_\gh}{\mb{X}_\bullet^{[\mb{p},\mb{q}]_{\gh}}}{\mb{X}_{DB}^{[\mb{p},\mb{q}]_\gh,+}}$.
	Since $\wtd{N}_\gh$ is equal to $(N_{\mb{X},DB,\bullet}^\mb{p})^{-1}$ on $\mb{X}_\bullet^{[\mb{p},\mb{q}]_\gh} \cap \mb{X}_\bullet^\mb{p}$, and since $N_{\mb{X},DB,\bullet}^\mb{[\mb{p},\mb{q}]_\gh}$ is equal to $N_{\mb{X},DB,\bullet}^\mb{p}$ on $\mb{X}_{DB}^{[\mb{p},\mb{q}]_\gh,+} \cap \mb{X}_{DB}^{\mb{p},+}$, we find that $\wtd{N}_\gh$ is the inverse of $N_{\mb{X},DB,\bullet}^{[\mb{p},\mb{q}]_\gh}$.
	Therefore $N_{\mb{X},DB,\bullet}^{[\mb{p},\mb{q}]_\gh}$ is an isomorphism, and since $[\mb{p},\mb{q}]_\gh$ satisfies the assumptions of Theorem \ref{thm:WP-char} (by Theorem \ref{thm:inclusion-interpolation}), $(P_\mb{X})_A^{[\mb{p},\mb{q}]_\gh}$ is well-posed.
	Furthermore, since $\wtd{N}_\gh = (N_{\mb{X},DB,\bullet}^\mb{p})^{-1}$ on $\mb{X}_\bullet^{[\mb{p},\mb{q}]_\gh} \cap \mb{X}_\bullet^\mb{p}$, the problems $(P_\mb{X})_A^{[\mb{p},\mb{q}]_\gh}$ and $(P_\mb{X})_A^{\mb{p}}$ are mutually well-posed.
	Transitivity of mutual well-posedness completes the proof of the first statement.
	When $\mb{X} = \mb{H}$ and $\gq(\mb{p}) \neq \gq(\mb{q})$, applying real interpolation with the same argument yields the second statement.
\end{proof}

Well-posedness can be extrapolated, preserving mutuality, within the classification region $J(\mb{X},DB)$, however we must restrict to exponents $\mb{p}$ with $\gq(\mb{p}) \in (-1,0)$.
This is a consequence of {\v{S}}ne{\u\i}berg's extrapolation theorem \cite{iS74}; the proof uses the same argument as that of Theorem \ref{thm:identn-openness}.

\begin{thm}[Extrapolation of well-posedness]\label{thm:wp-ext}\index{extrapolation!of well-posedness}
	Let $B = \hat{A}$, and suppose $\mb{p} \in J(\mb{X},DB)$ with $\gq(\mb{p}) \in (-1,0)$.
	If $\mb{X} = \mb{H}$ then further assume that $j(\mb{p}) \neq 0$, and if $\mb{X} = \mb{B}$ then assume $i(\mb{p}) > 1$.
	Suppose that $(P_\mb{X})_A^\mb{p}$ is well-posed.
	Then there exists a $(j,\gq)$-neighbourhood $O_\mb{p}$ of $\mb{p}$ such that for all $\mb{q} \in O_\mb{p}$ with $\theta(\mb{q}) \leq 0$, $(P_\mb{X})_A^{\mb{q}}$ and $(P_\mb{X})_A^\mb{p}$ are mutually well-posed.
\end{thm}

\begin{rmk}
	For $\mb{X} = \mb{H}$, the restriction $j(\mb{p}) \neq 0$ rules out $\BMO$-Sobolev spaces, which are not in the interior of any of our complex interpolation scales.
	Also with $\mb{X} = \mb{H}$, a corresponding result is true for $\gq(\mb{p}) \in \{-1,0\}$, but well-posedness is only obtained for $\mb{q}$ near $\mb{p}$ with $\gq(\mb{q}) = \gq(\mb{p})$.
	The proof uses the same argument.
	For $\mb{X} = \mb{B}$, the restriction $i(\mb{p}) > 1$ is due to the lack of a complex interpolation theorem for quasi-Banach $Z$-spaces.
\end{rmk}

This leads to an interesting corollary: a topological characterisation of mutual well-posedness.

\begin{cor}\index{well-posedness!mutual!topological characterisation}
	Let $B = \hat{A}$, and let $\mb{p},\mb{q} \in J(\mb{X},DB) \cap \mb{E}_{\text{fin}}$ with $\gq(\mb{p}),\gq(\mb{q}) \in (-1,0)$.
	Then $(P_\mb{X})_A^\mb{p}$ and $(P_\mb{X})_A^\mb{q}$ are mutually well-posed if and only if $\mb{p}$ and $\mb{q}$ are in the same connected component of $\wellp(P_\mb{X})_A \cap J(\mb{X},DB)$ (the set of all exponents in the classification region for which $(P_{\mb{X}})_A$ is well-posed).
\end{cor}

\begin{proof}
	If $(P_\mb{X})_A^\mb{p}$ and $(P_\mb{X})_A^\mb{q}$ are mutually well-posed, then Theorem \ref{thm:cwp-int} (interpolation of mutual well-posedness) implies that $\mb{p}$ and $\mb{q}$ lie in the same connected component of $\wellp(P_\mb{X})_A \cap J(\mb{X},DB)$.
	On the other hand, Theorem \ref{thm:wp-ext} (extrapolation of well-posedness) shows that the sets
	\begin{align*}
		&\{\mb{r} \in J(\mb{H},DB) : \text{$(P_\mb{H})_A^\mb{r}$ and $(P_\mb{H})_A^\mb{p}$ mutually well-posed, $j(\mb{r}) \neq 0$, $\gq(\mb{r}) \in (-1,0)$}\}, \\
		&\{\mb{r} \in J(\mb{B},DB) : \text{$(P_\mb{B})_A^\mb{r}$ and $(P_\mb{B})_A^\mb{p}$ mutually well-posed, $\gq(\mb{r}) \in (-1,0)$}\},
	\end{align*}
	and their complements in $\wellp(P_\mb{X})_A \cap J(\mb{X},DB)$, are both open.
	Therefore if $(P_\mb{X})_A^\mb{p}$ and $(P_\mb{X})_A^\mb{q}$ are not mutually well-posed, then $\mb{p}$ and $\mb{q}$ lie in different connected components of $\wellp(P_\mb{X})_A \cap J(\mb{X},DB)$.
\end{proof}

Now we present a $\heartsuit$-duality principle for well-posedness.

\begin{thm}[$\heartsuit$-duality of well-posedness]\label{thm:wp-duality}\index{duality!of well-posedness}
	Let $B = \hat{A}$, and suppose that $\mb{p},\mb{q} \in J(\mb{X},DB) \cap \mb{E}_{\text{fin}}$.
	If $(P_\mb{X})_A^{\mb{p}}$ and $(P_\mb{X})_A^\mb{q}$ are mutually well-posed, then $(P_\mb{X})_{A^*}^{\mb{p}^\heartsuit}$ and $(P_\mb{X})_{A^*}^{\mb{q}^\heartsuit}$ are mutually well-posed.
\end{thm}

\begin{rmk}
  If $i(\mb{p}) \in (1,\infty)$, then this statement is an equivalence.
  Furthermore, if $(P_\mb{X})_A^{\mb{p}}$ is \emph{compatibly} well-posed, then $(P_\mb{X})_{A^*}^{\mb{p}^\heartsuit}$ is also compatibly well-posed, since $\mb{e}^\heartsuit = \mb{e}$.
\end{rmk}

We point out the case where $\mb{p} = (1,s)$ with $s \in (-1,0]$: in this case the result says that well-posedness of a problem with coefficients $A$ and boundary data in the Hardy--Sobolev space $\dot{H}^1_s$ (resp. the Besov space $\dot{B}^{1,1}_s)$ implies well-posedness of the corresponding problem for $A^*$ with boundary data in the image of $\BMO$--Sobolev space $\dot{BMO}_{-s}$ (resp. the H\"older space $\dot{\gL}_{-s}$) under $D$.
We elaborate on this in Remark \ref{rmk:VMO} below.

\begin{proof}[Proof of Theorem \ref{thm:wp-duality}]
We will be sketchy because all the important details of this argument have already been done by the second author, Mourgoglou, and Stahlhut (see \cite[\textsection 12.2]{AS16} and \cite[\textsection 13]{AM15}).
Recall from Remark \ref{rmk:Bduals} that $\widehat{A^*} = NB^*N =: \td{B}$.
When $\mb{p}$ is finite, the pairing
\begin{equation*}
	\langle f,g \rangle_{\bbX^\mb{p}_{DB}}^N := \langle f, Ng \rangle_{\bbX^\mb{p}_{DB}}
\end{equation*}
is a duality pairing between $\bbX^\mb{p}_{DB}$ and $\bbX^{\mb{p}^\prime}_{\td{B} D}$.
By Proposition \ref{prop:D-mapping} we have that $\nm{Dg}_{\bbX_{D\td{B}}^{\mb{p}^\heartsuit}} \simeq \nm{g}_{\bbX_{\td{B}D}^{\mb{p}^\prime}}$ whenever $g \in \mc{D}(D) \cap \bbX^{\mb{p}^\prime}_{\td{B}D}$, and so the pairing
\begin{equation}\label{eqn:heartpairing}
	\langle f,g \rangle_{\bbX_{DB}^\mb{p}}^\heartsuit := \langle f, ND^{-1} g \rangle_{\bbX_{DB}^\mb{p}}
\end{equation}
is a duality pairing between $\bbX^\mb{p}_{DB}$ and $\mc{R}(D) \cap \bbX^{\mb{p}^\heartsuit}_{D\td{B}}$.
Since $\mb{p} \in J(\mb{X},DB)$, we have identified $\mb{X}_D^\mb{p} = \mb{X}_{DB}^\mb{p}$ and $\mb{X}_D^{\mb{p}^\heartsuit} = \mb{X}_{D\td{B}}^{\mb{p}^\heartsuit}$ as completions of $\bbX_{DB}^\mb{p}$ and $\bbX_{D\td{B}}^{\mb{p}^\heartsuit}$ respectively (using a simple modification of Proposition \ref{prop:dual-interval-incl} to make the second identification), and by density the pairing \eqref{eqn:heartpairing} extends to a duality pairing between $\mb{X}_{DB}^\mb{p}$ and $\mb{X}_{D\td{B}}^{\mb{p}^\heartsuit}$.
As in the proof of \cite[Lemma 13.3]{AM15}, this pairing realises $\mb{X}_{D\td{B}}^{\mb{p}^\heartsuit, \mp}$ as the dual of $\mb{X}_{DB}^{\mb{p},\pm}$, $\mb{X}_\perp^{\mb{p}^\heartsuit}$ as the dual of $\mb{X}_\parallel^\mb{p}$, and $\mb{X}_\parallel^{\mb{p}^\heartsuit}$ as the dual of $\mb{X}_\perp^{\mb{p}}$.
The remainder of the argument precisely follows the proof of \cite[Theorem 1.6]{AM15}.
\end{proof} 

\begin{rmk}\label{rmk:VMO}
	The restriction to finite exponents in Theorem \ref{thm:wp-duality} appears because the duality $\mb{X}^{\mb{p}^\prime} \simeq (\mb{X}^{\mb{p}})^\prime$ only holds when $\mb{p}$ is finite.
	When $\mb{p} = (\infty,s;0)$, for example, we do not have $(\dot{\BMO}_s)^\prime \simeq \dot{H}_{-s}^1$.
	However, we do have
	\begin{equation*}
		\VMO^\prime \simeq H^1,
	\end{equation*}
	where $\VMO$, the space of functions of \emph{vanishing mean oscillation},\index{space!VMO@$\VMO$} is the norm closure of $C_c^\infty$ in $\BMO$.\footnote{This is the $\VMO$ space of Coifman and Weiss \cite[\textsection 4]{CW77} rather than that first introduced by Sarason \cite{dS75}.}
	In this long remark, we will sketch how `fractional $\VMO$' spaces could be included in our results; for lack of concrete applications, we will keep this at the level of a sketch.
	Given our omission of proofs, the reader should consider this as an informed conjecture.
		
	For $s \in \bbR$, define $\dot{\VMO}_s$ to be the norm closure of $C_c^\infty$ in $\dot{\BMO}_s$.
	Then we have the duality
	\begin{equation*}
		\dot{\VMO}_s^\prime \simeq \dot{H}_{-s}^1.
	\end{equation*}
	For an operator $A$ satisfying the Standard Assumptions, we may define $A$-adapted fractional pre-$\VMO$ spaces $\bbVMO_{A,s}$ analogously to the spaces $\bbH_A^{(\infty,s;0)}$ by using \emph{vanishing tent spaces} $vT^\infty_{s;0}$.\index{space!tent!vanishing}
	These may be defined as the norm closure of $T^2 \cap T^\infty_{s;0}$ in $T^\infty_{s;0}$ (the weak-star closure returns $T^\infty_{s;0}$), and their properties have been established for example by Jiang and Yang \cite{JY11} when $s = 0$.
	In particular the duality
	\begin{equation*}
		(vT_{s;0}^\infty)^\prime \simeq T_{-s;0}^1
	\end{equation*}
	holds for all $s \in \bbR$.
	
	For auxiliary functions $\gy \in \gY_\infty^\infty$, define canonical completions $\gy \mb{VMO}_{DB,s}$ as the closures of $\bbQ_{\gy,DB} \bbVMO_{DB,s}$ in $vT_{s;0}^\infty$ (or equivalently, as the norm closures of $\bbQ_{\gy,DB} \bbVMO_{DB,s}$ in $T_{s;0}^\infty$, or even as the norm closures of $\bbQ_{\gy,DB} \bbH_{DB}^{(\infty,s;0)}$ in $T_{s;0}^\infty$).
	Then the duality
	\begin{equation*}
		(\gy \mb{VMO}_{DB,s})^\prime \simeq \td{\gy} \mb{H}_{B^* D}^{(1,-s;0)}
	\end{equation*}
	follows from the arguments of Section \ref{sec:completions}.
	
	One can follow the arguments of Theorem \ref{thm:mainthm-gtr2} and characterise solutions to $(\CR)_{DB}$ in $vT_{s;0}^\infty$ (with decay at infinity) as generalised Cauchy extensions of functions in $\mb{VMO}_{s,DB}^+$ when $(\infty,s;0)^\heartsuit = (1,-s;0) \in I(\mb{H},DB^*)$.
	As in this section, one can then characterise well-posedness of Regularity and Neumann problems with $\dot{\VMO}_s$ boundary data in terms of a corresponding map $N_{\mb{VMO},DB,\bullet}$ being an isomorphism.
	
	The $\heartsuit$-duality argument of Theorem \ref{thm:wp-duality} then says that for coefficients $A$, when $(1,-s;0) \in I(\mb{H},D\hat{A^*})$, well-posedness of a Regularity or Neumann problem for $L_A$ with data in $\dot{\VMO}_s$ implies well-posedness of the adjoint problem (i.e. that for $L_{A^*}$) with data in $\dot{H}^1_{-s}$.
	Furthermore, Theorem \ref{thm:wp-duality} applied to $L_{A^*}$ for the exponent $(1,-s;0)$ then implies that the original problem is, in fact, well-posed with data in $\dot{\BMO}_s$.
	This is \emph{a priori} stronger, so we deduce that well-posedness (of a problem) with data in $\dot{\BMO}_s$ is equivalent to well-posedness with data in $\dot{\VMO}_s$.
	Thus we can generalise observations made by the second author and Mourgoglou \cite{AM14} (where $s = 0$), which followed the observation of Dindos, Kenig, and Pipher \cite{DKP11} that $\BMO$ solvability for real coefficient equations is equivalent to $\VMO$ solvability.
\end{rmk}

\section{Regions of well-posedness for certain classes of coefficients}\label{sec:regproblem}\index{region!well-posedness!for specific coefficients}

The results above show that from well-posedness of a boundary value problem for an exponent $\mb{p} \in J(\mb{X},DB)$ with $B = \hat{A}$, we may deduce well-posedness for a larger range of exponents by $\heartsuit$-duality, interpolation, and extrapolation.
As an application of this principle we consider various classes of coefficients for which we already have some information on well-posedness.

\subsection{Real coefficients, $m=1$}\label{ssec:realcoeff}

First we consider the regularity problems $(R_\mb{X})_{A}^\mb{p}$ in the real scalar case: this is a good way of exhibiting the main features of our theory, and for the case of Hardy--Sobolev boundary data these results seem to be new.

Suppose that $m = 1$ (so that $L_A u = 0$ is a single equation rather than a system) and that the entries of $A$ are real.
In this setting, there exists a number $\ga \in (0,1]$ such that for every Euclidean ball $B = B(X_0,2r)$ in $\bbR^{1+n}_+$ and every solution $u$ to $L_A u = 0$ in $B$, we have
\begin{equation}\label{eqn:DGNM}\index{De Giorgi--Nash--Moser condition}
	|u(X) - u(X^\prime)| \lesssim \bigg( \frac{|X - X^\prime|}{r} \bigg)^\ga \bigg( \bariint_{B}^{} |u|^2 \bigg)^{1/2}
\end{equation}
for all $X,X^\prime$ in the smaller ball $B(X_0,r)$.
In this case say that the coefficients $A$ satisfy the \emph{De Giorgi--Nash--Moser condition of exponent $\ga$}.
The adjoint matrix $A^*$ will also satisfy a De Giorgi--Nash--Moser condition of (possibly different) exponent $\ga^*$.

The second author and Stahlhut \cite[Corollary 13.3]{AS16} show that in this case\footnote{In fact, a weaker assumption is needed there, and the parameters $\alpha,\alpha^*$ are not necessarily those from \eqref{eqn:DGNM}. Still, such parameters exist.} we have 
\begin{equation*}
	\left(\frac{n}{n+\ga}, p_+(DB)\right) \subset I_0(\mb{H},DB), \qquad
	\left(\frac{n}{n + \ga^*},p_+(D\td{B})\right) \subset I_0(\mb{H},D\td{B}),
\end{equation*}
where $\td{B} = \hat{A^*}$ (note that $\td{B} \neq B^*$) and where $p_+(DB), p_+(D\td{B}) > 2$.
Therefore by $\heartsuit$-duality (see Proposition \ref{prop:dual-interval-incl} and Remark \ref{rmk:Bduals}), we have
\begin{equation*}
(p^+(D\td{B})^\prime,\infty) \subset I_{-1}(\mb{H},DB), \qquad
(p^+(DB)^\prime,\infty) \subset I_{-1}(\mb{H},D\td{B}).
\end{equation*}
By interpolation (Theorem \ref{thm:inclusion-interpolation}) we then have that $I(\mb{H},DB)$ contains the region pictured in Figure \ref{fig:realrange}, and $I(\mb{B},DB)$ contains the interior of that region.
The point $\mb{x}_A$ here is defined as the pictured intersection, which is a function of $n$, $\ga$, and $p^+(D\td{B})$ that we need not compute explicitly.

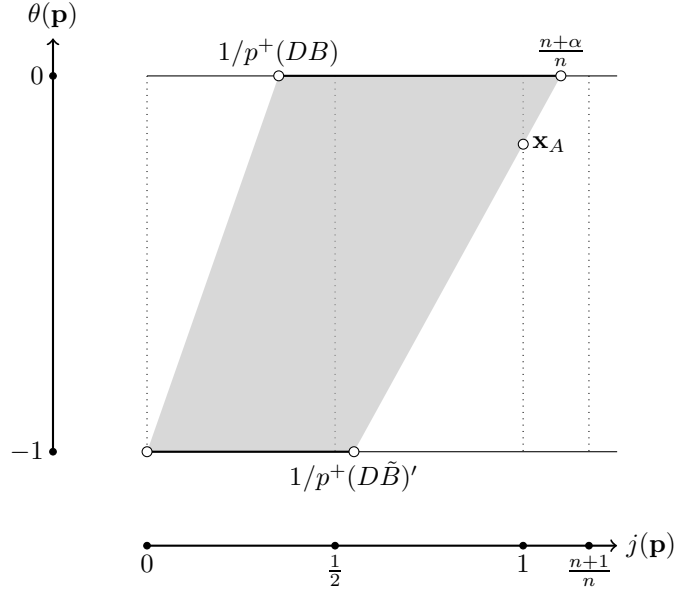
\begin{figure}
\caption{Exponents $\mb{p} \in I(\mb{H},DB)$, when $m=1$ and $A$ is real, with $B = \hat{A}$.}\label{fig:realrange}
\begin{center}
\begin{tikzpicture}[scale=2.5]
	%\draw [help lines] (0,0) grid (4,2);

	\newcommand*{\reg}{0.6}%intermediate regularity s_0 to show (modulus)

	%DB interval endpoints, can be tweaked to personal preference
	\coordinate (P00) at (1.7,2);
	\coordinate (P01) at (3.2,2);
	\coordinate (P10) at (1,0);
	\coordinate (P11) at (2.1,0);
	\coordinate (P2) at (1,0.6);
	\coordinate (P22) at (3, 1.8/1.1);

	\draw [thin] (1,2) -- (3.5,2); %s=0 axis
	\draw [thin] (1,0) -- (3.5,0); %s=-1 axis
	
	%1/p axis with labels
	\draw [thick,->] (1,-0.5) -- (3.5,-0.5);
	\draw [fill=black] (2,-0.5) circle [radius = .5pt];
	\node [below] at (2,-0.5) {$\frac{1}{2}$};
	\draw [fill=black] (3,-0.5) circle [radius = .5pt];
	\node [below] at (3,-0.5) {$1$};
	\draw [fill=black] (3.35,-0.5) circle [radius = .5pt];
	\node [below] at (3.35,-0.5) {$\frac{n+1}{n}$};
	\node [right] at (3.5,-0.5) {$j(\mb{p})$};
	\draw [fill=black] (1,-0.5) circle [radius = .5pt];
	\node [below] at (1,-0.5) {$0$};
	
	%s axis with labels
	\draw [thick,->] (0.5,0) -- (0.5,2.2);
	\node [above] at (0.5,2.2) {$\gq(\mb{p})$};
	\draw [fill=black] (0.5,2) circle [radius = .5pt];
	\node [left] at (0.5,2) {$0$};
	\draw [fill=black] (0.5,0) circle [radius = .5pt];
	\node [left] at (0.5,0) {$-1$};

	%vertical lines
	\draw [thin,dotted] (1,2) -- (1,0); %j=0 line
	\draw [thin,dotted] (2,2) -- (2,0); %j=1/2 line
	\draw [thin,dotted] (3,2) -- (3,0); %j=1 line
	\draw [thin,dotted] (3.35,2) -- (3.35,0); %j = n/(n+1) line
	
	%shade interior
	\path [fill=lightgray, opacity = 0.6] (P00)--(P01)--(P11)--(P10)--(P00);
	
	%s = 0 interval
	\draw [thick] (P00) -- (P01);
	\draw [thick] (P10) -- (P11);
	\draw [fill=white] (P00) circle [radius = .75pt];
	\node [above] at (P00) {$1/p^+(DB)$};
	\draw [fill=white] (P01) circle [radius = .75pt];
	\node [above] at (P01) {$\frac{n+\ga}{n}$};
	\draw [fill=white] (P10) circle [radius = .75pt];
	\draw [fill=white] (P11) circle [radius = .75pt];
	\node [below] at (P11) {$1/p^+(D\td{B})^\prime$};
	%\draw [fill=white] (P2) circle [radius = .75pt];
	%\node [left] at (P2) {$\mb{x}_{A^*}^\heartsuit$};
	\draw [fill=white] (P22) circle [radius = .75pt];
	\node [right] at (P22) {$\mb{x}_{A}$};	
\end{tikzpicture}
\end{center}
\end{figure}

There is also a corresponding diagram for $\td{B}$ that we have not pictured, including a corresponding exponent $\mb{x}_{A^*}$.
By applying $\heartsuit$-duality to the exponents $\mb{p} \in I(\mb{H},D\td{B})$ with $i(\mb{p}) \in (1,2)$, and another application of interpolation, we can increase these ranges to that pictured in Figure \ref{fig:realrange2}.

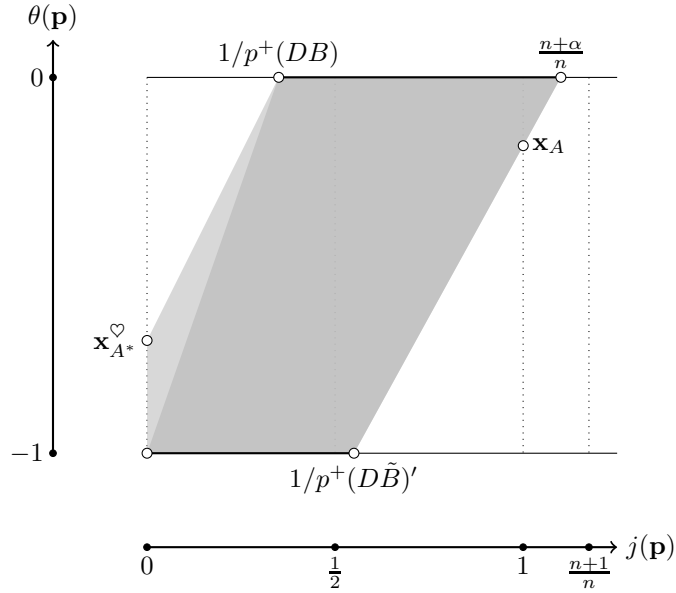
\begin{figure}
\caption{More exponents $\mb{p} \in I(\mb{H},DB)$, when $m=1$ and $A$ is real, with $B = \hat{A}$. The dark shaded region corresponds to Figure \ref{fig:realrange}.}\label{fig:realrange2}
\begin{center}
\begin{tikzpicture}[scale=2.5]
	%\draw [help lines] (0,0) grid (4,2);

	\newcommand*{\reg}{0.6}%intermediate regularity s_0 to show (modulus)

	%DB interval endpoints, can be tweaked to personal preference
	\coordinate (P00) at (1.7,2);
	\coordinate (P01) at (3.2,2);
	\coordinate (P10) at (1,0);
	\coordinate (P11) at (2.1,0);
	\coordinate (P2) at (1,0.6);
	\coordinate (P22) at (3, 1.8/1.1);

	\draw [thin] (1,2) -- (3.5,2); %s=0 axis
	\draw [thin] (1,0) -- (3.5,0); %s=-1 axis
	
	%1/p axis with labels
	\draw [thick,->] (1,-0.5) -- (3.5,-0.5);
	\draw [fill=black] (2,-0.5) circle [radius = .5pt];
	\node [below] at (2,-0.5) {$\frac{1}{2}$};
	\draw [fill=black] (3,-0.5) circle [radius = .5pt];
	\node [below] at (3,-0.5) {$1$};
	\draw [fill=black] (3.35,-0.5) circle [radius = .5pt];
	\node [below] at (3.35,-0.5) {$\frac{n+1}{n}$};
	\node [right] at (3.5,-0.5) {$j(\mb{p})$};
	\draw [fill=black] (1,-0.5) circle [radius = .5pt];
	\node [below] at (1,-0.5) {$0$};
	
	%s axis with labels
	\draw [thick,->] (0.5,0) -- (0.5,2.2);
	\node [above] at (0.5,2.2) {$\gq(\mb{p})$};
	\draw [fill=black] (0.5,2) circle [radius = .5pt];
	\node [left] at (0.5,2) {$0$};
	\draw [fill=black] (0.5,0) circle [radius = .5pt];
	\node [left] at (0.5,0) {$-1$};

	%vertical lines
	\draw [thin,dotted] (1,2) -- (1,0); %j=0 line
	\draw [thin,dotted] (2,2) -- (2,0); %j=1/2 line
	\draw [thin,dotted] (3,2) -- (3,0); %j=1 line
	\draw [thin,dotted] (3.35,2) -- (3.35,0); %j = n/(n+1) line
	
	%shade interior
	\path [fill=lightgray, opacity = 0.6] (P00)--(P01)--(P11)--(P10)--(P2)--(P00);
	\path [fill=lightgray, opacity = 0.8] (P00)--(P01)--(P11)--(P10)--(P00);
	
	%s = 0 interval
	\draw [thick] (P00) -- (P01);
	\draw [thick] (P10) -- (P11);
	\draw [fill=white] (P00) circle [radius = .75pt];
	\node [above] at (P00) {$1/p^+(DB)$};
	\draw [fill=white] (P01) circle [radius = .75pt];
	\node [above] at (P01) {$\frac{n+\ga}{n}$};
	\draw [fill=white] (P10) circle [radius = .75pt];
	\draw [fill=white] (P11) circle [radius = .75pt];
	\node [below] at (P11) {$1/p^+(D\td{B})^\prime$};
	\draw [fill=white] (P2) circle [radius = .75pt];
	\node [left] at (P2) {$\mb{x}_{A^*}^\heartsuit$};
	\draw [fill=white] (P22) circle [radius = .75pt];
	\node [right] at (P22) {$\mb{x}_{A}$};	
\end{tikzpicture}
\end{center}
\end{figure}

It has been shown that there exist $p_R(A) > 1$ (possibly small) and $0 < \ga^\sharp \leq \min(\ga,\ga^*)$ such that the Regularity problem $(R_\mb{H})_{A}^{(p,0)}$ is compatibly well-posed for all $p \in (n/(n+\ga^\sharp),p_R(A)]$, and likewise for $A^*$ (we can choose $\ga^\sharp \leq \min(\ga,\ga^*)$ which works for both $A$ and $A^*$).\footnote{The $p_R(A)$ endpoint of this result is due to Kenig and Rule in dimension $n+1 = 2$ \cite[Theorem 1.4]{KR09} and Hofmann, Kenig, Mayboroda, and Pipher in dimension $n + 1 \geq 3$ \cite[Corollary 1.2]{HKMP15.2}. The other endpoint is an extrapolation result of the second author and Mourgoglou \cite[\textsection 10.1]{AM14}.}
By the results of the previous paragraph, we have $(p,0) \in I_0(\mb{H},DB) \cap I_0(\mb{H},D\td{B})$ for all such $p$,\footnote{It is possible that $p > p_+(DB)$ or $p > p_+(D\td{B})$, in which case we have to restrict to small $p$.} and so we may apply $\heartsuit$-duality and interpolation as in the previous argument to deduce compatible well-posedness of $(R_\mb{H})_{A}^\mb{p}$ for $\mb{p}$ in the region pictured in Figure \ref{fig:realWP}, and of $(R_\mb{B})_{A}^\mb{p}$ in the interior of this region.\footnote{We can also deduce results for $BMO$-Sobolev spaces, which correspond to the unpictured $j(\mb{p}) = 0$ range.}

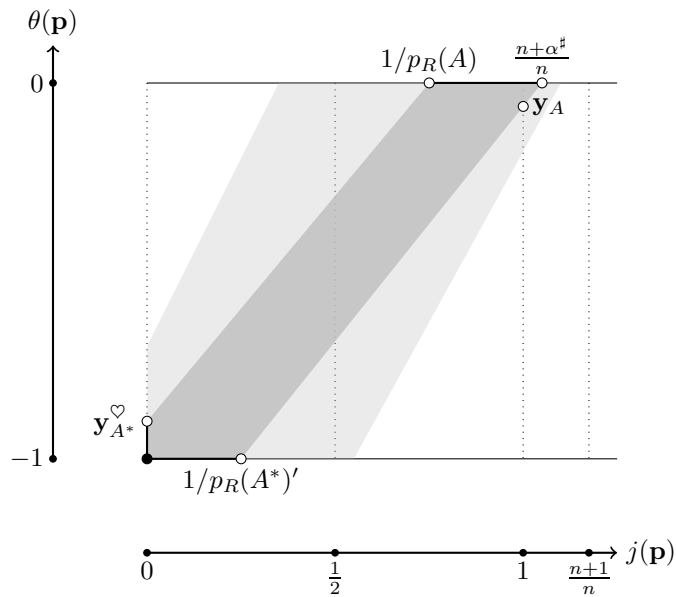
\begin{figure}
\caption{Exponents $\mb{p}$ for which $(R_\mb{H})_{A}^\mb{p}$ is compatibly well-posed (the dark shaded region). The light shaded region is from Figure \ref{fig:realrange2}.}\label{fig:realWP}
\begin{center}
\begin{tikzpicture}[scale=2.5]
	%\draw [help lines] (0,0) grid (4,2);

	\newcommand*{\reg}{0.6}%intermediate regularity s_0 to show (modulus)

	%DB interval endpoints, can be tweaked to personal preference
	\coordinate (P00) at (1.7,2);
	\coordinate (P01) at (3.2,2);
	\coordinate (P10) at (1,0);
	\coordinate (P11) at (2.1,0);
	\coordinate (P2) at (1,0.6);
	\coordinate (P22) at (3, 1.8/1.1);
	
	\coordinate (R00) at (2.5,2);
	\coordinate (R01) at (3.1,2);
	\coordinate (R10) at (1,0);
	\coordinate (R11) at (1.5,0);
	\coordinate (R2) at (1,0.2);
	\coordinate (R22) at (3,3/1.6);

	\draw [thin] (1,2) -- (3.5,2); %s=0 axis
	\draw [thin] (1,0) -- (3.5,0); %s=-1 axis
	
	%1/p axis with labels
	\draw [thick,->] (1,-0.5) -- (3.5,-0.5);
	\draw [fill=black] (2,-0.5) circle [radius = .5pt];
	\node [below] at (2,-0.5) {$\frac{1}{2}$};
	\draw [fill=black] (3,-0.5) circle [radius = .5pt];
	\node [below] at (3,-0.5) {$1$};
	\draw [fill=black] (3.35,-0.5) circle [radius = .5pt];
	\node [below] at (3.35,-0.5) {$\frac{n+1}{n}$};
	\node [right] at (3.5,-0.5) {$j(\mb{p})$};
	\draw [fill=black] (1,-0.5) circle [radius = .5pt];
	\node [below] at (1,-0.5) {$0$};
	
	%s axis with labels
	\draw [thick,->] (0.5,0) -- (0.5,2.2);
	\node [above] at (0.5,2.2) {$\gq(\mb{p})$};
	\draw [fill=black] (0.5,2) circle [radius = .5pt];
	\node [left] at (0.5,2) {$0$};
	\draw [fill=black] (0.5,0) circle [radius = .5pt];
	\node [left] at (0.5,0) {$-1$};

	%vertical lines
	\draw [thin,dotted] (1,2) -- (1,0); %j=0 line
	\draw [thin,dotted] (2,2) -- (2,0); %j=1/2 line
	\draw [thin,dotted] (3,2) -- (3,0); %j=1 line
	\draw [thin,dotted] (3.35,2) -- (3.35,0); %j = n/(n+1) line
	
	%shade interior
	\path [fill=lightgray, opacity = 0.3] (P00)--(P01)--(P11)--(P10)--(P2)--(P00);
	\path [fill=lightgray, opacity = 0.8] (R00)--(R01)--(R11)--(R10)--(R2)--(R00);
	
	%s = 0 interval
	\draw [thick] (R00) -- (R01);
	\draw [thick] (R10) -- (R11);
	\draw [thick] (1,0) -- (R2);
	\draw [fill=white] (R00) circle [radius = .75pt];
	\node [above] at (R00) {$1/p_R(A)$};
	\draw [fill=white] (R01) circle [radius = .75pt];
	\node [above] at (R01) {$\frac{n+\ga^\sharp}{n}$};
	\draw [fill=black] (R10) circle [radius = .75pt];
	\draw [fill=white] (R11) circle [radius = .75pt];
	\node [below] at (R11) {$1/p_R(A^*)^\prime$};
	\draw [fill=white] (R2) circle [radius = .75pt];
	\node [left] at (R2) {$\mb{y}_{A^*}^\heartsuit$};
	\draw [fill=white] (R22) circle [radius = .75pt];
	\node [right] at (R22) {$\mb{y}_A$};	
\end{tikzpicture}
\end{center}
\end{figure}

We can expand this region slightly for Besov spaces: applying $\heartsuit$-duality to compatible well-posedness of $(R_\mb{B})_{A^*}^\mb{p}$ for $\mb{p}$ in the open triangle with vertices $\mb{y}_{A^*}$ (defined pictorially in Figure \ref{fig:realWP}), $(n+\ga^\sharp/n, 0)$, and $(1,0)$, we find that $(R_\mb{B})_{A}^{(\infty,\ga;0)}$ is compatibly well-posed for all $\ga \in (-1,-1-\ga^\sharp)$. Therefore (after another iteration of interpolation) we have well-posedness of $(R_\mb{B})_A^\mb{p}$ for all $\mb{p}$ in the shaded region of Figure \ref{fig:realWPB}.
This is the same region obtained by Barton and Mayboroda for compatible well-posedness of $(R_\mb{B})_A^\mb{p}$ in this setting \cite[Figure 3.5]{BM16}.\footnote{The only difference is in the light shaded `region of applicability': ours depends on exponent $p^+(DB)$ from \cite{AS16}, while that in Barton--Mayboroda is in terms of the exponent appearing in Meyers' theorem \cite[Theorem 2]{nM63} (see also \cite[Lemma 2.12]{BM16}). It is not clear whether there is any relationship between these exponents.}
To recover the result of \cite[Corollary 3.24]{BM16}, one need only apply Lemma \ref{lem:decaylem} (which is valid for the region of $\mb{p}$ pictured in Figure \ref{fig:realWPB}, see Figure \ref{fig:decaylem2}) to remove the decay assumption at infinity from $(R_\mb{B})_{A}^{\mb{p}}$, and the trace theorem \cite[Theorem 6.3]{BM16} to replace our boundary condition with a trace condition.

In the case of dimension $n+1=2$, Neumann problems can be transposed to Regularity problems using the conjugate equation. This idea goes back to Morrey \cite{Mo66}. 
Thus the same results hold for Neumann problems in this case.

In the case that $A$ is symmetric and $n \geq 1$, in addition to the above assumptions, results of Kenig and Pipher \cite{KP93} imply that we have the additional information $p_R(A)  > 2$, and regions of compatible well-posedness of $(R_\mb{X})_A^\mb{p}$ can be expanded accordingly.
Note that the results there are proven in the context of the unit ball instead of the upper-half space, but similar arguments can be used.
An argument showing well-posedness at $p=2$ in our context has been given by the second author, Axelsson and McIntosh \cite[Theorem 2.2]{AAM10}.
From this, {\v{S}}ne{\u\i}berg's extrapolation gives $p_R(A)  > 2$.
The corresponding Neumann problems $(N_\mb{X})_A^\mb{p}$ are well-posed for an analogous range of $\mb{p}$ by the same arguments.

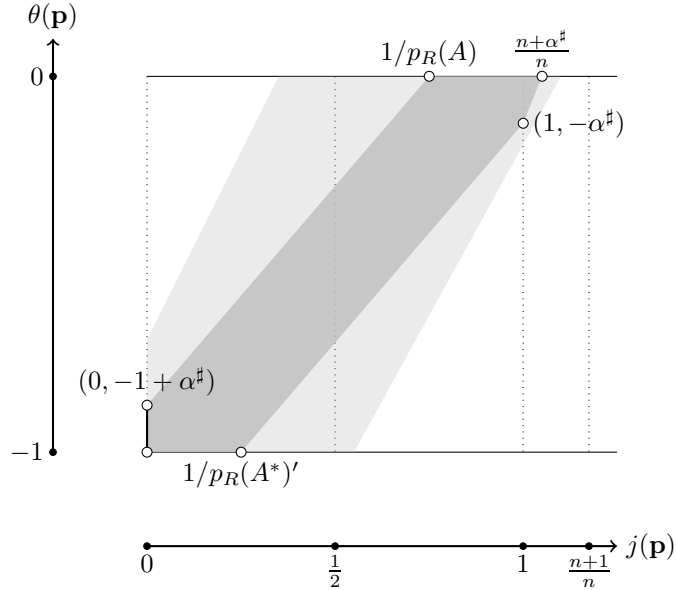
\begin{figure}\
\caption{Exponents $\mb{p}$ for which $(R_\mb{B})_{A}^\mb{p}$ is compatibly well-posed (the dark shaded region); this includes no exponents with $\gq(\mb{p})=0$ or $\gq(\mb{p})=-1$. The light shaded region is from Figure \ref{fig:realrange2}.}\label{fig:realWPB}
\begin{center}
\begin{tikzpicture}[scale=2.5]
	%\draw [help lines] (0,0) grid (4,2);

	\newcommand*{\reg}{0.6}%intermediate regularity s_0 to show (modulus)

	%DB interval endpoints, can be tweaked to personal preference
	\coordinate (P00) at (1.7,2);
	\coordinate (P01) at (3.2,2);
	\coordinate (P10) at (1,0);
	\coordinate (P11) at (2.1,0);
	\coordinate (P2) at (1,0.6);
	\coordinate (P22) at (3, 1.8/1.1);
	
	\coordinate (R00) at (2.5,2);
	\coordinate (R01) at (3.1,2);
	\coordinate (R10) at (1,0);
	\coordinate (R11) at (1.5,0);
	\coordinate (R2) at (1,0.25);
	\coordinate (R22) at (3,1.75);

	\draw [thin] (1,2) -- (3.5,2); %s=0 axis
	\draw [thin] (1,0) -- (3.5,0); %s=-1 axis
	
	%1/p axis with labels
	\draw [thick,->] (1,-0.5) -- (3.5,-0.5);
	\draw [fill=black] (2,-0.5) circle [radius = .5pt];
	\node [below] at (2,-0.5) {$\frac{1}{2}$};
	\draw [fill=black] (3,-0.5) circle [radius = .5pt];
	\node [below] at (3,-0.5) {$1$};
	\draw [fill=black] (3.35,-0.5) circle [radius = .5pt];
	\node [below] at (3.35,-0.5) {$\frac{n+1}{n}$};
	\node [right] at (3.5,-0.5) {$j(\mb{p})$};
	\draw [fill=black] (1,-0.5) circle [radius = .5pt];
	\node [below] at (1,-0.5) {$0$};
	
	%s axis with labels
	\draw [thick,->] (0.5,0) -- (0.5,2.2);
	\node [above] at (0.5,2.2) {$\gq(\mb{p})$};
	\draw [fill=black] (0.5,2) circle [radius = .5pt];
	\node [left] at (0.5,2) {$0$};
	\draw [fill=black] (0.5,0) circle [radius = .5pt];
	\node [left] at (0.5,0) {$-1$};

	%vertical lines
	\draw [thin,dotted] (1,2) -- (1,0); %j=0 line
	\draw [thin,dotted] (2,2) -- (2,0); %j=1/2 line
	\draw [thin,dotted] (3,2) -- (3,0); %j=1 line
	\draw [thin,dotted] (3.35,2) -- (3.35,0); %j = n/(n+1) line
	
	%shade interior
	\path [fill=lightgray, opacity = 0.3] (P00)--(P01)--(P11)--(P10)--(P2)--(P00);
	\path [fill=lightgray, opacity = 0.8] (R00)--(R01)--(R22)--(R11)--(R10)--(R2)--(R00);
	
	%s = 0 interval
	\draw [thick] (R10) -- (R2);
	\draw [fill=white] (R00) circle [radius = .75pt];
	\node [above] at (R00) {$1/p_R(A)$};
	\draw [fill=white] (R01) circle [radius = .75pt];
	\node [above] at (R01) {$\frac{n+\ga^\sharp}{n}$};
	\draw [fill=white] (R10) circle [radius = .75pt];
	\draw [fill=white] (R11) circle [radius = .75pt];
	\node [below] at (R11) {$1/p_R(A^*)^\prime$};
	\draw [fill=white] (R2) circle [radius = .75pt];
	\node [above] at (R2) {$(0,-1+\ga^\sharp)$};
	\draw [fill=white] (R22) circle [radius = .75pt];
	\node [right] at (R22) {$(1,-\ga^\sharp)$};	
\end{tikzpicture}
\end{center}
\end{figure}

\subsection{Constant coefficients}

Suppose that $A$ is constant.
By results of the second author and Stahlhut \cite[\textsection 3.4 and Theorem 5.1]{AS16}, we have $I_0(\mb{H},DB) = (n/(n+1),\infty)$, and therefore by $\heartsuit$-duality and interpolation we have
\begin{align*}
 	J(\mb{H},DB) &= I_\text{max}, \\
	J(\mb{B},DB) &= I_\text{max} \cap \{\mb{p} \in \mb{E} : \gq(\mb{p}) \in (-1,0)\}.
\end{align*}
Furthermore, it was observed by the second author and Mourgoglou \cite[\textsection 14.1]{AM15} that the Regularity and Neumann problems $(P_\mb{H})_A^\mb{p}$ are compatibly well-posed for all $\mb{p} \in I_0(\mb{H},DB)$.
By $\heartsuit$-duality and interpolation (arguing as we did for real symmetric scalar equations), we thus have compatible well-posedness of $(P_\mb{H})_A^\mb{p}$ for all $\mb{p} \in I_\text{max}$, and of $(P_\mb{B})_A^\mb{p}$ for all $\mb{p} \in I_\text{max}$ with $\gq(\mb{p}) \in (-1,0)$.
We note that these results may be more simply obtained by Fourier multiplier theory; our framework is overkill for constant coefficients.

\subsection{Self-adjoint coefficients}

Suppose that $A$ is self-adjoint, $A = A^*$, but not necessarily real and with $m=1$.
Then \cite[Theorem 2.2]{AAM10} shows that both Regularity and Neumann problems $(P_{\mb{H}})^{(2,0)}_A$ are compatibly well-posed, and by $\heartsuit$-duality $(P_{\mb{H}})^{(2,-1)}_A$ is also compatibly well-posed.
Interpolation then gives compatible well-posedness of $(P_{\mb{H}})_A^{(2,s)}$ and $(P_{\mb{B}})_A^{(2,s)}$ for all $s \in (-1,0)$, and extrapolation (via Theorem \ref{thm:wp-ext}) gives compatible well-posedness for exponents $\mb{p}$ sufficiently close to this line.

\subsection{Diagonal block coefficients}\label{subsec:block}

Suppose that $A$ has the form
\begin{equation*}
	A = \begin{bmatrix} A_{\perp \perp} & 0 \\ 0 & A_{\parallel \parallel} \end{bmatrix}.
\end{equation*}
By results of the second author and Mourgoglou \cite[\textsection 14.2]{AM15}, the problems $(R_\mb{H})_A^\mb{p}$ and $(N_\mb{H})_A^\mb{p}$ are compatibly well-posed for all $\mb{p} \in I_0(\mb{H},DB)$.
By $\heartsuit$-duality, we also have compatible well-posedness of these problems for all $\mb{p} \in I_0(\mb{H},DB^*)^\heartsuit$.
Interpolating between the intervals $I_0(\mb{H},DB)$ and $I_0(\mb{H},DB^*)^\heartsuit \cap \mb{E}_\text{fin}$ then yields a set of exponents $\mb{p}$ such that the problems $(P_\mb{H})_A^\mb{p}$ and $(P_\mb{B})_A^\mb{p}$ are all compatibly well-posed.

\subsection{Block triangular coefficients}\label{subsec:trian}

Suppose that $A$ has the lower triangular block form
\begin{equation*}
	A = \begin{bmatrix} A_{\perp \perp} & 0 \\ A_{\parallel \perp} & A_{\parallel \parallel} \end{bmatrix}.
\end{equation*}
Then results of the second author with McIntosh and Mourgoglou \cite[Theorem 6.2]{AMM13} imply that the Neumann problem $(N_\mb{H})_A^{(2,0)}$ is compatibly\footnote{The compatibility is not explicitly detailed there, but the argument is not too difficult. See \cite[Theorem 2.32]{AEN07}. Likewise for the Regularity problem below.} well-posed.
On the other hand, if $A$ has the upper triangular form
\begin{equation*}
	A = \begin{bmatrix} A_{\perp \perp} & A_{\perp \parallel} \\ 0 & A_{\parallel \parallel} \end{bmatrix}
\end{equation*}
then by \cite[Theorem 6.4]{AMM13} the Regularity problem $(R_{\mb{H}})_A^{(2,0)}$ is compatibly well-posed.
Since taking adjoints interchanges lower triangular and upper triangular block form matrices, we find by $\heartsuit$-duality that $(R_\mb{H})_A^{(2,-1)}$ is compatibly well-posed when $A$ is lower triangular, and $(N_{\mb{H}})_A^{(2,-1)}$ is compatibly well-posed when $A$ is upper triangular.
As in the previous examples we can use interpolation and extrapolation to deduce compatible well-posedness of these problems for larger open regions of exponents.

\subsection{An example with incompatible well-posedness}\index{well-posedness!incompatible}

Let $m=n=1$.
Since we are working in the plane, results of the second author and Stahlhut \cite[Proposition 3.11 and Theorem 5.1]{AS16}, along with the usual $\heartsuit$-duality and interpolation arguments, yield $J(\mb{H},DB) = I_\text{max}$ and $J(\mb{B},DB) = I_\text{max} \cap \{\gq(\mb{p}) \in (-1,0)\}$.

For all $k \in \bbR$ define the coefficient matrix
\begin{equation*}
	A_k(x) := \begin{bmatrix} 1 & k\sgn(x) \\ -k\sgn(x) & 1 \end{bmatrix}.
\end{equation*}
The following well-posedness result was proven by Axelsson \cite[Theorem 1.2]{aA10}.

\begin{thm}[Axelsson]
	Suppose $p \in (1/2,\infty)$. Then
	\begin{itemize}
		\item the Regularity problem $(R_\mb{H})_{A_k}^{(p,0)}$ is well-posed if $k \neq -\tan(\gp/2p)$, and
		\item the Neumann problem $(N_\mb{H})_{A_k}^{(p,0)}$ is well-posed if $k \neq \tan(\gp/2p)$.
	\end{itemize}
\end{thm}

However, results of Kenig, Koch, Pipher, and Toro \cite{KKPT00} and of Kenig and Rule \cite{KR09} (see \cite[Theorem 1.1]{aA10}) show that $(R_\mb{H})_{A_k}^{(p,0)}$ (resp. $(N_{\mb{H}})_{A_k}^{(p,0)}$ ) is not \emph{compatibly} well-posed when $k < -\tan(\gp/2p)$ (resp. $k > \tan(\gp/2p)$).
We will use our theory to deduce some results for more general exponents.
Let us consider only Regularity problems, as Neumann problems can be handled by the same arguments.

Rewriting the conditions above in terms of $p$, we find that $(R_{\mb{H}})_{A_k}^{(p,0)}$ is compatibly well-posed for all $p \in (1/2,\infty)$ when $k > 0$, and for all $p \in (\frac{1}{2},\frac{\gp}{2\arctan(-k)})$ when $k < 0$.
Furthermore, when $k < 0$, $(R_{\mb{H}})_{A_k}^{(p,0)}$ is (incompatibly) well-posed for all $p > \frac{\gp}{2\arctan(-k)}$.
Note that $\frac{\gp}{2\arctan(-k)} > 1$.

Let us fix $k < 0$ for convenience.
Since $A_k^* = A_{-k}$, $\heartsuit$-duality tells us that $(R_{\mb{H}})_{A_k}^{(p,-1)}$ is compatibly well-posed for all $p \in (1,\infty)$.
Therefore, by two iterations of $\heartsuit$-duality and interpolation, as we carried out in Subsection \ref{ssec:realcoeff}, we deduce that $(R_{\mb{H}})_{A_k}^{\mb{p}}$ and $(R_{\mb{H}})_{A_{-k}}^{\mb{p}}$ are compatibly well-posed for the exponents $\mb{p}$ pictured in Figures \ref{fig:rosen1} and \ref{fig:rosen2}.
Real interpolation then yields compatible well-posedness for the problems $(R_{\mb{B}})_{A_k}^{\mb{p}}$ for $\mb{p}$ in the interior of this region.
One can argue similarly to find regions of well-posedness for Neumann problems.

\begin{figure}
\caption{Exponents $\mb{p}$ for which $(R_\mb{H})_{A_k}^\mb{p}$ is compatibly well-posed (the dark shaded region), with $k < 0$. The thick dashed line consists of exponents $\mb{p}$ such that $(R_\mb{H})_{A_k}^\mb{p}$ is well-posed, but not compatibly well-posed. The exponent $\mb{x}_{-k}$ is defined geometrically in Figure \ref{fig:rosen2}. The region $j(\mb{p}) > 1$ is not to scale.}\label{fig:rosen1}
\begin{center}
\begin{tikzpicture}[scale=2.5]
	%\draw [help lines] (0,0) grid (4,2);

	%\newcommand*{\reg}{0.6}%intermediate regularity s_0 to show (modulus)

	%DB interval endpoints, can be tweaked to personal preference
	\coordinate (P00) at (1,2);
	\coordinate (P01) at (4,2);
	\coordinate (P10) at (1,0);
	\coordinate (P11) at (3,0);
	
	\coordinate (R00) at (2.9,2);
	\coordinate (R01) at (4,2);
	\coordinate (R10) at (1,0);
	\coordinate (R11) at (3,0);
	\coordinate (X) at (1,0.8);
	%\coordinate (R22) at (3,1.75);

	%\draw [thin] (1,2) -- (4,2); %s=0 axis
	%\draw [thin] (1,0) -- (4,0); %s=-1 axis
	
	%1/p axis with labels
	\draw [thick,->] (1,-0.5) -- (4.2,-0.5);
	\draw [fill=black] (2,-0.5) circle [radius = .5pt];
	\node [below] at (2,-0.5) {$\frac{1}{2}$};
	\draw [fill=black] (3,-0.5) circle [radius = .5pt];
	\node [below] at (3,-0.5) {$1$};
	\draw [fill=black] (4,-0.5) circle [radius = .5pt];
	\node [below] at (4,-0.5) {$2$};
	\node [right] at (4.2,-0.5) {$j(\mb{p})$};
	\draw [fill=black] (1,-0.5) circle [radius = .5pt];
	\node [below] at (1,-0.5) {$0$};
	
	%s axis with labels
	\draw [thick,->] (0.5,0) -- (0.5,2.2);
	\node [above] at (0.5,2.2) {$\gq(\mb{p})$};
	\draw [fill=black] (0.5,2) circle [radius = .5pt];
	\node [left] at (0.5,2) {$0$};
	\draw [fill=black] (0.5,0) circle [radius = .5pt];
	\node [left] at (0.5,0) {$-1$};
	\draw [fill=black] (0.5,1) circle [radius = .5pt];
	\node [left] at (0.5,1) {$-1/2$};

	%vertical lines
	\draw [thin,dotted] (1,2) -- (1,0); %j=0 line
	\draw [thin,dotted] (2,2) -- (2,0); %j=1/2 line
	\draw [thin,dotted] (3,2) -- (3,0); %j=1 line
	\draw [thin,dotted] (4,2) -- (4,0); %j = n/(n+1) line
	
	%shade interior
	%\path [fill=lightgray, opacity = 0.3] (P00)--(P01)--(P11)--(P10)--(P2)--(P00);
	\path [fill=lightgray, opacity = 0.8] (R00)--(R01)--(R11)--(R10)--(X)--(R00);
	
	%interval of non-compatible wellposedness
	\draw [thick,dashed] (1,2) -- (R00);
	
	%intervals
	\draw [thick] (R10) -- (X);
	\draw [thick] (R10) -- (R11);
	\draw [thick] (R00) -- (R01);

	%X
	\draw [fill=white] (X) circle [radius = .75pt];
	\node [left] at (X) {$\mb{x}_{-k}^\heartsuit$};
	
	\draw [fill=white] (R00) circle [radius = .75pt];
	\node [above] at (R00) {$(\frac{\gp}{2\arctan(-k)},1)$};
	\draw [fill=white] (R01) circle [radius = .75pt];
	%\node [above] at (R01) {$(1/2,0)$};
	\draw [fill=black] (R10) circle [radius = .75pt];
	\draw [fill=white] (R11) circle [radius = .75pt];
	\draw [fill=white] (1,2) circle [radius = .75pt];
	%\node [below] at (R11) {$(1,-1)$};
	%\draw [fill=white] (R2) circle [radius = .75pt];
	%\node [above] at (R2) {$(0,-1+\ga^\sharp)$};
	%\draw [fill=white] (R22) circle [radius = .75pt];
	%\node [right] at (R22) {$(1,-\ga^\sharp)$};	
	
\end{tikzpicture}
\end{center}
\end{figure}
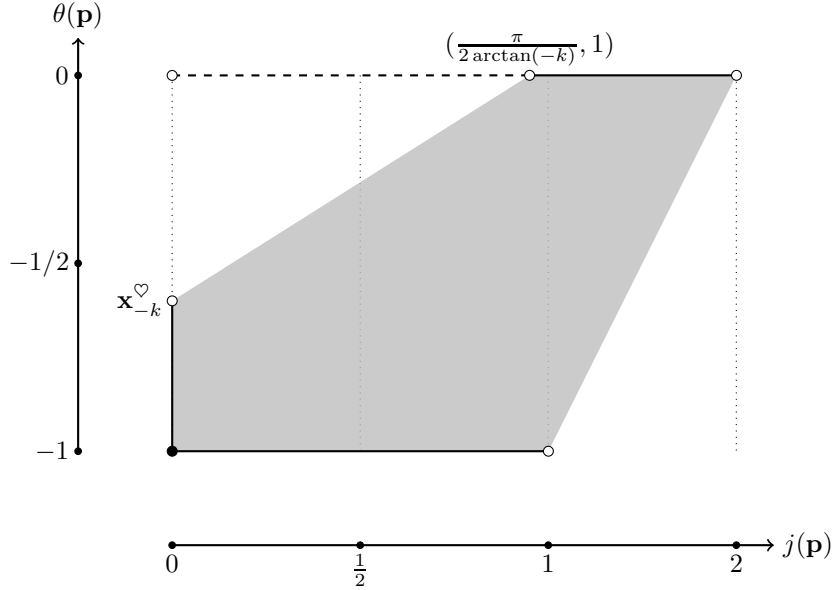

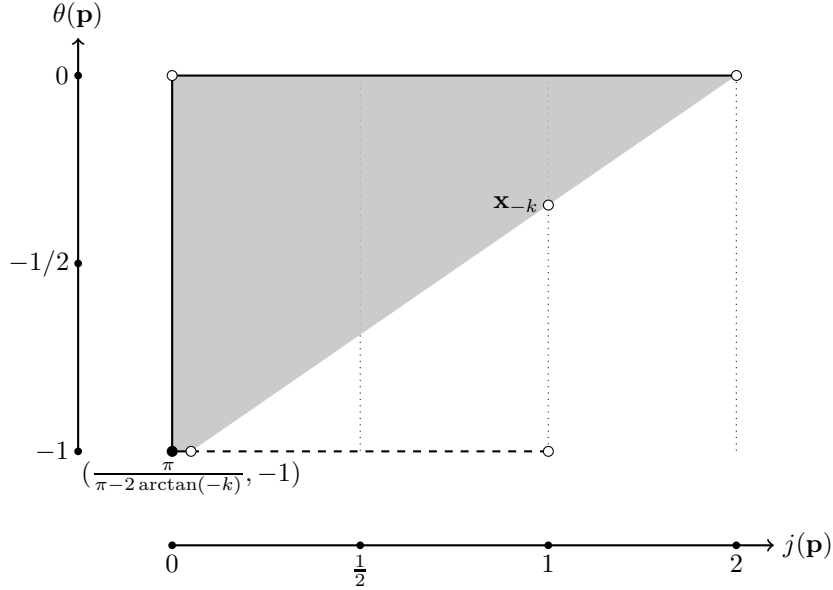
\begin{figure}
\caption{Exponents $\mb{p}$ for which $(R_\mb{H})_{A_{-k}}^\mb{p}$ is compatibly well-posed (the dark shaded region), with $k < 0$. The thick dashed line consists of exponents $\mb{p}$ such that $(R_\mb{H})_{A_{-k}}^\mb{p}$ is well-posed, but not compatibly well-posed. The exponent $\mb{x}_{-k}$ is defined as the pictured intersection. The region $j(\mb{p}) > 1$ is not to scale.}\label{fig:rosen2}
\begin{center}
\begin{tikzpicture}[scale=2.5]
	%\draw [help lines] (0,0) grid (4,2);

	%\newcommand*{\reg}{0.6}%intermediate regularity s_0 to show (modulus)

	%DB interval endpoints, can be tweaked to personal preference
	\coordinate (P00) at (1,2);
	\coordinate (P01) at (4,2);
	\coordinate (P10) at (1,0);
	\coordinate (P11) at (3,0);
	
	\coordinate (R01) at (4,2);
	\coordinate (R10) at (1,0);
	\coordinate (R11) at (3,0);
	\coordinate (X) at (3,3.8/2.9);
	\coordinate (Y) at (1.1,0);
	
	%\draw [thin] (1,2) -- (4,2); %s=0 axis
	%\draw [thin] (1,0) -- (4,0); %s=-1 axis
	
	%1/p axis with labels
	\draw [thick,->] (1,-0.5) -- (4.2,-0.5);
	\draw [fill=black] (2,-0.5) circle [radius = .5pt];
	\node [below] at (2,-0.5) {$\frac{1}{2}$};
	\draw [fill=black] (3,-0.5) circle [radius = .5pt];
	\node [below] at (3,-0.5) {$1$};
	\draw [fill=black] (4,-0.5) circle [radius = .5pt];
	\node [below] at (4,-0.5) {$2$};
	\node [right] at (4.2,-0.5) {$j(\mb{p})$};
	\draw [fill=black] (1,-0.5) circle [radius = .5pt];
	\node [below] at (1,-0.5) {$0$};
	
	%s axis with labels
	\draw [thick,->] (0.5,0) -- (0.5,2.2);
	\node [above] at (0.5,2.2) {$\gq(\mb{p})$};
	\draw [fill=black] (0.5,2) circle [radius = .5pt];
	\node [left] at (0.5,2) {$0$};
	\draw [fill=black] (0.5,0) circle [radius = .5pt];
	\node [left] at (0.5,0) {$-1$};
	\draw [fill=black] (0.5,1) circle [radius = .5pt];
	\node [left] at (0.5,1) {$-1/2$};

	%vertical lines
	\draw [thin,dotted] (1,2) -- (1,0); %j=0 line
	\draw [thin,dotted] (2,2) -- (2,0); %j=1/2 line
	\draw [thin,dotted] (3,2) -- (3,0); %j=1 line
	\draw [thin,dotted] (4,2) -- (4,0); %j = n/(n+1) line
	
	%shade interior
	%\path [fill=lightgray, opacity = 0.3] (P00)--(P01)--(P11)--(P10)--(P2)--(P00);
	\path [fill=lightgray, opacity = 0.8] (1,0)--(Y)--(4,2)--(1,2)--(1,0);
	
	%interval of non-compatible wellposedness
	\draw [thick,dashed] (Y) -- (3,0);
	
	%intervals
	\draw [thick] (1,0) -- (1,2);
	\draw [thick] (1,0) -- (Y);
	\draw [thick] (1,2) -- (4,2);

	%X
	\draw [fill=white] (X) circle [radius = .75pt];
	\node [left] at (X) {$\mb{x}_{-k}$};
	
	\draw [fill=white] (Y) circle [radius = .75pt];
	\node [below] at (Y) {$(\frac{\gp}{\gp - 2\arctan(-k)},-1)$};
	\draw [fill=white] (R01) circle [radius = .75pt];
	%\node [above] at (R01) {$(1/2,0)$};
	\draw [fill=black] (R10) circle [radius = .75pt];
	\draw [fill=white] (R11) circle [radius = .75pt];
	\draw [fill=white] (1,2) circle [radius = .75pt];
	%\node [below] at (R11) {$(1,-1)$};
	%\draw [fill=white] (R2) circle [radius = .75pt];
	%\node [above] at (R2) {$(0,-1+\ga^\sharp)$};
	%\draw [fill=white] (R22) circle [radius = .75pt];
	%\node [right] at (R22) {$(1,-\ga^\sharp)$};	
	
\end{tikzpicture}
\end{center}
\end{figure}

\begin{rmk}
We would like to extrapolate and find more exponents for which $(R_{\mb{H}})_{A_k}^{\mb{p}}$ is well-posed (but not compatibly well-posed).
However, our extrapolation theorem (Theorem \ref{thm:wp-ext}) only holds for exponents $\mb{p}$ with $\gq(\mb{p}) \in (-1,0)$, and we would need to extrapolate from exponents with $\gq(\mb{p}) = 0$.
We conjecture that there is a region of exponents $\mb{p}$, \emph{relatively open with respect to $I_{\text{max}}$}, such that $(R_{\mb{H}})_{A_k}^{\mb{p}}$ is incompatibly well-posed.
Such a region is pictured in Figure \ref{fig:rosen3}.
\end{rmk}

\begin{figure}
\caption{A \emph{conjectural} extrapolation of Figure \ref{fig:rosen1}. Two disjoint mutual well-posedness classes for $(R_\mb{H})_{A_k}^\mb{p}$, with $k < 0$. Again, the region $j(\mb{p}) > 1$ is not to scale.}\label{fig:rosen3}
\begin{center}
\begin{tikzpicture}[scale=2.5]
	%\draw [help lines] (0,0) grid (4,2);

	%\newcommand*{\reg}{0.6}%intermediate regularity s_0 to show (modulus)

	%DB interval endpoints, can be tweaked to personal preference
	\coordinate (P00) at (1,2);
	\coordinate (P01) at (4,2);
	\coordinate (P10) at (1,0);
	\coordinate (P11) at (3,0);
	
	\coordinate (R00) at (2.9,2);
	\coordinate (R01) at (4,2);
	\coordinate (R10) at (1,0);
	\coordinate (R11) at (3,0);
	\coordinate (X) at (1,0.8);
	%\coordinate (R22) at (3,1.75);

	%\draw [thin] (1,2) -- (4,2); %s=0 axis
	%\draw [thin] (1,0) -- (4,0); %s=-1 axis
	
	%1/p axis with labels
	\draw [thick,->] (1,-0.5) -- (4.2,-0.5);
	\draw [fill=black] (2,-0.5) circle [radius = .5pt];
	\node [below] at (2,-0.5) {$\frac{1}{2}$};
	\draw [fill=black] (3,-0.5) circle [radius = .5pt];
	\node [below] at (3,-0.5) {$1$};
	\draw [fill=black] (4,-0.5) circle [radius = .5pt];
	\node [below] at (4,-0.5) {$2$};
	\node [right] at (4.2,-0.5) {$j(\mb{p})$};
	\draw [fill=black] (1,-0.5) circle [radius = .5pt];
	\node [below] at (1,-0.5) {$0$};
	
	%s axis with labels
	\draw [thick,->] (0.5,0) -- (0.5,2.2);
	\node [above] at (0.5,2.2) {$\gq(\mb{p})$};
	\draw [fill=black] (0.5,2) circle [radius = .5pt];
	\node [left] at (0.5,2) {$0$};
	\draw [fill=black] (0.5,0) circle [radius = .5pt];
	\node [left] at (0.5,0) {$-1$};
	\draw [fill=black] (0.5,1) circle [radius = .5pt];
	\node [left] at (0.5,1) {$-1/2$};

	%vertical lines
	\draw [thin,dotted] (1,2) -- (1,0); %j=0 line
	\draw [thin,dotted] (2,2) -- (2,0); %j=1/2 line
	\draw [thin,dotted] (3,2) -- (3,0); %j=1 line
	\draw [thin,dotted] (4,2) -- (4,0); %j = n/(n+1) line
	
	%shade interior
	%\path [fill=lightgray, opacity = 0.3] (P00)--(P01)--(P11)--(P10)--(P2)--(P00);
	\path [fill=lightgray, opacity = 0.8] (R00)--(R01)--(R11)--(R10)--(X)--(R00);
	
	%interval of non-compatible wellposedness
	\draw [thick,dashed] (1,2) -- (R00);
	
	%intervals
	\draw [thick] (R10) -- (X);
	\draw [thick] (R10) -- (R11);
	\draw [thick] (R00) -- (R01);

	%X
	\draw [fill=white] (X) circle [radius = .75pt];
	\node [left] at (X) {$\mb{x}_{-k}^\heartsuit$};
	
	\draw [fill=white] (R00) circle [radius = .75pt];
	\node [above] at (R00) {$(\frac{\gp}{2\arctan(-k)},1)$};
	\draw [fill=white] (R01) circle [radius = .75pt];
	%\node [above] at (R01) {$(1/2,0)$};
	\draw [fill=black] (R10) circle [radius = .75pt];
	\draw [fill=white] (R11) circle [radius = .75pt];
	\draw [fill=white] (1,2) circle [radius = .75pt];
	%\node [below] at (R11) {$(1,-1)$};
	%\draw [fill=white] (R2) circle [radius = .75pt];
	%\node [above] at (R2) {$(0,-1+\ga^\sharp)$};
	%\draw [fill=white] (R22) circle [radius = .75pt];
	%\node [right] at (R22) {$(1,-\ga^\sharp)$};	
	
	%extrapolation
	\fill [lightgray, opacity= 0.3] (1,2) to[bend right] (R00) -- (1,2);
	\node [below] at (2,1.95) {?};
	
\end{tikzpicture}
\end{center}
\end{figure}
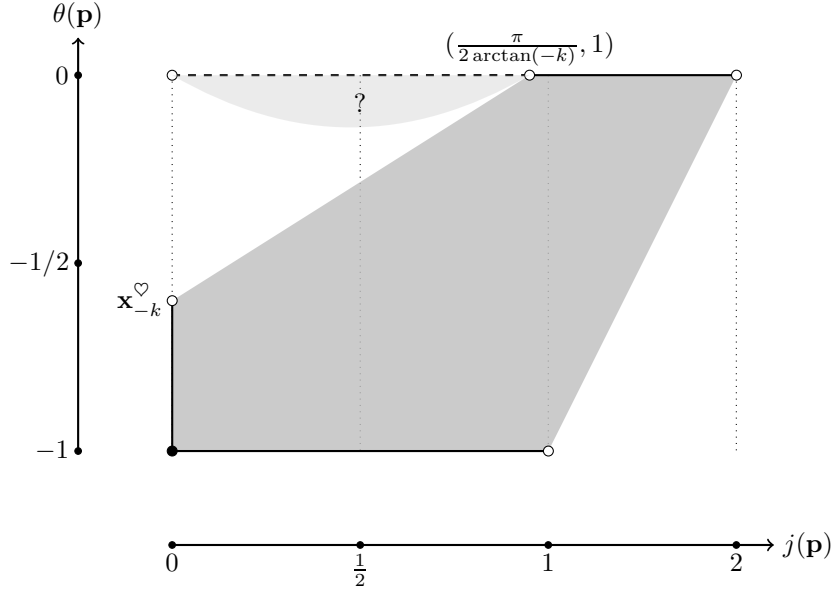

\section{Stability of well-posedness under perturbation of coefficients}\label{sec:wp-pert}\index{well-posedness!stability in coefficients}

In this section, we turn to studying stability in the coefficients of well-posedness, that is, how well-posedness results are affected by perturbing the coefficients of the system $L_A u = 0$ in $L^\infty$.

\begin{thm}[Stability of well-posedness in coefficients]\label{thm:wp-pert}
	Suppose that $A_0$ and $A$ are coefficients, let $B_0 = \widehat{A_0}$ and $B = \hat{A}$, and suppose $\mb{p} \in J(\mb{X},DB)$ for all $B$ with $\nm{B - B_0}_\infty$ sufficiently small, with uniform implicit constants in $B$.
	If $\mb{p} \in \wellp(P_\mb{X})_{A_0}$, and if $\nm{A_0 - A}_\infty$ is sufficiently small, then $\mb{p} \in \wellp(P_{\mb{X}})_{A}$.
\end{thm}

\begin{rmk}
	Note that in this theorem we make a rather strong \emph{a priori} assumption that $\mb{p}$ is in the classification region $J(\mb{X},DB)$ for all $B$ sufficiently close to $B_0$ (which includes $B = B_0$ in particular), and we also assume that the implicit constants in the norm equivalence $\bbX_{DB}^\mb{p} = \bbX_D^\mb{p}$ are uniform in $B$.
	We know that this assumption holds for certain $\mb{p}$ in certain cases; we return to this point in Remark \ref{rmk:pertn}.
\end{rmk}

\begin{proof}[Proof of Theorem \ref{thm:wp-pert}]
	By Theorem \ref{thm:WP-char} it suffices to show that
	\begin{equation*}
		\map{N_{\mb{X},DB,\bullet}^\mb{p}}{\mb{X}_{DB}^{\mb{p},+}}{\mb{X}_\bullet^\mb{p}}
	\end{equation*}
	is an isomorphism, where $\bullet$ denotes either $\perp$ or $\parallel$ depending on the boundary value problem in question.
	Recall that $N_{\mb{X},DB,\bullet}^\mb{p}$ is the restriction of $\map{N_\bullet}{\mb{X}_D^\mb{p}}{\mb{X}^\mb{p}_\bullet}$ to $\mb{X}_{DB}^{\mb{p},+} \subset \mb{X}_D^\mb{p}$, which is defined provided that $B$ is sufficiently close to $B_0$ thanks to our assumptions.
	Furthermore, since $DB$ has bounded $H^\infty$ functional calculus on $\mb{X}^\mb{p}_D$ for all $B$ sufficiently close to $B_0$, uniformly in $B$ (using $\mb{p} \in I(\mb{X},DB)$ and Proposition \ref{prop:comp-fc}, along with the assumptions), we get
	\begin{equation}\label{eqn:lip}
		\nm{\gc^+(DB) - \gc^+(DB_0)}_{\mc{L}(\mb{X}^\mb{p}_D)} \lesssim \nm{B - B_0}_\infty 
	\end{equation}
	by the results of Axelsson, Keith, and McIntosh \cite[\textsection 6]{AKM06}.\footnote{More precisely, this gives the corresponding estimate on $\overline{\mc{R}(D)}$. Proposition \ref{prop:comp-fc} then extends the estimate to $\mb{X}^\mb{p}_D$.}
	
	Write 
	\begin{equation*}
		N_{\mb{X},DB,\bullet}^\mb{p} = N_\bullet \gc^+(DB_0) \gc^+(DB) + N_\bullet (\gc^+(DB)- \gc^+(DB_0)).
	\end{equation*}
	Note that
	\begin{align*}
		\gc^+(DB_0) \gc^+(DB)
		&= (\gc^+(DB_0) - \gc^+(DB) + \gc^+(DB))\gc^+(DB) \\
		&= I - (\gc^+(DB) - \gc^+(DB_0))
	\end{align*}
	on $\mb{X}^{\mb{p},+}_{DB}$, so by the estimate \eqref{eqn:lip}, $\map{\gc^+(DB_0) \gc^+(DB)}{\mb{X}_{DB}^{\mb{p},+}}{\mb{X}_{DB_0}^{\mb{p},+}}$ is invertible by a Neumann series if $\nm{B - B_0}_\infty$ is sufficiently small.
	Since $\map{N_{\mb{X},DB_0,\bullet}^\mb{p}}{\mb{X}^{\mb{p},+}_{DB_0}}{\mb{X}^\mb{p}_\bullet}$ is an isomorphism (by Theorem \ref{thm:WP-char} and the assumptions), the composition
	\begin{equation*}
		\map{M := N_\bullet \gc^+(DB_0) \gc^+(DB) = N_{\mb{X},DB_0,\bullet}^\mb{p} \gc^+(DB_0) \gc^+(DB)}{\mb{X}^{\mb{p},+}_{DB}}{\mb{X}^\mb{p}_\bullet}
	\end{equation*}
	is also an isomorphism.
	One can then write
	\begin{equation*}
		N_{\mb{X},DB,\bullet}^\mb{p} = M \left(I + M^{-1} N_\bullet (\gc^+(DB)- \gc^+(DB_0)) \right)
	\end{equation*}
	and note that the second factor is also invertible by a Neumann series if $\nm{B - B_0}_\infty$ is sufficiently small.
	Hence $N_{\mb{X},DB,\bullet}^{\mb{p}}$ is an isomorphism, and we are done.
\end{proof}

\begin{rmk}
	The proof above shows that the inverse of $N_{\mb{X},DB,\bullet}^\mb{p}$ is given by an expression coming from multiple Neumann series involving the operators $(N_{\mb{X},DB_0,\bullet}^\mb{p})^{-1}$ and $\gc^+(DB) - \gc^+(DB_0)$.
	Therefore, if the assumptions of Theorem \ref{thm:wp-pert} are satisfied for two exponents $\mb{p},\mb{q}$, and if $(P_\mb{X})_{A_0}^\mb{p}$ and $(P_\mb{X})_{A_0}^\mb{q}$ are mutually well-posed, it follows that $(P_\mb{X})_A^\mb{p}$ and $(P_\mb{X})_A^{\mb{q}}$ remain mutually well-posed for $\nm{A - A_0}_\infty$ sufficiently small.
\end{rmk}

\begin{rmk}\label{rmk:pertn}
	There are three situations in which we can guarantee that the hypotheses of Theorem \ref{thm:wp-pert} are satisfied.
	
	First, for any coefficients $A$, the hypotheses are satisfied provided that $\mb{p} \in I_\text{min}$ (see Theorem \ref{thm:identn-etc}; in the Besov space case we also require $\gq(\mb{p}) \in (-1,0)$).
	The only thing which is not obvious here is uniformity of implicit constants: when $\gq(\mb{p}) = 0$ this is contained in the proof of \cite[Proposition 7.1]{AS16}, and the general case follows by $\heartsuit$-duality and interpolation (as in the derivation of the region $I_\text{min}$).
	
	Second: whenever $n=1$ the hypotheses are satisfied for all coefficients and for all $\mb{p} \in I_{\text{max}}$ (with $\gq(\mb{p}) \in (-1,0)$ in the Besov space case): this follows from \cite[Proposition 3.11 and Theorem 5.1]{AS16}, $\heartsuit$-duality, and interpolation.
		
	Finally, write $A$ in the transversal/tangential decomposition as in \eqref{eqn:A}, and suppose that $A_{\parallel \parallel}^*$ satisfies the De Giorgi--Nash--Moser condition\index{De Giorgi--Nash--Moser condition} of exponent $\ga$ in dimension $n$ (see \eqref{eqn:DGNM}).
	It is shown in \cite[\textsection 13]{AS16} that the hypotheses of Theorem \ref{thm:wp-pert} are satisfied for the exponent $(p,0)$ whenever $p \in (n/(n+\ga),p_+(DB))$, where $p_+(DB) > 2$.
	If in addition $A_{\parallel \parallel}$ satisfies the De Giorgi--Nash--Moser condition of exponent $\ga$, then it follows by $\heartsuit$-duality and interpolation that the hypotheses of Theorem \ref{thm:wp-pert} hold for all $\mb{p}$ in the region pictured in Figure \ref{fig:realrange2} (in the Besov space case, we again require $\gq(\mb{p}) \in (-1,0)$).
	
\end{rmk}

\section{The method of layer potentials}\label{sec:lps}\index{layer potential}
 
We conclude the monograph by explaining how the first-order approach relates to the method of layer potentials.
In solving boundary value problems for $L_{A}$, it has been a standing question whether solutions constructed by different methods agree, and in particular whether solutions are always given by layer potentials when their defining integrals exist. 
A consequence of Ros\'en's identification of the layer potentials in terms of Cauchy operators for $DB$ having assumed the De Giorgi--Nash--Moser condition on $A$ and $A^*$,\index{De Giorgi--Nash--Moser condition} is that solutions constructed via Cauchy operators coincide with solutions constructed by the method of layer potentials, and satisfy a layer potential representation
\begin{equation}\label{eqn:lprepn}
u(t,x)=\mc{S}_{t}(\partial_{\gn_A} u|_{t=0})(x) -  \mc{D}_{t}(u|_{t=0})(x)
\end{equation}
in some appropriate sense (modulo constants) \cite{aR13}. 
At the time, Ros\'en's identification applied when $X^{\mb{p}}= L^2_{\theta}$ for  $-1\le \theta \le 0$.
As a consequence of our results, we obtain a third representation theorem for exponents in the classification region.
We do not yet know what happens outside the classification region. 

\begin{thm}[Layer potential representation]\label{thm:layer}\index{elliptic equation!representation of solutions!by layer potentials} For $\mb{p} \in J(\mb{X},DB)$, any solution $u$ of $L_{A}u = 0$ with the interior control $\nabla u \in \wtd{X^{\mb{p}}}$ is given by the layer potential representation \eqref{eqn:lprepn} modulo constants. 
\end{thm}

We shall prove this theorem (and more) later in the section. In our general setting the operators in \eqref{eqn:lprepn} are generalised layer potential operators which we define below.
When $A$ and $A^*$ satisfy the De Giorgi--Nash--Moser condition, these are the more familiar integral operators that we review next.
We will also explain what we mean by `the method of layer potentials'.
This method has a long history in the case of the Laplace equation, or of equations with smooth coefficients.
In the case of elliptic equations $L_{A}u=0$ with non-smooth coefficients, solvability by means of layer potentials was first successfully developed in \cite{AAAHK11}.
See the references therein for historical background.

Suppose, for the moment, that $A$ and $A^*$ both satisfy the De Giorgi--Nash--Moser condition \eqref{eqn:DGNM} of some exponent.
Then for all $(t,x) \in \bbR^{1+n}$ there exists a fundamental solution $\gG_{(t,x)}$ for $L_{A^*}$ in $\bbR^{1+n}$ with pole at $(t,x)$.\footnote{Fundamental solutions were constructed in dimension $n+1 \geq 3$ by Hofmann and Kim \cite{HK07} at least for systems having pointwise accretivity bounds, and in dimension $n+1 \geq 2$ and accretivity in our sense  by Ros\'en \cite{aR13}.}
The fundamental solution $\tgG_{(t,x)}= (\tgG_{(t,x)}^{\, \beta, \alpha})$ is a $\bbC^{m\times m}$-valued locally integrable function on $\bbR^{1+n}$ satisfying
\begin{equation*}
	\dv A^* \nabla \tgG_{(t,x)} = \gd_{(t,x)}\mb{1} \qquad \text{in $\bbR^{1+n}$}
\end{equation*}
in the usual weak sense, where $\gd_{(t,x)}$ is the Dirac mass at $(t,x)$ and $\mb{1}$ is the identity matrix.
The assumption on $A^*$ guarantees existence and uniqueness of fundamental solutions, with various decay estimates.
A formal application of Green's formula tells us that
\begin{equation*}
u^\alpha(t,x)= \sum_{\beta=1}^m\int_{\bbR^n} \overline{{\tgG_{(t,x)}^{\, \beta, \alpha}(0,y)}} \, \partial_{\gn_A} u^\beta(0,y) dy -   \sum_{\beta=1}^m
\int_{\bbR^n} \overline{ \partial_{\gn_{A^*}} \tgG_{(t,x)}^{\, \beta, \alpha}(0,y)} \,  u^\beta (0,y)  \, dy
\end{equation*}
when $\alpha=1,\ldots,m$,
with our convention of the normal vector pointing in the $t$-direction (that is, inward for the upper half-space). 

For a (reasonable) function $\map{f}{\bbR^n}{\bbC^m}$ and for $(t,x) \in \bbR^{1+n}$, with $t\ne 0$, one is led to define the \emph{double layer potential} $\mc{D}_t f$ of $f$ by\index{layer potential!double}
\begin{equation*}
	\mc{D}_t f(x)^\alpha:= \int_{\bbR^n} \left( \partial_{\gn_{A^*}} \tgG^{\, \cdot , \alpha}_{(t,x)}(0,y), f(y) \right) \, dy \qquad (\alpha = 1,\ldots,m),
\end{equation*}
and the \emph{single layer potential} $\mc{S}_t f$ of $f$ by\index{layer potential!single}
\begin{equation*}
	\mc{S}_t f(x)^\alpha := \int_{\bbR^n} (\tgG^{\, \cdot, \alpha}_{(t,x)}(0,y), f(y)) \, dy  \qquad (\alpha = 1,\ldots,m).
\end{equation*}
Here,  $(z, \zeta)=\bar z \cdot \zeta$ for  $z,\zeta \in \bbC^m$.
The first question is whether these operators extend boundedly to spaces of boundary data: on $L^2(\bbR^n :  \bbC^m)$ for $\mc{D}_t$, and from $L^2(\bbR^n :  \bbC^m)$ into $L^2_{1}(\bbR^n :  \bbC^m)$ for $ \mc{S}_t$ (uniformly in $t>0$).
This was asked by Hofmann in \cite{H10}.
Another formal observation is that $ (\tgG_{(t,x)}(s,y))^*= \gG_{(s,y)}(t,x)$, where $\gG_{(s,y)}$ is the fundamental solution for $L_{A}$ with pole at $(s,y)$.
Note that $\gG_{(0,y)}$ is a solution to the equation $L_{A}u=0$ away from the pole $(0,y)$, so that the integrals `defining' $\mc{D}_t$ and $\mc{S}_t$ may be used to construct solutions.
Proving the invertibility of such operators is one way to solve boundary value problems (there are other ways; see below).
More precisely, one can solve Dirichlet problems for $L_A$ in $\bbR^{1+n}_+$ with boundary data $\gf$ by solving the double layer equation
\begin{equation*}
	\lim_{t \searrow 0} \mc{D}_t f = \gf,
\end{equation*}
and likewise one can solve Neumann problems for $L_A$ in $\bbR^{1+n}_+$ with boundary data $\gf$ by solving the single layer equation
\begin{equation*}
	\lim_{t \searrow 0} \partial_{\gn_A} \mc{S}_t f = \gf.
\end{equation*}
The corresponding solutions $u$ are then given by $u(t,x) = \mc{D}_t f(x)$ and $u(t,x) = \mc{S}_t f(x)$ respectively. 

It was shown by Ros\'en \cite{aR13} that these layer potential operators fall within the scope of the first-order framework, hence answering Hofmann's question on layer potential boundedness in the positive.  
Keeping the De Giorgi--Nash--Moser assumption on  $A$ and $A^*$, and writing $B = \hat{A}$ as usual, for all $f$ in a dense class in  $L^2(\bbR^n : \bbC^m)$ and $t \in \bbR \sm \{0\}$, Ros\'en shows that
\begin{equation}\label{eqn:dt}
	\mc{D}_t f = - \sgn(t) \left( \bbP_D e^{-t BD} \gc^{\sgn(t)}(BD)\bbP_{BD} \begin{bmatrix} f \\ 0 \end{bmatrix} \right)_{\perp}
\end{equation}
and
\begin{equation}\label{eqn:gradst}
	\nabla_A \mc{S}_t f = \sgn(t)  e^{-t DB} \gc^{\sgn(t)}(DB) \begin{bmatrix} f \\ 0 \end{bmatrix},
\end{equation}
where the vectors $\begin{bmatrix} f \\0 \end{bmatrix}$ are written with respect to the transversal/tangential splitting. 
In terms of Cauchy operators, on $\bbR^{1+n}_\pm$ this amounts to\index{layer potential!representation by Cauchy operator}
\begin{equation}\label{eqn:lps-cops}
	\mc{D} f = \mp \left( \bbP_D C_{BD}^\pm \bbP_{BD} \begin{bmatrix} f \\ 0 \end{bmatrix}\right)_\perp, \qquad \nabla_A \mc{S} f = \pm  C_{DB}^\pm  \begin{bmatrix} f \\ 0 \end{bmatrix} .
\end{equation}

Remark that vectors in the form $\begin{bmatrix} f \\0 \end{bmatrix}$ belong to $\overline{\mc{R}(D)}$ but not to $\overline{\mc{R}(BD)}$, hence the presence of the projection $\bbP_{BD}$.\footnote{This projection, and $\bbP_D$, were systematically forgotten in \cite{AS16} by the abuse of notation explained in section \ref{sec:further}, as this does not affect the definition in $L^2$. The stated results for the extensions there were correct. There is also a typo there: the four formulae (81-84) are written with $e^{+t...}$
instead of $e^{-t...}$ when $t<0$.} 

Now, we reverse the point of view.
The right hand sides of the expressions in \eqref{eqn:lps-cops} are defined for all coefficients $A$, whether or not the De Giorgi--Nash--Moser assumptions are satisfied.
Thus we may define abstract double and single layer potentials. 

We first define the abstract double layer potential. 
We need only consider $f$ in a dense class. 
For $f \in L^2(\bbR^n: \bbC^m)$, \eqref{eqn:dt} is well-defined by the functional calculus of the bisectorial operator $BD$. 
When $f$ belongs to  the classical Sobolev space $W^2_{1}(\bbR^n: \bbC^m)$, $\begin{bmatrix} f \\0 \end{bmatrix} \in \bbX^{\mb{p}+1}_{D}$  with $\mb{p}=(2,0)$ (Hardy-Sobolev and Besov spaces are the same in this case).
It follows from Proposition \ref{prop:commdiag} and Corollary \ref{cor:perpequal} that \eqref{eqn:dt} is well-defined as an element of $ (\bbX^{\mb{p}+1}_{BD} \cap \mc{D}(D))_{\perp}= (\bbX^{\mb{p}+1}_{D} \cap \mc{D}(D))_{\perp}$.
Furthermore, by the calculations in the proof of Theorem \ref{thm:secondrepresentation}, we have that $\mc{D} f$ is a weak solution of the equation $L_{A}u=0$ on each half-space. 

Next, we turn to the abstract single layer potential.
Assume $f\in L^2(\bbR^n: \bbC^m)$.
Then $h=\begin{bmatrix} f \\ 0 \end{bmatrix} \in \overline{\mc{R}(D)}$, and we can define
 \begin{equation}\label{eqn:single}
	\mc{S} f = \mp \left( D^{-1}C_{DB}^\pm \begin{bmatrix} f \\ 0 \end{bmatrix}\right)_\perp
\end{equation} 
as an element of $L^\infty(\bbR^\pm : \mb{X}^{(2,1)}_{\perp})$, where $D^{-1}$ is the inverse of the isomorphism $\map{D}{\mb{X}^{(2,1)}_{D}}{\overline{\mc{R}(D)}}$ (a nice Fourier multiplier).
By Corollary \ref{cor:similarity2}, we then have
\begin{equation*}
	C_{DB}h= DC_{BD}\bbP_{BD}D^{-1}h= D\bbP_{D}C_{BD}\bbP_{BD}D^{-1}h.
\end{equation*}
	Applying $D^{-1}$ and mapping into $\mb{X}^{(2,1)}_{D}$ yields
	\begin{equation*}
	D^{-1}C_{DB} h=\bbP_{D}C_{BD}\bbP_{BD}D^{-1}h,
	\end{equation*}
	hence
	\begin{equation*}
	(D^{-1}C_{DB} h)_{\perp}= (\bbP_{D}C_{BD}\bbP_{BD}D^{-1}h)_{\perp}= (C_{BD}\bbP_{BD}D^{-1}h)_{\perp},
	\end{equation*}
	where the last equality follows from Corollary \ref{cor:perpequal}.
	Taking the sign function into account, this shows that 
	$$\mc{S} f=  \mp (C_{BD}^\pm\bbP_{BD}D^{-1}h)_{\perp}$$
	on $\bbR^\pm$.
	It follows by the calculation of Theorem \ref{thm:secondrepresentation} that $\mc{S} f$ is also a solution to $L_{A}u=0$ in each half-space and  that $\nabla_{A}\mc{S} f$ satisfies the equality in \eqref{eqn:lps-cops}.  
	
Having defined the layer potentials abstractly on dense subspaces, the results of Section \ref{sec:similarity} yield bounded extensions of $\mc{S}$ and $\mc{D}$ on the classical smoothness spaces $\mb{X}^\mb{p}$  when $\mb{X}_{DB}^\mb{p}=\mb{X}_{D}^\mb{p}$.
More precisely, we can compute them as
\begin{equation*}
	\mc{S} f = \mp \left( D^{-1}C_{DB}^\pm \begin{bmatrix} f \\ 0 \end{bmatrix}\right)_\perp \qquad (f \in \mb{X}^\mb{p})
\end{equation*}
and
\begin{equation}\label{eqn:DLP-extended}
	\mc{D} f = \pm \left( \bbP_D C_{BD}^+ \bbP_{BD} \begin{bmatrix} f \\ 0 \end{bmatrix} \right)_\perp \qquad (f \in \mb{X}^{\mb{p}+1})
\end{equation}
using the appropriate extensions from Section \ref{sec:similarity} and \ref{sec:proj-BD}.
Note that it is necessary to use $\bbP_D$ in \eqref{eqn:DLP-extended} for arbitrary $f \in \mb{X}^{\mb{p}+1}$, although this does not matter for $f \in L^2$.
Thus for $\mb{p}\in J(\mb{X},DB)$, we get the uniform bounds
\begin{equation}\label{eqn:lp1}
	\sup_{t \neq 0} \nm{\nabla_A \mc{S}_t f}_{\mb{X}^\mb{p}} + \nm{\mc{S}_t f}_{\mb{X}^{\mb{p}+1}} \lesssim \nm{f}_{\mb{X}^\mb{p}}
\end{equation}
and
\begin{equation}\label{eqn:lp2}
	\sup_{t \neq 0} \nm{\nabla_A \mc{D}_t f}_{\mb{X}^\mb{p}} + \nm{\mc{D}_t f}_{\mb{X}^{\mb{p}+1}} \lesssim \nm{ f }_{\mb{X}^{\mb{p}+1}}.
\end{equation}
Compare these bounds with those of Barton and Mayboroda \cite[Chapter 3]{BM16}, in particular (3.11-16).
The inequalities \eqref{eqn:lp1} and \eqref{eqn:lp2} encapsulate a number of concrete inequalities; see the table in Section \ref{sec:exponent-space-table}.
 
We also obtain limits for these operators as $t \to 0^\pm$ (in $\mb{X}^\mb{p}$ or $\mb{X}^{\mb{p}+1}$ accordingly, and in the strong or the weak-star topology depending on whether $\mb{p}$ is finite).
In particular we can also recover the jump relations\index{layer potential!jump relations} with this formalism.
The equation
\begin{equation}
\label{eq:jumpst}
\nabla_{A}\mc{S}_{0+}f -\nabla_{A}\mc{S}_{0-}f= (\gc^+(DB)+ \gc^-(DB))  \begin{bmatrix} f \\ 0\end{bmatrix} =  \begin{bmatrix} f \\ 0\end{bmatrix}
\end{equation}
 encodes both the jump relation of the conormal derivative of $\mc{S}_{t}$ and the continuity of $\mc{S}_{t}$ across the boundary. 
Here we used that $\gc^+(DB)+ \gc^-(DB)=I$ on $\mb{X}^\mb{p}_{D} $, and that $\begin{bmatrix} f \\ 0\end{bmatrix} \in \mb{X}^\mb{p}_D$ when $f\in \mb{X}^\mb{p}$.
For the double layer operator, we have 
\begin{align}
\mc{D}_{0+}f -\mc{D}_{0-}f &= -\bigg(\bbP_D (\gc^+(BD)+ \gc^-(BD)) \bbP_{BD} \begin{bmatrix}  f \\ 0\end{bmatrix}\bigg)_{\perp} \label{eq:jumpdt} \\
	&= -\bigg(\bbP_D \bbP_{BD}  \begin{bmatrix} f \\ 0\end{bmatrix}\bigg)_{\perp }=  -f. \nonumber
\end{align}
This time we used that $\gc^+(BD)+ \gc^-(BD)=\bbP_{BD}$ on $\mb{X}^{\mb{p}+1}_{BD}$  and that \eqref{eq:rect} applies to $\begin{bmatrix}  f \\ 0\end{bmatrix} \in \mb{X}^{\mb{p}+1}_{D}$ when $f\in \mb{X}^{\mb{p}+1}$. 

For all exponents $\mb{p} \in J(\mb{X},DB)$, we have proven the tent/$Z$-space estimates  
\begin{equation*}
	\nm{ C_{DB}^+ h }_{X^\mb{p}} \lesssim \nm{ h }_{\bbX_D^\mb{p}}
\end{equation*}
(see Theorems \ref{thm:sgpnorm-abstract} and \ref{thm:sgpnorm-concrete} and those in \cite[\textsection 12.3]{AS16}.)
Applied to  $h=\begin{bmatrix} f \\ 0 \end{bmatrix}$ and taking extensions, these immediately yield  
\begin{equation}\label{eqn:nabS}
	\nm{ \nabla \mc{S}f }_{X^\mb{p}} \lesssim \nm{ f }_{\mb{X}^\mb{p}}
\end{equation}
and

\begin{equation}\label{eqn:nabD}
	\nm{ \nabla \mc{D}f }_{X^\mb{p}} \lesssim \nm{f}_{\mb{X}^{\mb{p}+1}}.
\end{equation}
Similar bounds for layer potentials on the lower half-space corresponding to \eqref{eqn:nabS} and \eqref{eqn:nabD} can also be derived.
Compare these results with those of Barton and Mayboroda \cite[Theorem 3.1]{BM16} concerning $\mc{S}$ and $\mc{D}$:  again, we recover their Besov space results (without the De Giorgi--Nash--Moser assumption that they use) and we obtain new Hardy--Sobolev space results.

With these results at hand, the proof of Theorem \ref{thm:layer} is easy.
We mostly reproduce calculations of \cite[Lemma 8.1]{AM14}.
By Theorem \ref{thm:firstrepresentation} and the assumption on $\mb{p}$, we have the semigroup representation for the conormal gradient $\nabla_{A}u(t,x)= e^{-tDB}\gc^+(DB)h(x)$,  with $h\in \mb{X}^{\mb{p}, +}_{DB}$ being the trace of $\nabla_{A}u$ at the boundary.
Thus $h_{\perp}=\partial_{\gn_A} u|_{t=0}$ and $h_{\parallel}= \nabla_{\parallel}u|_{t=0}$.
Write
\begin{equation*}
	e^{-tDB}\gc^+(DB)h= e^{-t DB} \gc^{+}(DB) \begin{bmatrix} h_{\perp} \\ 0 \end{bmatrix} +  
e^{-t DB} \gc^{+}(DB) \begin{bmatrix}0\\ h_{\parallel}  \end{bmatrix}
\end{equation*}
and, writing $u|_{t=0}=f$,   we know that 
$$ 
e^{-t DB} \gc^{+}(DB) \begin{bmatrix}0\\ h_{\parallel}  \end{bmatrix}= \nabla_{A}\bigg( \bbP_D e^{-tBD} \gc^+(BD)\bbP_{BD} \begin{bmatrix} f \\ 0\end{bmatrix}\bigg)_{\perp}
$$
as in the proof of Theorem \ref{thm:secondrepresentation}.
The second term is precisely $-\nabla_{A}\mc{D}_{t}f$.

Putting everything together, this gives us \eqref{eqn:lprepn}. 

Let us make some comments on the use of invertibility of layer potentials to solve  boundary value problems.
As mentioned, one can use invertibility of the operator $\mc{D}_{0+}$  for  Dirichlet problems and  that of the  operator $\partial_{\gn_A} \mc{S}_{0+}$  for  Neumann problems.
This is obtained (for example) for self-adjoint real equations ($m=1$) for $\mb{p}=(2,0)$, i.e. spaces of $L^2$ Dirichlet/Neumann data, in \cite[Theorem 1.13]{AAAHK11}, and for the lower half-space as well (looking at the $0^-$ limits).
Prior to this was the work of Verchota for Laplace's equation on Lipschitz domains \cite{Ver84}. 

Such invertibility is not necessary for well-posedness as characterised in Theorem \ref{thm:WP-char}.
But this is not the only way to prove well-posedness via layer potentials.
For regularity problems, one can use invertibility of the layer potential $\mc{S}_{0}$, as is done in \cite[Theorem 1.13]{AAAHK11} for real symmetric equations with $L^2$ regularity data, in \cite{HKMP15.2} for real equations with $L^p$ regularity data (for some $p>1$), and in \cite{AM14} for $H^p$ regularity data for some $p<1$.
For Neumann problems, one can use invertibility of $\partial_{\gn_A} \mc{D}_0$.
We give a quick sketch and refer to the proof of \cite[Theorem 1.11]{AM14} for details. 
The operator $\sgn(DB)$, written in matrix form according to the transversal/tangential decomposition, is
\begin{equation*}
	\sgn(DB) = \begin{bmatrix} \cdot & -2T \\2 \nabla_x \mc{S}_0 & \cdot \end{bmatrix}.
\end{equation*}
The operator $T$ will be discussed below.
The operator $\mc{S}_0$ is the single layer potential at $t=0$, given by
\begin{equation*}
	\mc{S}_0 f := \lim_{t \to 0} \mc{S}_t f,
\end{equation*}
recalling that \eqref{eq:jumpst} shows that the limits $t\to 0^\pm$ are the same for this operator.
For $\mb{p} \in J(\mb{X},DB)$, well-posedness of $(R_\mb{X})_A^\mb{p}$ on both half-spaces   is equivalent to invertibility of $\map{\nabla_x \mc{S}_0}{\mb{X}^\mb{p}_\perp}{\mb{X}^\mb{p}_\parallel}$ (a multiple of the bottom left component of $\sgn(DB)$), or equivalently, invertibility of the single layer potential $\mc{S}_0$ from $\mb{X}^\mb{p}$ (which is the same here as $\mb{X}_\perp^\mb{p}$) onto $\mb{X}^{\mb{p}+1}$ as in \cite[Theorem 1.11]{AM14}. 
Then, for $f \in \mb{X}^{\mb{p}+1}$, the solutions of the regularity problems with data $\nabla_{\parallel}f$ (that is, $f$ is the Dirichlet data)  on the upper and lower half-spaces are respectively given by 
\begin{equation*}
u^\pm= \mc{S}_t \mc{S}_0^{-1}f \qquad (t\in \bbR^\pm).
\end{equation*}
There is also a notion of mutual well-posedness on \emph{both} half-spaces, given by the conjunction of mutual well-posedness on each half-space separately.
As we did before, this is akin to demanding the agreement of inverses of the single layer potential for different $\mb{p}$.
We leave this to the reader.  

There are corresponding statements for  Neumann problems, using the invertibility of the top right component $-2T$ of $\sgn(DB)$. A calculation shows that 
 $T(\nabla f) = \partial_{\gn_A} \mc{D}_0 f$ (there is no jump here) and so invertibility of $T$ corresponds to invertibility of the conormal derivative of the double layer potential $\map{\partial_{\gn_A}\mc{D}_0}{\mb{X}^{\mb{p}+1}}{\mb{X}^\mb{p}}$ (again the latter space should be thought of as the transversal component).
The solutions to the Neumann problems in the upper and lower half-spaces with Neumann data $g$  are given respectively by 
\begin{equation*}
u^{\pm}= \mc{D}_{t}(\partial_{\nu_{A}}\mc{D}_{0})^{-1}g \qquad (t\in \bbR^\pm).
\end{equation*}
We remark that for the block diagonal case (see Section \ref{subsec:block}) both $\mc{S}_0$ and $\partial_{\gn_A} \mc{D}_0$ are invertible in the range where well-posedness was established: this is because $\sgn(DB)$ is an involution with zero diagonal entries (and this is because $A$ is diagonal if and only if $\hat{A}$ is diagonal, and then in this case $\sgn(DB) = DB[DB]^{-1}$ has zero diagonal entries).

For the triangular block form, it is either the invertibility of $\partial_{\gn_A} \mc{D}_0$ or that of $\mc{S}_0$ that holds, depending on the boundary value problem and on the value of $\mb{p}$ in the range of well-posedness. The results of \cite{AMM13} show this explicitly for the Regularity and Neumann problems with $L^2$ data.

We mention without proof that our results on stability in the coefficients (see Section \ref{sec:wp-pert}) allow one to perturb invertibility of boundary layer operators.
This was already observed in \cite{AAAHK11} in the following form: invertibility of \emph{a collection of} boundary layer operators is stable under perturbation.
Here, we observe that invertibility is preserved for each individual boundary layer operator, which is a little stronger.

In conclusion, we see that there are many different ways to obtain well-posedness (compatible or not) for  boundary value problems associated with $L_{A}u=0$ under our current assumptions on $A$.
The first order method essentially covers all results in the literature except for that in \cite{HKMP15}, concerning a Dirichlet problem with data in $L^p$ for large $p$, which uses harmonic measure methods for real equations.
We still lack the understanding to establish well-posedness for more general coefficients, with any currently available approach.

%\appendix
%    Include appendix "chapters" here.
%\include{}

\backmatter
\bibliographystyle{amsplain}
\bibliography{amenta-auscher-bib}

%    See note above about multiple indexes.
\printindex

\end{document}